%% file: main.tex
\documentclass{article}
\usepackage[utf8]{inputenc}
\usepackage[T1]{fontenc}
\usepackage{lmodern,textcomp}
\usepackage{subcaption}
\usepackage{booktabs}
\usepackage{amsthm,amsmath,amsfonts,amssymb,bm,dsfont,enumitem,graphicx,mathtools}
\usepackage{algorithm,algorithmic}
\usepackage[usenames]{color}
\usepackage{geometry}
\usepackage[colorlinks,
            linktoc=page, 
            linkcolor=red,
            anchorcolor=blue,
            citecolor=blue
            ]{hyperref}
\usepackage{mystyle}

\usepackage{natbib}

\input{notations}

\def\sz{c_z}
\def\su{\overline{c}}
\def\sl{\underline{c}}
\def\suD{\overline{c}_{D^*}}
\def\slD{\underline{c}_{D^*}}

\title{Covariates-Adjusted Mixed-Membership Estimation: A Novel Network Model with Optimal Guarantees \footnote{The research is in part supported by the NSF grants DMS-2412029, DMS-2210833, and DMS-2053832.}}
\author{Jianqing Fan\thanks{Department of Operations Research and Financial Engineering, Princeton University; 
\texttt{\{jqfan, jg5300, jikaih\}@princeton.edu}}
\qquad Jiawei Ge\footnotemark[2]
\qquad Jikai Hou\footnotemark[2]
}
\date{}

\begin{document}

\maketitle

\input{abstract}

\tableofcontents

\input{introduction}

\input{problem_setup}

\input{main_results}
\input{proof_idea}

\input{simulations}
\input{real_data}

\newpage
\bibliography{ref}
\bibliographystyle{apalike}

\newpage
\appendix
\allowdisplaybreaks
\input{appendix_preliminaries}
\input{apendix_local_geometry}

\input{appendix_nonconvex_iterates}

\input{appendix_nonconvex_debias}
\input{pf_geometry}
\input{pf_nonconvex_iterates}
\input{pf_nonconvex_debias}
\input{pf_main_results}

\input{technical_lemmas}

\end{document}

%% file: notations.tex
\def\argmin{{\arg\min}}

\def\bx{\mathbf{x}}
\def\by{\mathbf{y}}

\def\bbE{\mathbb{E}}

\def\bbP{\mathbb{P}}

\def\bbR{\mathbb{R}}

\def\cA{\mathcal{A}}

\def\cO{\mathcal{O}}
\def\cP{\mathcal{P}}

\def\cT{\mathcal{T}}

\def\diag{{\sf diag}}

\def\exp{{\sf exp}}

\def\eps{{\epsilon}}

\def\Tr{{\sf Tr}}

\def\vec{{\sf vec}}
\def\aug{{\sf aug}}
\def\diff{{\sf diff}}

%% file: abstract.tex
\begin{abstract}
This paper addresses the problem of mixed-membership estimation in networks, where the goal is to efficiently estimate the latent mixed-membership structure from the observed network.
Recognizing the widespread availability and valuable information carried by node covariates, we propose a novel network model that incorporates both community information, as represented by the Degree-Corrected Mixed Membership (DCMM) model, and node covariate similarities to determine connections.

We investigate the regularized maximum likelihood estimation (MLE) for this model and demonstrate that our approach achieves optimal estimation accuracy for both the similarity matrix and the mixed-membership, in terms of both the Frobenius norm and the entrywise loss. Since directly analyzing the original convex optimization problem is intractable, we employ nonconvex optimization to facilitate the analysis.
A key contribution of our work is identifying a crucial assumption that bridges the gap between convex and nonconvex solutions, enabling the transfer of statistical guarantees from the nonconvex approach to its convex counterpart. Importantly, our analysis extends beyond the MLE loss and the mean squared error (MSE) used in matrix completion problems, generalizing to all the convex loss functions.
Consequently, our analysis techniques extend to a broader set of applications, including ranking problems based on pairwise comparisons. 

Finally, simulation experiments validate our theoretical findings, and real-world data analyses confirm the practical relevance of our model.

\end{abstract}


%% file: introduction.tex
\section{Introduction}\label{sec:intro}


Network data plays a crucial role across various fields, ranging from finance \citep{fan2022simple, bhattacharya2023inferences} to social science \citep{adamic2005political,ji2022co}, where understanding its latent structure is essential for effective analysis and application.
A prominent model in this context is the Degree-Corrected Mixed Membership (DCMM) model, which models the structure of the network within the community regime.
However, in many practical scenarios, the connections between nodes are often influenced by more than just the community structure; they are also affected by specific covariates information associated with each node.
For example, on a professional networking platform, the connections between individuals are determined by diverse factors like their industry sector, educational background, and skill sets.
When observing whether two individuals are connected, covariates are often collected and significantly influence the network structure. Given the importance and availability of covariates, researchers have modified classical models to integrate this information, as seen in works like \cite{yan2018statistical, huang2018pairwise, ma2020universal}.

This work focuses on community detection while incorporating adjustments for these covariates.
Specifically, we propose a generative model for the entries of the observed adjacency matrix $A$. Given the observed covariates $\{z_i\}^n_{i=1}$ for $n$ individuals, the Bernoulli random variables $\{A_{ij}=A_{ji}:1\leq i<j\leq n\}$ are assumed to be mutually independent, and for each pair $i<j$:
\begin{align*}
\bbP\left(A_{ij}=1\mid z_i,z_j\right)=\frac{e^{z^{\top}_i H^{*} z_j+{\Gamma^{*}_{ij}}}}{1+e^{z^{\top}_i H^{*} z_j+{\Gamma^{*}_{ij}}}}.
\end{align*}
Here, the symmetric matrix $H^* \in \bbR^{p \times p}$ moderates the influence of covariates on edge formation, while $\Gamma^* = \Theta^* \Pi^* W^* \Pi^{\top} \Theta^{*}$ represents the component as in the DCMM model \citep{jin2017estimating}. $\Theta^* \in \bbR^{n \times n}$ captures degree heterogeneity, $\Pi^* \in \bbR^{n \times r}$ is the mixed membership profile matrix, and $W^* \in \bbR^{r \times r}$ reflects the connection probabilities between communities. The key insight is that both latent communities and covariates jointly influence network connections.
Unlike \cite{huang2018pairwise}, which assumes that $A_{ij}$ follows a Poisson distribution—allowing the use of spectral methods—our model deals with binary $A_{ij}$, more reflective of real-world connections.
The challenge, however, is that our model makes spectral methods inapplicable, requiring alternative approaches to handle the network structure effectively.
Our contributions are threefold:
\begin{enumerate}
\item From a methodological perspective, we introduce the Covariates-Adjusted Mixed Membership (CAMM) model. To estimate the model parameters, we propose a constrained regularized maximum likelihood estimator (MLE), which takes the following form: 
\begin{align*}
    \min_{H,\Gamma} \quad &\sum_{i \neq j} \left( \log(1 + e^{P_{ij}}) - A_{ij} P_{ij} \right) + \lambda \|\Gamma\|_{*} \\
    \text{ s.t. } \quad &P_{ij} = z_i^\top H z_j + \Gamma_{ij}, \notag\\
     &\cP_Z\Gamma=0,\quad \Gamma\cP_Z = 0,
\end{align*}
where $Z:=[z_1,\ldots,z_n]^{\top}\in\bbR^{n\times p}$ and $\cP_{Z}:=Z(Z^{\top}Z)^{-1}Z^{\top}$ represents the projection onto the column space of $Z$.
The objective function includes a standard logistic loss and a regularization term given by the nuclear norm, which acts as a convex surrogate for the rank function to capture the low-rank structure of $\Gamma$. The constraints are necessary to ensure the identifiability of the model.
This formulation results in a convex optimization problem, allowing for efficient solution methods.
By incorporating covariate adjustments, our model provides a principled approach to community detection in networks, making it a natural extension of classical models to handle real-world complexities.
Once the convex optimization problem is solved with solution $(\hat{H}_c, \hat{\Gamma}_c)$, we further apply the Mixed-SCORE algorithm \citep{jin2017estimating} to reconstruct the community memberships based on $\hat{\Gamma}_c$.

\item From a theoretical perspective, our contributions are: (1) We establish optimal statistical guarantees for the solutions of the convex optimization problem, specifically, $\|\hat{H}_c-H^{*}\|_F\lesssim {1/\sqrt{n}},
\|\hat{\Gamma}_c-\Gamma^{*}\|_{F}\lesssim \sqrt{n},
\|\hat{\Gamma}_c-\Gamma^{*}\|_{\infty}\lesssim {1/\sqrt{n}}$.
(2) We also provide optimal statistical guarantees for the reconstructed membership matrix $\hat{\Pi}_c$, specifically, $\|\hat{\Pi}_c-\Pi^{*}\|_{2,\infty}\lesssim {1/\sqrt{n}}$.
Our analysis of the convex optimization problem involves two key components: (i) analyzing the nonconvex gradient descent, and (ii) demonstrating the equivalence between the convex and nonconvex solutions. Due to the complexity of the logistic loss function—whose first derivative is not linear in the variables, unlike the mean square error commonly used in matrix completion problems \citep{chen2020noisy}—we employ the debiased estimator technique in the latter part of our analysis. 
We highlight that this approach can be generalized to all convex loss functions, making it potentially useful in a variety of contexts.
Furthermore, for the membership reconstruction, our analysis goes beyond the traditional sub-Gaussian noise assumption that is prevalent in the literature (e.g., \cite{jin2017estimating, bhattacharya2023inferences}) by incorporating results on the estimation of $\hat{\Gamma}_c$, which is critical for handling the more complex noise structures in our setting.

\item From an application perspective, we demonstrate through simulation studies that the estimation errors of the model parameters with respect to $n$ align perfectly with our optimal statistical guarantees, thereby verifying our theoretical results. Additionally, we validate the practical utility of our model by applying it to an S\&P 500 dataset, further showcasing its effectiveness in capturing complex network structures. We include $6$ popular covariates in our model and find that they explain a substantial part of the network. Furthermore, the recovered membership structure is highly consistent with the company sectors, and these results deepen our understanding of the underlying structure of the S\&P 500 companies.

\end{enumerate}

\subsection{Related work}
In this work, we focus on model-based community detection methods, where a probabilistic model that encodes the community structure is applied to effectively analyze the network data.
Widely recognized models in this field include the stochastic block model \citep{holland1983stochastic}, latent space models \citep{hoff2002latent,gao2020community}, mixture model \citep{newman2007mixture}, degree-corrected stochastic block model \citep{karrer2011stochastic}, and hierarchical block model \citep{peixoto2014hierarchical}.
However, these models do not account for the influence of covariates on the nodes' connections.
Recently, researchers have started to modify the classical models to incorporate covariates information.
Based on the relationship between covariates, community membership, and network structure, these modified models are generally divided into two categories: \emph{covariates-adjusted} models and \emph{covariates-assisted} models.

\paragraph{Covariates-adjusted network models}
Our work focuses on covariate-adjusted network models, where both covariates and community membership jointly influence the network structure. 
A concrete example is a citation network, where citations between papers depend on their research topics (community membership), and the likelihood of citation increases if the authors share similar attributes, such as working at the same institution or having similar academic backgrounds. 
Adjusting for these covariates is crucial for accurately recovering the true community memberships.
For covariates-adjusted network models, \cite{yan2018statistical} studied a directed network model, which captured the link homophily via incorporating covariates.
But their work did not take the potential community structures into consideration.
\cite{huang2018pairwise} introduced a pair-wise covariates-adjusted stochastic block model. They studied the MLE for the coefficients of the covariates and investigated both likelihood and spectral approaches for community detection.
\cite{ma2020universal} incorporated covariates information into latent space models, and presented two universal fitting algorithms: one based on nuclear norm penalization and the other based on projected gradient descent.
\cite{mu2022spectral} extended the generalized random dot product graph (GRDPG) to include vertex covariates, and conducted a comparative analysis of two model-based spectral algorithms: one utilizing only the adjacency matrix, and the other incorporating both the adjacency matrix and vertex covariates.
In contrast, our goal is to investigate a variant of the DCMM model that includes covariates adjustment into the network modeling.

\paragraph{Covariates-assisted network models}
Covariates-assisted network models refer to models where both the network structure and covariates incorporate information about community membership. A typical example is a social media interaction network. User interactions—such as likes, and comments—often depend on their shared interests (i.e., belonging to the same community). At the same time, the type of content users post or engage with (e.g., workout routines, photo-editing tips, or game reviews) is also driven by these shared interests. Integrating both covariates and network structure information can better reveal the underlying community memberships. Examples of work in this area include \cite{newman2007mixture,yan2021covariate,abbe2022, xu2023covariate,hu2024network}. However, covariates-assisted network models are not the primary focus of this paper.

\paragraph{Notation} 
We use $\|A\|$ to denote the spectral norm of matrix $A$, and $\|A\|_{\infty}$ for the entrywise $\ell_{\infty}$ norm.
Let $A_{m,\cdot}$ and $A_{\cdot,m}$ represent the $m$-th row and $m$-th column of matrix $A$, respectively. 
The Hadamard product (element-wise product) between two matrices $A$ and $B$ is denoted by $A \odot B$. 
We use $\sigma_{\max}(A)$ and $\sigma_{\min}(A)$ to denote the largest and smallest non-zero singular values of $A$, respectively, and correspondingly, $\lambda_{\max}(A)$ and $\lambda_{\min}(A)$ to denote the largest and smallest non-zero eigenvalues of $A$. 
The pseudoinverse of $A$ is denoted by $A^{\dagger}$.
The vectorization of a matrix $A:=[a_1,\ldots, a_m]$ is denoted by $\vec(A)$, which is obtained by stacking the rows of the matrix $A$ on top of one another, i.e., $\vec(A):=[a_1^{\top},\ldots, a_m^{\top}]^{\top}$.
For matrices $A_1,\ldots, A_{k}$, which may have different dimensions, we define
\begin{align*}
\vec\begin{bmatrix}
A_1\\
\vdots\\
A_k
\end{bmatrix}=
\begin{bmatrix}
\vec(A_1)\\
\vdots\\
\vec(A_k)
\end{bmatrix}.
\end{align*}
Finally, $f(n) \lesssim g(n)$ or $f(n) = O(g(n))$ means $\frac{|f(n)|}{|g(n)|} \leq C$ for some constant $C > 0$ when $n$ is sufficiently large; $f(n) \gtrsim g(n)$ means $\frac{|f(n)|}{|g(n)|} \geq C$ for some constant $C > 0$ when $n$ is sufficiently large; and $f(n) \asymp g(n)$ if and only if $f(n) \lesssim g(n)$ and $f(n) \gtrsim g(n)$.

%% file: problem_setup.tex
\section{Problem Setup}\label{sec:setup}
We consider an undirected graph with $n$ nodes and $r$ communities. 
The edge information is incorporated into a symmetric adjacency matrix $A=(A_{ij})\in\{0,1\}^{n\times n}$, namely $A_{ij}=1$ if there exists an edge between nodes $i$ and $j$ and $A_{ij}=0$ otherwise. 
We assume each node $i$ is associated with a degree heterogeneity parameter $\theta^{*}_i>0$, a community membership probability vector $\pi^{*}_i=(\pi^{*}_i(1),\ldots,\pi^{*}_i(r))^{\top}\in\bbR^r$, and a covariates vector $z_i\in\bbR^p$. 
Conditional on $\{z_i\}^n_{i=1}$, the Bernoulli random variables $\{A_{ij}=A_{ji}:1\leq i<j\leq n\}$ are assumed to be mutually independent, and for each pair $i<j$:
\begin{align}\label{model}
\bbP\left(A_{ij}=1\mid z_i,z_j\right)=\frac{\exp({z^{\top}_i H^{*} z_j+{\Gamma^{*}_{ij}}})}{1+\exp({z^{\top}_i H^{*} z_j+{\Gamma^{*}_{ij}}})}.
\end{align}
Here $\Gamma^{*}_{ij}$ represents the $(i,j)$ entry of $\Gamma^{*}:=\Theta^{*}\Pi^{*} W^{*}\Pi^{*\top}\Theta^{*}$ as in the DCMM model, where $\Theta^{*}:=\diag(\theta^{*}_1,\ldots, \theta^{*}_n)\in \bbR^{n\times n}$, $\Pi^{*}:=(\pi^{*}_1,\ldots,\pi^{*}_n)^{\top}\in\bbR^{n\times r}$ represents the mixed membership profile matrix, and $W^{*}\in\bbR^{r\times r}$ is a matrix capturing the relative connection probability between communities. 
Unlike the standard DCMM model, we do not assume $W^*$ to be nonnegative. This flexibility allows our model to capture both dense and sparse networks more effectively.  
We employ a symmetric matrix $H^{*}\in\bbR^{p\times p}$ to moderate how the covariates affect the edge formation. 
Only the adjacency matrix $A$ and the covariates $\{z_i\}^n_{i=1}$ are observed.

We impose the following identifiability condition for our model \eqref{model}.

\begin{assumption}\label{identifiability}
Let $ Z := \begin{bmatrix} z_1, \ldots, z_n \end{bmatrix}^{\top} \in \mathbb{R}^{n \times p} $. We assume that $\mathcal{P}_Z \Gamma^{*} = 0$, where $\mathcal{P}_Z := Z (Z^{\top} Z)^{-1} Z^{\top}$ denotes the projection onto the column space of $Z$. Additionally, we assume: (1) $|W^{*}_{i, i}| = 1$ for all $i\in [r]$, and (2) each community \( 1 \leq \ell \leq r \) contains at least one pure node, i.e., there exists some \( i \in [n] \) such that \( \pi^{*}_i(\ell) = 1 \).
\end{assumption}
The orthogonality between the column space of $Z$ and $\Gamma^*$ ensures the identifiability of the model parameters $(H^*, \Gamma^*)$. The remaining assumptions guarantee the identifiability of the DCMM model, as demonstrated in Proposition \ref{identidiablecondition}.


Due to the low-rank structure of $\Gamma^{*}$ and the constraint $\cP_Z\Gamma^{*}=0$, we consider the following constrained convex optimization problem:
\begin{align}\label{prob:cv}
    \min_{H,\Gamma} \quad &\sum_{i \neq j} \left( \log(1 + e^{P_{ij}}) - A_{ij} P_{ij} \right) + \lambda \|\Gamma\|_{*} \\
    \text{ s.t. } \quad &P_{ij} = z_i^\top H z_j + \Gamma_{ij}, \notag\\
     &\cP_Z\Gamma=0,\quad \Gamma\cP_Z = 0,\notag
\end{align}
where $\lambda>0$ is some regularization parameter and $\|\Gamma\|_{*}$ denotes the nuclear norm of $\Gamma$, enforcing the low-rank structure.
Let $(\hat{H}_c, \hat{\Gamma}_c)$ be the solution returned by \eqref{prob:cv}.
The primary goal of this paper is to establish optimal statistical guarantees for this obtained solution and subsequently reconstruct the mixed membership structure based on $\hat{\Gamma}_c$.

%% file: main_results.tex
\section{Main Results}\label{sec:results}

In this section, we present the key theoretical results of the paper, starting with the necessary assumptions in Section \ref{sec:assumption}, followed by the estimation guarantees for the proposed model in Section \ref{sec:est_results}, and concluding with the membership reconstruction results in Section \ref{sec:membership_results}.

We begin by introducing some additional notations that will be used throughout the following sections. Let the singular value decomposition (SVD) of $\Gamma^{*}$ be given by $\Gamma^{*} = U^{*} \Sigma^{*} V^{* \top}$, where $U^{*}, V^{*} \in \mathbb{R}^{n \times r}$. We denote the largest and smallest non-zero singular values of $\Gamma^{*}$ by $\sigma_{\max}$ and $\sigma_{\min}$, respectively, and define the condition number of $\Gamma^{*}$ as $\kappa := \sigma_{\max} / \sigma_{\min}$. 
Next, we define $X^{*} = U^{*} (\Sigma^{*})^{1/2} \in \mathbb{R}^{n \times r}$ and $Y^{*} = V^{*} (\Sigma^{*})^{1/2} \in \mathbb{R}^{n \times r}$, which ensures that $X^{* \top} X^{*} = Y^{* \top} Y^{*}$.

\subsection{Assumptions}\label{sec:assumption}

Before proceeding, we introduce several key model assumptions that are crucial for the development of our theoretical results. These assumptions relate to the structure of the covariates, the incoherence properties of the latent membership matrix $\Gamma^*$, and the characteristics of the Hessian matrix in the corresponding nonconvex optimization problem. These conditions form the basis for establishing the statistical guarantees presented in the following sections.

\begin{assumption}[Scale Assumption]\label{assumption:scales}
There exists constants $\sz$ and $c_P$ such that the following holds:
\begin{align*}
\max_{1\leq i\leq n}\|z_i\|_2\leq \sqrt{\sz},\quad
\max_{1\leq i,j\leq n}|P_{ij}^*|\leq c_P,
\end{align*}
where $P_{ij}^*:=z^{\top}_i H^{*} z_j+{\Gamma^{*}_{ij}}$.
\end{assumption}
Assumption \ref{assumption:scales} ensures that the interaction term \(P^*_{ij}\) stays within a controlled range, preventing the edge probabilities from becoming too close to either zero or one, which could lead to an ill-posed problem.


\begin{assumption}\label{assumption:eigenvalues}
We assume 
$Z^{\top}Z$ is full rank and 
there exists some constants $\su$ and $\sl$ such that
\begin{align*}
\sqrt{\sl }n\leq
\lambda_{\min}\left(Z^{\top}Z
\right)
\leq 
\lambda_{\max}\left(Z^{\top}Z
\right)
\leq \sqrt{\su }n.
\end{align*}
And, without loss of generality, we assume $\sl\leq 1 \leq \su$. This can always be achieved by rescale $\{z_i\}_{1\leq i\leq n}$ and adjust $c_z$ correspondingly.
\end{assumption}

Assumption \ref{assumption:eigenvalues} ensures that the covariance structure of the covariates contains sufficient information and prevents the covariates from collapsing into a lower-dimensional subspace, which would otherwise result in information loss and inaccurate estimation of $H^*$.
To recover the low-rank matrix $\Gamma^*$, we impose the commonly used incoherence assumption; see \cite{chen2020noisy} for an example.

\begin{assumption}[Incoherent]\label{assumption:incoherent}
We assume $\Gamma^*$ is $\mu$-incoherent, that is to say
\begin{align*}
\|U^*\|_{2,\infty}\leq \sqrt{\frac{\mu}n}\|U^*\|_F=\sqrt{\frac{\mu r}{n}},\quad
\|V^*\|_{2,\infty}\leq \sqrt{\frac{\mu}n}\|V^*\|_F=\sqrt{\frac{\mu r}{n}}.
\end{align*}
\end{assumption}

Our theoretical results leverage nonconvex optimization analysis, which will be discussed in Section \ref{sec:nonconvex_prob}. As an analog to Assumptions \ref{assumption:eigenvalues} and \ref{assumption:incoherent}, the following assumptions ensure the nonconvex optimization is well-behaved. We denote by $\cP^{\perp}_Z:=I_n-Z(Z^{\top}Z)^{-1}Z^{\top}\in\bbR^{n\times n}$ and 
$
\cP:=\begin{bmatrix}
I_{p^2} & &\\
 &  \cP^{\perp}_Z \otimes I_r& \\
 & & \cP^{\perp}_Z \otimes I_r
\end{bmatrix}\in\bbR^{(p^2+2nr)\times (p^2+2nr)}.
$
Consider
\begin{align*}
{D}^*:=
\sum_{i\neq j}
\frac{e^{{P}^*_{ij}}}{(1+e^{{P}^*_{ij}})^2}
\left(\vec\begin{bmatrix}
z_i z_j^{\top}\\
\frac1ne_ie_j^\top{Y}^*\\
\frac1ne_je_i^\top{X}^*
\end{bmatrix}\right)\left(\vec\begin{bmatrix}
z_i z_j^{\top}\\
\frac1ne_ie_j^\top{Y}^*\\
\frac1ne_je_i^\top{X}^*
\end{bmatrix}\right)^\top\in\bbR^{(p^2+2nr)\times (p^2+2nr)},
\end{align*}
which represents the Hessian matrix of the nonconvex counterpart at the ground truth $(H^*, X^*, Y^*)$. The following assumptions are required for $\cP D^* \cP$.

\begin{assumption}\label{assumption:D_eigen}
We assume there exists some constants $\slD$ and $\suD$ such that
\begin{align*}
\slD\leq \lambda_{\min}(\cP D^* \cP)\leq \lambda_{\max}(\cP D^* \cP)\leq \suD.
\end{align*}   
\end{assumption}

\begin{assumption}\label{assumption:2_infty}
We assume there exists some constants $c_{2,\infty}$ such that
\begin{align*}
\|I_{p^2+2nr}-(\cP D^* \cP)^{\dagger}(\cP D^* \cP)\|_{2,\infty}\leq c_{2,\infty}\sqrt{\frac{r^2 + p}{n}}.
\end{align*}
\end{assumption}
The convergence rate of the optimization algorithm depends on the condition number, which is the ratio of the largest and smallest eigenvalue of the Hessian matrix. Assumption \ref{assumption:D_eigen} is the nonconvex counterpart of Assumption \ref{assumption:eigenvalues}, and it ensures the eigenvalues of the Hessian matrix are balanced. While Assumption \ref{assumption:D_eigen} focuses on the non-zero eigenvalues of $\cP D^* \cP$, we emphasize here that $D^*$ has a null space with dimension $r^2$ and a mild condition is required for this null space, which is Assumption \ref{assumption:2_infty}. Assumption \ref{assumption:2_infty} can be viewed as an analog of Assumption \ref{assumption:incoherent}, and it is saying the projection onto the null space of $\cP D^* \cP$ is incoherent. 

Although Assumptions \ref{assumption:D_eigen} and \ref{assumption:2_infty} aid in the analysis of nonconvex optimization, the solution from the nonconvex optimization is in fact closely tied to that of the convex problem \eqref{prob:cv}. The following assumption is crucial in unveiling this connection.

\begin{assumption}\label{assumption:r+1}
We define a matrix $M^{*}$ such that
\begin{align*}
M^{*}_{ij}=
\begin{cases}
    \frac{e^{P_{ij}^*}}{(1+e^{P_{ij}^*})^2} & i\neq j\\
    0 & i=j.
\end{cases}
\end{align*}
Suppose $(\Delta_{H}, \Delta_{X}, \Delta_{Y})$ is given by
\begin{align*}
 \vec\begin{bmatrix}
\Delta_{H}\\
\Delta_{X}\\
\Delta_{Y}
\end{bmatrix}
=(\cP D^*\cP)^{\dagger}\vec\begin{bmatrix}
0\\
X^{*}\\
Y^{*}
\end{bmatrix}
\end{align*}
We assume that there exists a constant $\epsilon > 0$ such that
\begin{align*}
    \sigma_{r+1}\left(\cP_Z^{\perp}\left(\frac{1}{n}{M}^* \odot \left(Z\Delta_{H}Z^{\top}+\frac{\Delta_{X}{Y}^{*T}+{X}^*{\Delta^{\top}_{Y}}}{n}\right)\right)\cP_Z^{\perp}\right) <1-\epsilon.
\end{align*}
\end{assumption}

In fact, Assumption \ref{assumption:r+1} provides conditions that are nearly necessary and sufficient for the convex and nonconvex solutions to be equivalent. While this assumption may not seem intuitive at first, it is typically easy to satisfy in practical applications, with the upper bound $1-\epsilon$ often being quite small.
Specifically, Assumption \ref{assumption:r+1} holds in common settings such as stochastic block models.

\begin{proposition}\label{prop:assumption_example}
Assumption \ref{assumption:r+1} holds for the stochastic block model with two communities. More specifically, Assumption \ref{assumption:r+1} holds when $H^*=0$ and
\begin{align*}
\Gamma^*=
\begin{bmatrix}
p\mathbf{1}\mathbf{1}^{\top} & q\mathbf{1}\mathbf{1}^{\top}\\
q\mathbf{1}\mathbf{1}^{\top} & p\mathbf{1}\mathbf{1}^{\top}\\
\end{bmatrix},
\end{align*}
where $\mathbf{1}\in\bbR^{\frac{n}{2}\times 1}$ is an all one vector and $p>q$.
\end{proposition}

\subsection{Estimation results}\label{sec:est_results}

In this section, we present rigorous theoretical guarantees for the estimation of the model parameters. We demonstrate that, under the given assumptions, the solution $(\hat{H}_c, \hat{\Gamma}_c)$ obtained from the convex optimization problem \eqref{prob:cv} achieves optimal estimation errors for both the matrix $H^*$ up to the logarithmic terms, which captures the effects of the covariates, and the low-rank membership matrix $\Gamma^*$.

\begin{theorem}\label{thm:cv_est} 
Suppose Assumption \ref{assumption:scales}-\ref{assumption:r+1} hold and $n$ is sufficiently large. 
We have
\begin{align*}
&\|\hat{H}_c-H^{*}\|_F\lesssim \lambda\sqrt{\frac{\mu r \kappa}{n \sigma_{\min}}},\quad
\|\hat{\Gamma}_c-\Gamma^{*}\|_{F}\lesssim \lambda\kappa\sqrt{\mu r},\\
&\|\hat{\Gamma}_c-\Gamma^{*}\|_{\infty}\lesssim \mu r \kappa\left(\frac{\lambda\sigma_{\max}}{n^2}\sqrt{\mu r\left(1+\frac{n}{\sigma_{\max}}\right)}+\kappa\sqrt{\frac{\log n}{n}}\right)
\end{align*}
as long as $\lambda\gtrsim \frac{1}{\eps}\left(1+\frac{\mu r\sigma_{\max}}{n}\right)\sqrt{n\log n}$.
\end{theorem}

\begin{remark}
    Note that Theorem \ref{thm:cv_est} allows the rank $r$ and condition number $\kappa$ to grow with $n$. If we focus on the cases that $\mu, r, \kappa\asymp 1 $ and $\sigma_{\min}, \sigma_{\max}\asymp n$, then Theorem \ref{thm:cv_est} implies
    \begin{align*}
        \|\hat{H}_c-H^{*}\|_F\lesssim \sqrt{\frac{\log n}{n}}, \;\|\hat{\Gamma}_c-\Gamma^{*}\|_{F}\lesssim 1, \;\|\hat{\Gamma}_c-\Gamma^{*}\|_{\infty}\lesssim \sqrt{\frac{\log n}{n}}.
    \end{align*}
\end{remark}

\subsection{Membership reconstruction results}\label{sec:membership_results}
In this subsection, we shift focus to reconstructing the latent community memberships based on the estimated matrix $\hat{\Gamma}_c$.
We describe a vertex-hunting algorithm for efficiently estimating the mixed-membership vectors and provide theoretical bounds on the accuracy of the reconstructed memberships. We first state the identifiability condition as follows.

\begin{proposition}\label{identidiablecondition}
    Consider the DCMM model $\Gamma = \Theta \Pi W \Pi^\top \Theta$. If we assume (1) $|W_{ii}| = 1$ for all $i\in [r]$, (2) each community has at least one pure node, then the DCMM model is identifiable.
\end{proposition}


Inspired by \cite{jin2017estimating}, we consider the following three-step procedure (Algorithm \ref{alg:mem_est}):

\begin{algorithm}[H]
    \caption{Vertex Hunting and Membership Reconstruction}
    \begin{algorithmic}[1]\label{alg:mem_est}
        \STATE \textbf{Input:} Matrix $\hat{\Gamma}_c \in \mathbb{R}^{n \times n}$
        \STATE \textbf{Step 1 (Score step):}
        \begin{itemize}
            \item Obtain $(\hat{\lambda}_1,\hat{u}_1),\ldots, (\hat{\lambda}_r,\hat{u}_r)$, where $\hat{\lambda}_1,\ldots, \hat{\lambda}_r$ are the $r$ largest (in magnitude) eigenvalues of $\hat{\Gamma}_c$ and $\hat{u}_1,\ldots, \hat{u}_r$ are the corresponding eigenvectors.
            \item Obtain $\hat{R}=\begin{bmatrix}
                \hat{r}_1^{\top}\\
                \vdots\\
                \hat{r}_n^{\top}
            \end{bmatrix}:=
            [\hat{u}_2/\hat{u}_1,\ldots, \hat{u}_r/\hat{u}_1]\in\bbR^{n\times (r-1)}$.
        \end{itemize}

        \STATE \textbf{Step 2 (Vertex Hunting step):}
Run a convex hull algorithm on the $\{\hat{r}_i\}^n_{i=1}$. Denote vertices of the obtained convex hull by $\{\hat{v}_{\ell}\}_{\ell = 1}^r$.

        \STATE\textbf{Step 3 (Membership Reconstruction step):}
        \STATE For $1\leq \ell \leq r$, estimate
  $\hat{b}_1(\ell)=\left|\hat{\lambda}_1+\hat{v}_{\ell}^{\top}\diag(\hat{\lambda}_2,\ldots, \hat{\lambda}_r)\hat{v}_{\ell}\right|^{-1/2}$.
            
        \FOR{each $i \in [n]$}
        \STATE Solve  
        $
        \begin{cases} 
       \sum^r_{\ell =1}\hat{w}_i(\ell) \hat{v}_{\ell}=\hat{r}_i  \\
       \sum^r_{\ell =1}\hat{w}_i(\ell) =1 
        \end{cases}
       $ and obtain $\{\hat{w}_i(\ell)\}_{\ell=1}^r$.

        \STATE For $1\leq \ell\leq r$, let $\tilde{\pi}_i(\ell):=\max\left\{0,\frac{\hat{w}_i(\ell)}{\hat{b}_1(\ell)}\right\}$. And thus obtain $\tilde{\pi}_i\in\bbR^r$. 
        \STATE Obtain the estimator $\hat{\pi}_i:=\frac{\tilde{\pi}_i}{\|\tilde{\pi}_i\|_1}$.
        \ENDFOR

        \STATE \textbf{Output:}
        \[
        \hat{\Pi}_c :=
        \begin{bmatrix}
        \hat{\pi}_1^{\top} \\
        \vdots \\
        \hat{\pi}_n^{\top}
        \end{bmatrix} .
        \]
    \end{algorithmic}
\end{algorithm}

\begin{definition}[Efficient Vertex Hunting]\label{def:eff_vh}
A Vertex Hunting (VH) algorithm is efficient if it satisfies
\begin{align*}
\max_{1\leq \ell\leq r}\|\hat{v}_{\ell}-v_{\ell}^*\|_2\leq C \max_{1\leq i\leq n}\|\hat{r}_{i}-r_{i}^*\|_2
\end{align*}
for some constant $C$.
\end{definition}

\begin{remark}[Example of an Efficient VH Algorithm: Successive Projection] We present an example of an efficient VH algorithm known as Successive Projection (Algorithm \ref{alg:sp}).
\begin{algorithm}[H]
    \caption{Successive projection}
    \begin{algorithmic}[1]\label{alg:sp}
    \STATE \textbf{Input:} $\{\hat{r}_i\}^n_{i=1}$
    
    \STATE Initialize \( Y_i = (1, \hat{r}^{\top}_i)^{\top} \in \mathbb{R}^r \), for \( 1 \leq i \leq n \).
    
    \STATE At iteration \( \ell = 1, 2, \dots, r \): Find \( i_{\ell} = \arg \max_{1 \leq i \leq n} \|Y_i\|_2 \) and let \( a_{\ell} = Y_{i_{\ell}} / \|Y_{i_{\ell}}\|_2 \). Set the \( {\ell} \)-th estimated vertex as \( \hat{v}_{\ell} = \hat{r}_{i_{\ell}} \). Project all data points by updating \( Y_i \) to \( (I_r - a_{\ell} a_{\ell}^{\top}) Y_i \), for \( 1 \leq i \leq n \).
    
    \STATE \textbf{Output} \( \hat{v}_1, \hat{v}_2, \dots, \hat{v}_r \).
    \end{algorithmic}
\end{algorithm}
According to \cite{jin2017estimating}, Lemma 3.1, the successive projection method is an efficient VH algorithm. 
\end{remark}

Align with \cite{jin2017estimating}, we make the following assumptions.
\begin{assumption}\label{assumption:mem_est} We assume the following conditions hold.
\begin{enumerate}
    \item Let $\theta^{*}_{\max}:=\max_{1\leq i\leq n}\theta_i^{*}$, $\theta^{*}_{\min}:=\min_{1\leq i\leq n}\theta_i^{*}$ and 
    $\bar{\theta}^{*}_2:=\left(\frac{1}{n}\sum^n_{i=1}(\theta^{*}_i)^2\right)^{1/2}$. We assume there exists a constant $C_1$ such that $\theta^{*}_{\max}\leq C_1$ and a constant $C_2$ such that
    \begin{align*}
    \theta^{*}_{\max}\leq C_2 \theta^{*}_{\min}.
    \end{align*}

    \item Recall that $\Gamma^{*}:=\Theta^{*}\Pi^{*} W^{*}\Pi^{*\top}\Theta^{*}$. Let $G=r \|\theta^{*}\|^{-2}(\Pi^{*\top}\Theta^{* 2}\Pi^{*})\in\bbR^{r\times r}$. We assume $\|W^{*}\|_{\infty}\leq C$, $\|G\|\leq C$ and $\|G^{-1}\|\leq C$ for some constant $C$.

    \item Let $\lambda_{\ell}(W^{*}G)$ be the $\ell$-th largest right eigenvalue of $W^{*}G$ in magnitude, and $\eta_{\ell}\in\bbR^r$ be the associated right eigenvector, $1\leq \ell\leq r$. For a constant $c>0$ and a sequence $\{\beta_n\}_{n=1}^{\infty}$ such that $\beta_n\leq 1$, we assume
    \begin{align*}
    |\lambda_2(W^{*}G)|\leq (1-c)|\lambda_1(W^{*}G)|,\text{ and }
    c\beta_n\leq |\lambda_r(W^{*}G)|\leq |\lambda_2(W^{*}G)|\leq c^{-1}\beta_n.
    \end{align*}
    We also assume
    \begin{align*}
    \min_{1\leq \ell \leq r}\eta_1(\ell)>0,\text{ and }
    \frac{\max_{1\leq \ell \leq r}\eta_1(\ell)}{\min_{1\leq \ell \leq r}\eta_1(\ell)}\leq C.
    \end{align*}
\end{enumerate}
\end{assumption}

\begin{theorem}\label{thm:mem_est}
Let $\hat{\Pi}_c\in\bbR^{n\times r}$ be the membership estimation given by Algorithm \ref{alg:mem_est}. Suppose an efficient Vertex Hunting algorithm is available and Assumption \ref{assumption:mem_est} holds.
Under the assumptions and conditions of Theorem \ref{thm:cv_est}, it holds that
\begin{align*}
\max_{i\in [n]}\left\|\hat{\pi}_i -\pi^*_i\right\|_{1}\lesssim \lambda\left( \kappa^{1.5}\sqrt{\mu r}+\sqrt{\mu\kappa}r^{5/4}\right)\left(\frac{\mu \kappa^{0.5}}{\beta_n}+\frac{\kappa \mu^{1.5}}{\sqrt{\beta_n}} \right)\left(\frac{r}{\sqrt{n}\bar{\theta}_2^*}\right)^2.
\end{align*}
\end{theorem}
\begin{remark}
    Similar to Theorem \ref{thm:cv_est}, Theorem \ref{thm:mem_est} allows the rank $r$ and condition number $\kappa$ to grow with $n$ and $\beta_n, \bar{\theta}_2^*$ to decrease with $n$. In particular, if $\mu, r, \kappa,\beta_n\asymp 1 $, Theorem \ref{thm:mem_est} allows $(\log n/n)^{1/4}\ll\bar{\theta}_2^*$. If we focus on the cases that $\mu, r, \kappa,\beta_n, \bar{\theta}_2^*\asymp 1 $, then Theorem \ref{thm:mem_est} implies
    \begin{align*}
        \max_{i\in [n]}\left\|\hat{\pi}_i -\pi^*_i\right\|_{1}\lesssim\sqrt{\frac{\log n}{n}}.
    \end{align*}
\end{remark}

%% file: proof_idea.tex
\section{Proof Strategy and Key Innovations}\label{sec:proof_idea}
In this section, we outline our proof strategy and highlight the key technical contributions of this work. 
Directly analysis on convex problem \eqref{prob:cv} is unable to give the sophisticated control on $\|\hat{\Gamma}_c-\Gamma^{*}\|_{\infty}$.
To address this, we leverage nonconvex optimization to facilitate the analysis of the convex problem. 
Our approach consists of two main components: analyzing the nonconvex gradient descent using the leave-one-out technique and establishing the equivalence between the convex and nonconvex solutions.
While this analysis framework is well-established in the literature \citep{chen2020noisy}, our work extends it to handle the logistic loss, overcoming the limitations of existing methods that only apply to mean squared error (MSE) loss. 
Our approach can be further generalized to other convex loss functions. 
In the following sections, we describe these contributions in detail.


\subsection{Nonconvex problem}\label{sec:nonconvex_prob}
We begin by introducing a nonconvex optimization problem to aid our proof.
We reparameterize $\Gamma = XY^\top$, where $X, Y\in \mathbb{R}^{n\times r}$, and consider the following nonconvex problem as an alternative to \eqref{prob:cv}
\begin{align}\label{prob:ncv}
    \min_{H,X, Y} \quad &f( {H}, {X},  {Y}) :=  \sum_{i\neq j}\left(\log(1+e^{P_{ij}})  - A_{ij}P_{ij}\right)+\frac{\lambda}{2}\left\|X\right\|_F^2+\frac{\lambda}{2}\left\|Y\right\|_F^2\notag\\
    \text{ s.t. } \quad &P_{ij} = z_i^\top H z_j + ( {X} {Y}^\top)_{ij}, \notag\\
     &\cP_Z(X)=\cP_Z(Y)=0.
\end{align}
Here, we replace $\Gamma$ with $XY^\top$ and the nuclear norm $\|\Gamma\|_*$ with $(\|X\|_F^2+\|Y\|_F^2)/2$, motivated by the fact that for any rank-$r$ matrix $\Gamma$,
\begin{align*}
    \|\Gamma\|_* = \min_{X, Y\in \mathbb{R}^{n\times r}, XY^\top = \Gamma}\frac{1}{2}\left(\left\|X\right\|_F^2+\left\|Y\right\|_F^2\right)
\end{align*}
as shown in \cite{srebro2005rank,mazumder2010spectral}.
This reparameterization exploits the low-rank structure of $\Gamma^*$ and reduces the number of parameters from $O(n^2)$ to $O(nr)$, which allows us to better control $\|\hat{\Gamma}_c-\Gamma^{*}\|_{\infty}$. We solve this nonconvex problem using gradient descent.

Although one might be concerned that this nonconvex optimization depends on the value of $r$, which is unknown, we emphasize that this approach is purely an analytical tool to study the convex problem by defining a sequence of ancillary random vectors, rather than an algorithm to be directly applied. We initialize gradient descent at $H^0 = H^*$, $X^0 = X^*$, and $Y^0 = Y^*$, and run for a fixed number of iterations $t_0$. For $t = 0, \dots, t_0 - 1$, we compute:
\begin{align*}
\begin{bmatrix}
H^{t+1}\\
X^{t+1}\\
Y^{t+1}
\end{bmatrix}&=
\begin{bmatrix}
H^{t}-\eta\nabla_{H}f(H^t,X^t, Y^t)\\
\cP_Z^{\perp}\left(X^{t}-\eta\nabla_{X}f(H^t,X^t, Y^t)\right)\\
\cP_Z^{\perp}\left(Y^{t}-\eta\nabla_{Y}f(H^t,X^t, Y^t)\right)
\end{bmatrix}.
\end{align*}

We can show, with high probability, that there exists a sequence of rotation matrices $\{R^t\}_{t=0}^{t_0}$ such that:
\begin{align*}
    \|H^t-H^{*}\|_F, \|X^tR^t-X^{*}\|_{2,\infty}, \|Y^tR^t-Y^{*}\|_{2,\infty}\lesssim \frac{1}{\sqrt{n}}
\end{align*}
for all $0 \leq t \leq t_0$. See Lemma \ref{lem:ncv4} for more details. This implies that the nonconvex optimization path remains close to the true parameters $H^*$, $X^*$, and $Y^*$ throughout the iterations.


Furthermore, defining
\begin{align*}
&t^*:=\argmin_{0\leq t<t_0}\left\|
\cP\nabla f(H^t,X^t,Y^t)
\right\|_2,
\end{align*}
we can show that
\begin{align*}
    \left\|
\cP\nabla f(H^{t^*},X^{t^*},Y^{t^*})
\right\|_2\lesssim n^{-5}.
\end{align*}
See Lemma \ref{lem:ncv6} for more details. Therefore, if we define $(\hat{H},\hat{X},\hat{Y})=({H}^{t^{*}},{X}^{t^{*}}R^{t^{*}},{Y}^{t^{*}}R^{t^{*}})$ as the nonconvex solution, it then satisfies:
\begin{enumerate}
    \item The gradient of $f$ at $(\hat{H},\hat{X},\hat{Y})$ (after projection) is sufficiently small.
    \item The errors are well-controlled:
    \begin{align*}
        \|\hat{H}-H^{*}\|_F, \|\hat{X}-X^{*}\|_{2,\infty}, \|\hat{Y}-Y^{*}\|_{2,\infty}\lesssim \frac{1}{\sqrt{n}}.
    \end{align*}
    This further implies $\|\hat{X}\hat{Y}^\top-\Gamma^*\|_{\infty}\lesssim 1/\sqrt{n}$.
\end{enumerate}

\subsection{Bridge convex and nonconvex}
Although $(\hat{H},\hat{X},\hat{Y})$ is defined from a hypothetical algorithm that cannot be applied, we will show that the nonconvex solution is very close to the convex solution, in the sense that $(\hat{H},\hat{X}\hat{Y}^{\top})\approx (\hat{H}_c,\hat{\Gamma}_c)$. 
This allows us to transfer the theoretical guarantees of the nonconvex solution directly to the convex solution, leading to Theorem \ref{thm:cv_est}. 

Define $$L_c(H,\Gamma) = \sum_{i \neq j} \left( \log\left(1 + e^{z_i^\top H z_j + \Gamma_{ij}}\right) - A_{ij} \left(z_i^\top H z_j + \Gamma_{ij}\right)\right).$$
Notice that $(H,\Gamma)$ is the unique minimizer of the convex problem \eqref{prob:cv} if it satisfies:
\begin{align}
\begin{cases}
      \nabla_H L_c(H,\Gamma) = 0\\
      \cP^{\perp}_Z\nabla_\Gamma L_c(H,\Gamma)\cP^{\perp}_Z = -\lambda UV^T + \lambda W  \label{BridgeCVXNonCVX:KKT}\\
      \cP_Z\Gamma=0,\quad \Gamma\cP_Z = 0.
    \end{cases}
\end{align}
Here $\Gamma = U\Lambda V^T$ is the SVD of $\Gamma$ and $W\in T^{\perp}$ with $\|W\|<1$, where $T:=\{UA^{\top}+BV^{\top}\mid A, B\}$ is the tangent space of $\Gamma$.

Therefore, as long as we can verify \eqref{BridgeCVXNonCVX:KKT} for $(\hat{H},\hat{X}\hat{Y}^\top)$ approximately (recall that the gradient of $f$ at $(\hat{H},\hat{X},\hat{Y})$ is very small, but not exactly zero), we are able to show the different between the convex solution $(\hat{H}_c, \hat{\Gamma}_c)$ and nonconvex solution $(\hat{H}, \hat{X}\hat{Y}^\top)$ is extremely small.
This idea and corresponding analysis, first proposed by \cite{chen2020noisy}, was previously limited to the MSE because the derivative of the MSE is linear with respect to the variable. In this paper, we introduce a more elaborated approach to analyze the nonconvex solution and to verify \eqref{BridgeCVXNonCVX:KKT}, and this technique can be potentially extended to many other problems. 

For simplicity, let's assume the gradient $\cP\nabla f(\hat{H}, \hat{X}, \hat{Y})$ is exactly zero ($n^{-5}$ is sufficiently small). Then, one can show that the gradient of $L_c$ after projection can be expressed as:
\begin{align*}
    \cP^{\perp}_Z\nabla_\Gamma L_c(\hat{H},\hat{X}\hat{Y}^T)\cP^{\perp}_Z = -\lambda UV^T +\lambda W ,
\end{align*}
where $\Gamma = U\Lambda V^T$ is the SVD of $\hat{X}\hat{Y}^T$ and $W$ is a matrix from the orthogonal complement of the tangent space of $\hat{X}\hat{Y}^T$. In order to show $\hat{X}\hat{Y}^T$ is equivalent to the convex solution, it suffices to show $\|W\|<1$, which is equivalent to verifying that $\sigma_{r+1}(\cP^{\perp}_Z\nabla_\Gamma L_c(\hat{H},\hat{X}\hat{Y}^T)\cP^{\perp}_Z)<\lambda$. If the loss function in $L_c$ was the mean square error, then $\nabla_\Gamma L_c(\hat{H},\hat{X}\hat{Y}^T)$ would be linear with $\hat{H}$ and $\hat{X}\hat{Y}^T$, simplifying the analysis. However, since the logistic loss is used in our case, we need to analyze this gradient in more depth. To tackle this, we begin with the Taylor expansion of $\nabla L_c$ at $(H^*, \Gamma^*)$
\begin{align}
    \nabla_\Gamma L_c(\hat{H},\hat{X}\hat{Y}^T)\approx \nabla_\Gamma  L_c(H^*,\Gamma^*)+M^*\odot \left(Z(\hat{H} - H^*)Z^\top + (\hat{X}\hat{Y}^T - \Gamma^*)\right),\label{BridgeCVXNonCVX:Taylorexpansion}
\end{align}
where the matrix $M^*\in \mathbb{R}^{n\times n}$ is defined as:
\begin{align*}
M^{*}_{ij}=
\begin{cases}
    \frac{e^{P_{ij}^*}}{(1+e^{P_{ij}^*})^2} & i\neq j\\
    0 & i=j
\end{cases}
.
\end{align*}
The first term on the right-hand side of \eqref{BridgeCVXNonCVX:Taylorexpansion} is a mean zero random matrix which can be well controlled with the random matrix theory. We thus focus on the second term.

The key idea of our analysis is to isolate the bias and stochastic error in $\hat{H}-H^*$ and $\hat{X}\hat{T}^\top-\Gamma^*$. To achieve this, we define the corresponding debiased `estimator' as:
\begin{align}\label{BridgeCVXNonCVX:debias}
\vec\begin{bmatrix}
\hat{H}^d-\hat{H}\\
\hat{X}^d-\hat{X}\\
\hat{Y}^d-\hat{Y}
\end{bmatrix}:=-(\cP\hat{D}\cP)^{\dagger}\cP
\nabla L(\hat{H},\hat{X},\hat{Y}),
\end{align}
where \begin{align*}
\hat{D}:=
\sum_{i\neq j}\frac{e^{\hat{P}_{ij}}}{(1+e^{\hat{P}_{ij}})^2}
\left(\vec\begin{bmatrix}
z_i z_j^{\top}\\
\frac1ne_ie_j^\top\hat{Y}\\
\frac1ne_je_i^\top\hat{X}
\end{bmatrix}\right)\left(\vec\begin{bmatrix}
z_i z_j^{\top}\\
\frac1ne_ie_j^\top\hat{Y}\\
\frac1ne_je_i^\top\hat{X}
\end{bmatrix}\right)^\top.
\end{align*}
This debiased estimator can be viewed as running one Newton–Raphson step from $(\hat{H}, \hat{X},\hat{Y})$. Similarly, we run a Newton–Raphson step from $(H^*, X^*,Y^*)$ to define $(\bar{H},\bar{X},\bar{Y})$ as
\begin{align}\label{BridgeCVXNonCVX:quadraticapprox}
\vec\begin{bmatrix}
\bar{H}-H^*\\
\bar{X}-X^*\\
\bar{Y}-Y^*
\end{bmatrix}:=-(\cP D^*\cP)^{\dagger}\cP
\nabla L(H^*,X^*,Y^*).
\end{align}
This allows us to decompose $\hat{X}\hat{Y}^\top - \Gamma^*$ as:
\begin{align}
    \hat{X}\hat{Y}^\top - \Gamma^* = \underbrace{\left(\hat{X}\hat{Y}^\top -\hat{X}^d(\hat{Y}^{d})^\top\right)}_{(a)} +\underbrace{\left(\hat{X}^d (\hat{Y}^{d})^\top - \bar{X}\bar{Y}^\top \right)}_{(b)}+\underbrace{\left(\bar{X}\bar{Y}^\top  - \Gamma^*\right)}_{(c)}.\label{BridgeCVXNonCVX:decompose}
\end{align}
Note that $\hat{H}-H^*$ can be decomposed and analyzed similarly, so here we focus on $\hat{X}\hat{Y}^\top - \Gamma^*$ as an example.

Our key observation is that the term $(a)$ is the dominating term in \eqref{BridgeCVXNonCVX:decompose} as long as $\lambda$ is properly chosen, and we confirm this statement by controlling term $(b)$ and term $(c)$ accordingly. On the other hand, since $\hat{X}\hat{Y}^\top -\hat{X}^d(\hat{Y}^{d})^\top$ has an explicit form from \eqref{BridgeCVXNonCVX:debias}, we are able to fully characterize the error $\hat{X}\hat{Y}^\top - \Gamma^*$. 

To control the term $(b)$ and $(c)$, we first notice that $(\bar{X}, \bar{Y})$ is also explicitly defined by \eqref{BridgeCVXNonCVX:quadraticapprox}, so term $(c)$ can be controlled directly, as shown in Proposition \ref{prop:bar_dist}. 
In fact, since \eqref{BridgeCVXNonCVX:quadraticapprox} has nothing to do with $\lambda$, term $(c)$ can be controlled by term $(a)$ as long as $\lambda$ is properly chosen. 
When it comes to term $(b)$, it represents the difference between two Newton–Raphson steps with different initializations. 
Since we have shown the difference between these two initializations $(\hat{X}, \hat{Y})$ and $(X^*, Y^*)$ are well-controlled, the difference $\hat{X}^d(\hat{Y}^{d})^\top - \bar{X}\bar{Y}^\top$ shrinks further after the Newton–Raphson step. 
To be more concrete, in Theorem \ref{thm:ncv_debias}, we show that 
\begin{align*}
    \left\|\hat{X}^d(\hat{Y}^{d})^\top - \bar{X}\bar{Y}^\top\right\|_F\lesssim n^{1/4},
\end{align*}
which implies $\|\hat{X}^d(\hat{Y}^{d})^\top - \bar{X}\bar{Y}^\top\|_F \ll \|\hat{X}\hat{Y}^\top - \hat{X}^d\hat{Y}^{d\top}\|_F\asymp \sqrt{n}$. 

As for the main term $(a)$, its properties are guaranteed by Assumption \ref{assumption:r+1}. In fact, Assumption \ref{assumption:r+1} provides a necessary and sufficient condition for the equivalence between the convex and nonconvex solutions, according to our analysis.

Once we obtain the equivalence of the convex solution and nonconvex solutions, the theoretical guarantees for the nonconvex solution can be immediately transferred to the convex solution, which is the estimator we proposed. This is how we leverage the debiased `estimator' and uncertainty quantification to derive the error bounds for our estimator.

%% file: simulations.tex
\section{Simulation Studies}

In our experiments, we generated synthetic data to evaluate the performance of our model in estimating $ H $ and $ \Gamma $. For each trial, we randomly generated the ground truth parameters $ Z $, $ H $, and $ \Gamma $, based on predefined values for the number of nodes $ n $, the number of communities $ r=2$, and the dimension of covariates $ p=3$.

The matrix $ \Gamma $ was constructed as a symmetric matrix using the following process: First, we generated $ \Theta $, an $n\times n$ diagonal matrix, that represents individual node effects and is generated by drawing random values uniformly between  0.83 and 1.0. The community structure is encoded in the $W$ and $\Pi$ matrices. $W$, an $r \times r$ matrix, defines the interaction strength between communities. It is initialized with $-0.7$ for all off-diagonal values, representing weak inter-community connections, while the diagonal entries are set to 1 to indicate strong intra-community ties. The $\Pi$ matrix, an $n \times r$ matrix, represents the probability distribution of each node's affiliation across communities. The values in the first column of $\Pi$ are drawn from a Beta distribution with parameters $0.2$ and $0.2$, while the second column is defined such that each row sums to 1. This setup biases nodes to be closer to one of the two pure community types, $(1, 0)$ or $(0, 1)$. The overall matrix $\Gamma$ is then computed as $\Theta \Pi W \Pi^\top \Theta$, capturing the combined effects of both individual node attributes and community structures on connectivity.

The covariate matrix $Z$ is constructed to lie in the null space of $\Theta \Pi$, ensuring that it satisfies the orthogonality condition $\mathcal{P}_{Z} \Gamma = 0$. First, the null space of $(\Theta \Pi)^\top$ is computed, and a random orthogonal matrix is applied to the resulting null space matrix to generate an orthonormal basis for $Z$. Finally, $Z$ is scaled by $\sqrt{n}/2$ to standardize its values.
The symmetric $p \times p$ matrix $H$, which defines the influence of covariates on edge formation, is chosen as follows:

\begin{align*}
H = \begin{bmatrix} 
2.5 & 1 & -1 \\ 
1 & 1.5 & -0.5 \\ 
-1 & -0.5 & 2 
\end{bmatrix}.    
\end{align*}

Using the generated covariates $ Z $, symmetric interaction matrix $ H $, and matrix $ \Gamma $, we constructed the adjacency matrix $ A $ according to our model \eqref{model}, where the probability of an edge forming between two nodes is governed by a logistic function of their covariates and the corresponding entries of $ \Gamma $.

To estimate $ \hat{H} $ and $ \hat{\Gamma} $, we applied a Nesterov-accelerated gradient descent method with a nuclear norm penalty on $ \Gamma $ for regularization. For each value of $ n $, we repeated the simulation over $100$ runs to account for randomness in the data generation process.
We evaluated the model’s performance by calculating the absolute estimation errors $ \| \hat{H} - H^* \|_F $, $ \|\hat{\Gamma} - \Gamma \|_F$ and $\|\hat{\Gamma} - \Gamma \|_{\infty}$ for each run. 
The mean errors across all runs were computed for each $ n $.

Finally, in Figure \ref{fig:Simulation}, we visualized the results by plotting the mean estimation errors for $ H $ and $ \Gamma $ as functions of $ n $.

\begin{figure}[h]
    \centering
\includegraphics[width=0.7\textwidth]{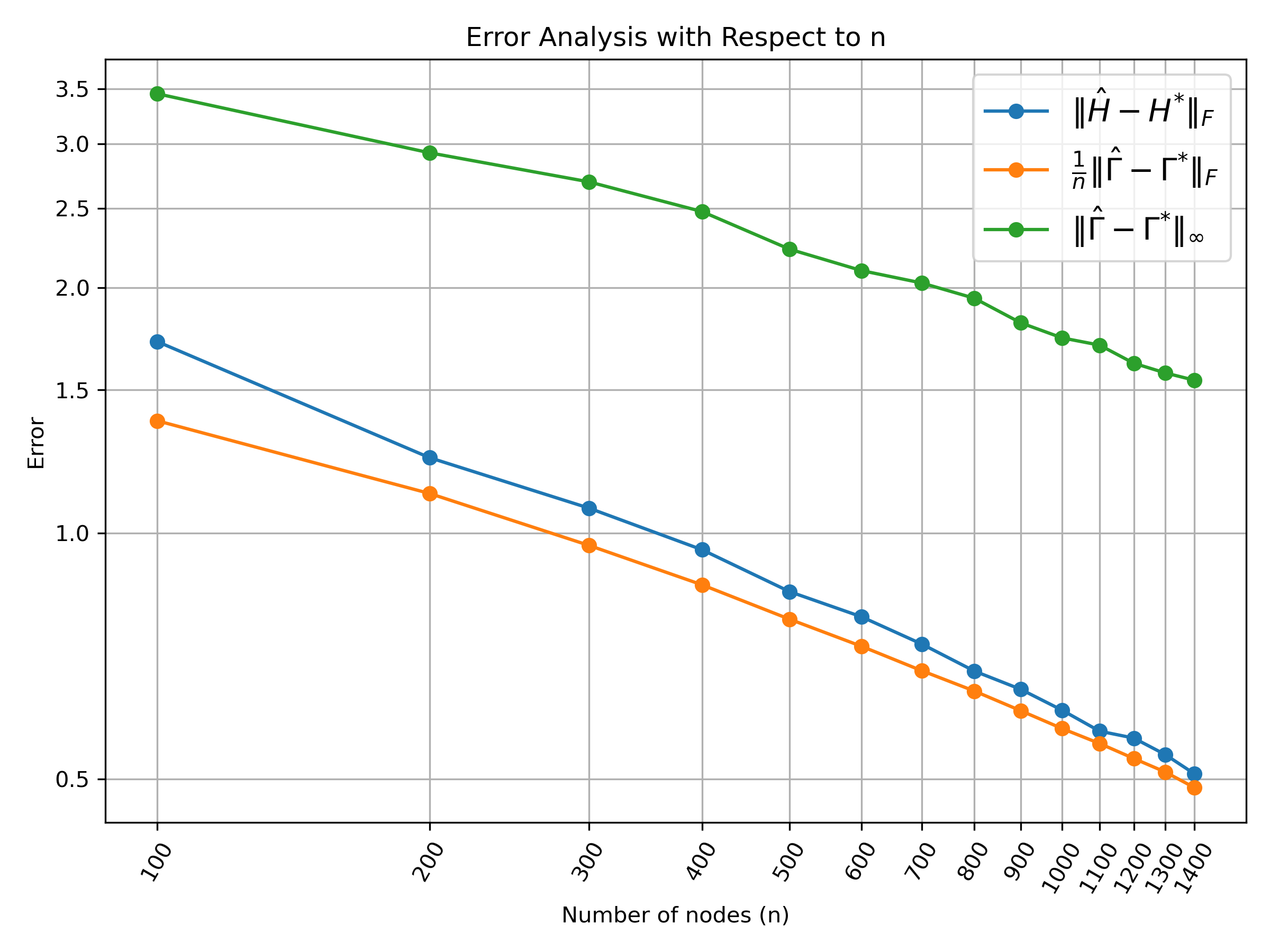}
    \caption{Log–log plot of the estimation error of $\hat{H}, \hat{\Gamma}$ measured by $\|\cdot\|_{F}$ and $\|\cdot\|_{\infty}$ vs. the number of nodes $n$. The results are reported for $r=2, p=3, \lambda=\sqrt{n}$ and are averaged over $100$ independent trials.}
    \label{fig:Simulation}
\end{figure}

%% file: real_data.tex
\section{Real Data Analysis}
In this section, we apply our model to a stock network. We use the daily return data of S\&P 500 stocks from November 10, 2021, to November 10, 2024, obtained from Wharton Research Data Services. The daily return is defined as the change in the total value of an investment in a common stock over a specified period per dollar of initial investment. The data is filtered to exclude assets with missing values and scaled by a factor of 100. 

To construct the stock network, we analyze the correlations of the processed data. Since much of the variation in stock excess returns is known to be driven by common factors, such as the Fama–French factors, we first remove the influence of these common factors. Specifically, we remove the first five principal components of the processed data matrix, which primarily represent the market portfolio. The network is then built using the correlation matrix of the idiosyncratic components (the residuals). Let $ \Sigma $ represent the correlation matrix of these idiosyncratic components. An edge is defined between nodes $ i $ and $ j $ if and only if $ \Sigma_{ij} > 0.16 $, resulting in the adjacency matrix $ A $.

We consider six covariates: price-to-earnings (PE), price-to-sales (PS), price-to-book (PB), price-to-free-cash-flow (PFCF), debt-to-equity ratio (DER), and return on equity (ROE). These covariates are constructed for each firm using financial data from November 10, 2021, to November 10, 2024.  
To preprocess the data, we first remove all infinite and missing values. Firms with no valid data remaining for certain covariates after this adjustment are excluded from the analysis. After preprocessing, we retain $ n = 492 $ companies in the network.  
For each firm, we compute the mean values of the relevant financial metrics over the given period. For example, when calculating the PE ratio, we first compute the mean values of price and earnings separately. If both mean values are positive, we compute their ratio. This ratio is then capped at a predefined lower bound for each covariate to mitigate extreme values and numerical instability. If one or both mean values are non-positive, we assign the predefined lower bound directly.  
In our experiment, we set the lower bounds as follows: $ 0.01 $ for PE, PS, and PB; $ 0.003 $ for PFCF; $ 0.03 $ for DER; and $ 0.3 $ for ROE. We then apply a logarithmic transformation to the obtained ratios and standardize each covariate across firms. This process results in a $ 492 \times 6 $ covariate matrix $ Z $.

Using our proposed model, along with the obtained adjacency matrix $ A $ and covariate matrix $ Z $, we employ a Nesterov-accelerated gradient descent method, initialized with zero matrices, to estimate $ \hat{H} $ and $ \hat{\Gamma} $. A nuclear norm penalty is applied to $ \Gamma $ for regularization. The regularization parameter is set to be $18$ and the estimated $\hat{\Gamma}$ has a rank of $4$. 
The scatter plot of the 3-dimensional eigenratio $ \hat{r}_i = \left[(\hat{u}_2)_i / (\hat{u}_1)_i, (\hat{u}_3)_i / (\hat{u}_1)_i, (\hat{u}_4)_i / (\hat{u}_1)_i \right]^\top $ for each stock exhibits a distinct tetrahedral structure. The four vertices of the tetrahedron correspond to the coordinates of four firms: Arch Capital Group (ACGL), PepsiCo (PEP), BXP, Inc. (BXP) and Pentair (PNR). Subsequently, we employ Algorithm \ref{alg:mem_est} to reconstruct the membership, yielding a $ 492 \times 4 $ estimated membership matrix $ \hat{\Pi} $.

In Figure \ref{fig:hat_Pi_sector}, we show the 3-dimensional scatter plot of $\hat{\Pi}_{i, 1:3}, i\in [492]$. (Since each row of $\hat{\Pi}$ adds up to $1$, the last column of $\hat{\Pi}$ can be simply expressed by the other three columns.) As we can see from Figure \ref{fig:hat_Pi_sector}, the estimated membership $ \hat{\Pi} $ shows a strong cluster effect. We mark companies from financials, real estate, consumer staples, and industrials sectors in black in the four subplots respectively. They occupy the four vertices, and there are very few other companies on those vertices. That is to say, we can observe a clear mixed membership structure behind the S\&P 500 companies, with financials, real estate, consumer staples, and industrials sectors being the vertices. In addition to that, the utilities sector also forms a cluster, which is located in the middle of the tetrahedral. It is also worth mentioning that the information technology companies are scattered in the central area of the entire tetrahedron, indicating the wide variety of technology companies.

\begin{figure}[h!]
    \centering
    \includegraphics[width=\textwidth]{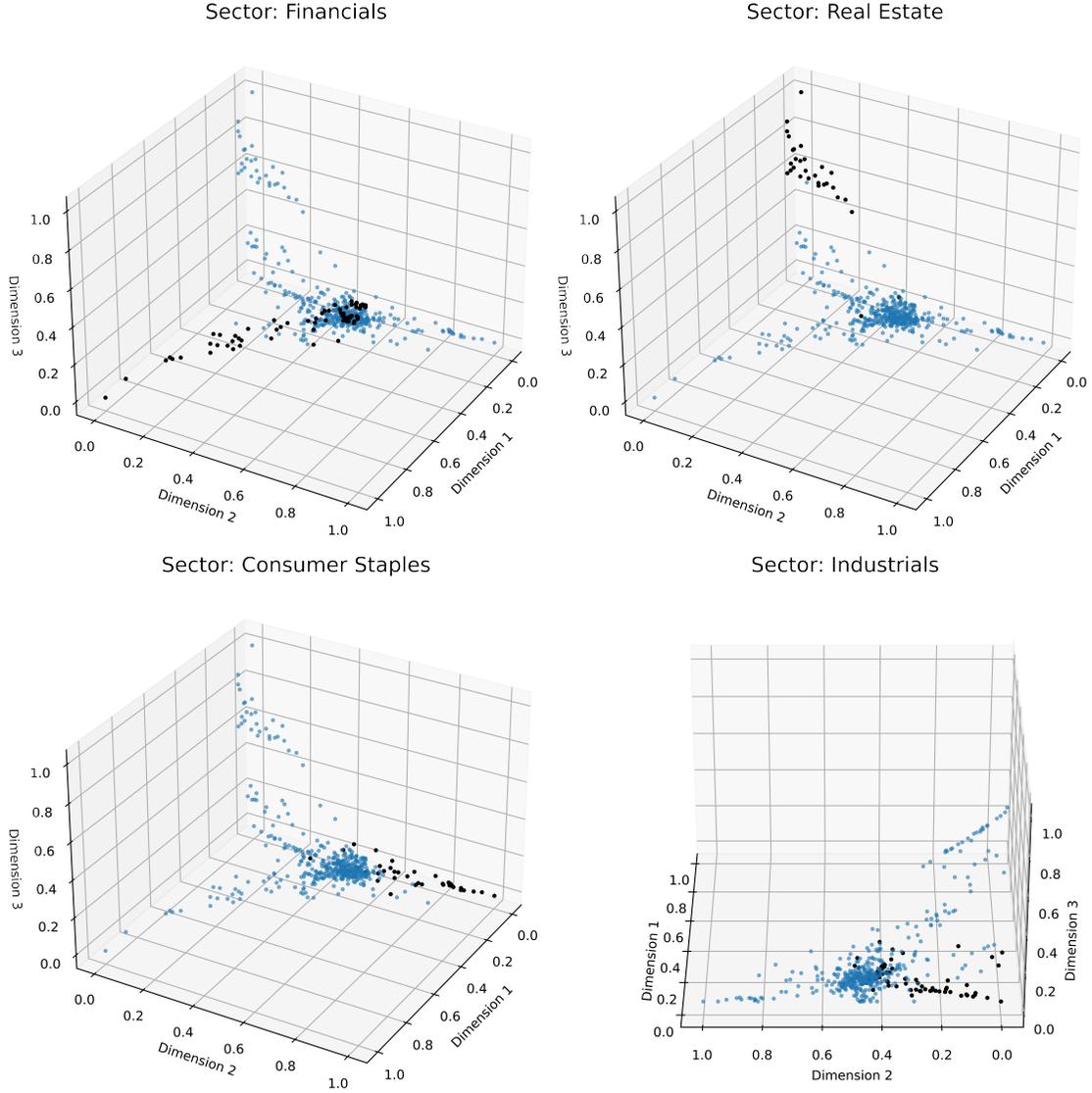}
    \caption{Sector plot for  $\hat{\Pi} $. Companies from financials/real estate/consumer staples/industrials sectors are marked in black in the top left/top right/bottom left/bottom right plots. The bottom right plot is rotated to better show the industrials sector.}
    \label{fig:hat_Pi_sector}
\end{figure}

With our observations on the membership structure shown above, let's now turn to the covariates part. Under our identifiability condition $\cP_Z\Gamma = \Gamma\cP_Z = 0$, we can view $\vec(H)$ as the regression coefficients of regressing $\vec(P)$ on 
\begin{align*}
    \left[ \begin{array}{c}
   \vec(z_1z_1^\top)      \\
   \vec(z_2z_1^\top)      \\
   \vec(z_3z_1^\top)      \\
   \vdots  \\
   \vec(z_{n-1}z_n^\top)      \\
   \vec(z_nz_n^\top)      \\
\end{array}
\right]  \in \mathbb{R}^{n^2 \times p^2}
\end{align*}
with $\text{mean}(\vec(\Gamma))$ being the intercept and $\vec(\Gamma) - \text{mean}(\vec(\Gamma))$ being the residuals. Recall $P$ is defined by $P = ZHZ^\top+\Gamma$ in \eqref{prob:cv}, which can be further written as
\begin{align*}
    P = \text{mean}(\vec(\Gamma)) +ZHZ^\top+\left(\Gamma - \text{mean}(\vec(\Gamma))\right)
\end{align*}
to ensure the mean of residue is $0$. Therefore, the $R^2$ of the described regression represents the proportion of $P$ that can be explained by the covariates. By the definition of $R^2$, we have
\begin{align*}
    R^2 = \frac{\left\|ZHZ^\top\right\|_F^2}{\left\|P - \text{mean}(\vec(\Gamma))\right\|_F^2}=\frac{\left\|ZHZ^\top\right\|_F^2}{\left\|ZHZ^\top\right\|_F^2+\left\|\Gamma-\text{mean}(\vec(\Gamma))\right\|_F^2}.
\end{align*}
Plugging our estimated $\hat{H}$ and $\hat{\Gamma}$, we get $R^2 = 0.586$, which means the covariates explain a significant part of $P$. In contrast, if we randomly shuffle the $n$ rows of the covariate matrix $Z$ and repeat the above calculation, it results in $R^2 = 0.0087$. This implies that our model extracts a substantial amount of information from the covariates.

We then compare the goodness of fit of our model to that of a model without covariate adjustment, given by $\bbP\left(A_{ij}=1\right) = e^{P^{*}_{ij}} / (1+e^{P^{*}_{ij}})$, where $P_{ij}^* = \Gamma_{ij}^*$. Since, on average,  $A_{ij}$ should be close to $e^{{P^{*}_{ij}}} / (1+e^{P^{*}_{ij}})$, we use a $\chi^2$-type of statistic
\begin{align}
    \text{ERROR} = \sum_{1\leq i<j\leq n}\frac{\left(A_{ij} - e^{\hat{P}_{ij}}/\left(1+e^{\hat{P}_{ij}}\right)\right)^2}{e^{\hat{P}_{ij}} / \left(1+e^{\hat{P}_{ij}}\right)} \label{def:error}
\end{align}
as a measurement to assess the goodness of fit of the estimator$\hat{P}_{ij}$. Note that the variance of $A_{ij}$ is $e^{P^*_{ij}} / \left(1+e^{P^*_{ij}}\right)$, and the denominator in \eqref{def:error} serves to normalize the mean squared error. We compute this error for both our model \eqref{prob:cv} and the model without covariates, where the latter’s estimate is obtained by solving the following optimization problem:
\begin{align*}
    \min_{\Gamma} \quad &\sum_{i \neq j} \left( \log(1 + e^{\Gamma_{ij}}) - A_{ij} \Gamma_{ij} \right) + \lambda \|\Gamma\|_{*}.
\end{align*}
The results are presented in Table \ref{goodnessoffit}. As one can see, the covariates contributes a substantial part to the goodness of fit of our model.

\begin{table}[h]
    \centering
      \caption{Goodness of fit comparison: our model (left), the model without covariates (middle), and the percentage decrease in error (right).}
    \renewcommand{\arraystretch}{1.2} 
    \begin{tabular}{lccc}
        \toprule
        & \textbf{With Covariates} & \textbf{Without Covariates} & \textbf{\% Decrease} \\
        \midrule
        \textbf{ERROR} & $55,441.40$ & $59,643.38$ & $7.05\%$ \\
        \bottomrule
    \end{tabular}
    \label{goodnessoffit}
\end{table}

We further evaluate how the covariates associated with each individual sector influence its specific position within the network. As an analog of the $R^2$, the sector-wise $ R^2 $ is calculated by grouping companies into their respective sectors. For each sector, the corresponding rows of the covariate matrix $ Z $ and the centered matrix $ \Gamma $ (i.e., $ \Gamma_{\text{centered}} = \Gamma - \text{mean}(\Gamma) $) are extracted. The $ R^2 $ for a sector is computed using the formula: 
\begin{align*}
R^2_{\text{sector}} = \frac{\| Z_{\text{sector}} H Z^T \|_F^2}{\| Z_{\text{sector}} H Z^T \|_F^2 + \| \Gamma_{\text{centered, sector}} \|_F^2},   
\end{align*}
where $ Z_{\text{sector}} $ is the sector-specific submatrix of $ Z $ and $ \Gamma_{\text{centered, sector}} $ is the corresponding rows of $ \Gamma_{\text{centered}}$. The numerator represents the explained variance for the sector, while the denominator represents the total variance. We report these values in Table \ref{tab:sector_Rsquare}. Sectors are ranked based on their $ R^2 $ values, providing a measure of how well the covariates explain variability within each sector. 

\begin{table}[h!]
\centering
\caption{Ranked Sector-wise $ R^2 $ Values}
\label{tab:sector_Rsquare}
\begin{tabular}{llc}
\toprule
\textbf{Rank} & \textbf{Sector}                & \textbf{$ R^2 $} \\ 
\midrule
1 & Utilities                 & 0.8370 \\ 
2 & Financials               & 0.6378 \\ 
3 & Health Care              & 0.6265 \\ 
4 & Real Estate              & 0.6022 \\ 
5 & Consumer Staples         & 0.5612 \\ 
6 & Consumer Discretionary   & 0.5482 \\ 
7 & Materials                & 0.5127 \\ 
8 & Energy                   & 0.5126 \\ 
9 & Information Technology   & 0.4651 \\ 
10 & Industrials              & 0.4404 \\ 
11 & Communication Services   & 0.4350 \\ 
\bottomrule
\end{tabular}
\end{table}


Next, we examine the overall impact of each covariate on the collective structure of the network.
More specifically, we consider six covariates: price-to-earnings (PE), price-to-sales (PS), price-to-book (PB), price-to-free-cash-flow (PFCF), debt-to-equity ratio (DER), and return on equity (ROE). For the $ j $-th covariate, we test the null hypothesis $ H_0: H[j, :] = H[:, j] = 0 $ to determine its significance. 

To perform the hypothesis test, we estimate $ \hat{H}_{\text{restricted}} $ and $ \hat{\Gamma}_{\text{restricted}} $, similar to the procedure for the full model. The key difference is the inclusion of an additional constraint, $ H[j, :] = H[:, j] = 0 $, which enforces the null hypothesis by setting the $ j $-th covariate's effect to zero. We then compute the objective function for both the full model (including all covariates) and the restricted model (with the null constraint applied). The objective function reflects the likelihood of the observed network under the model, regularized by the nuclear norm of $ \Gamma $.
The test statistic is calculated as $ \lambda_{\text{stat}} = 2 \times (\text{obj}_{\text{restricted}} - \text{obj}_{\text{full}}) $.
To assess significance, we randomly shuffle the $ j $-th column of the covariate matrix $ Z $ 1000 times, effectively decoupling the effect of the $ j $-th covariate. For each shuffled dataset, we compute the test statistic $ \lambda_{\text{stat\_shuffle}} $ using the same procedure. Table \ref{tab:null_test} reports the average, 95th percentile, and 99th percentile of $ \lambda_{\text{stat\_shuffle}} $ across these 1000 shuffles.  

We find that $ \lambda_{\text{stat}} $ statistics for PE, PS, PB, PFCF, and DER are significantly larger than the values obtained from the shuffled data, highlighting their statistical significance.  


\begin{table}[ht]
    \centering
    \renewcommand{\arraystretch}{1.2} 
    \caption{Test statistics results for six covariates}
    \begin{tabular}{lcccccc}
\hline
 & PE & PS & PB & PFCF & DER & ROE \\ \hline
$\lambda_{\text{stat}}$ & 190.74 & 1242.49 & 1046.17 & 1492.37 & 1277.67 & 36.71 \\ 
$\lambda_{\text{stat\_shuffle}}$-avg & 22.43 & 21.98 & 24.37 & 16.83 &  21.47 & 23.89\\ 
$\lambda_{\text{stat\_shuffle}}$-95\%quantile & 56.79 & 50.25 & 59.18 & 38.51 & 50.75 & 55.89\\ 
$\lambda_{\text{stat\_shuffle}}$-99\%quantile & 75.78 & 74.58 & 84.57 & 54.30 & 79.78 & 91.90\\ \hline
\end{tabular}
\label{tab:null_test}
\end{table}

%% file: appendix_preliminaries.tex
\section{Preliminaries}\label{sec:preliminaries}
As we have mentioned in Section \ref{sec:proof_idea}, our proof strategy leverages the analysis of nonconvex optimization. Since the condition number of Hessian matrix plays an important role in the analysis of gradient descent, we rescale our variables without changing the objective in the beginning to ensure the Hessian matrices involved in the analysis have small condition numbers. Specifically, we let
\begin{align*}
    &H_{\text{appendix}} = n H_{\text{original}}, \quad \Gamma_{\text{appendix}} = n \Gamma_{\text{original}}, \\
    &Z_{\text{appendix}} =  \frac{Z_{\text{original}}}{\sqrt{n}}, \quad \lambda_{\text{appendix}} =  \frac{\lambda_{\text{original}}}{n}.
\end{align*}
Note that this rescaling step does not change the value of objective at all. However, the Hessian matrices involved in our proof are now having balanced non-zero eigenvalues. We will use the variables with subscript `appendix' in the appendix sections, and the results proved in the appendix are transformed back to the original scale in the main body of this paper. For simplicity, we will omit the subscript `appendix' in the following content.

We define two types of logistic loss functions and their corresponding objectives. First, the nonconvex logistic loss is given by:
\begin{align}\label{obj:ncv_loss}
    L({H}, {X},  {Y})= \sum_{i\neq j}\log(1+e^{P_{ij}})  - A_{ij}P_{ij} ,
    \text{ where }P_{ij} = {z}_i^\top {H} {z}_j + \frac{( {X} {Y}^\top)_{ij}}{n}.
\end{align}
The nonconvex objective is defined as:
\begin{align}\label{obj:ncv}
    f( {H}, {X},  {Y}) = L({H}, {X},  {Y})+\frac{\lambda}{2}\left\| {X}\right\|_F^2 + \frac{\lambda}{2}\left\| {Y}\right\|_F^2.
\end{align}
Next, we introduce the convex logistic loss:
\begin{align}\label{obj:cv_loss}
    &L_c({H}, {\Gamma})= \sum_{i\neq j}\log(1+e^{P_{ij}})  - A_{ij}P_{ij}  ,
    \text{ where }P_{ij} =  {z}_i^\top {H} {z}_j + \frac{\Gamma_{ij}}{n}.
\end{align}
The convex objective is defined as:
\begin{align}\label{obj:cv}
    f_c( {\theta}, {H}, {\Gamma}) = L_c({H}, {\Gamma})+{\lambda}\left\| {\Gamma}\right\|_{*}.
\end{align}

We have the following proposition.
\begin{proposition}\label{prop:eigenvalues}
Suppose Assumption \ref{assumption:eigenvalues} holds. For rescaled $\{z_i\}^n_{i=1}$, we have 
$\sum_{1\leq i, j\leq n}
\vec(z_iz_j^{\top})\vec(z_iz_j^{\top})^{\top}$ is full rank and 
\begin{align*}
\sl\leq
\lambda_{\min}\left(\sum_{1\leq i, j\leq n}
 \vec(z_iz_j^\top)\vec(z_iz_j^\top)^{\top}
\right)
\leq 
\lambda_{\max}\left(\sum_{1\leq i, j\leq n}
 \vec(z_iz_j^\top)\vec(z_iz_j^\top)^{\top}
\right)
\leq \su.
\end{align*}
\end{proposition}
\begin{proof}[Proof of Proposition \ref{prop:eigenvalues}]
Note that
\begin{align*}
\vec(z_iz_j^\top)\vec(z_iz_j^\top)^{\top}
=(z_j \otimes z_i)(z_j \otimes z_i)^{\top}
=(z_j \otimes z_i)(z_j^{\top} \otimes z_i^{\top})
=(z_j z_j^{\top} )\otimes (z_i z_i^{\top} ).
\end{align*}
Thus, we have
\begin{align*}
\sum_{1\leq i, j\leq n}
 \vec(z_iz_j^\top)\vec(z_iz_j^\top)^{\top} = \sum_{1\leq i, j\leq n}(z_j z_j^{\top} )\otimes (z_i z_i^{\top} ) = (Z^{\top}Z)\otimes  (Z^{\top}Z).
\end{align*}
Consequently, after rescaling, it holds that
\begin{align*}
 &\lambda_{\max}\left(\sum_{1\leq i, j\leq n}
 \vec(z_iz_j^\top)\vec(z_iz_j^\top)^{\top}
\right)= \left(\lambda_{\max}(Z^{\top}Z)\right)^2\leq \su,\\
 &\lambda_{\min}\left(\sum_{1\leq i, j\leq n}
 \vec(z_iz_j^\top)\vec(z_iz_j^\top)^{\top}
\right)= \left(\lambda_{\min}(Z^{\top}Z)\right)^2\geq \sl.
\end{align*}
\end{proof}

%% file: apendix_local_geometry.tex
\section{Local geometry}\label{sec:geometry}
We define
\begin{align*}
f_{\aug}(H,X,Y):= f(H,X,Y)+\frac{c_{\aug}}{n^2}\|X^TX-Y^TY\|^2_{F},
\end{align*}
where $c_{\aug}=\frac{e^{2c_P}}{8(1+e^{2c_P})^2}$.
\begin{lemma}[Local geometry]\label{lem:convexity}
Let $\Delta=\begin{bmatrix}
\Delta_H\\ \Delta_X\\ \Delta_Y
\end{bmatrix}$ and 
\begin{align*}
\underline{C}:=\frac{e^{2c_P}}{(1+e^{2c_P})^2}\cdot\min\left\{\frac{\sl}{2},\frac{\sigma_{\min}}{20 n^2}\right\},\quad
\overline{C}:= \max\left\{\su,\frac{20\sigma_{\max}}{n^2}\right\}.
\end{align*}
Under Assumption \ref{assumption:scales}-\ref{assumption:incoherent}, with probability at least $1-n^{-10}$, we have
\begin{align*}
&\vec(\Delta)^T\nabla^2f_{\aug}(H,X,Y)\vec(\Delta)
\geq \underline{C}\|\Delta\|^2_F,\\
&\max\left\{\left\|\nabla^2f_{\aug}(H,X,Y)\right\|,\left\|\nabla^2f(H,X,Y)\right\|\right\}\leq \overline{C}
\end{align*}
for $(H,X,Y)$ and $\Delta$ obeying:
\begin{itemize}
    \item $\cP_Z(X)=\cP_Z(Y)=\cP_Z(\Delta_X)=\cP_Z(\Delta_Y)=0$.
    \item $\|H-H^{*}\|_{F}\leq c_2\sqrt{n},\ \left\|
\begin{bmatrix}
XR-X^*\\
YR-Y^*
\end{bmatrix}\right\|_{2,\infty}\leq c_3$.

\item
$\begin{bmatrix}\Delta_X\\\Delta_Y\end{bmatrix}$ lying in the set
\begin{align*}
\left\{
\begin{bmatrix}
X_1\\Y_1
\end{bmatrix}
\hat{R}-
\begin{bmatrix}
X_2\\Y_2
\end{bmatrix}\bigg|
\left\|\begin{bmatrix}
X_2-X^*\\Y_2-Y^*
\end{bmatrix}\right\|\leq c_4 \sqrt{n},
\ \hat{R}:=\argmin_{R\in\cO^{r\times r}}
\left\|\begin{bmatrix}
X_1\\Y_1
\end{bmatrix}
R-
\begin{bmatrix}
X_2\\Y_2
\end{bmatrix}\right\|_F
\right\}.
\end{align*}

\end{itemize}

\end{lemma}
\begin{proof}
See Appendix \ref{sec:pf_geometry}.
\end{proof}




%% file: appendix_nonconvex_iterates.tex
\section{Properties of the nonconvex iterates}\label{sec:nonconvex_iterates}

In this section, we study the gradient descent starting from the ground truth $(\theta^*,H^*,X^*,Y^*)$. Note that this algorithm cannot be implemented in practice because we do not have access to the ground truth parameters. More specifically, we consider the following algorithm:
\begin{algorithm}[H]
\caption{Gradient Descent}\label{GD}
\begin{algorithmic}[1]
\STATE {\bf Initialize:} $H^0=H^{*}$, $X^0=X^{*}$, and $Y^0=Y^{*}$
\FOR{$t=0,\ldots,T-1$}
\STATE Update
\begin{align}\label{alg:update_M}
&H^{t+1}:=H^{t}-\eta\nabla_{H}f(H^t,F^t)\\\label{alg:update_F}
&F^{t+1}:=
\begin{bmatrix}
X^{t+1}\\
Y^{t+1}
\end{bmatrix}=\
\begin{bmatrix}
\cP_Z^{\perp}\left(X^{t}-\eta\nabla_{X}f(H^t,F^t)\right)\\
\cP_Z^{\perp}\left(Y^{t}-\eta\nabla_{Y}f(H^t,F^t)\right)
\end{bmatrix}
\end{align}
\ENDFOR
\end{algorithmic}
\end{algorithm}
Here $\eta$ is the step-size, and the update in \eqref{alg:update_F} is to guarantee that $\cP_Z(X^{t+1})=\cP_Z(Y^{t+1})=0$ always holds on the trajectory.

\paragraph{Leave-one-out objective} For $1\leq m\leq n$, we define the following leave-one-out objective
\begin{align*}
L^{(m)}(H,X,Y)&=\sum_{\substack{
        i \neq j \\
        i, j \neq m
    }}\log(1+e^{P_{ij}})  - A_{ij}P_{ij}\\
    &+\quad \sum_{i\neq m}\left\{\log(1+e^{P_{im}})+\log(1+e^{P_{mi}})-\frac{e^{P^{*}_{im}}}{1+e^{P^{*}_{im}}}P_{im}-\frac{e^{P^{*}_{mi}}}{1+e^{P^{*}_{mi}}}P_{mi}\right\}
\end{align*}
and $f^{(m)}$, $f^{(m)}_{\aug}$, $H^{t+1,(m)}$, $F^{t+1,(m)}$ are defined correspondingly.

\paragraph{Properties}
Let
\begin{align*}
R^t&:=\argmin_{R\in\cO^{r\times r}}\|F^tR-F^*\|_F,\\
R^{t,(m)}&:=\argmin_{R\in\cO^{r\times r}}\|F^{t,(m)}R-F^*\|_F,\\
O^{t,(m)}&:=\argmin_{R\in\cO^{r\times r}}\|F^{t,(m)}R-F^tR^t\|_F.
\end{align*}
We will inductively prove the following lemmas.

\begin{lemma}\label{lem:ncv1}
Suppose Assumption \ref{assumption:scales}-\ref{assumption:2_infty} holds.
For all $0\leq t\leq t_0$, we have
\begin{align*}
\left\|
\begin{bmatrix}
H^t-H^{*}\\
F^tR^t-F^{*}
\end{bmatrix}
\right\|_F\leq c_{11}\sqrt{n}.
\end{align*}
\end{lemma}
\begin{proof}
See Appendix \ref{pf:lem_ncv1}.
\end{proof}

\begin{lemma}\label{lem:ncv2}
Suppose Assumption \ref{assumption:scales}-\ref{assumption:2_infty} holds.
For all $0\leq t\leq t_0$, we have
\begin{align*}
\max_{1\leq m\leq n}
\left\|
\begin{bmatrix}
H^{t,(m)}-H^t\\
F^{t,(m)}O^{t,(m)}-F^tR^t
\end{bmatrix}
\right\|_F\leq c_{21}.
\end{align*}
\end{lemma}
\begin{proof}
See Appendix \ref{pf:lem_ncv2}.
\end{proof}

\begin{lemma}\label{lem:ncv3}
Suppose Assumption \ref{assumption:scales}-\ref{assumption:2_infty} holds.
For all $0\leq t\leq t_0$, we have
\begin{align*}
\max_{1\leq m\leq n}
\left\|
\begin{bmatrix}
\left(F^{t,(m)}R^{t,(m)}-F^{*}\right)_{m,\cdot}\\
\left(F^{t,(m)}R^{t,(m)}-F^{*}\right)_{m+n,\cdot}
\end{bmatrix}
\right\|_2\leq c_{31}
\end{align*}
\end{lemma}
\begin{proof}
See Appendix \ref{pf:lem_ncv3}.
\end{proof}

\begin{lemma}\label{lem:ncv4}
Suppose Assumption \ref{assumption:scales}-\ref{assumption:2_infty} holds.
For all $0\leq t\leq t_0$, we have
\begin{align*}
\|H^t-H^{*}\|_F\leq c_{11}\sqrt{n},\ \|F^tR^t-F^{*}\|_{2,\infty}\leq c_{41},
\end{align*}
where $c_{41}=5\kappa c_{21}+c_{31}$.
\end{lemma}
\begin{proof}
See Appendix \ref{pf:lem_ncv4}.
\end{proof}

\begin{lemma}\label{lem:ncv5}
Suppose Assumption \ref{assumption:scales}-\ref{assumption:2_infty} holds.
For all $0\leq t\leq t_0$, we have
\begin{align}
&\|X^{tT}X^t-Y^{tT}Y^t\|_F\leq c_{51}\eta n^2,\ \|(X^{t,(m)})^TX^{t,(m)}-(Y^{t,(m)})^TY^{t,(m)}\|_F\leq c_{51}\eta n^2\notag\\\label{ineq:gradient}
&f(H^t,F^t)\leq f(H^{t-1},F^{t-1})-\frac{\eta}{2}
\left\|
\cP\nabla f(H^{t-1},F^{t-1})
\right\|_2^2.
\end{align}
\end{lemma}
\begin{proof}
See Appendix \ref{pf:lem_ncv5}.
\end{proof}

\begin{lemma}\label{lem:ncv6}
If Lemma \ref{lem:ncv1}-Lemma \ref{lem:ncv4} hold for all $0\leq t\leq t_0$ and Lemma \ref{lem:ncv5} holds for all $1\leq t\leq t_0$, we then have
\begin{align*}
\min_{0\leq t< t_0}
\left\|
\cP\nabla f(H^{t},F^{t})
\right\|_2
\lesssim n^{-5},
\end{align*}
as long as $\eta t_0\geq n^{12}$.
\end{lemma}
\begin{proof}
See Appendix \ref{pf:lem_ncv6}.
\end{proof}

%% file: appendix_nonconvex_debias.tex
\section{Properties of debiased nonconvex estimator}\label{sec:nonconvex_debias}

Let 
\begin{align*}
&t^*:=\argmin_{0\leq t<t_0}\left\|
\cP\nabla f(H^{t},F^{t})
\right\|_2.
\end{align*}
And we denote $(\hat{H},\hat{X},\hat{Y})=({H}^{t^{*}},{X}^{t^{*}}R^{t^{*}},{Y}^{t^{*}}R^{t^{*}})$. It then holds that
\begin{align}\label{ncv:est}
&\|\hat{H}-H^{*}\|_F\leq c_{11}\sqrt{n},\ \|\hat{X}-X^{*}\|_{2,\infty}\leq c_{41},\
\|\hat{Y}-Y^{*}\|_{2,\infty}\leq c_{41},\\\label{ncv:grad}
&\left\|\cP\nabla f(\hat{H},\hat{X},\hat{Y})
\right\|_2\lesssim n^{-5}.
\end{align}
Moreover, we have
\begin{align*}
\cP\vec
\begin{bmatrix}
\hat{H}\\
\hat{X}\\
\hat{Y}
\end{bmatrix}
=\vec
\begin{bmatrix}
\hat{H}\\
\hat{X}\\
\hat{Y}
\end{bmatrix}.
\end{align*}

Let
\begin{align*}
\hat{D}:=
\sum_{i\neq j}\frac{e^{\hat{P}_{ij}}}{(1+e^{\hat{P}_{ij}})^2}
\left(\vec\begin{bmatrix}
z_i z_j^{\top}\\
\frac1ne_ie_j^\top\hat{Y}\\
\frac1ne_je_i^\top\hat{X}
\end{bmatrix}\right)\left(\vec\begin{bmatrix}
z_i z_j^{\top}\\
\frac1ne_ie_j^\top\hat{Y}\\
\frac1ne_je_i^\top\hat{X}
\end{bmatrix}\right)^\top.
\end{align*}
We define the debiased estimator $(\hat{H}^d,\hat{X}^d,\hat{Y}^d)$ as
\begin{align}\label{def:debias}
\vec
\begin{bmatrix}
\hat{H}^d-\hat{H}\\
\hat{X}^d-\hat{X}\\
\hat{Y}^d-\hat{Y}
\end{bmatrix}:=-(\cP\hat{D}\cP)^{\dagger}\cP
\nabla L(\hat{H},\hat{X},\hat{Y}),
\end{align}
which then satisfies:
\begin{align}\label{property:debias}
\cP\left(\nabla L(\hat{H},\hat{X},\hat{Y})+\hat{D}
\vec
\begin{bmatrix}
\hat{H}^d-\hat{H}\\
\hat{X}^d-\hat{X}\\
\hat{Y}^d-\hat{Y}
\end{bmatrix}
\right)=0
\text{ and }
\cP\vec
\begin{bmatrix}
\hat{H}^d\\
\hat{X}^d\\
\hat{Y}^d
\end{bmatrix}=
\vec
\begin{bmatrix}
\hat{H}^d\\
\hat{X}^d\\
\hat{Y}^d
\end{bmatrix}.
\end{align}
Here the second condition leads to the fact that $\cP_Z(\hat{X}^d)=\cP_Z(\hat{Y}^d)=0$.

Similarly, let
\begin{align*}
{D}^*:=
\sum_{i\neq j}
\frac{e^{{P}^*_{ij}}}{(1+e^{{P}^*_{ij}})^2}
\left(\vec\begin{bmatrix}
z_i z_j^{\top}\\
\frac1ne_ie_j^\top{Y}^*\\
\frac1ne_je_i^\top{X}^*
\end{bmatrix}\right)\left(\vec\begin{bmatrix}
z_i z_j^{\top}\\
\frac1ne_ie_j^\top{Y}^*\\
\frac1ne_je_i^\top{X}^*
\end{bmatrix}\right)^\top.
\end{align*}
We define 
\begin{align}\label{def:bar}
\vec
\begin{bmatrix}
\bar{H}-H^*\\
\bar{X}-X^*\\
\bar{Y}-Y^*
\end{bmatrix}:=-(\cP D^*\cP)^{\dagger}\cP
\nabla L(H^*,X^*,Y^*),
\end{align}
which then satisfies
\begin{align}\label{property:bar}
\cP\left(
\nabla L(H^*,X^*,Y^*)+D^*
\vec
\begin{bmatrix}
\bar{H}-H^*\\
\bar{X}-X^*\\
\bar{Y}-Y^*
\end{bmatrix}
\right)=0
\text{ and }
\cP \vec \begin{bmatrix}
\bar{H}\\
\bar{X}\\
\bar{Y}
\end{bmatrix}
=\vec
\begin{bmatrix}
\bar{H}\\
\bar{X}\\
\bar{Y}
\end{bmatrix}.
\end{align}
Here the second condition leads to the fact that $\cP_Z(\bar{X})=\cP_Z(\bar{Y})=0$.

The distance between the debiased estimator and the original estimator can be captured by the following proposition.
\begin{proposition}\label{prop:debias_dist}
We have
\begin{align*}
 \left\|\begin{bmatrix}
\hat{H}^d-\hat{H}\\
\hat{X}^d-\hat{X}\\
\hat{Y}^d-\hat{Y}
\end{bmatrix}\right\|_F
\leq  c_a \sqrt{n},
\quad \text{where} \quad
c_a\asymp \frac{\lambda}{\slD}\sqrt{\frac{\mu r\sigma_{\max}}{n}}.
\end{align*}
\end{proposition}
\begin{proof}
See Appendix \ref{pf:prop_debias_dist}.
\end{proof}

Similarly, we have the following proposition.
\begin{proposition}\label{prop:bar_dist}
We have
\begin{align*}
 \left\|\begin{bmatrix}
 \bar{H}-{H^{*}}\\
\bar{X}-{X^{*}}\\
\bar{Y}-{Y^{*}}
\end{bmatrix}\right\|_F
\leq  c'_a \sqrt{n},
\quad \text{where} \quad
c'_a\asymp \frac{\sqrt{\mu r \sigma_{\max}\log n}}{\slD n}.
\end{align*}
Moreover, it holds that
\begin{align*}
 \left\|\begin{bmatrix}
 \bar{H}-{H^{*}}\\
\bar{X}-{X^{*}}\\
\bar{Y}-{Y^{*}}
\end{bmatrix}\right\|_{\infty}
\leq  c_b ,
\quad \text{where} \quad
c_b\asymp  \sqrt{\frac{(1+e^{c_P})^2\log n}{\slD e^{c_P}}}.
\end{align*}
\end{proposition}
\begin{proof}
See Appendix \ref{pf:prop_bar_dist}. 
\end{proof}

We can then establish the following theorem.
\begin{theorem}\label{thm:ncv_debias}
Under Assumption \ref{assumption:2_infty}, it holds that
\begin{align*}
\left\|
\begin{bmatrix}
 \hat{H}^d-\bar{H}\\
\frac{1}{n}\left(\hat{X}^{d}(\hat{Y}^{d})^T-\bar{X}\bar{Y}^T\right)
\end{bmatrix}
\right\|_F
\lesssim c_{d} n^{1/4},
\end{align*}
where 
\begin{align*}
c_d\asymp\sqrt{\frac{2(1+e^{c_P})^2}{ \sl e^{c_P}}}\left(\frac{\su \mu r\sigma_{\max}}{n^2}\right)^{1/4}(c_a+c_{11})^{3/2}.
\end{align*}
\end{theorem}
\begin{proof}
See Appendix \ref{pf:ncv_debias}. 
\end{proof}


%% file: pf_geometry.tex
\section{Proofs of Section \ref{sec:geometry}}\label{sec:pf_geometry}

\paragraph{Observations}
Based on the constraints on $(H,X,Y)$, it can be seen that: 
\begin{align*}
 &\|XR-X^*\|\leq \|XR-X^*\|_{F}\leq \sqrt{n}\|XR-X^*\|_{2,\infty}\leq c_3 \sqrt{n}\\
 &\|YR-Y^*\|\leq \|YR-Y^*\|_{F}\leq \sqrt{n}\|YR-Y^*\|_{2,\infty}\leq c_3 \sqrt{n}.
\end{align*}
This further implies that
\begin{align*}
&\|XY^T-X^*Y^{*T}\|_{F}\\
&=\|(XR-X^*)(YR)^T+X^{*}(YR-Y^*)^T\|_F\\
&\leq \|XR-X^*\|_F\|YR\|+\|X^{*}\|\|YR-Y^*\|_F\\
&\leq \|XR-X^*\|_F\|YR-Y^*\|+\|XR-X^*\|_F\|Y^*\|+\|X^{*}\|\|YR-Y^*\|_F\\
&\leq 3 c_3\sqrt{\sigma_{\max}n},
\end{align*}
where we use the fact that $\|X^*\|=\|Y^*\|=\sqrt{\sigma_{\max}}\geq c_3\sqrt{n}$. Moreover, we have
\begin{align*}
&\|XY^T-X^*Y^{*T}\|_{\infty}\\
&=\|(XR-X^*)(YR)^T+X^{*}(YR-Y^*)^T\|_{\infty}\\
&\leq \|XR-X^*\|_{2,\infty}\|YR\|_{2,\infty}+\|X^{*}\|_{2,\infty}\|YR-Y^*\|_{2,\infty} \tag{Cauchy}\\
&\leq \|XR-X^*\|_{2,\infty}\|YR-Y^*\|_{2,\infty}+ \|XR-X^*\|_{2,\infty}\|Y^*\|_{2,\infty}+\|X^{*}\|_{2,\infty}\|YR-Y^*\|_{2,\infty} \\
&\leq 3c_3 \sqrt{\frac{\mu r \sigma_{\max}}{n}},
\end{align*}
where we use the fact that $c_3, \|X^*\|_{2,\infty}, \|Y^*\|_{2,\infty}\leq \sqrt{\frac{\mu r \sigma_{\max}}{n}}$. Thus, we obtain
\begin{align*}
|P_{ij}|&\leq |P^*_{ij}|+|P_{ij}-P^*_{ij}|\\
&\leq |P^*_{ij}|+\|H-H^*\|\|z_i\|\|z_j\|+\frac1n \|XY^T-X^*Y^{*T}\|_{\infty}\\
&\leq |P^*_{ij}|+\frac{\sz}{n}\|H-H^*\|+\frac1n \|XY^T-X^*Y^{*T}\|_{\infty}\\
&\leq |P^*_{ij}|+\frac{1}{\sqrt{n}}\left(\sz c_2+\frac{3c_3 \sqrt{\mu r \sigma_{\max}}}{n}\right).
\end{align*}
Based on Assumption \ref{assumption:scales}, we know $|P^*_{ij}|\leq c_P$ and this leads to the fact that $|P_{ij}|\leq 2c_P$ as long as $n\gg 1/ c^2_P$.
We will use the above observations in the following proofs.

\begin{lemma}\label{lem:POmega}
Define $P_\Omega(\cdot)$ as 
\begin{align*}
    [P_\Omega(A)]_{ij} = \begin{cases}
      A_{ij}, &\quad  i\neq j \\
      0, &\quad i= j,
    \end{cases} 
    \quad \quad A\in \mathbb{R}^{n\times n},
\end{align*}
which removes the diagonal entries of $n\times n$ matrices. Consider $X, Y\in \mathbb{R}^{n\times r}$ satisfies 
\begin{align*}
    \left\|\begin{bmatrix}
X\\
Y
\end{bmatrix}R - \begin{bmatrix}
X^*\\
Y^*
\end{bmatrix}\right\|_{2, \infty} \leq \frac{1}{6}\sqrt{\frac{\sigma_{\min}}{\kappa n}}
\end{align*}
for a rotation matrix $R\in \mathcal{O}^{r\times r}$ and let $\cT$ be the tangent space of $XY^\top$
\begin{align*}
    \cT = \left\{UA^\top + BV^\top | A, B\in \mathbb{R}^{n\times r}\right\},
\end{align*}
where $XY^\top = U\Sigma V^\top$ is the SVD of $XY^\top$. We denote by $P_\cT$ the projection operator which projects $n\times n$ matrices to space $\cT$. Then we have
\begin{align*}
    \left\|P_\Omega( P_\cT (A))\right\|_F\geq \frac{9}{10}\left\|P_\cT (A)\right\|_F, \quad \forall A\in \mathbb{R}^{n\times n},
\end{align*}
as long as $n\gg \kappa^2\mu r$. As a directly corollary, we have 
\begin{align*}
    \sum_{i=1}^n(( P_\cT (A))_{ii}^2\leq \frac{1}{5}\left\|P_\cT (A)\right\|_F^2, \quad \forall A\in \mathbb{R}^{n\times n},
\end{align*}
\end{lemma}
\begin{proof}
    Without loss of generality we can assume $A\in \cT$, otherwise we can place $A$ with $P_\cT (A)$ and the statement is not affected. Then the statement can be written as
    \begin{align*}
        \left\|P_\Omega( A)\right\|_F\geq \frac{9}{10}\left\|A\right\|_F, \quad \forall A\in \cT.
    \end{align*}
    It is equivalent to show 
    \begin{align*}
        \sum_{i=1}^n A_{ii}^2 \leq \frac{19}{100}\left\|A\right\|_F^2, \quad \forall A\in \cT.
    \end{align*}
    Again since $P_\cT(A) = A$, the above statement is also equivalent to
    \begin{align*}
        \sum_{i=1}^n [P_T(A)]_{ii}^2 \leq \frac{19}{100}\left\|A\right\|_F^2, \quad \forall A\in \cT.
    \end{align*}
    To verify this, we begin with the explicit expression $P_\cT(A) = UU^\top A + AVV^\top - UU^\top A VV^\top$. Then we have
    \begin{align}
        \sum_{i=1}^n [P_\cT(A)]_{ii}^2 &= \sum_{i=1}^n [UU^\top A + AVV^\top - UU^\top A VV^\top]_{ii}^2 \notag\\
        &\leq 2\sum_{i=1}^n [UU^\top A(I - VV^\top)]_{ii}^2 + [AVV^\top]_{ii}^2 \notag\\
        &\leq 2\sum_{i=1}^n \left\|[UU^\top]_{i, \cdot}\right\|_2^2 \left\|[A(I - VV^\top)]_{\cdot,i}\right\|_2^2 + \left\|[A]_{i, \cdot}\right\|_2^2 \left\|[VV^\top]_{\cdot,i}\right\|_2^2 \notag\\
        &\leq 2\left\|UU^\top\right\|_{2,\infty}^2 \left\|A(I - VV^\top)\right\|_F^2 + 2\left\|A\right\|_F^2 \left\|VV^\top\right\|_{2,\infty}^2 \notag\\
        &\leq 2\left\|A\right\|_F^2\left(\left\|U\right\|_{2,\infty}^2 + \left\|V\right\|_{2,\infty}^2\right).\label{lemPOmegaproofeq1}
    \end{align}
    
    It remains to control $\|U\|_{2,\infty}$ and $\|V\|_{2,\infty}$. By definition we can write $U = XY^\top V\Sigma^{-1}$. As a result, we have
    \begin{align}
        \left\|U\right\|_{2,\infty} &\leq\left\|X\right\|_{2,\infty} \left\|Y\right\| \left\|V\right\|\left\|\Sigma^{-1}\right\| = \left\|X\right\|_{2,\infty} \left\|Y\right\| \left\|\Sigma^{-1}\right\| \notag\\
        &\leq \frac{\left(\left\|X^*\right\|_{2,\infty} + \left\|XR - X^*\right\|_{2,\infty}\right)\left(\left\|Y^*\right\| + \left\|YR - Y^*\right\|\right)}{\sigma_{\min} - \left\|XY^\top - \Gamma^*\right\|} \notag\\
        &\leq \frac{(2\sqrt{\sigma_{\max}\mu r / n})(2\sqrt{\sigma_{\max}})}{0.5\sigma_{\min} } = 8\kappa \sqrt{\frac{\mu r}{n}},\label{lemPOmegaproofeq2}
    \end{align}
    since $\left\|XR - X^*\right\|_{2,\infty}\leq \sqrt{\sigma_{\max}\mu r/n}$, $\left\|YR - Y^*\right\|\leq \sqrt{\sigma_{\max}}$ and 
    \begin{align*}
        \left\|XY^\top - \Gamma^*\right\|&\leq \left\|XR - X^*\right\| \left\|YR\right\| + \left\|X^*\right\|\left\|YR - Y^*\right\| \\
        &\leq \left\|XR - X^*\right\| (\left\|Y^*\right\| + \left\|YR- Y^*\right\|) + \left\|X^*\right\|\left\|YR - Y^*\right\| \\
        &\leq 3\sqrt{\sigma_{\max}} \sqrt{n}\left\|\begin{bmatrix}
X\\
Y
\end{bmatrix}R - \begin{bmatrix}
X^*\\
Y^*
\end{bmatrix}\right\|_{2, \infty} \leq \frac{1}{2}\sigma_{\min}.
    \end{align*}
    Similarly, for $V$ we also have $\left\|V\right\|_{2,\infty}\leq 8\kappa \sqrt{\mu r /n}$.
    
    Combine \eqref{lemPOmegaproofeq1} and \eqref{lemPOmegaproofeq2} we get
    \begin{align*}
        \sum_{i=1}^n [P_\cT(A)]_{ii}^2\leq \frac{256\kappa^2 \mu r}{n}\left\|A\right\|_F^2.
    \end{align*}
    Therefore, we have 
    \begin{align*}
        \sum_{i=1}^n A_{ii}^2 \leq \frac{19}{100}\left\|A\right\|_F^2, \quad \forall A\in \cT,
    \end{align*}
    as long as $n\gg \kappa^2\mu r$.
    
\end{proof}

\begin{lemma}\label{lem:balance}
It holds that
\begin{align*}
&\left(\frac{\sigma_{\min}}{4}-10(c_3+c_4) \sqrt{n\sigma_{\max}}\right)\left(\|\Delta_X\|_F^2+\|\Delta_Y\|_F^2\right)\\
&\leq  
\|\Delta_{X}Y^T+X\Delta_{Y}^T\|^2_F+\frac{1}{4}\|\Delta_X^T X+X^T\Delta_X-\Delta_Y^T Y-Y^T\Delta_Y\|_F^2\\
&\leq 16 \sigma_{\max}\left(\|\Delta_X\|_F^2+\|\Delta_Y\|_F^2\right).
\end{align*}
\end{lemma}
\begin{proof}
    To show the upper bound, note that 
\begin{align*}
&\|\Delta_{X}Y^T+X\Delta_{Y}^T\|^2_F\leq \left(\|Y\|\|\Delta_X\|_F+\|X\|\|\Delta_Y\|_F\right)^2   \leq \left(\|X\|^2+\|Y\|^2\right)\left(\|\Delta_X\|_F^2+\|\Delta_Y\|_F^2\right),\\
 &\|\Delta_X^T X+X^T\Delta_X-\Delta_Y^T Y-Y^T\Delta_Y\|_F^2\leq 4\left(\|X\|^2+\|Y\|^2\right)\left(\|\Delta_X\|_F^2+\|\Delta_Y\|_F^2\right).
\end{align*}
Thus, it holds that
\begin{align*}
&\|\Delta_{X}Y^T+X\Delta_{Y}^T\|^2_F+\frac{1}{4}\|\Delta_X^T X+X^T\Delta_X-\Delta_Y^T Y-Y^T\Delta_Y\|_F^2\\
&\leq
2\left(\|X\|^2+\|Y\|^2\right)\left(\|\Delta_X\|_F^2+\|\Delta_Y\|_F^2\right).
\end{align*}
By Weyl's inequality, we have
\begin{align*}
|\sigma_{\max}(X)-\sigma_{\max}(X^{*})|\leq \|XR-X^{*}\|\leq c_3 \sqrt{n},\  |\sigma_{\max}(Y)-\sigma_{\max}(Y^{*})|\leq \|YR-Y^{*}\|\leq c_3 \sqrt{n}.
\end{align*}
As long as $c_3\sqrt{n}\ll \sqrt{\sigma_{\max}}$, we have 
\begin{align*}
\sigma_{\max}(X)\leq 2\sqrt{\sigma_{\max}},\ 
\sigma_{\max}(Y)\leq 2\sqrt{\sigma_{\max}},
\end{align*}
which implies
\begin{align*}
\|\Delta_{X}Y^T+X\Delta_{Y}^T\|^2_F+\frac{1}{4}\|\Delta_X^T X+X^T\Delta_X-\Delta_Y^T Y-Y^T\Delta_Y\|_F^2  \leq 16 \sigma_{\max} \left(\|\Delta_X\|_F^2+\|\Delta_Y\|_F^2\right).
\end{align*}

For the lower bound, note that 
\begin{align*}
&\left\| \Delta_X Y^{T} + X \Delta_Y^{T} \right\|_F^2 + \frac{1}{4} \left\| \Delta_X^{T} X + X^{T} \Delta_X -\Delta_Y^T Y-Y^T\Delta_Y\right\|_F^2 \\
&= \left\| \Delta_X Y^{T} \right\|_F^2 + \left\| X \Delta_Y^{T} \right\|_F^2 + \frac12\left\| \Delta_X^{T} X \right\|_F^2 + \frac{1}{2} \left\| \Delta_Y^{T} Y \right\|_F^2 - \left\langle \Delta_X^{T} X, \Delta_Y^{T} Y \right\rangle  \\
&\quad+\frac{1}{2} \left\langle X^{T} \Delta_X, \Delta_X^{T} X \right\rangle+ \frac{1}{2} \left\langle Y^{T} \Delta_Y, \Delta_Y^{T} Y \right\rangle+ \left\langle \Delta_X^T X, Y^{T} \Delta_Y \right\rangle \\
&= \left\| \Delta_X Y^{T} \right\|_F^2 + \left\| X \Delta_Y^{T} \right\|_F^2 + \frac12\left\| \Delta_X^{T} X - \Delta_Y^{T} Y \right\|_F^2 + \frac{1}{2} \left\langle X^{T} \Delta_X + Y^T\Delta_Y, \Delta_X^{T} X + \Delta_Y^{T} Y \right\rangle \\
&= \left\| \Delta_X Y^{T} \right\|_F^2 + \left\| X \Delta_Y^{T} \right\|_F^2 + \frac12\left\| \Delta_X^{T} X - \Delta_Y^{T} Y \right\|_F^2 \\
&\quad+ \frac{1}{2} \left\langle (XR)^{T} (\Delta_X R)+(YR)^T(\Delta_Y R), (\Delta_X R)^{T} (XR) + (\Delta_YR)^{T} (YR) \right\rangle \\
&= \left\| \Delta_X Y^{T} \right\|_F^2 + \left\| X \Delta_Y^{T} \right\|_F^2 + \frac12\left\| \Delta_X^{T} X - \Delta_Y^{T} Y \right\|_F^2 \\
&\quad+ \frac{1}{2} \left\langle X_2^{T} (\Delta_X R)+Y_2^T(\Delta_Y R), (\Delta_X R)^TX_2 + (\Delta_Y R)^T Y_2 \right\rangle + \mathcal{E}_1 \\
&=  \left\| \Delta_X Y^{T} \right\|_F^2 + \left\| X \Delta_Y^{T} \right\|_F^2 + \frac12\left\| \Delta_X^{T} X - \Delta_Y^{T} Y \right\|_F^2  + \frac12\left\| X_2^{T} \Delta_X +Y_2^T\Delta_Y  \right\|_F^2 + \mathcal{E}_1,
\end{align*}
where 
\begin{align*}
\mathcal{E}_1 &:=  \left\langle (XR - X_2)^{T} \Delta_X R + (YR - Y_2)^{T} \Delta_Y R, (\Delta_X R)^{T} X_2 + (\Delta_Y R)^{T} Y_2 \right\rangle \\
& \quad+ \frac{1}{2} \left\langle (X R- X_2)^{T} \Delta_X R + (YR - Y_2)^{T} \Delta_Y R, (\Delta_X R)^{T} (XR - X_2) + (\Delta_Y R)^{T} (YR - Y_2) \right\rangle.
\end{align*}
Based on the fact that $|\langle A,B\rangle|\leq\|A\|_F\|B\|_F$ and $\|AB\|_F\leq \|A\|\|B\|_F$, we have
\begin{align*}
|\mathcal{E}_1|\leq 
\left((\|XR-X_2\|+\|YR-Y_2\|)(\|X_2\|+\|Y_2\|)+\frac12 (\|XR-X_2\|+\|YR-Y_2\|)^2\right)\left(\|\Delta_X\|^2_F+\|\Delta_Y\|^2_F\right).
\end{align*}
By the definition of $X_2,Y_2$, as long as $c_4\sqrt{n}\ll \sqrt{\sigma_{\max}}$, we have 
\begin{align*}
\|X_2\|\leq \|X_2-X^*\|+\|X^*\|\leq 2 \sqrt{\sigma_{\max}},\ \|Y_2\|\leq \|Y_2-X^*\|+\|Y^*\|\leq 2 \sqrt{\sigma_{\max}}.
\end{align*}
Moreover, on the observations, we have
\begin{align*}
&\|XR-X_2\|\leq \|XR-X^*\|+\|X_2-X^*\|\leq (c_3+c_4)\sqrt{n},\notag\\
&\|YR-Y_2\|\leq \|YR-Y^*\|+\|Y_2-Y^*\|\leq (c_3+c_4) \sqrt{n}.
\end{align*}
Thus, we have
\begin{align*}
 |\mathcal{E}_1|\leq 10(c_3+c_4) \sqrt{n\sigma_{\max}}  \left(\|\Delta_X\|^2_F+\|\Delta_Y\|^2_F\right).
\end{align*}
As a result, we have
\begin{align*}
 &\left\| \Delta_X Y^{T} + X \Delta_Y^{T} \right\|_F^2 + \frac{1}{4} \left\| \Delta_X^{T} X + X^{T} \Delta_X -\Delta_Y^T Y-Y^T\Delta_Y\right\|_F^2 \\
 &\geq  \left\| \Delta_X Y^{T} \right\|_F^2 + \left\| X \Delta_Y^{T} \right\|_F^2  - |\mathcal{E}_1|\\
 &\geq \frac{\sigma_{\min}}{4}\left(\left\| \Delta_X \right\|_F^2 + \left\|  \Delta_Y \right\|_F^2 \right)- |\mathcal{E}_1|\\
 &\geq \left(\frac{\sigma_{\min}}{4}-10(c_3+c_4) \sqrt{n\sigma_{\max}}  \right)\left(\left\| \Delta_X \right\|_F^2 + \left\|  \Delta_Y \right\|_F^2 \right).
\end{align*}
Here we use Weyl's inequality to obtain that
\begin{align*}
|\sigma_{\min}(X)-\sigma_{\min}(X^{*})|\leq \|XR-X^{*}\|\leq c_3 \sqrt{n},\  |\sigma_{\min}(Y)-\sigma_{\min}(Y^{*})|\leq \|YR-Y^{*}\|\leq c_3 \sqrt{n}.
\end{align*}
As long as $c_3\sqrt{n}\ll \sqrt{\sigma_{\min}}$, we have
\begin{align*}
  \sigma_{\min}(X)\geq \frac12\sqrt{\sigma_{\min}}, \ \sigma_{\min}(Y)\geq\frac12\sqrt{\sigma_{\min}}.
\end{align*}
\end{proof}

\begin{proof}[Proof of Lemma \ref{lem:convexity}]
According to the definition of $\nabla^2 L(H,X,Y)$, we have
\begin{align*}
\vec(\Delta)^T\nabla^2 L(H,X,Y)\vec(\Delta)
&=\sum_{i\neq j}\frac{e^{P_{ij}}}{(1+e^{P_{ij}})^2}\left(\langle\Delta_{H},z_iz_j^T\rangle+\frac1n\langle\Delta_{X}Y^T+X\Delta_{Y}^T,e_i e_j^T\rangle\right)^2\\
&\quad +\frac2n\sum_{i\neq j}\left(\frac{e^{P_{ij}}}{1+e^{P_{ij}}}-A_{ij}\right)\langle\Delta_X,e_ie_j^T\Delta_Y\rangle.
\end{align*}
Thus, it holds that
\begin{align}\label{eq:convexity}
&\vec(\Delta)^T\nabla^2 f_{\aug}(H,X,Y)\vec(\Delta)\notag\\
=&\vec(\Delta)^T\nabla^2 L(H,X,Y)\vec(\Delta)
+\lambda\|\Delta_X\|_F^2+\lambda\|\Delta_Y\|_F^2+\frac{2c_{\aug}}{n^2}\|\Delta_X^T X+X^T\Delta_X-\Delta_Y^T Y-Y^T\Delta_Y\|_F^2\notag\\
=&\sum_{i\neq j}\frac{e^{P_{ij}}}{(1+e^{P_{ij}})^2}\left(\langle\Delta_{H},z_iz_j^T\rangle+\frac1n\langle\Delta_{X}Y^T+X\Delta_{Y}^T,e_i e_j^T\rangle\right)^2+\lambda\|\Delta_X\|_F^2+\lambda\|\Delta_Y\|_F^2\notag\\
&+\frac{2c_{\aug}}{n^2}\|\Delta_X^T X+X^T\Delta_X-\Delta_Y^T Y-Y^T\Delta_Y\|_F^2+\frac2n\sum_{i\neq j}\left(\frac{e^{P_{ij}}}{1+e^{P_{ij}}}-A_{ij}\right)\langle\Delta_X,e_ie_j^T\Delta_Y\rangle.   
\end{align}
We first deal with the last term, which is the only term that contains $A_{ij}$ and thus has randomness. Note that
\begin{align*}
& \frac2n\sum_{i\neq j}\left(\frac{e^{P_{ij}}}{1+e^{P_{ij}}}-A_{ij}\right)\langle\Delta_X,e_ie_j^T\Delta_Y\rangle\\
& =\frac2n\sum_{i\neq j}\left(\frac{e^{P_{ij}}}{1+e^{P_{ij}}}-\frac{e^{P^{*}_{ij}}}{1+e^{P^{*}_{ij}}}\right)\langle\Delta_X,e_ie_j^T\Delta_Y\rangle+\frac2n\sum_{i\neq j}\left(\frac{e^{P^{*}_{ij}}}{1+e^{P^{*}_{ij}}}-A_{ij}\right)\langle\Delta_X,e_ie_j^T\Delta_Y\rangle.  
\end{align*}
For the first term, we have
\begin{align*}
&\left|\frac2n\sum_{i\neq j}\left(\frac{e^{P_{ij}}}{1+e^{P_{ij}}}-\frac{e^{P^{*}_{ij}}}{1+e^{P^{*}_{ij}}}\right)\langle\Delta_X,e_ie_j^T\Delta_Y\rangle\right|\\
=&\left|\frac2n\sum_{i\neq j}\left(\frac{e^{P_{ij}}}{1+e^{P_{ij}}}-\frac{e^{P^{*}_{ij}}}{1+e^{P^{*}_{ij}}}\right)(\Delta_X\Delta_Y^T)_{ij}\right|\\
\leq & \frac{1}{n}\sum_{i\neq j}|P_{ij}-P_{ij}^*||(\Delta_X\Delta_Y^T)_{ij}|\tag{\text{mean-value theorem}}\\
\leq & \frac{1}{n}\|\Delta_X\Delta_Y^T\|_F\cdot\sqrt{\sum_{i\neq j}(P_{ij}-P_{ij}^*)^2}\tag{\text{Cauchy}}.
\end{align*}
Note that
\begin{align}\label{ineq:sq_bound}
&\sum_{i\neq j}(P_{ij}-P_{ij}^*)^2\leq \sum_{i,j}(P_{ij}-P_{ij}^*)^2=\sum_{i, j}\left(\langle H-H^*,z_iz_j^T\rangle+\frac1n\langle XY^T-X^*Y^{*T},e_i e_j^T\rangle\right)^2\notag\\
=&\left\|Z(H-H^*)Z^T + \frac{1}{n}(XY^T - X^*Y^{*T})\right\|_F^2 = \left\|Z(H-H^*)Z^T\right\|_F^2+\left\|\frac{1}{n}(XY^T - X^*Y^{*T})\right\|_F^2\notag\\
\leq &\su \|H-H^{*}\|^2_{F}+\frac{1}{n^2}\|XY^T-X^{*}Y^{*T}\|^2_{F},
\end{align}
where the last equation follows from the fact that $\cP_Z(XY^T - X^*Y^{*T})=0$ and $\cP_Z(Z(H-H^*)Z^T)=Z(H-H^*)Z^T$.
Thus, we have
\begin{align*}
 &\left|\frac2n\sum_{i\neq j}\left(\frac{e^{P_{ij}}}{1+e^{P_{ij}}}-\frac{e^{P^{*}_{ij}}}{1+e^{P^{*}_{ij}}}\right)\langle\Delta_X,e_ie_j^T\Delta_Y\rangle\right|\\
 \leq &  \frac{1}{n}\|\Delta_X\Delta_Y^T\|_F\cdot\left(\sqrt{\su}\|H-H^{*}\|_{F}+\frac{1}{n}\|XY^T-X^{*}Y^{*T}\|_{F}\right)\\
 \leq &  \sqrt{\frac{1}{n}}\|\Delta_X\Delta_Y^T\|_F\cdot\left(c_2\sqrt{\su}+\frac{3c_3\sqrt{\sigma_{\max}}}{n}\right)\tag{By the observations}\\
 \leq & c\frac{\sigma_{\min}}{n^{5/2}}\|\Delta_X\Delta_Y^T\|_F\tag{as long as $c_2\sqrt{\su}+\frac{3c_3\sqrt{\sigma_{\max}}}{n}\leq c\frac{\sigma_{\min}}{n^2}$ for some $c$}\\
 \leq & c\frac{\sigma_{\min}}{n^{5/2}}\left(\|\Delta_X\|_F^2+\|\Delta_Y\|_F^2\right)\tag{$ab\leq a^2+b^2$}
\end{align*}
For the second term, note that $(\frac{e^{P^{*}_{ij}}}{1+e^{P^{*}_{ij}}}-A_{ij})(\Delta_X\Delta_Y^T)_{ij}$ is mean-zero $|(\Delta_X\Delta_Y^T)_{ij}|^2$-subgaussian variable. By the independency, it holds that
\begin{align*}
\sum_{i\neq j}\left(\frac{e^{P^{*}_{ij}}}{1+e^{P^{*}_{ij}}}-A_{ij}\right)(\Delta_X\Delta_Y^T)_{ij}
\end{align*}
is mean-zero $\|\Delta_X\Delta_Y^T\|^2_F$-subgaussian variable. Thus, with probability at least $1-n^{-10}$, we have
\begin{align*}
  \left|\sum_{i\neq j}\left(\frac{e^{P^{*}_{ij}}}{1+e^{P^{*}_{ij}}}-A_{ij}\right)(\Delta_X\Delta_Y^T)_{ij}  \right|\lesssim \|\Delta_X\Delta_Y^T\|_F\sqrt{\log n}.
\end{align*}
As a result, with probability at least $1-n^{-10}$, we have
\begin{align*}
&\left|\frac2n\sum_{i\neq j}\left(\frac{e^{P^{*}_{ij}}}{1+e^{P^{*}_{ij}}}-A_{ij}\right)\langle\Delta_X,e_ie_j^T\Delta_Y\rangle\right|\\
&\lesssim \frac{\sqrt{\log n}}{n}\|\Delta_X\Delta_Y^T\|_F\\
 &\leq \frac{\sigma_{\min}}{n^{5/2}}\|\Delta_X\Delta_Y^T\|_F\tag{as long as $\sqrt{\frac{\log n}{n}}\ll \frac{\sigma_{\min}}{n^2}$ }\\
 &\leq \frac{\sigma_{\min}}{n^{5/2}}\left(\|\Delta_X\|_F^2+\|\Delta_Y\|_F^2\right)\tag{$ab\leq a^2+b^2$}.
\end{align*}
To summarize, we show that with probability at least $1-n^{-10}$,
\begin{align*}
    \left|\frac2n\sum_{i\neq j}\left(\frac{e^{P_{ij}}}{1+e^{P_{ij}}}-A_{ij}\right)\langle\Delta_X,e_ie_j^T\Delta_Y\rangle\right|\leq c\frac{\sigma_{\min}}{n^{5/2}}\left(\|\Delta_X\|_F^2+\|\Delta_Y\|_F^2\right).
\end{align*}

For the rest of the terms, note that
\begin{align}\label{ineq:convexity_upper}
&\sum_{i\neq j}\frac{e^{P_{ij}}}{(1+e^{P_{ij}})^2}\left(\langle\Delta_{H},z_iz_j^T\rangle+\frac1n\langle\Delta_{X}Y^T+X\Delta_{Y}^T,e_i e_j^T\rangle\right)^2+\frac{2c_{\aug}}{n^2}\|\Delta_X^T X+X^T\Delta_X-\Delta_Y^T Y-Y^T\Delta_Y\|_F^2   \notag \\
\leq &\frac14 \left\|Z\Delta_H Z^T + \frac{1}{n}(\Delta_{X}Y^T+X\Delta_{Y}^T)\right\|_F^2 +\frac{2c_{\aug}}{n^2}\|\Delta_X^T X+X^T\Delta_X-\Delta_Y^T Y-Y^T\Delta_Y\|_F^2    \notag\\
\leq& \frac{\su}{4}\|\Delta_{H}\|^2_F+\frac{1}{4n^2} \|\Delta_{X}Y^T+X\Delta_{Y}^T\|^2_F+\frac{2c_{\aug}}{n^2}\|\Delta_X^T X+X^T\Delta_X-\Delta_Y^T Y-Y^T\Delta_Y\|_F^2\notag\\
\leq& \su\|\Delta_{H}\|^2_F+\frac{1}{n^2} \|\Delta_{X}Y^T+X\Delta_{Y}^T\|^2_F+\frac{1}{4n^2}\|\Delta_X^T X+X^T\Delta_X-\Delta_Y^T Y-Y^T\Delta_Y\|_F^2 \notag \\
\leq & \su\|\Delta_{H}\|^2_F+\frac{16\sigma_{\max}}{n^2} \left(\|\Delta_{X}\|^2_F+\|\Delta_{Y}\|^2_F\right),
\end{align}
where the second inequality follows from Assumption \ref{assumption:eigenvalues}, the third inequality follows from the fact that $8c_{\aug}\leq 1$ and the last inequality follows from Lemma \ref{lem:balance}. Combine this with \eqref{eq:convexity}, we final obtain that with probability at least $1-n^{-10}$, we have
\begin{align*}
 &\vec(\Delta)^T\nabla^2 f_{\aug}(H,X,Y)\vec(\Delta)\notag\\  
 \leq& \su\|\Delta_{H}\|^2_F+ \left(\frac{16\sigma_{\max}}{n^2}+\lambda +c\frac{\sigma_{\min}}{n^{5/2}}\right)\left(\left\| \Delta_X \right\|_F^2 + \left\|  \Delta_Y \right\|_F^2 \right)\\
 \leq& \su\|\Delta_{H}\|^2_F+\frac{20\sigma_{\max}}{n^2}\left(\left\| \Delta_X \right\|_F^2 + \left\|  \Delta_Y \right\|_F^2 \right)
\end{align*}
as long as $\lambda, c\frac{\sigma_{\min}}{n^{5/2}}\ll \frac{\sigma_{\max}}{n^2}$. In other words, we obtain
\begin{align*}
\left\|\nabla^2f_{\aug}(H,X,Y)\right\|\leq \max\left\{\su,\frac{20\sigma_{\max}}{n^2}\right\}.
\end{align*}
It's easy to see that the above upper bound also holds for $\|\nabla^2f(H,X,Y)\|$.

Now let's focus on the lower bond. One can see that
\begin{align}
    &\sum_{i\neq j}\frac{e^{P_{ij}}}{(1+e^{P_{ij}})^2}\left(\langle\Delta_{H},z_iz_j^T\rangle+\frac1n\langle\Delta_{X}Y^T+X\Delta_{Y}^T,e_i e_j^T\rangle\right)^2  \notag\\
\geq &\frac{e^{2c_P}}{(1+e^{2c_P})^2}\sum_{i\neq j}\left(\langle\Delta_{H},z_iz_j^T\rangle+\frac1n\langle\Delta_{X}Y^T+X\Delta_{Y}^T,e_i e_j^T\rangle\right)^2 \notag\\
= &\frac{e^{2c_P}}{(1+e^{2c_P})^2}\left(\sum_{i,j}\left(z_i^\top\Delta_{H}z_j+\frac1n(\Delta_{X}Y^T+X\Delta_{Y}^T)_{ij}\right)^2- \sum_{i=1}^n\left(z_i^\top\Delta_{H}z_i+\frac1n(\Delta_{X}Y^T+X\Delta_{Y}^T)_{ii}\right)^2\right)  \notag\\
= & \frac{e^{2c_P}}{(1+e^{2c_P})^2} \left(\left\|Z\Delta_H Z^T\right\|_F^2+\left\|\frac{1}{n}(\Delta_{X}Y^T+X\Delta_{Y}^T)\right\|_F^2-\sum_{i=1}^n\left(z_i^\top\Delta_{H}z_i+\frac1n(\Delta_{X}Y^T+X\Delta_{Y}^T)_{ii}\right)^2\right) \notag\\ 
\geq & \frac{e^{2c_P}}{(1+e^{2c_P})^2} \left(\sl\left\|\Delta_H \right\|_F^2+\left\|\frac{1}{n}(\Delta_{X}Y^T+X\Delta_{Y}^T)\right\|_F^2-\sum_{i=1}^n\left(z_i^\top\Delta_{H}z_i+\frac1n(\Delta_{X}Y^T+X\Delta_{Y}^T)_{ii}\right)^2\right)\label{eq:stronglycvxlemeq1}.
\end{align}
Here the last equation follows from the fact that $\cP_Z(\Delta_{X}Y^T+X\Delta_{Y}^T)=0$ and $\cP_Z(Z\Delta_H Z^T)=Z\Delta_H Z^T$, and the inequality follows from Assumption \ref{assumption:eigenvalues}. On the one hand, one can control the last term as
\begin{align}
    &\sum_{i=1}^n\left(z_i^\top\Delta_{H}z_i+\frac1n(\Delta_{X}Y^T+X\Delta_{Y}^T)_{ii}\right)^2 \leq 9\sum_{i=1}^n \left(\left(z_i^\top\Delta_{H}z_i\right)^2+ \frac{\left(\Delta_{X}Y^T\right)_{ii}^2+\left(X\Delta_{Y}^T\right)_{ii}^2}{n^2}\right) \notag \\
    \lesssim & \sum_{i=1}^n \left(\left\|z_i\right\|_2^4\left\|\Delta_H\right\|^2 +\frac{\left\|(\Delta_X)_{i, :}\right\|_2^2\left\|(Y)_{i, :}\right\|_2^2+\left\|(\Delta_Y)_{i, :}\right\|_2^2\left\|(X)_{i, :}\right\|_2^2}{n^2}\right) \notag \\
     \lesssim &\left(\sum_{i=1}^n \frac{c_z^2}{n^2}\left\|\Delta_H\right\|^2 \right)+\frac{\left\|\Delta_X\right\|_F^2\left\|Y\right\|_{2,\infty}^2+\left\|\Delta_Y\right\|_F^2\left\|X\right\|_{2,\infty}^2}{n^2} \notag \\
    \lesssim  & \frac{c_z^2}{n}\left\|\Delta_H\right\|^2+ \frac{\left\|\Delta_X\right\|_F^2\left\|Y^*\right\|_{2,\infty}^2+\left\|\Delta_Y\right\|_F^2\left\|X^*\right\|_{2,\infty}^2}{n^2}\notag \\
    \lesssim & \frac{\left\|\Delta_H\right\|^2_F+\left\|\Delta_X\right\|^2_F+\left\|\Delta_Y\right\|^2_F}{n} \notag\\
    \ll &
\frac{\sl}{100}\left\|\Delta_H\right\|^2_F+\frac{\sigma_{\min}}{100n^2} \left(\left\|\Delta_X\right\|^2_F+\left\|\Delta_Y\right\|^2_F\right).\label{eq:stronglycvxlemeq2}
\end{align}
On the other hand, by Lemma \ref{lem:balance} we have
\begin{align}
    &\frac{e^{2c_P}}{(1+e^{2c_P})^2} \left\|\frac{1}{n}(\Delta_{X}Y^T+X\Delta_{Y}^T)\right\|_F^2+ \frac{2c_{\aug}}{n^2}\|\Delta_X^T X+X^T\Delta_X-\Delta_Y^T Y-Y^T\Delta_Y\|_F^2 \notag \\
    = &\frac{e^{2c_P}}{(1+e^{2c_P})^2} \left\|\frac{1}{n}(\Delta_{X}Y^T+X\Delta_{Y}^T)\right\|_F^2+ \frac{e^{2c_P}}{4n^2(1+e^{2c_P})^2}\|\Delta_X^T X+X^T\Delta_X-\Delta_Y^T Y-Y^T\Delta_Y\|_F^2 \notag \\
    \geq & \frac{e^{2c_P}}{n^2(1+e^{2c_P})^2}\left(\frac{\sigma_{\min}}{4}-10(c_3+c_4) \sqrt{n\sigma_{\max}}\right)\left(\|\Delta_X\|_F^2+\|\Delta_Y\|_F^2\right) \notag \\
    \geq &\frac{e^{2c_P}\sigma_{\min}}{8n^2(1+e^{2c_P})^2}\left(\|\Delta_X\|_F^2+\|\Delta_Y\|_F^2\right)  \label{eq:stronglycvxlemeq3}
\end{align}
as long as $\sigma_{\min}\geq 80(c_3+c_4)\sqrt{n\sigma_{\max}}$. Combine \eqref{eq:stronglycvxlemeq2} and \eqref{eq:stronglycvxlemeq3} with \eqref{eq:stronglycvxlemeq1} we get
\begin{align*}
    &\sum_{i\neq j}\frac{e^{P_{ij}}}{(1+e^{P_{ij}})^2}\left(\langle\Delta_{H},z_iz_j^T\rangle+\frac1n\langle\Delta_{X}Y^T+X\Delta_{Y}^T,e_i e_j^T\rangle\right)^2 + \frac{2c_{\aug}}{n^2}\|\Delta_X^T X+X^T\Delta_X-\Delta_Y^T Y-Y^T\Delta_Y\|_F^2 \\
    \geq & \frac{\sl e^{2c_P}}{2(1+e^{2c_P})^2} \left\|\Delta_H\right\|_F^2 + \frac{e^{2c_P}\sigma_{\min}}{8n^2(1+e^{2c_P})^2}\left(\|\Delta_X\|_F^2+\|\Delta_Y\|_F^2\right) - \frac{e^{2c_P}\sigma_{\min}}{100n^2(1+e^{2c_P})^2}\left(\|\Delta_X\|_F^2+\|\Delta_Y\|_F^2\right) \\
    \geq &\frac{\sl e^{2c_P}}{2(1+e^{2c_P})^2} \left\|\Delta_H\right\|_F^2 + \frac{e^{2c_P}\sigma_{\min}}{10n^2(1+e^{2c_P})^2}\left(\|\Delta_X\|_F^2+\|\Delta_Y\|_F^2\right).
\end{align*}
Plugging this in \eqref{eq:convexity} we get
\begin{align*}
    \vec(\Delta)^T\nabla^2 f_{\aug}(H,X,Y)\vec(\Delta)\geq& \frac{\sl e^{2c_P}}{2(1+e^{2c_P})^2} \left\|\Delta_H\right\|_F^2 + \frac{e^{2c_P}\sigma_{\min}}{10n^2(1+e^{2c_P})^2}\left(\|\Delta_X\|_F^2+\|\Delta_Y\|_F^2\right) \\
    &-c\frac{\sigma_{\min}}{n^{5/2}}\left(\|\Delta_X\|_F^2+\|\Delta_Y\|_F^2\right) \\
    \geq & \frac{\sl e^{2c_P}}{2(1+e^{2c_P})^2} \left\|\Delta_H\right\|_F^2 + \frac{e^{2c_P}\sigma_{\min}}{20n^2(1+e^{2c_P})^2}\left(\|\Delta_X\|_F^2+\|\Delta_Y\|_F^2\right) \\
    \geq & \underline{C}\left(\|\Delta_H\|_F^2+\|\Delta_X\|_F^2+\|\Delta_Y\|_F^2\right)
\end{align*}
as long as $n \gg c^2.$

\end{proof}

%% file: pf_nonconvex_iterates.tex
\section{Proofs of Section \ref{sec:nonconvex_iterates}}\label{sec:pf_nonconvex_iterates}

We define
\begin{align*}
f_{\diff}(X,Y):=\frac{c_{\aug}}{n^2}\|X^TX-Y^TY\|^2_{F}.
\end{align*}
Thus, we have $f_{\aug}(H,X,Y)=f(H,X,Y)+f_{\diff}(X,Y)$.
Note that for any $H$, $F$, and $R\in\cO^{r\times r}$, we have
\begin{align*}
f(H,FR)=f(H,F),\ \nabla_{H} f(H,FR)=\nabla_{H} f(H,F), \ \nabla_{F} f(H,FR)=\nabla_{F} f(H,F)R, 
\end{align*}
which will be used in the following proofs.
We first present the following lemmas, which will be constantly used in the proofs.
\begin{lemma}\label{lem:gradient_norm}
Let $H^*, F^*$ be the ground truth parameters and $\lambda\gtrsim\sqrt{\frac{\log n}{n}}$. Under Assumption \ref{assumption:scales}, it holds with probability at least $1-n^{-10}$ that
\begin{align*}
\|\nabla_{H} f(H^*,F^*)\|_F\lesssim\sz\sqrt{p\log n},
\quad \|\nabla_{F} f(H^*,F^*)\|_F\lesssim \lambda   \left(\|X^{*}\|_F+\|Y^{*}\|_F\right).
\end{align*}
\end{lemma}
\begin{proof}[Proof of Lemma \ref{lem:gradient_norm}]
In the following, we will bound $\|\nabla_{H} f(H^*,F^*)\|_F$ and $\|\nabla_{F} f(H^*,F^*)\|_F$, respectively.
To bound $\|\nabla_{H} f(H^*,F^*)\|_F$, note that
\begin{align*}
 \nabla_{H} f(H^*,F^*)
=\nabla_{H} L(H^*,F^*)
=\sum_{i\neq j}\left(\frac{e^{P_{ij}^*}}{1+e^{P_{ij}^*}}-A_{ij}\right)z_iz_j^T.
\end{align*}
By Assumption \ref{assumption:scales}, we have
\begin{align*}
\left\|\sum_{i\neq j}\bbE\left[\left(\frac{e^{P_{ij}^*}}{1+e^{P_{ij}^*}}-A_{ij}\right)^2z_iz_j^Tz_jz_i^T\right]\right\|\lesssim \sum_{i\neq j}\|z_i\|_2^2\|z_j\|_2^2
\asymp \sz^2.
\end{align*}
Thus, by matrix Bernstein inequality, with probability at least $1-n^{-10}$, we have
\begin{align*}
    \left\|\sum_{i\neq j}\left(\frac{e^{P_{ij}^*}}{1+e^{P_{ij}^*}}-A_{ij}\right)z_iz_j^T\right\|\lesssim \sz\sqrt{\log n},
\end{align*}
which implies that $\|\nabla_{H} f(H^*,F^*)\|_F\leq \sqrt{p}\|\nabla_{H} f(H^*,F^*)\|\lesssim \sz\sqrt{p\log n}$.

We then bound $\|\nabla_{F} f(H^*,F^*)\|_F$. Note that
\begin{align*}
\|\nabla_{F} f(H^*,F^*)\|_F&\leq \|\nabla_{X} f(H^*,F^*)\|_F+\|\nabla_{Y} f(H^*,F^*)\|_F\\
&\leq \|\nabla_{X} L(H^*,F^*)\|_F+\|\nabla_{Y} L(H^*,F^*)\|_F+\lambda\left(\|X^{*}\|_F+\|Y^{*}\|_F\right)\\
&\leq \|\nabla_{\Gamma}L_c(H^*,\Gamma^*)\|\left(\|X^{*}\|_F+\|Y^{*}\|_F\right)+\lambda\left(\|X^{*}\|_F+\|Y^{*}\|_F\right),
\end{align*}
where the last inequality follows from the fact that $\|AB\|_F\leq \|A\|\|B\|_F$. Here 
\begin{align*}
\nabla_{\Gamma}L_c(H^*,\Gamma^*)=\frac{1}{n}\sum_{i\neq j}\left(\frac{e^{P_{ij}^*}}{1+e^{P_{ij}^*}}-A_{ij}\right)e_ie_j^T.
\end{align*}
Note that
\begin{align*}
\left\|\sum_{i\neq j}\bbE\left[\left(\frac{e^{P_{ij}^*}}{1+e^{P_{ij}^*}}-A_{ij}\right)^2e_ie_j^Te_je_i^T\right]\right\|\lesssim n\left\|\sum^n_{i=1}e_ie_i^T\right\|=n.
\end{align*}
Thus, by matrix Bernstein inequality, with probability at least $1-n^{-10}$, we have $\|\nabla_{\Gamma}L_c(H^*,\Gamma^*)\|\lesssim\sqrt{\frac{\log n}{n}}\lesssim\lambda$. Consequently, it holds that
\begin{align*}
 \|\nabla_{F} f(H^*,F^*)\|_F\lesssim \lambda   \left(\|X^{*}\|_F+\|Y^{*}\|_F\right).
\end{align*}
\end{proof}

\begin{lemma}\label{lem:gradient_spectral}
Suppose Lemma \ref{lem:ncv1} holds for the $t$-th iteration. Under Assumption \ref{assumption:eigenvalues}, we have
\begin{align*}
&\left\|\frac1n\sum_{i\neq j}\left(\frac{e^{P_{ij}^t}}{1+e^{P_{ij}^t}}-A_{ij}\right)e_i e_j^T\right\|\\
&\lesssim   \frac{1}{n}\left(\sqrt{\su}\|H^t-H^{*}\|_{F}+\frac{1}{n}\|X^*\|\|F^tR^t-F^{*}\|_{F}\right)+\sqrt{\frac{\log n}{n}}\\
&\lesssim \sqrt{\frac{c_{11}^2\overline{C}+\log n}{n}}.
\end{align*}
\end{lemma}
\begin{proof}[Proof of Lemma \ref{lem:gradient_spectral}]
We denote 
\begin{align*}
D^t:=\frac1n\sum_{i\neq j}\left(\frac{e^{P_{ij}^t}}{1+e^{P_{ij}^t}}-A_{ij}\right)e_i e_j^T.
\end{align*}   
Note that
\begin{align*}
\|D^t\|\leq \left\|\frac1n\sum_{i\neq j}\left(\frac{e^{P_{ij}^t}}{1+e^{P_{ij}^t}}-\frac{e^{P_{ij}^*}}{1+e^{P_{ij}^*}}\right)e_i e_j^T\right\|
+\left\|\frac1n\sum_{i\neq j}\left(\frac{e^{P_{ij}^*}}{1+e^{P_{ij}^*}}-A_{ij}\right)e_i e_j^T\right\|.
\end{align*}
For the first term, we have
\begin{align*}
&\left\|\frac1n\sum_{i\neq j}\left(\frac{e^{P_{ij}^t}}{1+e^{P_{ij}^t}}-\frac{e^{P_{ij}^*}}{1+e^{P_{ij}^*}}\right)e_i e_j^T\right\|\\
&\leq  \frac1n\left\|\sum_{i\neq j}\left(\frac{e^{P_{ij}^t}}{1+e^{P_{ij}^t}}-\frac{e^{P_{ij}^*}}{1+e^{P_{ij}^*}}\right)e_i e_j^T\right\|_F\\
&=\frac1n\sqrt{\sum_{i\neq j}\left(\frac{e^{P_{ij}^t}}{1+e^{P_{ij}^t}}-\frac{e^{P_{ij}^*}}{1+e^{P_{ij}^*}}\right)^2}\\
&\leq \frac{1}{4n}\sqrt{\sum_{i\neq j}\left(P^t_{ij}-P^*_{ij}\right)^2}\tag{by mean value theorem}\\
&\lesssim \frac{1}{n}\left(\sqrt{\su}\|H^t-H^{*}\|_{F}+\frac{1}{n}\|X^*\|\|F^tR^t-F^{*}\|_{F}\right),
\end{align*}
where the last inequality follows from the same argument as \eqref{ineq:sq_bound} and
\begin{align*}
&\|X^tY^{tT}-X^{*}Y^{*T}\|_{F}\\
&=\|(X^tR^t-X^*)(Y^tR^t)^T+X^{*}(Y^tR^t-Y^*)^T\|_F\\
&\leq \|X^tR^t-X^*\|_F\|Y^tR^t\|+\|X^{*}\|\|Y^tR^t-Y^*\|_F\\
&\leq \|X^tR^t-X^*\|_F\|Y^tR^t-Y^*\|+\|X^tR^t-X^*\|_F\|Y^*\|+\|X^{*}\|\|Y^tR^t-Y^*\|_F\\
&\leq 3 \|X^{*}\|\|F^tR^t-Y^*\|_F.
\end{align*}
For the second term, as bound $\|\nabla_{\Gamma}L_c(H^*,\Gamma^*)\|$ in the proof of Lemma \ref{lem:gradient_norm}, we have with probability at least $1-n^{-10}$ that
\begin{align*}
 \left\|\frac1n\sum_{i\neq j}\left(\frac{e^{P_{ij}^*}}{1+e^{P_{ij}^*}}-A_{ij}\right)e_i e_j^T\right\|\lesssim \sqrt{\frac{\log n}{n}}   .
\end{align*}
As a result, we have
\begin{align*}
 \|D^t\|
 &\lesssim  \frac{1}{n}\left(\sqrt{\su}\|H^t-H^{*}\|_{F}+\frac{1}{n}\|X^*\|\|F^tR^t-F^{*}\|_{F}\right)+\sqrt{\frac{\log n}{n}}\\
 &\lesssim c_{11}\sqrt{\frac{\overline{C}}{n}}+\sqrt{\frac{\log n}{n}} \tag{recall the definition of $\overline{C}$}\\
 &\lesssim \sqrt{\frac{c_{11}^2\overline{C}+\log n}{n}}.
\end{align*}
We then finish the proof.
\end{proof}

\subsection{Proofs of Lemma \ref{lem:ncv1}}\label{pf:lem_ncv1}
Suppose Lemma \ref{lem:ncv1}-Lemma \ref{lem:ncv5} hold for the $t$-th iteration. In the following, we prove Lemma \ref{lem:ncv1} for the $(t+1)$-th iteration. 
By the gradient decent update, we have
\begin{align*}
\vec
\begin{bmatrix}
H^{t+1}\\
F^{t+1}
\end{bmatrix}&=\cP\vec
\begin{bmatrix}
H^{t}-\eta\nabla_{H}f(H^t,F^t)\\
F^{t}-\eta\nabla_{F}f(H^t,F^t)
\end{bmatrix},
\end{align*}
which then gives
\begin{align}\label{ineq30:pf_lem_ncv1}
\left\|\begin{bmatrix}
H^{t+1}-H^{*}\\
F^{t+1}R^t-F^{*}
\end{bmatrix}\right\|_F
=\left\|\vec\begin{bmatrix}
H^{t+1}-H^{*}\\
F^{t+1}R^t-F^{*}
\end{bmatrix}\right\|_2
\leq 
\left\|
\vec
\begin{bmatrix}
H^t-H^*-\eta\nabla_{H}f(H^t,F^tR^t)
\\
F^tR^t-F^*-\eta\nabla_F f(H^t,F^tR^t)
\end{bmatrix}\right\|_2.
\end{align}
Consequently, we only need to bound the RHS of \eqref{ineq30:pf_lem_ncv1}. Note that
\begin{align*}
&\vec\begin{bmatrix}
H^t-H^*-\eta\nabla_{H}f(H^t,F^tR^t)
\\
F^tR^t-F^*-\eta\nabla_F f(H^t,F^tR^t)
\end{bmatrix}\\
&=\vec\begin{bmatrix}
H^t-H^{*}\\
F^tR^t-F^*
\end{bmatrix}-\eta\nabla f(H^t,F^tR^t)\\
&=\vec\begin{bmatrix}
H^t-H^{*}\\
F^tR^t-F^*
\end{bmatrix}-\eta\left(\nabla f_{\aug}(H^t,F^tR^t)-\nabla f_{\aug}(H^*,F^*)\right)+\eta\nabla f_{\diff}(F^tR^t)-\eta\nabla f_{\aug}(H^*,F^*).
\end{align*}
Also, notice that
\begin{align*}
&\nabla f_{\aug}(H^t,F^tR^t)-\nabla f_{\aug}(H^*,F^*)\\
&=\int^1_0\nabla^2 f_{\aug}\left((H^*,F^*)+\tau(H^t-H^*,F^tR^t-F^*)\right)d \tau\cdot \vec\begin{bmatrix}
H^t-H^{*}\\
F^tR^t-F^*
\end{bmatrix}.
\end{align*}
Thus, we have
\begin{align*}
&\vec\begin{bmatrix}
H^t-H^*-\eta\nabla_{H}f(H^t,F^tR^t)
\\
F^tR^t-F^*-\eta\nabla_F f(H^t,F^tR^t)
\end{bmatrix}\\
&=\left(I-\eta\int^1_0\nabla^2 f_{\aug}\left((H^*,F^*)+\tau(H^t-H^*,F^tR^t-F^*)\right)d \tau\right)\cdot \vec\begin{bmatrix}
H^t-H^{*}\\
F^tR^t-F^*
\end{bmatrix}\\
&\quad +\eta\nabla f_{\diff}(F^tR^t)-\eta\nabla f_{\aug}(H^*,F^*).
\end{align*}
For notation simplicity, we denote 
\begin{align*}
A:=\int^1_0\nabla^2 f_{\aug}\left((H^*,F^*)+\tau(H^t-H^*,F^tR^t-F^*)\right)d \tau.
\end{align*}
Since Lemma \ref{lem:ncv4} holds for the $t$-th iteration, we know $A$ satisfies the local geometry properties as outlined in Lemma \ref{lem:convexity} as long as $c_{11}\leq c_2$, $c_{41}\leq c_3$.
By triangle inequality, we have
\begin{align*}
&\left\|\vec\begin{bmatrix}
H^t-H^*-\eta\nabla_{H}f(H^t,F^tR^t)
\\
F^tR^t-F^*-\eta\nabla_F f(H^t,F^tR^t)
\end{bmatrix}\right\|_2\\
&\leq \underbrace{\left\|(I-\eta A)\vec\begin{bmatrix}
H^t-H^{*}\\
F^tR^t-F^*
\end{bmatrix}\right\|_2}_{(1)}+\eta\underbrace{\|\nabla f_{\diff}(F^tR^t)\|_2}_{(2)}+\eta\underbrace{\|\nabla f_{\aug}(H^*,F^*)\|_2}_{(3)}.
\end{align*}
In the following, we bound (1)-(3), respectively.
\begin{enumerate}
\item We first bound (1).
Note that
\begin{align*}
(1)^2&=\vec
\begin{bmatrix}
H^t-H^{*}\\
F^tR^t-F^{*}
\end{bmatrix}^T
(I-2\eta A+\eta^2 A^2)
\vec
\begin{bmatrix}
H^t-H^{*}\\
F^tR^t-F^{*}
\end{bmatrix}\\
&=\left\|
\begin{bmatrix}
H^t-H^{*}\\
F^tR^t-F^{*}
\end{bmatrix}
\right\|_F^2-2\eta\vec
\begin{bmatrix}
H^t-H^{*}\\
F^tR^t-F^{*}
\end{bmatrix}^T A \ 
\vec
\begin{bmatrix}
H^t-H^{*}\\
F^tR^t-F^{*}
\end{bmatrix}\\
&\quad+\eta^2\vec
\begin{bmatrix}
H^t-H^{*}\\
F^tR^t-F^{*}
\end{bmatrix}^TA^2\ \vec
\begin{bmatrix}
H^t-H^{*}\\
F^tR^t-F^{*}
\end{bmatrix}.
\end{align*}
By Lemma \ref{lem:convexity}, we have
\begin{align*}
& \vec
\begin{bmatrix}
H^t-H^{*}\\
F^tR^t-F^{*}
\end{bmatrix}^T A \ 
\vec
\begin{bmatrix}
H^t-H^{*}\\
F^tR^t-F^{*}
\end{bmatrix}\geq \underline{C}\left\|
\begin{bmatrix}
H^t-H^{*}\\
F^tR^t-F^{*}
\end{bmatrix}
\right\|_F^2,\\
&\vec
\begin{bmatrix}
H^t-H^{*}\\
F^tR^t-F^{*}
\end{bmatrix}^TA^2\ \vec
\begin{bmatrix}
H^t-H^{*}\\
F^tR^t-F^{*}
\end{bmatrix}\leq \overline{C}^2\left\|
\begin{bmatrix}
H^t-H^{*}\\
F^tR^t-F^{*}
\end{bmatrix}
\right\|_F^2,
\end{align*}
where the second inequality holds since by Lemma \ref{lem:convexity}, we know $\|A\|\leq  \overline{C}$. As a result, we have
\begin{align*}
(1)^2&\leq \left(1+\eta^2\overline{C}^2-2\underline{C}\eta\right)
\left\|
\begin{bmatrix}
H^t-H^{*}\\
F^tR^t-F^{*}
\end{bmatrix}
\right\|_F^2\\
&\leq  \left(1-\underline{C}\eta\right)
\left\|
\begin{bmatrix}
H^t-H^{*}\\
F^tR^t-F^{*}
\end{bmatrix}
\right\|_F^2,
\end{align*}
where the second inequality holds as long as $\eta\overline{C}^2\leq \underline{C}$.
This implies 
\begin{align*}
(1)&\leq \left(1-\frac{\underline{C}}{2}\eta\right)
\left\|
\begin{bmatrix}
H^t-H^{*}\\
F^tR^t-F^{*}
\end{bmatrix}
\right\|_F,
\end{align*}
which follows from the fact that $\sqrt{1-x}\leq 1-\frac{x}{2}$.

\item We then bound (2). Note that
\begin{align*}
&\nabla_X f_{\diff}(F^tR^t)=\frac{4c_{\aug}}{n^2}X^t(X^{tT}X^t-Y^{tT}Y^t)R^t,\\
&\nabla_Y f_{\diff}(F^tR^t)=\frac{4c_{\aug}}{n^2}Y^t(Y^{tT}Y^t-X^{tT}X^t)R^t.
\end{align*}
Thus, we have
\begin{align*}
\|\nabla f_{\diff}(F^tR^t)\|_2
&=\left\|
\vec
\begin{bmatrix}
0\\
\nabla_X f_{\diff}(F^tR^t)\\
\nabla_Y f_{\diff}(F^tR^t)
\end{bmatrix}
\right\|_2
=\left\|
\begin{bmatrix}
\nabla_X f_{\diff}(F^tR^t)\\
\nabla_Y f_{\diff}(F^tR^t)
\end{bmatrix}
\right\|_F\\
&\leq \frac{4c_{\aug}}{n^2}\left(\|X^t(X^{tT}X^t-Y^{tT}Y^t)R^t\|_F+\|Y^t(Y^{tT}Y^t-X^{tT}X^t)R^t\|_F\right)\\
&\leq \frac{4c_{\aug}}{n^2}\left(\|X^t\|+\|Y^t\|\right)\|X^{tT}X^t-Y^{tT}Y^t\|_F.
\end{align*}
Since Lemma \ref{lem:ncv1} holds for the $t$-th iteration, we have
\begin{align*}
\|F^t\|=\|F^tR^t\|\leq \|F^tR^t-F^*\|+\|F^*\|\leq 2\|F^*\|,
\end{align*}
where the last inequality holds because by Lemma \ref{lem:ncv1}, we have $\|F^tR^t-F^*\|\ll \|F^*\|$.
Consequently, we have
\begin{align*}
 (2)=\|\nabla f_{\diff}(F^tR^t)\|_2\leq \frac{16c_{\aug}}{n^2} \|F^*\|  \|X^{tT}X^t-Y^{tT}Y^t\|_F.
\end{align*}

\item We then bound (3).
Note that $X^{*T}X^{*}=Y^{*T}Y^{*}$. Thus, we have $\nabla f_{\diff}(F^*)=0$, which implies $\nabla f_{\aug}(H^*,F^*)=\nabla f(H^*,F^*)$.
By Lemma \ref{lem:gradient_norm}, we have 
\begin{align*}
(3)=\|\nabla f(H^*,F^*)\|_2\lesssim c_z\sqrt{p\log n}+\lambda\sqrt{\mu r \sigma_{\max}}\lesssim \lambda\sqrt{\mu r \sigma_{\max}}
\end{align*}
as long as $\sz^2 p\ll n$.

\end{enumerate}

Consequently, we conclude that
\begin{align*}
&\left\|\vec\begin{bmatrix}
H^t-H^*-\eta\nabla_{H}f(H^t,F^tR^t)
\\
F^tR^t-F^*-\eta\nabla_F f(H^t,F^tR^t)
\end{bmatrix}\right\|_2\\
&\leq\left(1-\frac{\underline{C}}{2}\eta\right)
\left\|
\begin{bmatrix}
H^t-H^{*}\\
F^tR^t-F^{*}
\end{bmatrix}
\right\|_F+\frac{16 \eta c_{\aug}}{n^2} \|F^*\|  \|X^{tT}X^t-Y^{tT}Y^t\|_F+\lambda\eta\sqrt{\mu r \sigma_{\max}}.
\end{align*}
Recall \eqref{ineq30:pf_lem_ncv1}, we then have
\begin{align*}
 &\left\|\begin{bmatrix}
H^{t+1}-H^{*}\\
F^{t+1}R^t-F^{*}
\end{bmatrix}\right\|_F\\  
&\leq\left(1-\frac{\underline{C}}{2}\eta\right)
\left\|
\begin{bmatrix}
H^t-H^{*}\\
F^tR^t-F^{*}
\end{bmatrix}
\right\|_F+\frac{16 \eta c_{\aug}}{n^2} \|F^*\|  \|X^{tT}X^t-Y^{tT}Y^t\|_F+\lambda\eta\sqrt{\mu r \sigma_{\max}}.
\end{align*}
By Lemma \ref{ar:lem1}, we know $\|F^*\|\leq 2\sqrt{\sigma_{\max}}$. By Lemma \ref{lem:ncv5}, we have $\|X^{tT}X^t-Y^{tT}Y^t\|_F\leq c_{51}\eta n^2$. 
We then obtain that
\begin{align*}
 &\left\|\begin{bmatrix}
H^{t+1}-H^{*}\\
F^{t+1}R^t-F^{*}
\end{bmatrix}\right\|_F\\   
&\leq\left(1-\frac{\underline{C}}{2}\eta\right)
\left\|
\begin{bmatrix}
H^t-H^{*}\\
F^tR^t-F^{*}
\end{bmatrix}
\right\|_F
+{32 c_{\aug}c_{51}\sqrt{\sigma_{\max}}\eta^2}+\lambda\eta\sqrt{\mu r \sigma_{\max}}.
\end{align*}
Since Lemma \ref{lem:ncv1} holds for the $t$-th iteration, we have
\begin{align*}
 &\left\|\begin{bmatrix}
H^{t+1}-H^{*}\\
F^{t+1}R^t-F^{*}
\end{bmatrix}\right\|_F\\   
&\leq\left(1-\frac{\underline{C}}{2}\eta\right)c_{11}\sqrt{n}+{32 c_{\aug}c_{51}\sqrt{\sigma_{\max}}\eta^2}+\lambda\eta\sqrt{\mu r \sigma_{\max}}\\
&\leq c_{11}\sqrt{n}
\end{align*}
as long as $\lambda\sqrt{\frac{ \mu r\sigma_{\max}}{n}}\lesssim c_{11} \underline{C} $ and $\frac{ c_{51}\sqrt{\sigma_{\max}}\eta}{\sqrt{n}}\lesssim c_{11} \underline{C} $.
Finally, by the definition of $R^{t+1}$, we have
\begin{align*}
 \|F^{t+1}R^{t+1}-F^{*}\|_F\leq  \|F^{t+1}R^{t}-F^{*}\|_F.   
\end{align*}
Consequently, we have
\begin{align*}
&\left\|\begin{bmatrix}
H^{t+1}-H^{*}\\
F^{t+1}R^{t+1}-F^{*}
\end{bmatrix}\right\|_F  \leq c_{11}\sqrt{n}.
\end{align*}

\subsection{Proofs of Lemma \ref{lem:ncv2}}\label{pf:lem_ncv2}
Suppose Lemma \ref{lem:ncv1}-Lemma \ref{lem:ncv5} hold for the $t$-th iteration. In the following, we prove Lemma \ref{lem:ncv2} for the $(t+1)$-th iteration.
More specifically, we fix $m$ and aim to bound
\begin{align*}
\left\|
\begin{bmatrix}
H^{t+1,(m)}-H^{t+1}\\
F^{t+1,(m)}O^{t+1,(m)}-F^{t+1}R^{t+1}
\end{bmatrix}
\right\|_F.
\end{align*}
\begin{claim}\label{claim:lem_ncv2}
It holds that
\begin{align*}
\|F^{t+1}R^{t+1}-F^{t+1,(m)}O^{t+1,(m)}\|_F\leq \|F^{t+1}R^{t}-F^{t+1,(m)}O^{t,(m)}\|_F.
\end{align*}
\end{claim}
\begin{proof}[Proof of Claim]
By the definition of $O^{t+1,(m)}$, for any $O\in\cO^{r\times r}$, we have
\begin{align*}
\|F^{t+1}R^{t+1}-F^{t+1,(m)}O^{t+1,(m)}\|_F\leq \|F^{t+1}R^{t+1}-F^{t+1,(m)}O\|_F.
\end{align*}
Choosing $O=O^{t,(m)}(R^t)^{-1}R^{t+1}$, we then have
\begin{align*}
\|F^{t+1}R^{t+1}-F^{t+1,(m)}O\|_F&=\|F^{t+1}R^{t+1}-F^{t+1,(m)}O^{t,(m)}(R^t)^{-1}R^{t+1}\|_F\\
&=\|F^{t+1}-F^{t+1,(m)}O^{t,(m)}(R^t)^{-1}\|_F\\
&=\|F^{t+1}R^t-F^{t+1,(m)}O^{t,(m)}\|_F,
\end{align*}
which then finishes the proofs.
\end{proof}

By Claim \ref{claim:lem_ncv2}, we have
\begin{align*}
 \left\|
\begin{bmatrix}
H^{t+1,(m)}-H^{t+1}\\
F^{t+1,(m)}O^{t+1,(m)}-F^{t+1}R^{t+1}
\end{bmatrix}
\right\|_F\leq 
 \left\|
\begin{bmatrix}
H^{t+1}-H^{t+1,(m)}\\
F^{t+1}R^t-F^{t+1,(m)}O^{t,(m)}
\end{bmatrix}
\right\|_F.
\end{align*}
Moreover, by the gradient decent update, we have
\begin{align*}
\vec
&\begin{bmatrix}
H^{t+1}-H^{t+1,(m)}\\
F^{t+1}R^t-F^{t+1,(m)}O^{t,(m)}
\end{bmatrix}\\
&=\cP\vec
\begin{bmatrix}
\left(H^t-\eta\nabla_{H}f(H^t,F^t)\right)-\left(H^{t,(m)}-\eta\nabla_{H}f^{(m)}(H^{t,(m)},F^{t,(m)})\right)\\
(F^{t}R^t-\eta\nabla_F f(H^t,F^tR^t))-(F^{t,(m)}O^{t,(m)}-\eta\nabla_F f^{(m)}(H^{t,(m)},F^{t,(m)}O^{t,(m)}))
\end{bmatrix},
\end{align*}
which further implies that
\begin{align}\label{ineq30:pr_lem_ncv2}
 &\left\|
\begin{bmatrix}
H^{t+1,(m)}-H^{t+1}\\
F^{t+1,(m)}O^{t+1,(m)}-F^{t+1}R^{t+1}
\end{bmatrix}
\right\|_F\notag\\
&\leq 
 \left\|\vec
\begin{bmatrix}
\left(H^t-\eta\nabla_{H}f(H^t,F^t)\right)-\left(H^{t,(m)}-\eta\nabla_{H}f^{(m)}(H^{t,(m)},F^{t,(m)})\right)\\
(F^{t}R^t-\eta\nabla_F f(H^t,F^tR^t))-(F^{t,(m)}O^{t,(m)}-\eta\nabla_F f^{(m)}(H^{t,(m)},F^{t,(m)}O^{t,(m)}))
\end{bmatrix}
 \right\|_2.
\end{align}
Thus, we only need to control the RHS of \eqref{ineq30:pr_lem_ncv2}.

Notice that $\nabla_{H} f(H,F)=\nabla_{H} f_{\aug}(H,F)$ and $\nabla_{H} f^{(m)}(H,F)=\nabla_{H} f^{(m)}_{\aug}(H,F)$, we have
\begin{align*}
&H^t-\eta\nabla_{H}f(H^t,F^t)-\left(H^{t,(m)}-\eta\nabla_{H}f^{(m)}(H^{t,(m)},F^{t,(m)})\right)\\
&=H^t-\eta\nabla_{H}f_{\aug}(H^t,F^t)
-
\left(
H^{t,(m)}-\eta\nabla_{H}f^{(m)}_{\aug}(H^{t,(m)},F^{t,(m)})
\right)\\
&=H^t-H^{t,(m)}-\eta\left(\nabla_{H}f_{\aug}(H^t,F^t)-\nabla_{H}f_{\aug}(H^{t,(m)},F^{t,(m)})\right)\\
&\quad +\eta\left(\nabla_{H}f^{(m)}_{\aug}(H^{t,(m)},F^{t,(m)})-\nabla_{H}f_{\aug}(H^{t,(m)},F^{t,(m)})\right)\\
&=H^t-H^{t,(m)}-\eta\left(\nabla_{H}f_{\aug}(H^t,F^tR^t)-\nabla_{H}f_{\aug}(H^{t,(m)},F^{t,(m)}O^{t,(m)})\right)\\
&\quad +\eta\left(\nabla_{H}f^{(m)}(H^{t,(m)},F^{t,(m)}O^{t,(m)})-\nabla_{H}f(H^{t,(m)},F^{t,(m)}O^{t,(m)})\right).
\end{align*}
Moreover, we have
\begin{align*}
& F^{t}R^t-\eta\nabla_F f(H^t,F^tR^t)-(F^{t,(m)}O^{t,(m)}-\eta\nabla_F f^{(m)}(H^{t,(m)},F^{t,(m)}O^{t,(m)}))\\
&=F^{t}R^t-F^{t,(m)}O^{t,(m)}-\eta\left(\nabla_F f_{\aug}(H^t,F^tR^t)-\nabla_F f_{\aug}(H^{t,(m)},F^{t,(m)}O^{t,(m)})\right)\\
&\quad+\eta\left(\nabla_F f^{(m)}(H^{t,(m)},F^{t,(m)}O^{t,(m)})-\nabla_F f(H^{t,(m)},F^{t,(m)}O^{t,(m)})\right)\\
&\quad+\eta\left(\nabla_F f_{\diff}(H^t,F^tR^t)-\nabla_F f_{\diff}(H^{t,(m)},F^{t,(m)}O^{t,(m)})\right).
\end{align*}
As a result, we have
\begin{align*}
\vec
&\begin{bmatrix}
\left(H^t-\eta\nabla_{H}f(H^t,F^t)\right)-\left(H^{t,(m)}-\eta\nabla_{H}f^{(m)}(H^{t,(m)},F^{t,(m)})\right)\\
(F^{t}R^t-\eta\nabla_F f(H^t,F^tR^t))-(F^{t,(m)}O^{t,(m)}-\eta\nabla_F f^{(m)}(H^{t,(m)},F^{t,(m)}O^{t,(m)}))
\end{bmatrix}\\
&=\underbrace{\left(I-\eta A\right)
\vec\begin{bmatrix}
H^{t}-H^{t,(m)}\\
F^{t}R^t-F^{t,(m)}O^{t,(m)}
\end{bmatrix}}_{(1)}\\
&\quad +\eta\underbrace{\left(\nabla f^{(m)}(H^{t,(m)},F^{t,(m)}O^{t,(m)})-\nabla f(H^{t,(m)},F^{t,(m)}O^{t,(m)})\right)}_{(2)}\\
&\quad+\eta\underbrace{\left(\nabla f_{\diff}(H^t,F^tR^t)-\nabla f_{\diff}(H^{t,(m)},F^{t,(m)}O^{t,(m)})\right)}_{(3)}
\end{align*}
where
\begin{align*}
A=\int^1_0\nabla^2 f_{\aug}\left((H^{t,(m)},F^{t,(m)}O^{t,(m)})+\tau (H^t-H^{t,(m)},F^tR^t-F^{t,(m)}O^{t,(m)})\right)d\tau.
\end{align*}
In the following, we bound the Frobenius norm of (1)-(3), respectively.
\begin{enumerate}
\item  We first bound (1).
Since Lemma \ref{lem:ncv2} and Lemma \ref{lem:ncv4} hold for the $t$-th iteration, we have
\begin{align*}
&\left\|H^{t,(m)}+\tau\left(H^t-H^{t,(m)}\right)-H^*\right\|_{F}   \leq \|H^t-H^*\|_{F}+(1-\tau)\|H^{t,(m)}-H^t\|_{F}\leq c_{11}\sqrt{n}+c_{21}\leq c_2\sqrt{n}\\
&\left\|F^{t,(m)}O^{t,(m)}+\tau\left(F^tR^t-F^{t,(m)}O^{t,(m)}\right)-F^*\right\|_{2,\infty}   \\
&\quad\leq \|F^tR^t-F^*\|_{2,\infty}+(1-\tau)\|F^{t,(m)}O^{t,(m)}-F^tR^t\|_{2,\infty}\\
&\quad\leq \|F^tR^t-F^*\|_{2,\infty}+\|F^{t,(m)}O^{t,(m)}-F^tR^t\|_{F}\leq c_{41}+c_{21}\leq c_3.
\end{align*}
Thus Lemma \ref{lem:convexity} can be applied to bound the Frobenius norm of (1). Following the same argument as bounding term (1) in Appendix \ref{pf:lem_ncv1}, we have
\begin{align*}
\|(1)\|_2\leq \left(1-\frac{\underline{C}}{2}\eta\right)
\left\|
\begin{bmatrix}
H^{t}-H^{t,(m)}\\
F^{t}R^t-F^{t,(m)}O^{t,(m)}
\end{bmatrix}
\right\|_F.
\end{align*}

\item  We then bound (2).
Note that
\begin{align*}
(2)&=\nabla L^{(m)}(H^{t,(m)},F^{t,(m)}O^{t,(m)})-\nabla L(H^{t,(m)},F^{t,(m)}O^{t,(m)})\\
&=\vec\left(\sum_{i\neq m}\left(\frac{e^{P_{im}^*}}{1+e^{P_{im}^*}}-A_{im}\right)
\begin{bmatrix}
z_iz_m^T\\
\frac1n e_ie_m^T Y^{t,(m)}O^{t,(m)}\\
\frac1n e_me_i^T X^{t,(m)}O^{t,(m)}
\end{bmatrix}
+
\sum_{i\neq m}\left(\frac{e^{P_{mi}^*}}{1+e^{P_{mi}^*}}-A_{mi}\right)
\begin{bmatrix}
z_mz_i^T\\
\frac1n e_me_i^T Y^{t,(m)}O^{t,(m)}\\
\frac1n e_ie_m^T X^{t,(m)}O^{t,(m)}
\end{bmatrix}\right).
\end{align*}
Since the first and second terms are similar, we only focus on bounding the norm of the first term in the following.
Notice that
\begin{align*}
    &\max\Bigg\{\left\|\sum_{i\neq m}\bbE\left[\left(\frac{e^{P_{im}^*}}{1+e^{P_{im}^*}}-A_{im}\right)^2\vec(z_iz_m^T)\vec(z_iz_m^T)^T\right]\right\|,\\
    &\qquad\qquad\left\|\sum_{i\neq m}\bbE\left[\left(\frac{e^{P_{im}^*}}{1+e^{P_{im}^*}}-A_{im}\right)^2\vec(z_iz_m^T)^T\vec(z_iz_m^T)\right]\right\|\Bigg\}\\
    &\quad\leq \sum_{i\neq m}\|\vec(z_iz_m^T)\|_2^2=\sum_{i\neq m}\|z_iz_m^T\|_F^2=\sum_{i\neq m}\|z_i\|^2_2\|z_m\|_2^2\lesssim \frac{\sz^2}{n}.
\end{align*}
Thus by matrix Berstein's inequality, we have with probability at least $1-n^{-10}$ that
\begin{align}\label{ineq2:lem_ncv2}
  \left\|\sum_{i\neq m}\left(\frac{e^{P_{im}^*}}{1+e^{P_{im}^*}}-A_{im}\right)z_iz_m^T\right\|_F=\left\|\sum_{i\neq m}\left(\frac{e^{P_{im}^*}}{1+e^{P_{im}^*}}-A_{im}\right)\vec(z_iz_m^T)\right\|_2 \lesssim \sz\sqrt{\frac{\log n}{n}}.
\end{align}
Notice that
\begin{align*}
    &\max\Bigg\{\left\|\sum_{i\neq m}\bbE\left[\left(\frac{e^{P_{im}^*}}{1+e^{P_{im}^*}}-A_{im}\right)^2\vec(e_me_i^T Y^{t,(m)}O^{t,(m)})\vec(e_me_i^T Y^{t,(m)}O^{t,(m)})^T\right]\right\|,\\
    &\qquad\qquad\left\|\sum_{i\neq m}\bbE\left[\left(\frac{e^{P_{im}^*}}{1+e^{P_{im}^*}}-A_{im}\right)^2\vec(e_me_i^T Y^{t,(m)}O^{t,(m)})^T\vec(e_me_i^T Y^{t,(m)}O^{t,(m)})\right]\right\|\Bigg\}\\
    &\quad\leq \sum_{i\neq m}\|\vec(e_me_i^T Y^{t,(m)}O^{t,(m)})\|_2^2\\
    &\quad=\sum_{i\neq m}\|e_me_i^T Y^{t,(m)}O^{t,(m)}\|_F^2\\
    &\quad=\sum_{i\neq m}\Tr\left(e_me_i^T Y^{t,(m)}(Y^{t,(m)})^{T}e_i e_m^T\right)\\
    &\quad = \sum_{i\neq m}\left( Y^{t,(m)}(Y^{t,(m)})^{T}\right)_{ii}\\
    &\quad \leq \|Y^{t,(m)}\|_F^2\\
    &\quad\leq (\|Y^{t,(m)}O^{t,(m)}-Y^{t}R^t\|_F+\|Y^{t}R^t-Y^*\|_F+\|Y^*\|_F)^2\lesssim \mu r\sigma_{\max},
\end{align*}
where the last equation follows from Assumption \ref{assumption:incoherent} and the fact that Lemma \ref{lem:ncv1} and Lemma \ref{lem:ncv2} hold for the $t$-th iteration.
Thus by matrix Berstein's inequality, we have with probability at least $1-n^{-10}$ that
\begin{align}\label{ineq3:lem_ncv2}
  &\left\|\frac1n\sum_{i\neq m}\left(\frac{e^{P_{im}^*}}{1+e^{P_{im}^*}}-A_{im}\right)e_me_i^T Y^{t,(m)}O^{t,(m)}\right\|_F\notag\\
  &=\left\|\frac1n\sum_{i\neq m}\left(\frac{e^{P_{im}^*}}{1+e^{P_{im}^*}}-A_{im}\right)\vec(e_me_i^T Y^{t,(m)}O^{t,(m)})\right\|_2  \notag\\
  &\lesssim \frac{\sqrt{\mu r \sigma_{\max}\log n}}{n}.
\end{align}
Similarly, we have with probability at least $1-n^{-10}$ that
\begin{align}\label{ineq4:lem_ncv2}
  \left\|\frac1n\sum_{i\neq m}\left(\frac{e^{P_{im}^*}}{1+e^{P_{im}^*}}-A_{im}\right)e_ie_m^T X^{t,(m)}O^{t,(m)}\right\|_F\lesssim \frac{\sqrt{\mu r \sigma_{\max}\log n}}{n}.
\end{align}
Combine\eqref{ineq2:lem_ncv2}, \eqref{ineq3:lem_ncv2} and \eqref{ineq4:lem_ncv2}, we conclude that $ \|(2)\|_2\lesssim \frac{\sqrt{\mu r \sigma_{\max}\log n}}{n}$.
\item  We then bound (3). 
Notice that
\begin{align*}
\|F^{t,(m)}\|=\|F^{t,(m)}O^{t,(m)}\|\leq \|F^{t,(m)}O^{t,(m)}-F^tR^t\|+\|F^tR^t-F^*\|+\|F^*\|\leq 2\|F^*\|.
\end{align*}
Then following the same argument as bounding term (2) in Appendix \ref{pf:lem_ncv1}, we have
\begin{align*}
&\|\nabla_F f_{\diff}(H^t,F^tR^t)\|_F\lesssim\frac{c_{\aug}}{n^2}\|F^*\|\|X^{tT}X^t-Y^{tT}Y^t\|_F,\\
&\|\nabla_F f_{\diff}(H^{t,(m)},F^{t,(m)}O^{t,(m)})\|_F\lesssim\frac{c_{\aug}}{n^2}\|F^*\|\|X^{t,(m)T}X^{t,(m)}-Y^{t,(m)T}Y^{t,(m)}\|_F.
\end{align*}
Thus, it holds that
\begin{align*}
\|(3)\|_2&\lesssim \frac{c_{\aug}}{n^2}\|F^*\|\left(\|X^{tT}X^t-Y^{tT}Y^t\|_F+\|X^{t,(m)T}X^{t,(m)}-Y^{t,(m)T}Y^{t,(m)}\|_F\right)\\
&\lesssim \eta c_{51} c_{\aug}\sqrt{\sigma_{\max}}.
\end{align*}
\end{enumerate}

Combine the bounds of Frobenius norm of (1)-(3), we conclude that
\begin{align*}
&\left\|\vec\begin{bmatrix}
\left(H^t-\eta\nabla_{H}f(H^t,F^t)\right)-\left(H^{t,(m)}-\eta\nabla_{H}f^{(m)}(H^{t,(m)},F^{t,(m)})\right)\\
(F^{t}R^t-\eta\nabla_F f(H^t,F^tR^t))-(F^{t,(m)}O^{t,(m)}-\eta\nabla_F f^{(m)}(H^{t,(m)},F^{t,(m)}O^{t,(m)}))
\end{bmatrix}\right\|_2\\
&\leq \left(1-\frac{\underline{C}}{2}\eta\right)
\left\|
\begin{bmatrix}
H^{t}-H^{t,(m)}\\
F^{t}R^t-F^{t,(m)}O^{t,(m)}
\end{bmatrix}
\right\|_F
+c\eta\frac{\sqrt{\mu r \sigma_{\max}\log n}}{n}+c\eta^2 c_{51} c_{\aug}\sqrt{\sigma_{\max}}\\
&\leq \left(1-\frac{\underline{C}}{2}\eta\right)
c_{21}
+c\eta\frac{\sqrt{\mu r \sigma_{\max}\log n}}{n}+c\eta^2 c_{51} c_{\aug}\sqrt{\sigma_{\max}}\\
&\leq c_{21}
\end{align*}
as long as $\frac{\sqrt{\mu r \sigma_{\max}\log n}}{\underline{C}n}\lesssim c_{21}$ and $\eta\ll \frac{\underline{C}c_{21}}{c_{51} c_{\aug}\sqrt{\sigma_{\max}}}$.
By \eqref{ineq30:pr_lem_ncv2}, we further have
\begin{align*}
\left\|
\begin{bmatrix}
H^{t+1}-H^{t+1,(m)}\\
F^{t+1}R^{t+1}-F^{t+1,(m)}O^{t+1,(m)}
\end{bmatrix}
\right\|_F \leq c_{21}.
\end{align*}
Finally, by the arbitrariness of $m$, we finish the proofs.

\subsection{Proofs of Lemma \ref{lem:ncv3}}\label{pf:lem_ncv3}
Suppose Lemma \ref{lem:ncv1}-Lemma \ref{lem:ncv5} hold for the $t$-th iteration. In the following, we prove Lemma \ref{lem:ncv3} for the $(t+1)$-th iteration. Note that
\begin{align*}
&\nabla_X L^{(m)}(H,X,Y)\\
&=\frac1n\sum_{\substack{
        i \neq j \\
        i, j \neq m
    }}\left(\frac{e^{P_{ij}}}{1+e^{P_{ij}}}-A_{ij}\right)e_ie_j^TY\\
&\quad+\frac1n\sum_{i\neq m}\left(\frac{e^{P_{im}}}{1+e^{P_{im}}}-\frac{e^{P_{im}^*}}{1+e^{P_{im}^*}}\right)e_ie_m^TY+\frac1n\sum_{i\neq m}\left(\frac{e^{P_{mi}}}{1+e^{P_{mi}}}-\frac{e^{P_{mi}^*}}{1+e^{P_{mi}^*}}\right)e_me_i^TY,
\end{align*}
where the $m$-th row of the first and second terms are all zeros. Thus, by the gradient descent update, we have
\begin{align*}
&\left(F^{t+1,(m)}R^{t+1,(m)}-F^{*}\right)_{m,\cdot}\\
&=\left(X^{t,(m)}_{m,\cdot}-\eta\left\{\frac1n\sum_{i\neq m}\left(\frac{e^{P_{mi}^{t,(m)}}}{1+e^{P_{mi}^{t,(m)}}}-\frac{e^{P_{mi}^*}}{1+e^{P_{mi}^*}}\right)e_i^TY^{t,(m)}+\lambda X^{t,(m)}_{m,\cdot}\right\}\right)R^{t+1,(m)}-X^*_{m,\cdot}\\
&=\left\{X^{t,(m)}_{m,\cdot}R^{t,(m)}-X^*_{m,\cdot}-\eta\left(\frac1n\sum_{i\neq m}\left(\frac{e^{P_{mi}^{t,(m)}}}{1+e^{P_{mi}^{t,(m)}}}-\frac{e^{P_{mi}^*}}{1+e^{P_{mi}^*}}\right)e_i^TY^{t,(m)}+\lambda X^{t,(m)}_{m,\cdot}\right)R^{t,(m)}\right\}\\
&+\left\{X^{t,(m)}_{m,\cdot}R^{t,(m)}-\eta\left(\frac1n\sum_{i\neq m}\left(\frac{e^{P_{mi}^{t,(m)}}}{1+e^{P_{mi}^{t,(m)}}}-\frac{e^{P_{mi}^*}}{1+e^{P_{mi}^*}}\right)e_i^TY^{t,(m)}+\lambda X^{t,(m)}_{m,\cdot}\right)R^{t,(m)}\right\}\left(\left(R^{t,(m)}\right)^{-1}R^{t+1,(m)}-I_r\right)
\end{align*}
By the mean value theorem, we have for some $\{c_i\}$ that
\begin{align*}
& \sum_{i\neq m}\left(\frac{e^{P^{t,(m)}_{mi}}}{1+e^{P^{t,(m)}_{mi}}}-\frac{e^{P^{*}_{mi}}}{1+e^{P^{*}_{mi}}}\right)e_i^TY^{t,(m)}\\
&= \sum_{i\neq m}\frac{e^{c_i}}{(1+e^{c_i})^2}\left(P^{t,(m)}_{mi}-P^{*}_{mi}\right)e_i^TY^{t,(m)}\\
& = \sum_{i\neq m}\frac{e^{c_i}}{(1+e^{c_i})^2}\left(\langle H^{t,(m)}-H^*,z_mz_i^T\rangle\right)e_i^TY^{t,(m)}\\
& \quad +\frac1n\sum_{i\neq m}\frac{e^{c_i}}{(1+e^{c_i})^2}\left(X^{t,(m)}(Y^{t,(m)})^T-X^{*}Y^{*T}\right)_{mi}e_i^TY^{t,(m)}.
\end{align*}
Note that
\begin{align*}
&X^{t,(m)}(Y^{t,(m)})^T-X^{*}Y^{*T}\\
&=\left(X^{t,(m)}R^{t,(m)}-X^*\right)\left(Y^{*}\right)^T+(X^{t,(m)}R^{t,(m)})\left(Y^{t,(m)}R^{t,(m)}-Y^*\right)^T.
\end{align*}
Thus, we further have
\begin{align*}
& \sum_{i\neq m}\left(\frac{e^{P^{t,(m)}_{mi}}}{1+e^{P^{t,(m)}_{mi}}}-\frac{e^{P^{*}_{mi}}}{1+e^{P^{*}_{mi}}}\right)e_i^TY^{t,(m)}\\
&= \sum_{i\neq m}\frac{e^{c_i}}{(1+e^{c_i})^2}\left(\langle H^{t,(m)}-H^*,z_mz_i^T\rangle\right)e_i^TY^{t,(m)}\\
& \quad +\left(X^{t,(m)}R^{t,(m)}-X^*\right)_{m,\cdot}^{\top}\left(\frac1n\sum_{i\neq m}\frac{e^{c_i}}{(1+e^{c_i})^2}(Y^*_{i,\cdot})e_i^TY^{t,(m)}\right)\\
& \quad +\frac1n\sum_{i\neq m}\frac{e^{c_i}}{(1+e^{c_i})^2}\left((X^{t,(m)}R^{t,(m)})\left(Y^{t,(m)}R^{t,(m)}-Y^*\right)^T\right)_{mi}e_i^TY^{t,(m)}.
\end{align*}
Consequently, we have
\begin{align}\label{eq21:lem_ncv3}
&\left(F^{t+1,(m)}R^{t+1,(m)}-F^{*}\right)_{m,\cdot}\notag\\
&=\left(I_r-\frac{\eta}{n^2}\sum_{i\neq m}\frac{e^{c_i}}{(1+e^{c_i})^2}(Y^*_{i,\cdot})e_i^TY^{t,(m)}R^{t,(m)}\right)\left(X^{t,(m)}R^{t,(m)}-X^*\right)_{m,\cdot}\notag\\
&\quad -\frac{\eta}{n}\sum_{i\neq m}\frac{e^{c_i}}{(1+e^{c_i})^2}\left(\langle H^{t,(m)}-H^*,z_mz_i^T\rangle\right)e_i^TY^{t,(m)}R^{t,(m)}\notag\\
&\quad-\frac{\eta}{n^2}\sum_{i\neq m}\frac{e^{c_i}}{(1+e^{c_i})^2}\left((X^{t,(m)}R^{t,(m)})\left(Y^{t,(m)}R^{t,(m)}-Y^*\right)^T\right)_{mi}e_i^TY^{t,(m)}R^{t,(m)}-\eta\lambda X^{t,(m)}_{m,\cdot}R^{t,(m)}\notag\\
&+\left\{X^{t,(m)}_{m,\cdot}R^{t,(m)}-\eta\left(\frac1n\sum_{i\neq m}\left(\frac{e^{P_{mi}^{t,(m)}}}{1+e^{P_{mi}^{t,(m)}}}-\frac{e^{P_{mi}^*}}{1+e^{P_{mi}^*}}\right)e_i^TY^{t,(m)}+\lambda X^{t,(m)}_{m,\cdot}\right)R^{t,(m)}\right\}\left(\left(R^{t,(m)}\right)^{-1}R^{t+1,(m)}-I_r\right)\notag\\
&=\left(I_r-\frac{\eta}{n^2}\sum_{i\neq m}\frac{e^{c_i}}{(1+e^{c_i})^2}(Y^*_{i,\cdot})(Y^*_{i,\cdot})^{\top}\right)\left(X^{t,(m)}R^{t,(m)}-X^*\right)_{m,\cdot}+r_1,
\end{align}
where
\begin{align*}
r_1&=-\frac{\eta}{n^2}\underbrace{\left(\sum_{i\neq m}\frac{e^{c_i}}{(1+e^{c_i})^2}(Y^*_{i,\cdot})\left(Y^{t,(m)}R^{t,(m)}-Y^{*}\right)_{i,\cdot}^{\top}\right)\left(X^{t,(m)}R^{t,(m)}-X^*\right)_{m,\cdot}}_{(a)}\notag\\
&\quad-\frac{\eta}{n}\underbrace{\sum_{i\neq m}\frac{e^{c_i}}{(1+e^{c_i})^2}\left(\langle H^{t,(m)}-H^*,z_mz_i^T\rangle\right)e_i^TY^{t,(m)}R^{t,(m)}}_{(b)}\\
&\quad-\frac{\eta}{n^2}\underbrace{\sum_{i\neq m}\frac{e^{c_i}}{(1+e^{c_i})^2}\left((X^{t,(m)}R^{t,(m)})\left(Y^{t,(m)}R^{t,(m)}-Y^*\right)^T\right)_{mi}e_i^TY^{t,(m)}R^{t,(m)}}_{(c)}-\eta\lambda \underbrace{X^{t,(m)}_{m,\cdot}R^{t,(m)}}_{(d)}\\
&+\underbrace{\left\{X^{t,(m)}_{m,\cdot}R^{t,(m)}-\eta\left(\frac1n\sum_{i\neq m}\left(\frac{e^{P_{mi}^{t,(m)}}}{1+e^{P_{mi}^{t,(m)}}}-\frac{e^{P_{mi}^*}}{1+e^{P_{mi}^*}}\right)e_i^TY^{t,(m)}+\lambda X^{t,(m)}_{m,\cdot}\right)R^{t,(m)}\right\}\left(\left(R^{t,(m)}\right)^{-1}R^{t+1,(m)}-I_r\right)}_{(e)}
\end{align*}
We bound $\|r_1\|_2$ in the following.
\begin{enumerate}
    
\item For (a), by Cahuchy-Schwarz, we have
\begin{align*}
\| (a)   \|_2\leq \left\|\left(X^{t,(m)}R^{t,(m)}-X^*\right)_{m,\cdot}\right\|_2\|Y^{t,(m)}R^{t,(m)}-Y^*\|_F\|Y^*\|_F.
\end{align*}
Note that
\begin{align*}
&\left\|\left(X^{t,(m)}R^{t,(m)}-X^*\right)_{m,\cdot}\right\|_2\leq c_{31}\\
&\|Y^{t,(m)}R^{t,(m)}-Y^*\|_F\\
&\quad\leq \|Y^{t,(m)}R^{t,(m)}-Y^tR^t\|_F+\|Y^{t}R^{t}-Y^*\|_F\\
&\quad\leq 5\kappa\|Y^{t,(m)}O^{t,(m)}-Y^tR^t\|_F+\|Y^{t}R^{t}-Y^*\|_F\tag{by Lemma \ref{ar:lem2}}\\
&\quad\leq 5\kappa c_{21}+c_{11}\sqrt{n}\\
&\quad \lesssim c_{11}\sqrt{n}\\
&\|Y^*\|_F\leq \sqrt{\mu r\sigma_{\max}}.
\end{align*}
Thus, we have
\begin{align}\label{ineq15:lem_ncv3}
\| (a)   \|_2\lesssim c_{11}c_{31}     \sqrt{\mu r\sigma_{\max}n}.
\end{align}

\item For (b), note that
\begin{align*}
&\left\|\sum_{i\neq m}\frac{e^{c_i}}{(1+e^{c_i})^2}\left(\langle H^{t,(m)}-H^*,z_mz_i^T\rangle\right)e_i\right\|_2\\
&\leq\frac{1}{4}\sqrt{\sum_{i\neq m}\left|\langle H^{t,(m)}-H^*,z_mz_i^T\rangle\right|^2}\\
&\leq \frac{1}{4}\| H^{t,(m)}-H^*\|\|z_m\|_2\sqrt{\sum_{i\neq m}\|z_i\|_2^2}\\
&\leq \frac{\sz}{4\sqrt{n}}\| H^{t,(m)}-H^*\|.
\end{align*}
Moreover, we have
\begin{align*}
    \|Y^{t,(m)}\|\leq  \|F^{t,(m)}O^{t,(m)}-F^tR^t\|+\|F^tR^t-F^*\|+\|F^*\|\leq 2\|F^*\|.
\end{align*}
Thus, we have
\begin{align}\label{ineq16:lem_ncv3}
\|(b)\|_2\leq \frac{\sz}{\sqrt{n}}\| H^{t,(m)}-H^*\|\|F^*\|\lesssim \sqrt{\sigma_{\max}}\sz c_{11}.
\end{align}

\item For (c), note that
\begin{align*}
&\left\|\sum_{i\neq m}\frac{e^{c_i}}{(1+e^{c_i})^2}\left((X^{t,(m)}R^{t,(m)})\left(Y^{t,(m)}R^{t,(m)}-Y^*\right)^T\right)_{mi}e_i\right\|_2\\
&\leq \frac14 \left\|\left((X^{t,(m)}R^{t,(m)})\left(Y^{t,(m)}R^{t,(m)}-Y^*\right)^T\right)_{m,\cdot}\right\|_2\\
&= \frac14 \left\|(X^{t,(m)}R^{t,(m)})^{\top}_{m,\cdot}\left(Y^{t,(m)}R^{t,(m)}-Y^*\right)^T\right\|_2\\
&\leq \frac14  \left\|(X^{t,(m)}R^{t,(m)})_{m,\cdot}\right\|_2\left\|Y^{t,(m)}R^{t,(m)}-Y^*\right\|\\
&\lesssim  \left\|F^*\right\|_{2,\infty}\left\|Y^{t,(m)}R^{t,(m)}-Y^*\right\|.
\end{align*}
Thus, we have
\begin{align}\label{ineq17:lem_ncv3}
\|(c)\|_2\lesssim \|F^*\|\left\|F^*\right\|_{2,\infty}\left\|Y^{t,(m)}R^{t,(m)}-Y^*\right\|\lesssim \sqrt{\mu r}c_{11}\sigma_{\max}.
\end{align}

\item For (d), we have
\begin{align}\label{ineq18:lem_ncv3}
\|(d)\|_2\leq \left\|X^{t,(m)}R^{t,(m)}\right\|_{2,\infty}\leq \left\|F^{t,(m)}R^{t,(m)}-F^*\right\|_{2,\infty}+\|F^*\|_{2,\infty}\leq2\|F^*\|_{2,\infty}\lesssim \sqrt{\frac{\mu r \sigma_{\max}}{n}}.
\end{align}

\item Finally, we bound (e).
We denote
\begin{align*}
(1)&:=\left\{X^{t,(m)}_{m,\cdot}R^{t,(m)}-X^*_{m,\cdot}-\eta\left(\frac1n\sum_{i\neq m}\left(\frac{e^{P_{mi}^{t,(m)}}}{1+e^{P_{mi}^{t,(m)}}}-\frac{e^{P_{mi}^*}}{1+e^{P_{mi}^*}}\right)e_i^TY^{t,(m)}+\lambda X^{t,(m)}_{m,\cdot}\right)R^{t,(m)}\right\}\\
&=\left(I_r-\frac{\eta}{n^2}\sum_{i\neq m}\frac{e^{c_i}}{(1+e^{c_i})^2}(Y^*_{i,\cdot})(Y^*_{i,\cdot})^{\top}\right)\left(X^{t,(m)}R^{t,(m)}-X^*\right)_{m,\cdot}\notag\\
&\quad-\frac{\eta}{n^2}(a)-\frac{\eta}{n}(b)-\frac{\eta}{n^2}(c)-\eta\lambda (d).
\end{align*}
Note that, by Cauchy-Schwartz, we have
\begin{align*}
&\left\|\sum_{i\neq m}\frac{e^{c_i}}{(1+e^{c_i})^2}(Y^*_{i,\cdot})(Y^*_{i,\cdot})^{\top}\right\|\leq \|Y^*\|_F^2.
\end{align*}
Thus, it can be seen that $\|(1)\|_2\leq \|F^*\|_{2,\infty}$ and we have
\begin{align*}
\|(1)+X^*_{m,\cdot}\|_2\leq \|(1)\|_2+\|X^*_{m,\cdot}\|_2\leq\|(1)\|_2+\|F^*\|_{2,\infty}\leq 2\|F^*\|_{2,\infty}.
\end{align*}
Note that
\begin{align*}
(e)=\left((1)+X^*_{m,\cdot}\right)\left(\left(R^{t,(m)}\right)^{-1}R^{t+1,(m)}-I_r\right).
\end{align*}
Regarding the term $\left(R^{t,(m)}\right)^{-1}R^{t+1,(m)}-I_r$, we have the following claim.
\begin{claim}\label{claim:rotation}
With probability at least $1-n^{-10}$, we have
\begin{align*}
\left\|\left(R^{t,(m)}\right)^{-1}R^{t+1,(m)}-I_r\right\|\lesssim \frac{\eta}{n}\left\| F^{t, (m)}R^{t, (m)}-F^*\right\|.
\end{align*}
\end{claim}
Consequently, we have
\begin{align}\label{ineq19:lem_ncv3}
\|(e)\|_2&\leq \|(1)+X^*_{m,\cdot}\|_2\left\|\left(R^{t,(m)}\right)^{-1}R^{t+1,(m)}-I_r\right\|\notag\\
&\lesssim\frac{\eta}{n}\left\| F^{t, (m)}R^{t, (m)}-F^*\right\|\|F^*\|_{2,\infty}\notag\\
&\lesssim \frac{\eta}{n}\sqrt{\mu r\sigma_{\max}}c_{11}.
\end{align}

It remains to prove Claim \ref{claim:rotation}.
\begin{proof}[Proof of Claim \ref{claim:rotation}]
To facilitate analysis, we introduce an auxiliary point $\tilde{F}^{t+1,(m)}:=\begin{bmatrix}
\tilde{X}^{t+1,(m)}\\
\tilde{Y}^{t+1,(m)}
\end{bmatrix}$ where
\begin{align*}
&\tilde{X}^{t+1,(m)}\\
&={X}^{t,(m)}{R}^{t,(m)}-\eta\left[\frac1n\sum_{\substack{
        i \neq j \\
        i, j \neq m
    }}\left(\frac{e^{P_{ij}^{t,(m)}}}{1+e^{P_{ij}^{t,(m)}}}-A_{ij}\right)e_ie_j^T+\frac1n\sum_{i\neq m}\left(\frac{e^{P_{im}^{t,(m)}}}{1+e^{P_{im}^{t,(m)}}}-\frac{e^{P_{im}^{*}}}{1+e^{P_{im}^{*}}}\right)(e_ie_m^T+e_me_i^T)\right]Y^*\\
&\quad-\eta\lambda X^*-\frac{4c_{\aug}\eta}{n^2}X^*({R}^{t,(m)})^T\left(({X}^{t,(m)})^T{X}^{t,(m)}-({Y}^{t,(m)})^T{Y}^{t,(m)}\right){R}^{t,(m)},\\
&\tilde{Y}^{t+1,(m)}\\
&={Y}^{t,(m)}{R}^{t,(m)}-\eta\left[\frac1n\sum_{\substack{
        i \neq j \\
        i, j \neq m
    }}\left(\frac{e^{P_{ij}^{t,(m)}}}{1+e^{P_{ij}^{t,(m)}}}-A_{ij}\right)e_je_i^T+\frac1n\sum_{i\neq m}\left(\frac{e^{P_{im}^{t,(m)}}}{1+e^{P_{im}^{t,(m)}}}-\frac{e^{P_{im}^{*}}}{1+e^{P_{im}^{*}}}\right)(e_ie_m^T+e_me_i^T)\right]X^*\\
&\quad-\eta\lambda Y^*-\frac{4c_{\aug}\eta}{n^2}Y^*({R}^{t,(m)})^T\left(({Y}^{t,(m)})^T{Y}^{t,(m)}-({X}^{t,(m)})^T{X}^{t,(m)}\right){R}^{t,(m)}.
\end{align*}
We have the following claim.
\begin{claim}\label{claim2:rotation}
It holds that
\begin{align*}
I_r=\argmin_{R\in\cO^{r\times r}}\left\|\tilde{F}^{t+1,(m)}R-F^*\right\|_F,\text{ and } \sigma_{\min}\left(\tilde{F}^{t+1,(m)T}F^{*}\right)\geq \sigma_{\min}/2.
\end{align*}
\begin{proof}[Proof of Claim \ref{claim2:rotation}]
See Claim 4 in \cite{chen2020noisy}.
\end{proof}
\end{claim}
With this claim at hand, by Lemma \ref{ar:lem5} with $S=\tilde{F}^{t+1,(m)T}F^{*}$ and $K=\left({F}^{t+1,(m)}R^{t,(m)}-\tilde{F}^{t+1,(m)}\right)^TF^*$, we have
\begin{align}\label{ineq1:claim2_rotation}
&\left\|\left(R^{t,(m)}\right)^{-1}R^{t+1,(m)}-I_r\right\|\\
&=\left\|\text{sgn}(S+K)-\text{sgn}(S)\right\|\tag{by Claim \ref{claim2:rotation} and the definition of $R^{t+1,(m)}$}\\
&\leq \frac{1}{\sigma_{\min}\left(\tilde{F}^{t+1,(m)T}F^{*}\right)}\left\|\left({F}^{t+1,(m)}R^{t,(m)}-\tilde{F}^{t+1,(m)}\right)^TF^*\right\|\tag{by Lemma \ref{ar:lem5}}\\
&\leq \frac{2}{\sigma_{\min}}\left\|{F}^{t+1,(m)}R^{t,(m)}-\tilde{F}^{t+1,(m)}\right\|\|F^*\|\tag{by Claim \ref{claim2:rotation}}.
\end{align}
Here  \(\text{sgn}(A) = UV^\top\) for a matrix \( A \) with SVD \( U \Sigma V^\top \). Note that
\begin{align*}
    &F^{t+1, (m)} R^{t, (m)} - \tilde{F}^{t+1, (m)} \\
    &= -\eta \begin{bmatrix} B & 0 \\ 0 & B^\top \end{bmatrix} \begin{bmatrix} Y^{t, (m)}R^{t, (m)}-Y^*  \\ X^{t, (m)}R^{t, (m)}-X^* \end{bmatrix} + \frac{4c_{\aug}\eta}{n^2} \begin{bmatrix} X^\star \\ -Y^\star \end{bmatrix} R^{t, (m) \top} C R^{t, (m)} - \eta \lambda\begin{bmatrix}   X^{t, (m)}R^{t, (m)}-X^*\\ Y^{t, (m)}R^{t, (m)}-Y^* \end{bmatrix} ,
\end{align*}
where we denote
\begin{align*}
 &B := \frac1n\sum_{\substack{
        i \neq j \\
        i, j \neq m
    }}\left(\frac{e^{P_{ij}^{t,(m)}}}{1+e^{P_{ij}^{t,(m)}}}-A_{ij}\right)e_ie_j^T+\frac1n\sum_{i\neq m}\left(\frac{e^{P_{im}^{t,(m)}}}{1+e^{P_{im}^{t,(m)}}}-\frac{e^{P_{im}^{*}}}{1+e^{P_{im}^{*}}}\right)(e_ie_m^T+e_me_i^T),\\ 
 &C := X^{t, (m) \top} X^{t, (m)} - Y^{t, (m) \top} Y^{t, (m)}.
\end{align*}
This enables us to obtain
\begin{align*}
 &\left\|F^{t+1, (m)} R^{t, (m)} - \tilde{F}^{t+1, (m)}\right\| \\
& \leq \eta \left\| B \right\| \left\| F^{t, (m)}R^{t, (m)}-F^*\right\| + \frac{4c_{\aug}\eta}{n^2} \left\| F^\star \right\| \left\| C \right\|_F +\eta \lambda\left\| F^{t, (m)}R^{t, (m)}-F^*\right\|.   
\end{align*}
We then bound $\|B\|$ in the following.
Note that
\begin{align*}
B&=\frac1n\sum_{i\neq j}\left(\frac{e^{P_{ij}^{t,(m)}}}{1+e^{P_{ij}^{t,(m)}}}-\frac{e^{P_{ij}^{*}}}{1+e^{P_{ij}^{*}}}\right)e_ie_j^T +   \frac1n\sum_{\substack{
        i \neq j \\
        i, j \neq m
    }}\left(\frac{e^{P_{ij}^{*}}}{1+e^{P_{ij}^{*}}}-A_{ij}\right)e_ie_j^T.
\end{align*}
For the first term, we have
\begin{align*}
&\left\|\sum_{i\neq j}\left(\frac{e^{P_{ij}^{t,(m)}}}{1+e^{P_{ij}^{t,(m)}}}-\frac{e^{P_{ij}^{*}}}{1+e^{P_{ij}^{*}}}\right)e_ie_j^T\right\|\\
&\leq \left\|\sum_{i\neq j}\left(\frac{e^{P_{ij}^{t,(m)}}}{1+e^{P_{ij}^{t,(m)}}}-\frac{e^{P_{ij}^{*}}}{1+e^{P_{ij}^{*}}}\right)e_ie_j^T\right\|_F\\
&=\sqrt{\sum_{i\neq j}\left(\frac{e^{P_{ij}^{t,(m)}}}{1+e^{P_{ij}^{t,(m)}}}-\frac{e^{P_{ij}^{*}}}{1+e^{P_{ij}^{*}}}\right)^2}\\
&=\frac14\sqrt{\sum_{i\neq j}\left(P_{ij}^{t,(m)}-P_{ij}^{*}\right)^2}\tag{by mean value theorem}\\
&\lesssim \sqrt{\su}\|H^{t,(m)}-H^{*}\|_{F}+\frac{1}{n}\|F^*\|\|F^{t,(m)}R^{t,(m)}-F^{*}\|_{F}.
\end{align*}
For the second term, same as bounding $\|\nabla_{\Gamma}L_c(H^*,\Gamma^*)\|$ in the proof of Lemma \ref{lem:gradient_norm}, we have
\begin{align*}
\left\| \frac1n\sum_{\substack{
        i \neq j \\
        i, j \neq m
    }}\left(\frac{e^{P_{ij}^{*}}}{1+e^{P_{ij}^{*}}}-A_{ij}\right)e_ie_j^T\right\|\lesssim\sqrt{\frac{\log n}{n}}.
\end{align*}
Consequently, we have
\begin{align*}
\|B\|&\lesssim \frac{1}{n}\left(\sqrt{\su}\|H^{t,(m)}-H^{*}\|_{F}+\frac{1}{n}\|F^*\|\|F^{t,(m)}R^{t,(m)}-F^{*}\|_{F}\right)+\sqrt{\frac{\log n}{n}}\\
&\lesssim \sqrt{\frac{\overline{C}}{n}}c_{11}+\sqrt{\frac{\log n}{n}}.
\end{align*}
Thus, we have
\begin{align*}
 &\left\|F^{t+1, (m)} R^{t, (m)} - \tilde{F}^{t+1, (m)}\right\| \\
& \leq c\eta\sqrt{\frac{\overline{C}c^2_{11}+\log n}{n}}\left\| F^{t, (m)}R^{t, (m)}-F^*\right\| + \frac{4c_{\aug}\eta}{n^2} \left\| F^\star \right\| \left\|  X^{t, (m) \top} X^{t, (m)} - Y^{t, (m) \top} Y^{t, (m)} \right\|_F \\
&\quad +\eta \lambda\left\| F^{t, (m)}R^{t, (m)}-F^*\right\|.   
\end{align*}
By \eqref{ineq1:claim2_rotation}, we obtain
\begin{align*}
 &\left\|\left(R^{t,(m)}\right)^{-1}R^{t+1,(m)}-I_r\right\|\\
 &\leq \frac{2\|F^*\|}{\sigma_{\min}}  \Bigg(c\eta\sqrt{\frac{\overline{C}c^2_{11}+\log n}{n}}\left\| F^{t, (m)}R^{t, (m)}-F^*\right\| + \frac{4c_{\aug}\eta}{n^2} \left\| F^\star \right\| \left\|  X^{t, (m) \top} X^{t, (m)} - Y^{t, (m) \top} Y^{t, (m)} \right\|_F \\
 &\qquad\qquad\qquad+\eta \lambda\left\| F^{t, (m)}R^{t, (m)}-F^*\right\|\Bigg)\\
  &\lesssim \frac{\|F^*\|}{\sigma_{\min}} \eta \Bigg(\sqrt{\frac{\overline{C}c^2_{11}+\log n}{n}}\left\| F^{t, (m)}R^{t, (m)}-F^*\right\| + c_{\aug}c_{51}\eta\sqrt{\sigma_{\max}} +\lambda\left\| F^{t, (m)}R^{t, (m)}-F^*\right\|\Bigg)\\
&\lesssim \frac{\eta\sqrt{\sigma_{\max}}}{\sigma_{\min}}\left(\sqrt{\frac{\overline{C}c^2_{11}+\log n}{n}}+\lambda\right) \left\| F^{t, (m)}R^{t, (m)}-F^*\right\| \tag{as long as $\eta$ small enough}\\
&\lesssim \frac{\eta}{n}\left\| F^{t, (m)}R^{t, (m)}-F^*\right\|
\end{align*}
as long as $\frac{n\sqrt{\sigma_{\max}}}{\sigma_{\min}}\left(\sqrt{\frac{\overline{C}c^2_{11}+\log n}{n}}+\lambda\right)\ll 1$.
We then prove Claim \ref{claim:rotation}.

\end{proof}
\end{enumerate}

Combine \eqref{ineq15:lem_ncv3}, \eqref{ineq16:lem_ncv3}, \eqref{ineq17:lem_ncv3}, \eqref{ineq18:lem_ncv3}, \eqref{ineq19:lem_ncv3}, we obtain
\begin{align*}
&\|r_1\|_2\\
&\lesssim \eta\left(\frac{1}{n^{3/2}} c_{11}c_{31} \sqrt{\mu r\sigma_{\max}}+\frac{1}{n}\sqrt{\sigma_{\max}}\sz c_{11}+\frac{1}{n^2}\sqrt{\mu r}c_{11}\sigma_{\max}+\lambda \sqrt{\frac{\mu r \sigma_{\max}}{n}}+ \frac{1}{n}\sqrt{\mu r\sigma_{\max}}c_{11}\right).
\end{align*}

Similarly, we have 
\begin{align}\label{eq20:lem_ncv3}
&\left(F^{t+1,(m)}R^{t+1,(m)}-F^{*}\right)_{m+n,\cdot}\notag\\
&=\left(I_r-\frac{\eta}{n^2}\sum_{i\neq m}\frac{e^{c_i}}{(1+e^{c_i})^2}(X^*_{i,\cdot})(X^*_{i,\cdot})^{\top}\right)\left(Y^{t,(m)}R^{t,(m)}-Y^*\right)_{m,\cdot}+r_2,
\end{align}
where
\begin{align*}
&\|r_2\|_2\\
&\lesssim \eta\left(\frac{1}{n^{3/2}} c_{11}c_{31} \sqrt{\mu r\sigma_{\max}}+\frac{1}{n}\sqrt{\sigma_{\max}}\sz c_{11}+\frac{1}{n^2}\sqrt{\mu r}c_{11}\sigma_{\max}+\lambda \sqrt{\frac{\mu r \sigma_{\max}}{n}}+ \frac{1}{n}\sqrt{\mu r\sigma_{\max}}c_{11}\right).
\end{align*}

By\eqref{eq21:lem_ncv3} and \eqref{eq20:lem_ncv3}, we obtain
\begin{align}\label{ineq58:pf_lem_ncv3}
&\begin{bmatrix}
\left(F^{t+1,(m)}R^{t,(m)}-F^{*}\right)_{m,\cdot}\\
\left(F^{t+1,(m)}R^{t,(m)}-F^{*}\right)_{m+n,\cdot}
\end{bmatrix}\notag\\
&=A\begin{bmatrix}
\left(F^{t,(m)}R^{t,(m)}-F^{*}\right)_{m,\cdot}\\
\left(F^{t,(m)}R^{t,(m)}-F^{*}\right)_{m+n,\cdot}
\end{bmatrix}+
\begin{bmatrix}
r_1\\
0
\end{bmatrix}
+
\begin{bmatrix}
0\\
r_2
\end{bmatrix}
\end{align}
where
\begin{align*}
A&=
\begin{bmatrix}
  I_r-\frac{\eta}{n^2}\sum_{i\neq m}\frac{e^{c_i}}{(1+e^{c_i})^2}(Y^*_{i,\cdot})(Y^*_{i,\cdot})^{\top} & 0 \\
 0 & I_r-\frac{\eta}{n^2}\sum_{i\neq m}\frac{e^{c_i}}{(1+e^{c_i})^2}(X^*_{i,\cdot})(X^*_{i,\cdot})^{\top}\\
\end{bmatrix}\\
&=I_{2r}-\eta\sum_{i\neq m}\frac{e^{c_i}}{(1+e^{c_i})^2}
\begin{bmatrix}
\frac{1}{n}Y^*_{i,\cdot}\\
0
\end{bmatrix}^{\otimes 2}
-\eta\sum_{i\neq m}\frac{e^{c_i}}{(1+e^{c_i})^2}
\begin{bmatrix}
0\\
\frac{1}{n}X^*_{i,\cdot}
\end{bmatrix}^{\otimes 2}\\
&=I_{2r}-\eta\sum_{i\neq m}\frac{e^{c_i}}{(1+e^{c_i})^2}
\begin{bmatrix}
\frac{1}{n}e_i^{\top}Y^*\\
0
\end{bmatrix}^{\otimes 2}
-\eta\sum_{i\neq m}\frac{e^{c_i}}{(1+e^{c_i})^2}
\begin{bmatrix}
0\\
\frac{1}{n}e_i^{\top}X^*
\end{bmatrix}^{\otimes 2}\\
&=I_{2r}-\eta\sum_{i\neq m}\frac{e^{c_i}}{(1+e^{c_i})^2}
\left(
\begin{bmatrix}
\frac{1}{n}e_i^{\top}Y^*\\
0
\end{bmatrix}^{\otimes 2}
+
\begin{bmatrix}
0\\
\frac{1}{n}e_i^{\top}X^*
\end{bmatrix}^{\otimes 2}
\right).
\end{align*}

Notice that $|P^{t,(m)}_{im}|\leq 2|P^{*}_{im}|$. Thus, by Assumption \ref{assumption:scales}, we have $|c_i|\leq 2|P^{*}_{im}|\leq 2c_P$, which implies
\begin{align*}
\frac{e^{2c_P}}{(1+e^{2c_P})^2}\leq\frac{e^{c_i}}{(1+e^{c_i})^2}\leq \frac14.
\end{align*}
Denote 
\begin{align*}
B:=\sum_{i\neq m}\frac{e^{c_i}}{(1+e^{c_i})^2}
\left(
\begin{bmatrix}
\frac{1}{n}e_i^{\top}Y^*\\
0
\end{bmatrix}^{\otimes 2}
+
\begin{bmatrix}
0\\
\frac{1}{n}e_i^{\top}X^*
\end{bmatrix}^{\otimes 2}
\right).
\end{align*}
We then have
\begin{align}\label{ineq56:pf_lem_ncv3}
&\left\|
A\begin{bmatrix}
\left(F^{t,(m)}R^{t,(m)}-F^{*}\right)_{m,\cdot}\\
\left(F^{t,(m)}R^{t,(m)}-F^{*}\right)_{m+n,\cdot}
\end{bmatrix}
\right\|^2_2\notag\\
&=\begin{bmatrix}
\left(F^{t,(m)}R^{t,(m)}-F^{*}\right)_{m,\cdot}\\
\left(F^{t,(m)}R^{t,(m)}-F^{*}\right)_{m+n,\cdot}
\end{bmatrix}^{\top}(I-2\eta B+\eta^2 B^2)
\begin{bmatrix}
\left(F^{t,(m)}R^{t,(m)}-F^{*}\right)_{m,\cdot}\\
\left(F^{t,(m)}R^{t,(m)}-F^{*}\right)_{m+n,\cdot}
\end{bmatrix}\notag\\
&=\left\|
\begin{bmatrix}
\left(F^{t,(m)}R^{t,(m)}-F^{*}\right)_{m,\cdot}\\
\left(F^{t,(m)}R^{t,(m)}-F^{*}\right)_{m+n,\cdot}
\end{bmatrix}
\right\|^2_2
-2\eta \begin{bmatrix}
\left(F^{t,(m)}R^{t,(m)}-F^{*}\right)_{m,\cdot}\\
\left(F^{t,(m)}R^{t,(m)}-F^{*}\right)_{m+n,\cdot}
\end{bmatrix}^{\top}B
\begin{bmatrix}
\left(F^{t,(m)}R^{t,(m)}-F^{*}\right)_{m,\cdot}\\
\left(F^{t,(m)}R^{t,(m)}-F^{*}\right)_{m+n,\cdot}
\end{bmatrix}\notag\\
&\quad +
\eta^2\begin{bmatrix}
\left(F^{t,(m)}R^{t,(m)}-F^{*}\right)_{m,\cdot}\\
\left(F^{t,(m)}R^{t,(m)}-F^{*}\right)_{m+n,\cdot}
\end{bmatrix}^{\top}B^2
\begin{bmatrix}
\left(F^{t,(m)}R^{t,(m)}-F^{*}\right)_{m,\cdot}\\
\left(F^{t,(m)}R^{t,(m)}-F^{*}\right)_{m+n,\cdot}
\end{bmatrix}
\end{align}

Denote
\begin{align*}
B_1:=\sum_{i\neq m}\frac{e^{c_i}}{(1+e^{c_i})^2}
\begin{bmatrix}
\frac{1}{n}e_i^{\top}Y^*
\end{bmatrix}^{\otimes 2}
,\text{ and }  B_2:=\sum_{i\neq m}\frac{e^{c_i}}{(1+e^{c_i})^2}
\begin{bmatrix}
\frac{1}{n}e_i^{\top}X^*
\end{bmatrix}^{\otimes 2}.
\end{align*}
It can be seen that
\begin{align*}
&B_1\preccurlyeq \frac{1}{4n^2}\sum^n_{i=1}Y^*_{i,\cdot}Y^{* T}_{i,\cdot}=\frac{1}{4n^2} Y^{* T}Y^{*}\preccurlyeq \frac{\sigma_{\max}}{4n^2} I_{2r},\\
&B_2\preccurlyeq \frac{1}{4n^2}\sum^n_{i=1}X^*_{i,\cdot}X^{* T}_{i,\cdot}=\frac{1}{4n^2} X^{* T}X^{*}\preccurlyeq \frac{\sigma_{\max}}{4n^2} I_{2r}.
\end{align*}
Thus, $\|B\|\leq \|B_1\| + \|B_2\|\leq \frac{\sigma_{\max}}{2n^2}$. On the other hand, we want to lower bound the smallest eigenvalue of $B$. For any $\begin{bmatrix}
\bx\\
\by
\end{bmatrix}$, where $\bx,\by\in\bbR^{r}$, we have
\begin{align}
\begin{bmatrix}
\bx\\
\by
\end{bmatrix}^{\top}B
\begin{bmatrix}
\bx\\
\by
\end{bmatrix}
&={\bx}^{\top}
B_1
{\bx}
+
{\by}^{\top}
B_2
{\by}.\label{pf:lem_ncv3eq3}
\end{align}
An important observation is that $B_1, B_2$ are submatrices of \( D^* \), a fact we will leverage in the subsequent proofs.
We denote by $v\in \mathbb{R}^{p^2+2nr}$ such that 
\begin{align*}
    v_k =
    \begin{cases}
      x_j & \text{if } k = p^2+(m-1)r+j \text{ for some } j\in [r]\\
      0 & \text{otherwise}
    \end{cases}.
\end{align*}
Then we know that 
\begin{align}
{\bx}^{\top}
B_1
{\bx}
\geq \frac{4 e^{2 c_P}}{(1+e^{2 c_P})^2} v^\top D^* v
\geq \frac{4 e^{2 c_P}}{(1+e^{2 c_P})^2} v^\top \cP D^*\cP v
\label{pf:lem_ncv3eq1}.
\end{align}
Denote by $Q:=(\cP D^*\cP)^{\dagger}\cP D^*\cP$, then Assumption \ref{assumption:D_eigen} implies
\begin{align}
    v^\top \cP D^*\cP v\geq \slD \left\|Q v\right\|_2^2.\label{pf:lem_ncv3eq2}
\end{align}
On the other hand, we have
\begin{align*}
    \left\|Q v\right\|_2^2&\geq \sum_{j=1}^r (Qv)_{p^2+(m-1)r+j}^2 \\
    &=\sum_{j=1}^r \left(v_{p^2+(m-1)r+j} - ((I-Q)v)_{p^2+(m-1)r+j}\right)^2 \\
    &\geq \sum_{j=1}^r v_{p^2+(m-1)r+j}^2 - 2\left(\sum_{j=1}^r \left| v_{p^2+(m-1)r+j}\right|\right)\left\|(I-Q)v\right\|_\infty \\
    &\geq \left\|v\right\|_2^2 - 2\sqrt{r}\left\|v\right\|_2\left\|I-Q\right\|_{2,\infty}\left\|v\right\|_2 \geq  \left(1-2c_{2,\infty}\sqrt{\frac{r^3+rp}{n}}\right)\left\|v\right\|_2^2
\end{align*}
according to Assumption \ref{assumption:2_infty}. Therefore, as long as $n\geq 16(r^3+rp) c_{2,\infty}^2$, we have $\left\|Q v\right\|_2^2\geq \left\|v\right\|_2^2 / 2$. Combine this with \eqref{pf:lem_ncv3eq1} and \eqref{pf:lem_ncv3eq2} we get 
\begin{align*}
{\bx}^{\top}
B_1
{\bx}\geq \frac{4 e^{2 c_P}}{(1+e^{2 c_P})^2}\cdot \frac{\slD}{2}\left\|v\right\|_2^2 = \frac{2 \slD e^{2 c_P}}{(1+e^{2 c_P})^2}\left\|\bx\right\|_2^2. 
\end{align*}
Similarly, we have 
\begin{align*}
{\by}^{\top}
B_2
{\by}\geq\frac{2 \slD e^{2 c_P}}{(1+e^{2 c_P})^2}\left\|\by\right\|_2^2. 
\end{align*}
Plugging these in \eqref{pf:lem_ncv3eq3} we know that 
\begin{align*}
    \begin{bmatrix}
\bx\\
\by
\end{bmatrix}^{\top}B
\begin{bmatrix}
\bx\\
\by
\end{bmatrix}\geq \frac{2 \slD e^{2 c_P}}{(1+e^{2 c_P})^2}\left\|\bx\right\|_2^2+\frac{2 \slD e^{2 c_P}}{(1+e^{2 c_P})^2}\left\|\by\right\|_2^2
=\frac{2 \slD e^{2 c_P}}{(1+e^{2 c_P})^2}\left\|
\begin{bmatrix}
\bx\\
\by
\end{bmatrix}
\right\|^2_2.
\end{align*}
Since this holds for all $\bx,\by\in\bbR^{r}$, we know that 
\begin{align*}
    B\succcurlyeq\frac{2 \slD e^{2 c_P}}{(1+e^{2 c_P})^2} I_{2r}.
\end{align*}
To sum up, as long as $n\geq 16(r^3+rp) c_{2,\infty}^2$, we have 
\begin{align*}
   \frac{2 \slD e^{2 c_P}}{(1+e^{2 c_P})^2}  I_{2r}\preccurlyeq B\preccurlyeq \frac{\sigma_{\max}}{2n^2} I_{2r}.
\end{align*}

By \eqref{ineq56:pf_lem_ncv3}, we then have
\begin{align*}
&\left\|A
\begin{bmatrix}
\left(F^{t+1,(m)}R^{t,(m)}-F^{*}\right)_{m,\cdot}\\
\left(F^{t+1,(m)}R^{t,(m)}-F^{*}\right)_{m+n,\cdot}
\end{bmatrix}
\right\|_2^2\notag\\
&\leq \left(1- \frac{4 \slD e^{2 c_P}}{(1+e^{2 c_P})^2}\eta +\frac{\sigma_{\max}^2}{4n^4} \eta^2 \right)
\left\|
\begin{bmatrix}
\left(F^{t+1,(m)}R^{t,(m)}-F^{*}\right)_{m,\cdot}\\
\left(F^{t+1,(m)}R^{t,(m)}-F^{*}\right)_{m+n,\cdot}
\end{bmatrix}
\right\|_2^2\notag\\
&\leq \left(1-\frac{2 \slD e^{2 c_P}}{(1+e^{2 c_P})^2}\eta\right)\left\|
\begin{bmatrix}
\left(F^{t+1,(m)}R^{t,(m)}-F^{*}\right)_{m,\cdot}\\
\left(F^{t+1,(m)}R^{t,(m)}-F^{*}\right)_{m+n,\cdot}
\end{bmatrix}
\right\|_2^2
\end{align*}
as long as $\eta\leq \frac{8n^2e^{2c_{P}}\slD}{\sigma_{\max}^2(1+e^{2c_P})^2}$.
Recall equation \eqref{ineq58:pf_lem_ncv3}. Consequently, we have
\begin{align*}
&\left\|
\begin{bmatrix}
\left(F^{t+1,(m)}R^{t,(m)}-F^{*}\right)_{m,\cdot}\\
\left(F^{t+1,(m)}R^{t,(m)}-F^{*}\right)_{m+n,\cdot}
\end{bmatrix}
\right\|_2\\
&\leq
\left\|A
\begin{bmatrix}
\left(F^{t,(m)}R^{t,(m)}-F^{*}\right)_{m,\cdot}\\
\left(F^{t,(m)}R^{t,(m)}-F^{*}\right)_{m+n,\cdot}
\end{bmatrix}
\right\|_2+\|r_1\|_2+\|r_2\|_2\\
&\leq \left(1-\frac{ \slD e^{2 c_P}}{(1+e^{2 c_P})^2}\eta\right)c_{31}\\
&\quad +c\eta\left(\frac{1}{n^{3/2}} c_{11}c_{31} \sqrt{\mu r\sigma_{\max}}+\frac{1}{n}\sqrt{\sigma_{\max}}\sz c_{11}+\frac{1}{n^2}\sqrt{\mu r}c_{11}\sigma_{\max}+\lambda \sqrt{\frac{\mu r \sigma_{\max}}{n}}+ \frac{1}{n}\sqrt{\mu r\sigma_{\max}}c_{11}\right)\\
&\leq c_{31}
\end{align*}
as long as $\max\{\frac{\sz}{n}\sqrt{\overline{C}\mu r\sigma_{\max}}c_{11}, \lambda \sqrt{\frac{\mu r \sigma_{\max}}{n}}\}\lesssim \frac{ \slD e^{2 c_P}}{(1+e^{2 c_P})^2} c_{31}$.

\subsection{Proofs of Lemma \ref{lem:ncv4}}\label{pf:lem_ncv4}
Suppose Lemma \ref{lem:ncv1}-Lemma \ref{lem:ncv5} hold for the $t$-th iteration. In the following, we prove Lemma \ref{lem:ncv4} for the $(t+1)$-th iteration.
Since Lemma \ref{lem:ncv1}-Lemma \ref{lem:ncv5} hold for the $t$-th, we know Lemma \ref{lem:ncv1} holds for the $t+1$-th iteration, which implies $\|H^{t+1}-H^*\|_F\leq c_{11}\sqrt{n}$. 

For $1\leq m\leq n$, we have
\begin{align*}
\left\|\left(F^{t+1}R^{t+1}-F^*\right)_{m,\cdot}\right\|_2
&\leq \left\|\left(F^{t+1,(m)}R^{t+1,(m)}-F^*\right)_{m,\cdot}\right\|_2+ \left\|\left(F^{t+1,(m)}R^{t+1,(m)}-F^{t+1}R^{t+1}\right)_{m,\cdot}\right\|_2\\
&\leq \left\|\left(F^{t+1,(m)}R^{t+1,(m)}-F^*\right)_{m,\cdot}\right\|_2+ \left\|F^{t+1,(m)}R^{t+1,(m)}-F^{t+1}R^{t+1}\right\|_F\\
&\leq \left\|\left(F^{t+1,(m)}R^{t+1,(m)}-F^*\right)_{m,\cdot}\right\|_2+ 5\kappa\left\|F^{t+1,(m)}O^{t+1,(m)}-F^{t+1}R^{t+1}\right\|_F,
\end{align*}
and 
\begin{align*}
\left\|\left(F^{t+1}R^{t+1}-F^*\right)_{m+n,\cdot}\right\|_2
&\leq \left\|\left(F^{t+1,(m)}R^{t+1,(m)}-F^*\right)_{m+n,\cdot}\right\|_2+ \left\|\left(F^{t+1,(m)}R^{t+1,(m)}-F^{t+1}R^{t+1}\right)_{m+n,\cdot}\right\|_2\\
&\leq \left\|\left(F^{t+1,(m)}R^{t+1,(m)}-F^*\right)_{m+n,\cdot}\right\|_2+ \left\|F^{t+1,(m)}R^{t+1,(m)}-F^{t+1}R^{t+1}\right\|_F\\
&\leq \left\|\left(F^{t+1,(m)}R^{t+1,(m)}-F^*\right)_{m+n,\cdot}\right\|_2+ 5\kappa\left\|F^{t+1,(m)}O^{t+1,(m)}-F^{t+1}R^{t+1}\right\|_F,
\end{align*}
where the last inequalities follow from Lemma \ref{ar:lem2}.
It then holds that
\begin{align*}
&\max_{1\leq m\leq n}\left\|\left(F^{t+1}R^{t+1}-F^*\right)_{m,\cdot}\right\|_2\\
&\leq \max_{1\leq m\leq n} \left\|\left(F^{t+1,(m)}R^{t+1,(m)}-F^*\right)_{m,\cdot}\right\|_2+5\kappa\max_{1\leq m\leq n}\left\|F^{t+1,(m)}O^{t+1,(m)}-F^{t+1}R^{t+1}\right\|_F\\
&\leq c_{31}+5\kappa c_{21}=c_{41},\\
&\max_{1\leq m\leq n}\left\|\left(F^{t+1}R^{t+1}-F^*\right)_{m+n,\cdot}\right\|_2\\
&\leq \max_{1\leq m\leq n} \left\|\left(F^{t+1,(m)}R^{t+1,(m)}-F^*\right)_{m+n,\cdot}\right\|_2+5\kappa\max_{1\leq m\leq n}\left\|F^{t+1,(m)}O^{t+1,(m)}-F^{t+1}R^{t+1}\right\|_F\\
&\leq c_{31}+5\kappa c_{21}=c_{41},\\
\end{align*}
where the last inequalities hold since Lemma \ref{lem:ncv2} and \ref{lem:ncv3} hold for the $(t+1)$-th iteration. 
We then finish the proof.

\subsection{Proofs of Lemma \ref{lem:ncv5}}\label{pf:lem_ncv5}
Suppose Lemma \ref{lem:ncv1}-Lemma \ref{lem:ncv5} hold for the $t$-th iteration. In the following, we prove Lemma \ref{lem:ncv5} for the $(t+1)$-th iteration.


We first show $\|(X^{t+1})^TX^{t+1}-(Y^{t+1})^TY^{t+1}\|_F\leq c_{51}\eta n^2$. We denote 
\begin{align*}
&A^t=(X^{t})^TX^{t}-(Y^{t})^TY^{t},\ A^{t+1}=(X^{t+1})^TX^{t+1}-(Y^{t+1})^TY^{t+1},\\
&D^t=\frac1n\sum_{i\neq j}\left(\frac{e^{P_{ij}^t}}{1+e^{P_{ij}^t}}-A_{ij}\right)e_i e_j^T.
\end{align*}
Same as Lemma 15 in \cite{chen2020noisy}, it can be seen that
\begin{align}\label{ineq1:lem_ncv5}
\|A^{t+1}\|_F\leq (1-\lambda\eta)\|A^{t}\|_F+\eta^2\|Y^{tT}D^{tT}D^{t}Y^{t}-X^{tT}D^{tT}D^{t}X^{t}\|_F.
\end{align}
Note that
\begin{align*}
\|Y^{tT}D^{tT}D^{t}Y^{t}-X^{tT}D^{tT}D^{t}X^{t}\|_F
&\leq     \|Y^{tT}D^{tT}D^{t}Y^{t}\|_F+\|X^{tT}D^{tT}D^{t}X^{t}\|_F\\
&\leq \|Y^t\|\|D^t\|^2\|Y^t\|_F+\|X^t\|\|D^t\|^2\|X^t\|_F.
\end{align*}
Moreover, by Lemma \ref{lem:ncv1} for the $t$-th iteration, we have
\begin{align*}
&\|X^t\|\leq \|X^tR^t-X^*\|+\|X^*\|\leq 2\|X^*\|,\quad  \|X^t\|_F\leq \|X^tR^t-X^*\|_F+\|X^*\|_F\leq 2\|X^*\|_F,\\
&\|Y^t\|\leq 2\|Y^*\|,\ \|Y^t\|_F\leq 2\|Y^*\|_F.
\end{align*}
Thus, we have
\begin{align*}
 \|Y^{tT}D^{tT}D^{t}Y^{t}-X^{tT}D^{tT}D^{t}X^{t}\|_F
 &\leq 4\|Y^*\|\|Y^*\|_F\|D^t\|^2+4\|X^*\|\|X^*\|_F  \|D^t\|^2
 &\lesssim \sqrt{\mu r}\sigma_{\max}\|D^t\|^2.
\end{align*}
By Lemma \ref{lem:gradient_spectral}, we have
\begin{align*}
 \|D^t\|&\lesssim \sqrt{\frac{c_{11}^2\overline{C}+\log n}{n}}.
\end{align*}
This implies
\begin{align}\label{ineq2:lem_ncv5}
 \|Y^{tT}D^{tT}D^{t}Y^{t}-X^{tT}D^{tT}D^{t}X^{t}\|_F   \lesssim \sqrt{\mu r}\sigma_{\max}\|D^t\|^2\lesssim \frac{\sqrt{\mu r}\sigma_{\max}(c_{11}^2\overline{C}+\log n)}{n}.
\end{align}
Combine \eqref{ineq1:lem_ncv5} and \eqref{ineq2:lem_ncv5}, we have
\begin{align*}
\|A^{t+1}\|_F
&\leq (1-\lambda\eta)\|A^{t}\|_F+C\eta^2\frac{\sqrt{\mu r}\sigma_{\max}(c_{11}^2\overline{C}+\log n)}{n}\\
&\leq (1-\lambda\eta)c_{51}\eta n^2+C\eta^2\frac{\sqrt{\mu r}\sigma_{\max}(c_{11}^2\overline{C}+\log n)}{n}\\
&\leq c_{51}\eta n^2
\end{align*}
as long as $\lambda\eta<1$ and $c_{51}\lambda n^3\gg \sqrt{\mu r}\sigma_{\max} (c_{11}^2\overline{C}+\log n)$.
The upper bound on the leave-one-out sequences can be derived similarly.

We then prove \eqref{ineq:gradient} in the following.
By the gradient descent update, we have
\begin{align*}
&f(H^{t+1},F^{t+1})\\
&=f(H^{t+1},F^{t+1}R^t)\\
&=f\left(H^t-\eta\nabla_{H}f(H^t,F^t), 
F^tR^t-\eta\begin{bmatrix}
  \cP_Z^{\perp} & 0\\
  0 & \cP_Z^{\perp}
\end{bmatrix}\nabla_F f(H^t,F^tR^t)\right)\\
&=f(H^t,F^tR^t)-\eta\left\langle\nabla_{H} f(H^t,F^tR^t),\nabla_{H}f(H^t,F^t)\right\rangle\\
&\quad -\eta\left\langle\nabla_F f(H^t,F^tR^t),\begin{bmatrix}
  \cP_Z^{\perp} & 0\\
  0 & \cP_Z^{\perp}
\end{bmatrix}\nabla_F f(H^t,F^tR^t)\right\rangle\\
&\quad +\frac{\eta^2}{2}
\vec\begin{bmatrix}
\nabla_{H}f(H^t,F^t)\\
\begin{bmatrix}
  \cP_Z^{\perp} & 0\\
  0 & \cP_Z^{\perp}
\end{bmatrix}\nabla_F f(H^t,F^tR^t)
\end{bmatrix}^T \nabla^2 f(\tilde{H},\tilde{F})
\vec\begin{bmatrix}
\nabla_{H}f(H^t,F^t)\\
\begin{bmatrix}
  \cP_Z^{\perp} & 0\\
  0 & \cP_Z^{\perp}
\end{bmatrix}\nabla_F f(H^t,F^tR^t)
\end{bmatrix}\\
&=f(H^t,F^tR^t)-\eta\left\|\nabla_{H}f(H^t,F^t)\right\|_F^2 -\eta\left\|\begin{bmatrix}
  \cP_Z^{\perp} & 0\\
  0 & \cP_Z^{\perp}
\end{bmatrix}\nabla_F f(H^t,F^tR^t)\right\|^2_{F}\\
&\quad +\frac{\eta^2}{2}
\left(\cP\nabla f(H^t,F^tR^t)\right)^T \nabla^2 f(\tilde{H},\tilde{F})
\cP\nabla f(H^t,F^tR^t)\\
&=f(H^t,F^tR^t)-\eta\left\|\cP\nabla f(H^t,F^t)\right\|_2^2 \\
&\quad  +\frac{\eta^2}{2}
\left(\cP\nabla f(H^t,F^tR^t)\right)^T \nabla^2 f(\tilde{H},\tilde{F})
\cP\nabla f(H^t,F^tR^t).
\end{align*}
By Lemma \ref{lem:convexity}, we have 
\begin{align*}
\left(\cP\nabla f(H^t,F^tR^t)\right)^T \nabla^2 f(\tilde{H},\tilde{F})
\cP\nabla f(H^t,F^tR^t)\leq \overline{C}\left\|\cP\nabla f(H^t,F^tR^t)\right\|^2_2
=\overline{C}\left\|\cP\nabla f(H^t,F^t)\right\|^2_2.
\end{align*}
Thus, it holds that
\begin{align*}
&f(H^{t+1},F^{t+1})\\
&\leq f(H^t,F^t)-\left(\eta-\frac{\overline{C}}{2}\eta^2\right)\left\|\cP\nabla f(H^t,F^t)\right\|^2_2\\
&\leq f(H^t,F^t)-\frac{\eta}{2}
\left\|\cP\nabla f(H^t,F^t)\right\|^2_2
\end{align*}
as long as $\overline{{C}}\eta\leq 1$. We then finish the proofs.

\subsection{Proofs of Lemma \ref{lem:ncv6}}\label{pf:lem_ncv6}
\begin{proof}
Summing \eqref{ineq:gradient} from $t=1$ to $t=t_0-1$ leads to 
\begin{align*}
&f(H^{t_0},F^{t_0})\leq f(H^{0},F^{0})-\frac{\eta}{2}\sum^{t_0-1}_{t=0}\left\|\cP\nabla f(H^t,F^t)\right\|^2_2.
\end{align*}
Thus, we have
\begin{align}\label{ineq2:lem_ncv6}
&\min_{0\leq t<t_0}\left\|\cP\nabla f(H^t,F^t)\right\|_2\notag\\
&\leq \left(\frac{1}{t_0}\sum^{t_0-1}_{t=0}\left\|\cP\nabla f(H^t,F^t)\right\|^2_2\right)^{1/2}\notag\\
&\leq \left(\frac{2}{\eta t_0}\left(f(H^*,F^*)-f(H^{t_0},F^{t_0})\right)\right)^{1/2},
\end{align}
where we use the fact that $(H^0,F^0)=(H^*,F^*)$. 
Thus, it remains to control $f(H^*,F^*)-f(H^{t_0},F^{t_0})$. Note that
\begin{align*}
 f(H^{t_0},F^{t_0})
 &=f(H^{t_0},F^{t_0}R^{t_0})\\
 &=f(H^*,F^*)+\left\langle\nabla f(H^*,F^*),\vec
 \begin{bmatrix}
H^{t_0}-H^*\\
F^{t_0}R^{t_0}-F^*
 \end{bmatrix}
 \right\rangle\\
 &\quad+\frac12\vec\left(
 \begin{bmatrix}
H^{t_0}-H^*\\
F^{t_0}R^{t_0}-F^*
 \end{bmatrix}
 \right)^T \nabla^2 f(\tilde{H},\tilde{F})
 \vec\left(
 \begin{bmatrix}
H^{t_0}-H^*\\
F^{t_0}R^{t_0}-F^*
 \end{bmatrix}
 \right),
\end{align*}
where $(\tilde{H},\tilde{F})$ lies in the line segment connecting $(H^*,F^*)$ and $(H^{t_0},F^{t_0}R^{t_0})$. By triangle inequality, we have
\begin{align}\label{ineq1:lem_ncv6}
 &f(H^*,F^*)- f(H^{t_0},F^{t_0})\notag\\
 &\leq\|\nabla_{H} f(H^*,F^*)\|_F\|H^{t_0}-H^*\|_F+\|\nabla_{F} f(H^*,F^*)\|_F\|F^{t_0}R^{t_0}-F^*\|_F\notag\\
 &\quad+\frac12\left|\vec\left(
 \begin{bmatrix}
H^{t_0}-H^*\\
F^{t_0}R^{t_0}-F^*
 \end{bmatrix}
 \right)^T \nabla^2 f(\tilde{H},\tilde{F})
 \vec\left(
 \begin{bmatrix}
H^{t_0}-H^*\\
F^{t_0}R^{t_0}-F^*
 \end{bmatrix}
 \right)\right|.
\end{align}
By Lemma \ref{lem:ncv4}, it holds that
\begin{align*}
&\|\tilde{H}-H^*\|_{F}\leq \|{H}^{t_0}-H^*\|_{F}\leq c_{11}\sqrt{n}\leq c_2\sqrt{n},\\
&\|\tilde{F}-F^*\|_{2,\infty}\leq \|{F}^{t_0}{R}^{t_0}-F^*\|_{\infty}\leq c_{41}\leq c_3.
\end{align*}
Thus, by Lemma \ref{lem:convexity}, we have
\begin{align*}
&\frac{1}{2}\left|\vec\left(
 \begin{bmatrix}
H^{t_0}-H^*\\
F^{t_0}R^{t_0}-F^*
 \end{bmatrix}
 \right)^T \nabla^2 f(\tilde{H},\tilde{F})
 \vec\left(
 \begin{bmatrix}
H^{t_0}-H^*\\
F^{t_0}R^{t_0}-F^*
 \end{bmatrix}
 \right)\right|\\
 &\leq \frac{\overline{C}}{2}\left\|
  \begin{bmatrix}
H^{t_0}-H^*\\
F^{t_0}R^{t_0}-F^*
 \end{bmatrix}
 \right\|_F^2\\
 & \leq \frac{\overline{C} c_{11}^2 n}{2},
\end{align*}
where the last inequality follows from Lemma \ref{lem:ncv1}.

Consequently, by \eqref{ineq1:lem_ncv6} and Lemma \ref{lem:gradient_norm}, we have shown that
\begin{align*}
&f(H^*,F^*)- f(H^{t_0},F^{t_0})\\
&\lesssim \sz\sqrt{p\log n}\|H^{t_0}-H^*\|_F+\lambda \sqrt{\mu r\sigma_{\max}}\|F^{t_0}R^{t_0}-F^*\|_F+\overline{C} c_{11}^2 n,
\end{align*}
where we use the fact that $\|X^*\|_F\leq \sqrt{\mu r \sigma_{\max}}$ by Assumption \ref{assumption:incoherent}.
By Lemma \ref{lem:ncv1}, this further implies that
\begin{align*}
&f(H^*,F^*)- f(H^{t_0},F^{t_0})\\
&\lesssim (\sz\sqrt{p\log n}+\lambda \sqrt{\mu r\sigma_{\max}})\cdot c_{11}\sqrt{n}+\overline{C} c_{11}^2 n\\
&\lesssim n^2
\end{align*}
as long as $(\sz\sqrt{p\log n}+\lambda \sqrt{\mu r\sigma_{\max}})\cdot c_{11}\sqrt{n}+\overline{C} c_{11}^2 n\ll n^2$.
Thus, we have
\begin{align*}
\frac{2}{\eta t_0}\left(f(H^*,F^*)-f(H^{t_0},F^{t_0})\right)\lesssim\frac{n^2}{\eta t_0}\leq n^{-10}
\end{align*}
as long as $\eta t_0\geq n^{12}$ (which holds as long as $t_0$ large enough). Together with \eqref{ineq2:lem_ncv6}, we finish the proofs.

\end{proof}

%% file: pf_nonconvex_debias.tex
\section{Proofs of Section \ref{sec:nonconvex_debias}}\label{sec:pf_nonconvex_debias}

We first prove some useful lemmas in the following.

\begin{lemma}\label{lem:D}
Under Assumption \ref{assumption:D_eigen}, we have
\begin{align*}
\|\cP\hat{D}\cP-\cP D^* \cP\|\lesssim  \frac{\hat{c}}{\sqrt{n}}
\end{align*}
and 
\begin{align*}
\frac{\slD}{2}\leq \lambda_{\min}(\cP \hat{D} \cP)\leq \lambda_{\max}(\cP  \hat{D}  \cP)\leq {2\suD}.
\end{align*}
Here 
\begin{align*}
\hat{c}=\frac{(1+e^{c_P})^2}{e^{c_P}}\left( \sz c_{11}\suD+\frac{\sqrt{\mu r\sigma_{\max}}}{n}( c_{41}\suD+c_{11})\right).
\end{align*}
\end{lemma}
\begin{proof}[Proof of Lemma \ref{lem:D}]

We first control $\|\cP\hat{D}\cP-\cP D^*\cP\|$ in the following.
Note that
\begin{align*}
&\hat{D}- D^*\notag\\
&=
\sum_{i\neq j}\left(\frac{e^{\hat{P}_{ij}}}{(1+e^{\hat{P}_{ij}})^2}-\frac{e^{P^{*}_{ij}}}{(1+e^{P^{*}_{ij}})^2}\right)\begin{bmatrix}
z_iz_j^T\\
\frac1ne_ie_j^TY^*\\
\frac1ne_je_i^TX^*
\end{bmatrix}^{\otimes 2}\notag\\
&\quad+\sum_{i\neq j}\frac{e^{\hat{P}_{ij}}}{(1+e^{\hat{P}_{ij}})^2}\left(
\begin{bmatrix}
z_iz_j^T\\
\frac1ne_ie_j^T\hat{Y}\\
\frac1ne_je_i^T\hat{X}
\end{bmatrix}^{\otimes 2}
-\begin{bmatrix}
z_iz_j^T\\
\frac1ne_ie_j^TY^*\\
\frac1ne_je_i^TX^*
\end{bmatrix}^{\otimes 2}
\right)
\end{align*}
Thus, it holds that
\begin{align}\label{eq1:delta_diff}
&\|\cP\hat{D}\cP-\cP D^*\cP\|\notag\\
&\leq 
\left\|\sum_{i\neq j}\left(\frac{e^{\hat{P}_{ij}}}{(1+e^{\hat{P}_{ij}})^2}-\frac{e^{P^{*}_{ij}}}{(1+e^{P^{*}_{ij}})^2}\right)
\cP
\begin{bmatrix}
z_iz_j^T\\
\frac1ne_ie_j^TY^*\\
\frac1ne_je_i^TX^*
\end{bmatrix}^{\otimes 2}\cP\right\|\notag\\
&\quad+\left\|\sum_{i\neq j}\frac{e^{\hat{P}_{ij}}}{(1+e^{\hat{P}_{ij}})^2}\left(
\begin{bmatrix}
z_iz_j^T\\
\frac1ne_ie_j^T\hat{Y}\\
\frac1ne_je_i^T\hat{X}
\end{bmatrix}^{\otimes 2}
-\begin{bmatrix}
z_iz_j^T\\
\frac1ne_ie_j^TY^*\\
\frac1ne_je_i^TX^*
\end{bmatrix}^{\otimes 2}
\right)\right\|
\end{align}
For the first term, we have
\begin{align*}
&\left\|\sum_{i\neq j}\left(\frac{e^{\hat{P}_{ij}}}{(1+e^{\hat{P}_{ij}})^2}-\frac{e^{P^{*}_{ij}}}{(1+e^{P^{*}_{ij}})^2}\right)\cP\begin{bmatrix}
z_iz_j^T\\
\frac1ne_ie_j^TY^*\\
\frac1ne_je_i^TX^*
\end{bmatrix}^{\otimes 2}\cP\right\|\\
&\leq \left\|\sum_{i\neq j}\left|\frac{e^{\hat{P}_{ij}}}{(1+e^{\hat{P}_{ij}})^2}-\frac{e^{P^{*}_{ij}}}{(1+e^{P^{*}_{ij}})^2}\right|\cP\begin{bmatrix}
z_iz_j^T\\
\frac1ne_ie_j^TY^*\\
\frac1ne_je_i^TX^*
\end{bmatrix}^{\otimes 2}\cP\right\|\\
&\leq\frac14 \left\|\sum_{i\neq j}\left|\hat{P}_{ij}-P^*_{ij}\right|\cP\begin{bmatrix}
z_iz_j^T\\
\frac1ne_ie_j^TY^*\\
\frac1ne_je_i^TX^*
\end{bmatrix}^{\otimes 2}\cP\right\|\tag{by mean-value theorem}\\
&\leq \frac14 \max_{i\neq j}|\hat{P}_{ij}-P^*_{ij}|
 \left\|\cP\sum_{i\neq j}\begin{bmatrix}
z_iz_j^T\\
\frac1ne_ie_j^TY^*\\
\frac1ne_je_i^TX^*
\end{bmatrix}^{\otimes 2}\cP\right\|\\
&\leq \frac{(1+e^{c_P})^2}{4e^{c_P}} \max_{i\neq j}|\hat{P}_{ij}-P^*_{ij}|
 \left\|\cP D^{*}\cP\right\|.
\end{align*}
Note that
\begin{align*}
\max_{i\neq j}|\hat{P}_{ij}-P^*_{ij}|
&\lesssim\frac{\sz}{n}\|\hat{H}-H^*\|_{F}+\frac1n\|F^{*}\|_{2,\infty}\|\hat{F}-F^*\|_{2,\infty},\\
&\lesssim \frac{\sz c_{11}}{\sqrt{n}}+\frac{\sqrt{\mu r\sigma_{\max}}}{n^{3/2}}c_{41},
\end{align*}
where the last inequality follows from \eqref{ncv:est} and Assumption \ref{assumption:incoherent}. 
Since we have $\|\cP D^*\cP\|\leq \suD$, it then holds that
\begin{align}\label{eq2:delta_diff}
 &\left\|\sum_{i\neq j}\left(\frac{e^{\hat{P}_{ij}}}{(1+e^{\hat{P}_{ij}})^2}-\frac{e^{P^{*}_{ij}}}{(1+e^{P^{*}_{ij}})^2}\right)\cP\begin{bmatrix}
\frac{1}{\sqrt{n}}(e_i+e_j)\\
z_iz_j^T\\
\frac1ne_ie_j^TY^*\\
\frac1ne_je_i^TX^*
\end{bmatrix}^{\otimes 2}\cP\right\|\notag\\
&\lesssim \frac{(1+e^{c_P})^2}{e^{c_P}}\cdot \left(\frac{\sz c_{11}}{\sqrt{n}}+\frac{\sqrt{\mu r\sigma_{\max}}}{n^{3/2}}c_{41}\right)\cdot \suD.   
\end{align}

For the second term, note that for any $\Delta$,
\begin{align*}
&\left|\vec(\Delta)^T\left(\sum_{i\neq j}\frac{e^{\hat{P}_{ij}}}{(1+e^{\hat{P}_{ij}})^2}\left(
\begin{bmatrix}
z_iz_j^T\\
\frac1ne_ie_j^T\hat{Y}\\
\frac1ne_je_i^T\hat{X}
\end{bmatrix}^{\otimes 2}
-\begin{bmatrix}
z_iz_j^T\\
\frac1ne_ie_j^TY^*\\
\frac1ne_je_i^TX^*
\end{bmatrix}^{\otimes 2}
\right)\right)\vec(\Delta)\right|\\
&=\Bigg|
\sum_{i\neq j}\frac{e^{\hat{P}_{ij}}}{(1+e^{\hat{P}_{ij}})^2}\cdot\\
&\qquad\Bigg(
\frac{2}{n}\langle\Delta_{H}, z_iz_j^T\rangle \langle e_i^T\Delta_X, e_j^T(\hat{Y}-Y^*)\rangle+\frac{2}{n}\langle\Delta_{H}, z_iz_j^T\rangle \langle e_j^T\Delta_Y, e_i^T(\hat{X}-X^*)\rangle\\
&\qquad\quad +\frac{2}{n^2}\langle e^T_j\Delta_Y, e^T_i\hat{X}\rangle \langle e_i^T\Delta_X, e_j^T(\hat{Y}-Y^*)\rangle+\frac{2}{n^2}\langle e^T_i\Delta_X, e^T_j Y^{*}\rangle \langle e_j^T\Delta_Y, e_i^T(\hat{X}-X^*)\rangle\\
&\qquad\quad +\frac{1}{n^2}\langle e^T_i\Delta_X, e^T_j\hat{Y}\rangle \langle e_i^T\Delta_X, e_j^T(\hat{Y}-Y^*)\rangle+\frac{1}{n^2}\langle e^T_i\Delta_X, e^T_j Y^{*}\rangle \langle e_i^T\Delta_X, e_j^T(\hat{Y}-Y^*)\rangle\\
&\qquad\quad +\frac{1}{n^2}\langle e^T_j\Delta_Y, e^T_i\hat{X}\rangle \langle e_j^T\Delta_Y, e_i^T(\hat{X}-X^*)\rangle+\frac{1}{n^2}\langle e^T_j\Delta_Y, e^T_i X^{*}\rangle \langle e_j^T\Delta_Y, e_i^T(\hat{X}-X^*)\rangle
\Bigg)
\Bigg|.
\end{align*}
Note that
\begin{align*}
&\left|\frac{2}{n}
\sum_{i\neq j}\frac{e^{\hat{P}_{ij}}}{(1+e^{\hat{P}_{ij}})^2}
\langle\Delta_{H}, z_iz_j^T\rangle \langle e_i^T\Delta_X, e_j^T(\hat{Y}-Y^*)\rangle
\right|\\
&\lesssim \frac{1}{n}
\sqrt{\sum_{i\neq j}|\langle\Delta_{H}, z_iz_j^T\rangle|^2}\cdot
\sqrt{\sum_{i\neq j}| \langle e_i^T\Delta_X, e_j^T(\hat{Y}-Y^*)\rangle |^2}\\
&\lesssim\frac{1}{n}\cdot \sz\|\Delta_{H}\|_F\cdot\|\Delta_{X}\|_F\|\hat{Y}-Y^*\|_F\\
&\lesssim_{(i)} \frac{\sz c_{11}}{\sqrt{n}}\|\Delta_{H}\|_F\cdot\|\Delta_{X}\|_F\\
&\lesssim \frac{\sz c_{11}}{\sqrt{n}}(\|\Delta_{H}\|_F^2+\|\Delta_{X}\|_F^2),
\end{align*}
where (i) follows from Lemma \ref{lem:ncv1}.
Similar arguments hold for the other terms, for which we omit the proofs. We then have
\begin{align*}
&\left|\vec(\Delta)^T\left(\sum_{i\neq j}\frac{e^{\hat{P}_{ij}}}{(1+e^{\hat{P}_{ij}})^2}\left(
\begin{bmatrix}
z_iz_j^T\\
\frac1ne_ie_j^T\hat{Y}\\
\frac1ne_je_i^T\hat{X}
\end{bmatrix}^{\otimes 2}
-\begin{bmatrix}
z_iz_j^T\\
\frac1ne_ie_j^TY^*\\
\frac1ne_je_i^TX^*
\end{bmatrix}^{\otimes 2}
\right)\right)\vec(\Delta)\right|\\  
&\lesssim\frac{\sz c_{11}+\sqrt{\mu r \sigma_{\max}/n^2}c_{11}}{\sqrt{n}}(\|\Delta_{H}\|_F^2+\|\Delta_{X}\|_F^2+\|\Delta_{Y}\|_F^2)\\
&=\frac{\sz c_{11}+\sqrt{\mu r \sigma_{\max}/n^2}c_{11}}{\sqrt{n}}\|\vec(\Delta)\|_2^2.
\end{align*}
Consequently, we have
\begin{align}\label{eq3:delta_diff}
&\left\|\sum_{i\neq j}\frac{e^{\hat{P}_{ij}}}{(1+e^{\hat{P}_{ij}})^2}\left(
\begin{bmatrix}
\frac{1}{\sqrt{n}}(e_i+e_j)\\
z_iz_j^T\\
\frac1ne_ie_j^T\hat{Y}\\
\frac1ne_je_i^T\hat{X}
\end{bmatrix}^{\otimes 2}
-\begin{bmatrix}
\frac{1}{\sqrt{n}}(e_i+e_j)\\
z_iz_j^T\\
\frac1ne_ie_j^TY^*\\
\frac1ne_je_i^TX^*
\end{bmatrix}^{\otimes 2}
\right)\right\| \notag\\
&\lesssim \frac{\sz c_{11}+\sqrt{\mu r \sigma_{\max}/n^2}c_{11}}{\sqrt{n}}.
\end{align}
Combine \eqref{eq1:delta_diff}, \eqref{eq2:delta_diff} and \eqref{eq3:delta_diff}, we have
\begin{align*}
&\|\cP\hat{D}\cP-\cP D^*\cP\|\\
&\lesssim  \frac{(1+e^{c_P})^2}{e^{c_P}}\cdot \left(\frac{\sz c_{11}}{\sqrt{n}}+\frac{\sqrt{\mu r\sigma_{\max}}}{n^{3/2}}c_{41}\right)\cdot \suD+\frac{\sz c_{11}+\sqrt{\mu r \sigma_{\max}/n^2}c_{11}}{\sqrt{n}}\\
&\lesssim \frac{1}{\sqrt{n}} \frac{(1+e^{c_P})^2}{e^{c_P}}\left( \sz c_{11}\suD+\frac{\sqrt{\mu r\sigma_{\max}}}{n}( c_{41}\suD+c_{11})\right)\\
&=  \frac{\hat{c}}{\sqrt{n}} .
\end{align*}
By Weyl's inequality, we have
\begin{align*}
|\lambda_i(\cP\hat{D}\cP)-\lambda_i(\cP D^*\cP)|\lesssim \|\cP \hat{D}\cP-\cP D^*\cP\|\lesssim \frac{\hat{c}}{\sqrt{n}}.
\end{align*}
Since $\slD\leq \lambda_{\min}(\cP D^* \cP)\leq \lambda_{\max}(\cP D^* \cP)\leq \suD$, we then have $\frac{\slD}{2}\leq \lambda_{\min}(\cP \hat{D} \cP)\leq \lambda_{\max}(\cP  \hat{D}  \cP)\leq {2\suD}$ as long as $n$ is large enough.

\end{proof}

\subsection{Proofs of Proposition \ref{prop:debias_dist}}\label{pf:prop_debias_dist}
By \eqref{def:debias}, we have
\begin{align*}
\left\|\begin{bmatrix}
\hat{H}^d-\hat{H}\\
\hat{X}^d-\hat{X}\\
\hat{Y}^d-\hat{Y}
\end{bmatrix}\right\|_F
&=\|(\cP \hat{D}\cP)^{\dagger}\cP\nabla L(\hat{H},\hat{X},\hat{Y})\|_F\\
&\leq 
\frac{1}{\lambda_{\min}(\cP \hat{D}\cP )}\|\cP\nabla L(\hat{H},\hat{X},\hat{Y})\|_F.
\end{align*}
Note that by \eqref{ncv:grad}, we have
\begin{align*}
\left\|\cP\nabla L(\hat{H},\hat{X},\hat{Y})
+
\begin{bmatrix}
0\\
\lambda\hat{X}\\
\lambda\hat{Y}
\end{bmatrix}\right\|_F
\lesssim n^{-5},
\end{align*}
which then gives
\begin{align*}
\|\cP \nabla L(\hat{H},\hat{X},\hat{Y})\|_F
&\leq c n^{-5}+\lambda (\|\hat{X}\|_F+\|\hat{Y}\|_F)  \\
&\leq  c n^{-5}+\lambda (\|\hat{X}-X^*\|_F+\|\hat{Y}-Y^*\|_F)+ \lambda (\|X^*\|_F+\|Y^*\|_F)\\
&\lesssim \lambda \|X^{*}\|_F\\
&\lesssim \lambda\sqrt{\mu r\sigma_{\max}}
\end{align*}
as long as $\|\hat{X}-X^*\|_F\ll \|X^*\|_F$ and $n^{-5}\ll \lambda \|X^*\|_F$.
By Lemma \ref{lem:D}, we have $\lambda_{\min}(\cP \hat{D}\cP )\geq \slD/2$.
As a result, we have
\begin{align*}
 \left\|\begin{bmatrix}
\hat{H}^d-\hat{H}\\
\hat{X}^d-\hat{X}\\
\hat{Y}^d-\hat{Y}
\end{bmatrix}\right\|_F
\lesssim \frac{{\lambda}\sqrt{\mu r\sigma_{\max}}}{\slD}.
\end{align*}

\subsection{Proofs of Proposition \ref{prop:bar_dist}}\label{pf:prop_bar_dist}
By \eqref{def:bar}, we have
\begin{align*}
\left\|\begin{bmatrix}
\bar{H}-{H^{*}}\\
\bar{X}-{X^{*}}\\
\bar{Y}-{Y^{*}}
\end{bmatrix}\right\|_F
&=\|(\cP D^* \cP)^{\dagger} \cP\nabla L({H^{*}},{X^{*}},{Y^{*}})\|_F\\
&\leq \frac{1}{\lambda_{\min}(\cP {D^{*}}\cP )}\|\cP\nabla L({H^{*}},{X^{*}},{Y^{*}})\|_F\\
&\leq \frac{1}{\lambda_{\min}(\cP {D^{*}}\cP )}\|\nabla L({H^{*}},{X^{*}},{Y^{*}})\|_F.
\end{align*}
By Lemma \ref{lem:gradient_norm}, we have
\begin{align*}
\|\nabla L({H^{*}},{X^{*}},{Y^{*}})\|_F
&\lesssim \sz\sqrt{p\log n}+\sqrt{\frac{\log n}{n}} (\|{X^{*}}\|_F+\|{Y^{*}}\|_F)  \\
&\lesssim \sqrt{\frac{\mu r \sigma_{\max}\log n}{n}}.
\end{align*}
Note that $\lambda_{\min}(\cP {D^{*}}\cP )\geq \slD$.
As a result, we have
\begin{align*}
 \left\|\begin{bmatrix}
\bar{H}-{H^{*}}\\
\bar{X}-{X^{*}}\\
\bar{Y}-{Y^{*}}
\end{bmatrix}\right\|_F
\lesssim \frac{1}{\slD}\sqrt{\frac{\mu r \sigma_{\max}\log n}{n}}.
\end{align*}

In order to show the second part of the result, we introduce the following lemma first.
\begin{lemma}\label{concentrationlemma}
    Consider some fixed constants $a_{ij}$ for $ i\neq j\in [n]$, and random variable 
\begin{align*}
		X = \sum_{i\neq j} a_{ij}\left(\frac{e^{P_{ij}^*}}{1+e^{P_{ij}^*}}-A_{ij}\right).
\end{align*}
Then with probability at least $1-O(n^{-11})$ we have 
\begin{align*}
    |X|\lesssim \sqrt{\frac{(1+e^{c_P})^2}{e^{c_P}}\text{Var}[X]\cdot\log n}
\end{align*}
\end{lemma}
\begin{proof}
    Denote by $X_{ij} = a_{ij}(e^{P_{ij}^*}/(1+e^{P_{ij}^*})-A_{ij})$. Then we know that $\mathbb{E} X_{ij} = 0$ and $|X_{ij}|\leq a_{ij}$. Therefore, by Hoeffding inequality, with probability at least $1-O(n^{-11})$, we have
	\begin{align}
	\left|\sum_{i\neq j}X_{ij}\right|&\lesssim \left(\log n\cdot\sum_{i\neq j}a_{ij}^2\right)^{\frac{1}{2}}.\label{concentrationlemmaeq1}
	\end{align}
 On the other hand, since $A_{ij}$ are independent random variables, we know that
	\begin{align*}
		\text{Var}[X] &= \sum_{i\neq j}\text{Var}[X_{ij}]=\sum_{i\neq j}a_{ij}^2 \frac{e^{P_{ij}^*}}{(1+e^{P_{ij}^*})^2} \geq  \frac{e^{c_P}}{(1+e^{c_P})^2}\sum_{i\neq j}a_{ij}^2.
	\end{align*}
As a result, we have $\sum_{i\neq j}a_{ij}^2\lesssim e^{-c_P}(1+e^{c_P})^2 \text{Var}[X]$. Combine this with \eqref{concentrationlemmaeq1} we get
	\begin{align*}
		|X| = \left|\sum_{i\neq j}X_{ij}\right|&\lesssim \sqrt{\frac{(1+e^{c_P})^2}{e^{c_P}}\text{Var}[X]\cdot\log n}
	\end{align*}
	with probability exceeding $1-O(n^{-11})$.
\end{proof}

Let's come back to control
\begin{align*}
 \left\|\begin{bmatrix}
 \bar{H}-{H^{*}}\\
\bar{X}-{X^{*}}\\
\bar{Y}-{Y^{*}}
\end{bmatrix}\right\|_{\infty}.
\end{align*}
According to the definition of $\bar{H}, \bar{X},\bar{Y}$ from \eqref{def:bar}, we know that each entry of $ \bar{H}-{H^{*}}, \bar{X}-{X^{*}}, \bar{Y}-{Y^{*}}$ can be written as linear combinations of $e^{P_{ij}^*}/(1+e^{P_{ij}^*})-A_{ij}$, since $\nabla L(H^*, X^*, Y^*)$ is a linear combination of $e^{P_{ij}^*}/(1+e^{P_{ij}^*})-A_{ij}$. Then by Lemma \ref{concentrationlemma}, we know that given any index $i$ we have 
\begin{align*}
\left|\begin{bmatrix}
 \bar{H}-{H^{*}}\\
\bar{X}-{X^{*}}\\
\bar{Y}-{Y^{*}}
\end{bmatrix}_i\right|\lesssim \sqrt{\frac{(1+e^{c_P})^2}{e^{c_P}}\text{Var}\left[\begin{bmatrix}
 \bar{H}-{H^{*}}\\
\bar{X}-{X^{*}}\\
\bar{Y}-{Y^{*}}
\end{bmatrix}_i\right]\cdot\log n}
\end{align*}
with probability at least $1-O(n^{-11})$. Taking a union bound for all indices $i$ we know that 
\begin{align}
 \left\|\begin{bmatrix}
  \bar{H}-{H^{*}}\\
\bar{X}-{X^{*}}\\
\bar{Y}-{Y^{*}}
\end{bmatrix}\right\|_{\infty}\lesssim \sqrt{\frac{(1+e^{c_P})^2}{e^{c_P}}\max_{i}\text{Var}\left[\begin{bmatrix}
 \bar{H}-{H^{*}}\\
\bar{X}-{X^{*}}\\
\bar{Y}-{Y^{*}}
\end{bmatrix}_{i}\right]\cdot\log n}\label{pf:prop:bar_disteq1}
\end{align}
with probability at least $1-O(n^{-10})$. On the other hand, from \eqref{def:bar} we know that 
\begin{align*}
\text{Var}\begin{bmatrix}
\bar{H}-H^*\\
\bar{X}-X^*\\
\bar{Y}-Y^*
\end{bmatrix}=(\cP D^*\cP)^{\dagger}.
\end{align*}
Therefore, one can see that
\begin{align*}
\max_{i}\text{Var}\left[\begin{bmatrix}
\bar{H}-H^* \\
\bar{X}-{X^{*}}\\
\bar{Y}-{Y^{*}}
\end{bmatrix}_{i}\right]\leq \left\|\text{Var}\begin{bmatrix}
\bar{H}-H^* \\
\bar{X}-{X^{*}}\\
\bar{Y}-{Y^{*}}
\end{bmatrix}\right\|=\left\|(\cP D^*\cP)^{\dagger}\right\|\leq \slD^{-1}.
\end{align*}
Plugging this in \eqref{pf:prop:bar_disteq1} we get 
\begin{align*}
 \left\|\begin{bmatrix}
  \bar{H}-{H^{*}}\\
\bar{X}-{X^{*}}\\
\bar{Y}-{Y^{*}}
\end{bmatrix}\right\|_{\infty}\lesssim \sqrt{\frac{(1+e^{c_P})^2}{\slD e^{c_P}}\cdot\log n}
\end{align*}
with probability at least $1-O(n^{-10})$.

\subsection{Proofs of Theorem \ref{thm:ncv_debias}}\label{pf:ncv_debias}

We first prove the following lemma.
\begin{lemma}\label{lem:debias_bar}
Under Assumption \ref{assumption:2_infty}, we have
\begin{align*}
\|\hat{F}^d-\bar{F}\|_{2, \infty}\leq c_{\infty}\sqrt{r},
\end{align*}
where 
\begin{align*}
c_{\infty}\asymp \frac{\lambda\hat{c}}{\slD^2}\sqrt{\frac{\mu r\sigma_{\max}}{n}}+ \frac{c_{11}\hat{c}}{\slD}+ c_{2,\infty}c_{11}\sqrt{r^2+p}
\end{align*}
and $\hat{c}$ is defined in Lemma \ref{lem:D}.
\end{lemma}
\begin{proof}[Proof of Lemma \ref{lem:debias_bar}]
By \eqref{def:debias} and \eqref{def:bar}, we have
\begin{align}\label{ineq1:pf_debias_bar}
\begin{bmatrix}
\hat{H}^d-\bar{H}\\
\hat{X}^d-\bar{X}\\
\hat{Y}^d-\bar{Y}
\end{bmatrix}
&=\begin{bmatrix}
\hat{H}-{H}^*\\
\hat{X}-{X}^*\\
\hat{Y}-{Y}^*
\end{bmatrix}
+\left((\cP D^*\cP)^{\dagger}-(\cP \hat{D}\cP)^{\dagger}\right)\cP\nabla L(\hat{H},\hat{X},\hat{Y})\notag\\
&\quad+ (\cP D^*\cP)^{\dagger}\cP\left(\nabla L({H}^*,{X}^*,{Y}^*)-\nabla L(\hat{H},\hat{X},\hat{Y})\right).
\end{align}
For notation simplicity, we denote
\begin{align*}
V^*:=
\begin{bmatrix}
{H}^*\\
{X}^*\\
{Y}^*
\end{bmatrix} \text{ and }
\hat{V}:=
\begin{bmatrix}
\hat{H}\\
\hat{X}\\
\hat{Y}
\end{bmatrix}.
\end{align*}
We can further decompose the third term on the RHS of \eqref{ineq1:pf_debias_bar} as 
\begin{align*}
 &(\cP D^*\cP)^{\dagger}\cP\left(\nabla L({H}^*,{X}^*,{Y}^*)-\nabla L(\hat{H},\hat{X},\hat{Y})\right)\\
=&(\cP D^*\cP)^{\dagger}\cP\left(\int^1_0 \left(\nabla^2 L(\hat{V}+t(V^*-\hat{V}))-\nabla^2 L(V^*)\right)dt(V^*-\hat{V})\right)\\
&+(\cP D^*\cP)^{\dagger}\cP\left(\nabla^2 L(V^*)-D^*\right)(V^*-\hat{V})
+\left((\cP D^*\cP)^{\dagger}(\cP D^*\cP)-I\right)(V^*-\hat{V})+(V^*-\hat{V}).
\end{align*}
Consequently, we have
\begin{align}
&\left\|
\begin{bmatrix}
\hat{X}^d-\bar{X}\\
\hat{Y}^d-\bar{Y}
\end{bmatrix}
\right\|_{\infty}\notag\\
&\leq 
\underbrace{\left\|\left((\cP D^*\cP)^{\dagger}-(\cP \hat{D}\cP)^{\dagger}\right)\cP\nabla L(\hat{V})\right\|_2}_{(1)}\notag\\
&\quad +
\underbrace{\left\|(\cP D^*\cP)^{\dagger}\cP\left(\int^1_0 \left(\nabla^2 L(\hat{V}+t(V^*-\hat{V}))-\nabla^2 L(V^*)\right)dt(V^*-\hat{V})\right)\right\|_2}_{(2)}\notag\\
&\quad +
\underbrace{\left\|(\cP D^*\cP)^{\dagger}\cP\left(\nabla^2 L(V^*)-D^*\right)(V^*-\hat{V})\right\|_2}_{(3)}
+
\underbrace{\left\|\left((\cP D^*\cP)^{\dagger}(\cP D^*\cP)-I\right)(V^*-\hat{V})\right\|_{\infty}}_{(4)}.
\end{align}

In the following, we bound (1)-(4), respectively.
\begin{enumerate}
\item For (1), by Theorem 3.3 in \cite{stewart1977perturbation}, we have
\begin{align*}
\|(\cP D^*\cP)^{\dagger}-(\cP \hat{D}\cP)^{\dagger}\|
&\leq \frac{1+\sqrt{5}}{2}\max\{\|(\cP\hat{D}\cP)^{\dagger}\|^2, \|(\cP D^*\cP)^{\dagger}\|^2\}\cdot \|\cP\hat{D}\cP-\cP D^*\cP\|.
\end{align*}
By Lemma \ref{lem:D}, we have
\begin{align*}
\max\{\|(\cP\hat{D}\cP)^{\dagger}\|^2, \|(\cP D^*\cP)^{\dagger}\|^2\}\lesssim \frac{1}{\slD^2}, \quad
\|\cP\hat{D}\cP-\cP D^*\cP\|\lesssim  \frac{\hat{c}}{\sqrt{n}}.
\end{align*}
Thus, we obtain
\begin{align*}
\|(\cP D^*\cP)^{\dagger}-(\cP \hat{D}\cP)^{\dagger}\|\lesssim \frac{\hat{c}}{\slD^2\sqrt{n}}.
\end{align*}
Further, as shown in the proof of Proposition \ref{prop:debias_dist}, we have
\begin{align*}
\|\cP\nabla L(\hat{V})\|_F\lesssim\lambda\sqrt{\mu r\sigma_{\max}}.
\end{align*}
Consequently, we have
\begin{align*}
(1)\leq \|(\cP D^*\cP)^{\dagger}-(\cP \hat{D}\cP)^{\dagger}\|\|\cP\nabla L(\hat{V})\|_F\lesssim \frac{\lambda\hat{c}}{\slD^2}\sqrt{\frac{\mu r\sigma_{\max}}{n}}.
\end{align*}

\item  We then bound (2). Denote $V^t=\hat{V}+t(V^*-\hat{V})$ and define $D^t$ correspondingly. Following the same argument as in the proof of Lemma \ref{lem:D}, we have
\begin{align*}
\|\cP D^{t}\cP-\cP D^* \cP\|\lesssim  \frac{\hat{c}}{\sqrt{n}}.
\end{align*}
Further, as already being shown in the proof of Lemma \ref{lem:convexity},
we have
\begin{align*}
\|\nabla^2 L(V^t)-D^t\|\lesssim   \frac{1}{n}\cdot\left(\sqrt{\su}\|\hat{H}-H^{*}\|_{F}+\frac{1}{n}\|X^*\|\|\hat{F}-F^*\|_{F}\right)\lesssim \frac{c_{11}}{\sqrt{n}}
\end{align*}
and
\begin{align*}
 \|\nabla^2 L(V^*)-D^*\|\lesssim \frac{\sqrt{\log n}}{n}.
\end{align*}
Consequently, we have for all $t\in [0,1]$
\begin{align*}
&\left\|\cP\left(\nabla^2 L(\hat{V}+t(V^*-\hat{V}))-\nabla^2 L(V^*)\right)\cP\right\|\\
&\leq \|\nabla^2 L(V^t)-D^t\|+\|\cP D^{t}\cP-\cP D^* \cP\|+ \|\nabla^2 L(V^*)-D^*\|\lesssim \frac{\hat{c}}{\sqrt{n}}.
\end{align*}
Thus, we have
\begin{align*}
(2)\lesssim \frac{\hat{c}}{\sqrt{n}}\|(\cP D^*\cP)^{\dagger}\|\|V^*-\hat{V}\|_F\leq  \frac{c_{11}\hat{c}}{\slD} .
\end{align*}

\item For (3), we have
\begin{align*}
(3)\leq \|(\cP D^*\cP)^{\dagger}\|\|\nabla^2 L(V^*)-D^*\|\|V^*-\hat{V}\|_F \lesssim \frac{1}{\slD}\cdot \frac{\sqrt{\log n}}{n} \cdot c_{11}\sqrt{n}=\frac{c_{11}}{\slD}\sqrt{\frac{\log n}{n}}.
\end{align*}

\item Finally, for (4), we have
\begin{align*}
(4)
\leq \left\|I-(\cP D^*\cP)^{\dagger}(\cP D^*\cP)\right\|_{2,\infty}\|V^*-\hat{V}\|_F\leq c_{2,\infty}c_{11}\sqrt{r^2+p},
\end{align*}
where the last inequality follows from Assumption \ref{assumption:2_infty}.

\end{enumerate}

Combine the bounds for (1)-(4), we have
\begin{align*}
&\left\|
\begin{bmatrix}
\hat{X}^d-\bar{X}\\
\hat{Y}^d-\bar{Y}
\end{bmatrix}
\right\|_{\infty}
\lesssim
\frac{\lambda\hat{c}}{\slD^2}\sqrt{\frac{\mu r\sigma_{\max}}{n}}+ \frac{c_{11}\hat{c}}{\slD}+ c_{2,\infty}c_{11}\sqrt{r^2+p}=c_{\infty}.
\end{align*}
Consequently, it holds that
\begin{align*}
\|\hat{F}^d-\bar{F}\|_{2,\infty}\leq \sqrt{r}\|\hat{F}^d-\bar{F}\|_{\infty}\lesssim c_{\infty} \sqrt{r}.
\end{align*}
We then finish the proofs.
\end{proof}

With Lemma \ref{lem:debias_bar} in hand, we then prove Theorem \ref{thm:ncv_debias} in the following.
\begin{proof}[Proof of Theorem \ref{thm:ncv_debias}]

For notation simplicity, given $V=\begin{bmatrix}
H\\
X\\
Y
\end{bmatrix}$, we let $c(V):=\begin{bmatrix}
H\\
XY^\top
\end{bmatrix}$, which is the convex counterpart of $V$. 
We denote
\begin{align*}
&\Delta
=\begin{bmatrix}
\Delta_{H}\\
\Delta_{\Gamma}
\end{bmatrix}
:=c(\bar{V})-c(\hat{V}^d),\ 
\begin{bmatrix}
\Delta_X\\
\Delta_Y
\end{bmatrix}
:=\begin{bmatrix}
\bar{X}-\hat{X}^d\\
\bar{Y}-\hat{Y}^d
\end{bmatrix},\\
&\Delta'
=\begin{bmatrix}
\Delta'_{H}\\
\Delta'_{\Gamma}
\end{bmatrix}
:=c(\hat{V}^d)-c(V^*),\ 
\begin{bmatrix}
\Delta'_X\\
\Delta'_Y
\end{bmatrix}
=\begin{bmatrix}
\hat{X}^d-X^*\\
\hat{Y}^d-Y^*
\end{bmatrix},\\
&\Delta^{''}
=\begin{bmatrix}
\Delta''_{H}\\
\Delta''_{\Gamma}
\end{bmatrix}
:=c(\bar{V})-c(\hat{V}),\ 
\begin{bmatrix}
\Delta''_X\\
\Delta''_Y
\end{bmatrix}
=\begin{bmatrix}
\bar{X}-\hat{X}\\
\bar{Y}-\hat{Y}
\end{bmatrix},\\
&\Delta^{'''}
=\begin{bmatrix}
\Delta^{'''}_{H}\\
\Delta^{'''}_{\Gamma}
\end{bmatrix}
:=c(\bar{V})-c(V^*),\ 
\begin{bmatrix}
\Delta'''_X\\
\Delta'''_Y
\end{bmatrix}
=\begin{bmatrix}
\bar{X}-X^*\\
\bar{Y}-Y^*
\end{bmatrix},\\
&\hat{\Delta}
=\begin{bmatrix}
\hat{\Delta}_{H}\\
\hat{\Delta}_{\Gamma}
\end{bmatrix}
:=c(\hat{V}^d)-c(\hat{V}),\ 
\begin{bmatrix}
\hat{\Delta}_X\\
\hat{\Delta}_Y
\end{bmatrix}
=\begin{bmatrix}
\hat{X}^d-\hat{X}\\
\hat{Y}^d-\hat{Y}
\end{bmatrix},
\end{align*}
which will be used in the following proofs.
By Proposition \ref{prop:debias_dist}, Proposition \ref{prop:bar_dist} and Lemma \ref{lem:ncv1}, all the Frobenius norms related to \(H, X, Y\) (e.g., \(\|\Delta_H\|_F\), \(\|\Delta_X\|_F\), \(\|\Delta_Y\|_F\)) are bounded by $(c_a+c_{11})\sqrt{n}$. Additionally, all the Frobenius norms related to \(\Gamma\) (e.g., \(\|\Delta_{\Gamma}\|_F\)) are bounded by $(c_a+c_{11}) \sqrt{\mu r\sigma_{\max}n}$.

We define the quadratic approximation of the convex loss (defined in \eqref{obj:cv_loss}) as
\begin{align*}
\bar{L}_c\left(c(V)\right)
:={L}_c\left(c(V^*)\right)+\nabla {L}_c\left(c(V^*)\right)\left(c(V)-c(V^*)\right)+\frac12 \left(c(V)-c(V^*)\right)^\top\nabla^2 {L}_c\left(c(V^*)\right)\left(c(V)-c(V^*)\right),
\end{align*}
which implies for all $V$
\begin{align}\label{eq2:pf_debias}
\nabla\bar{L}_c\left(c(V)\right)=\nabla {L}_c\left(c(V^*)\right)+\nabla^2 {L}_c\left(c(V^*)\right)\left(c(V)-c(V^*)\right).
\end{align}
Note that
\begin{align*}
\nabla L_c(c(\hat{V}^d))
&=\nabla L_c(c({V}^*))+\int^1_{0}\nabla^2L_c\left(c({V}^*)+t(c(\hat{V}^d)-c({V}^*))\right)(c(\hat{V}^d)-c({V}^*)) dt\\
&=\nabla\bar{L}_c\left(c(\bar{V})\right)-\nabla^2 {L}_c\left(c(V^*)\right)\left(c(\bar{V})-c(V^*)\right)\\
&\quad+\int^1_{0}\nabla^2L_c\left(c({V}^*)+t(c(\hat{V}^d)-c({V}^*))\right)(c(\hat{V}^d)-c({V}^*)) dt,
\end{align*}
where the second equation follows from \eqref{eq2:pf_debias} with $V=\bar{V}$. Rearranging the terms gives
\begin{align*}
&\nabla^2 {L}_c\left(c(V^*)\right)\left(c(\bar{V})-c(\hat{V}^d)\right)\\
=& \int^1_{0}\left(\nabla^2L_c\left(c({V}^*)+t(c(\hat{V}^d)-c({V}^*))\right)-\nabla^2 L_c(c(V^*))\right)dt(c(\hat{V}^d)-c({V}^*)) \\
&-\nabla L_c(c(\hat{V}^d))+   \nabla\bar{L}_c\left(c(\bar{V})\right).
\end{align*}
By multiplying both sides with $c(\bar{V})-c(\hat{V}^d)$, we have
\begin{align}\label{eq3:pf_debias}
 &\left(c(\bar{V})-c(\hat{V}^d)\right)^\top \nabla^2{L}_c\left(c(V^*)\right)\left(c(\bar{V})-c(\hat{V}^d)\right)\notag\\
 =& \underbrace{\left(c(\bar{V})-c(\hat{V}^d)\right)^\top\left(\int^1_{0}\nabla^2L_c\left(c({V}^*)+t(c(\hat{V}^d)-c({V}^*))\right)-\nabla^2 L_c(c(V^*))dt\right)(c(\hat{V}^d)-c({V}^*))}_{(1)}\notag\\
&-\underbrace{\nabla L_c(c(\hat{V}^d))^\top\left(c(\bar{V})-c(\hat{V}^d)\right)}_{(2)}+   \underbrace{\nabla\bar{L}_c\left(c(\bar{V})\right)^\top\left(c(\bar{V})-c(\hat{V}^d)\right)}_{(3)}.   
\end{align}
For the LHS of \eqref{eq3:pf_debias}, we have
\begin{align}
 &\left(c(\bar{V})-c(\hat{V}^d)\right)^T \nabla^2{L}_c\left(c(V^*)\right)\left(c(\bar{V})-c(\hat{V}^d)\right)\notag\\
=&\sum_{i,j}\frac{e^{P_{ij}^*}}{(1+e^{P_{ij}^*})^2}\left(\langle\Delta_{H},z_iz_j^T\rangle+\frac1n\langle \Delta_{\Gamma},e_i e_j^T\rangle\right)^2 - \sum_{i=1}^n\frac{e^{P_{ii}^*}}{(1+e^{P_{ii}^*})^2}\left(\langle\Delta_{H},z_iz_i^T\rangle+\frac1n\langle \Delta_{\Gamma},e_i e_i^T\rangle\right)^2\notag\\
\geq& \frac{e^{c_P}}{(1+e^{c_P})^2} \sum_{i,j}\left(\langle\Delta_{H},z_iz_j^T\rangle+\frac1n\langle \Delta_{\Gamma},e_i e_j^T\rangle\right)^2 -\frac{1}{4}\sum_{i=1}^n\left(\langle\Delta_{H},z_iz_i^T\rangle+\frac1n\langle \Delta_{\Gamma},e_i e_i^T\rangle\right)^2\notag\\
\geq& \frac{e^{c_P}}{(1+e^{c_P})^2} \left(\left\|Z\Delta_H Z^\top\right\|_F^2+\frac{1}{n^2}\left\|\Delta_\Gamma\right\|_F^2\right)-\frac{1}{2}\sum_{i=1}^n\left((z_i^\top\Delta_H z_i)^2+\frac{(\Delta_\Gamma)_{ii}^2}{n^2}\right)\notag\\
\geq & \frac{e^{c_P}}{(1+e^{c_P})^2} \left(\sl\left\|\Delta_H\right\|_F^2+\frac{1}{n^2}\left\|\Delta_\Gamma\right\|_F^2\right) - \frac{c_z^2}{2n}\left\|\Delta_H\right\|_F^2 - \frac{1}{2}\sum_{i=1}^n\frac{(\Delta_\Gamma)_{ii}^2}{n^2} \notag\\
\geq & \frac{\sl e^{c_P}}{2(1+e^{c_P})^2} \left(\left\|\Delta_H\right\|_F^2+\frac{1}{n^2}\left\|\Delta_\Gamma\right\|_F^2\right)  - \frac{1}{2}\sum_{i=1}^n\frac{(\Delta_\Gamma)_{ii}^2}{n^2}\label{LHS:pf_debias}
\end{align}
as long as $n\geq (1+e^{c_P})^2c_z^2/(\sl e^{c_P})$. To control $\sum_{i=1}^n(\Delta_\Gamma)_{ii}^2$, we write $\Delta_\Gamma = \Delta_\Gamma^{'''}- \Delta_\Gamma^{'}$. Then $\sum_{i=1}^n(\Delta_\Gamma)_{ii}^2\leq 2\sum_{i=1}^n(\Delta_\Gamma^{'''})_{ii}^2 + 2\sum_{i=1}^n(\Delta_\Gamma^{'})_{ii}^2$. One can see that
\begin{align*}
    \sum_{i=1}^n(\Delta_\Gamma^{'})_{ii}^2  &= \sum_{i=1}^n \left(\hat{X}^d\hat{Y}^{d\top} - X^*Y^{*\top}\right)_{ii}^2 = \sum_{i=1}^n \left((\hat{X}^d-X^*)\hat{Y}^{d\top} +X^* (\hat{Y}^{d}-Y^{*})^\top\right)_{ii}^2 \\
    &\leq \sum_{i=1}^n \left((\hat{X}^d-X^*)_{i, \cdot}(\hat{Y}^{d})_{i, \cdot}^\top +(X^*)_{i, \cdot} (\hat{Y}^{d}-Y^{*})_{i, \cdot}^\top\right)^2 \\
    &\leq 2\sum_{i=1}^n \left(\left\|\hat{X}^d-X^*)_{i, \cdot}\right\|_2^2\left\|(\hat{Y}^{d})_{i, \cdot}\right\|_2^2 +\left\|(X^*)_{i, \cdot}\right\|_2^2 \left\|(\hat{Y}^{d}-Y^{*})_{i, \cdot}\right\|_2^2\right) \\
    &\leq 2 \left(\left\|\hat{X}^d-X^*\right\|_F^2\left\|\hat{Y}^{d}\right\|_{2,\infty}^2 +\left\|X^*\right\|_{2,\infty}^2 \left\|\hat{Y}^{d}-Y^{*}\right\|_F^2\right)
\end{align*}
By Proposition \ref{prop:debias_dist}, Lemma \ref{lem:ncv1} and Assumption \ref{assumption:incoherent} we know that 
\begin{align*}
    &\left\|\hat{X}^d-X^*\right\|_F, \left\|\hat{Y}^d-Y^*\right\|_F\lesssim c_a\sqrt{n}+c_{11}\sqrt{n},\\
    &\left\|X^*\right\|_{2,\infty}, \left\|Y^*\right\|_{2,\infty}\leq \sqrt{\frac{\mu r \sigma_{max}}{n}},\\
    &\left\|\hat{Y}^{d}\right\|_{2,\infty}\leq \left\|\hat{Y}^{d} - Y^*\right\|_{F} + \left\|Y^*\right\|_{2,\infty}\leq \sqrt{\frac{\mu r \sigma_{max}}{n}}+c_a\sqrt{n}.
\end{align*}
Therefore, we have
\begin{align*}
    \sum_{i=1}^n(\Delta_\Gamma^{'})_{ii}^2\leq 4(c_a+c_{11})^2n \left(\sqrt{\frac{\mu r \sigma_{max}}{n}}+c_a\sqrt{n}\right)^2.
\end{align*}
Similarly, for $\sum_{i=1}^n(\Delta_\Gamma^{'''})_{ii}^2$ we also have
\begin{align*}
    \sum_{i=1}^n(\Delta_\Gamma^{'''})_{ii}^2\leq 4(c_a'+c_{11})^2n \left(\sqrt{\frac{\mu r \sigma_{max}}{n}}+c_a'\sqrt{n}\right)^2.
\end{align*}
Combine them together, we know that
\begin{align*}
    \sum_{i=1}^n(\Delta_\Gamma)_{ii}^2\leq 2\sum_{i=1}^n(\Delta_\Gamma^{'''})_{ii}^2 + 2\sum_{i=1}^n(\Delta_\Gamma^{'})_{ii}^2\leq 8((c_a+c_{11})^2+(c_a'+c_{11})^2)n\left(\sqrt{\frac{\mu r \sigma_{max}}{n}}+(c_a'+c_a)\sqrt{n}\right)^2.
\end{align*}
Plugging this back to \eqref{LHS:pf_debias}, we know that 
\begin{align}
    &\left(c(\bar{V})-c(\hat{V}^d)\right)^T \nabla^2{L}_c\left(c(V^*)\right)\left(c(\bar{V})-c(\hat{V}^d)\right) \notag\\
    \geq& \frac{ \sl e^{c_P}}{2(1+e^{c_P})^2} \left(\left\|\Delta_H\right\|_F^2+\frac{1}{n^2}\left\|\Delta_\Gamma\right\|_F^2\right)  - C_0 \left(\sqrt{\frac{\mu r \sigma_{max}}{n^2}}+c_a'+c_a\right)^2,\label{eq26:pf_debias}
\end{align}
where we define $C_0:=4((c_a+c_{11})^2+(c_a'+c_{11})^2)$.

For the RHS of \eqref{eq3:pf_debias}, we will bound (1), (2) and (3), respectively.
\begin{enumerate}
\item We first bound (1). We denote
\begin{align*}
\begin{bmatrix}
H^t\\
\Gamma^t
\end{bmatrix}
:=c({V}^*)+t(c(\hat{V}^d)-c({V}^*)),\ 
P_{ij}^t:=z_i^TH^tz_j+\frac{\Gamma^t_{ij}}{n}.
\end{align*}
It then holds that
\begin{align*}
 &\nabla^2L_c\left(c({V}^*)+t(c(\hat{V}^d)-c({V}^*))\right)-\nabla^2 L_c(c(V^*))\\
 =&\sum_{i\neq j}\left(\frac{e^{P_{ij}^t}}{(1+e^{P_{ij}^t})^2}-\frac{e^{P_{ij}^*}}{(1+e^{P_{ij}^*})^2}\right)
   \begin{bmatrix}
  \vec(z_iz_j^T)\\
  \frac1n\vec(e_ie_j^T)
  \end{bmatrix}^{\otimes 2}.
\end{align*}
Thus, we have
\begin{align*}
&\left(c(\bar{V})-c(\hat{V}^d)\right)^T\left(\nabla^2L_c\left(c({V}^*)+t(c(\hat{V}^d)-c({V}^*))\right)-\nabla^2 L_c(c(V^*))\right)(c(\hat{V}^d)-c({V}^*))\\
=&
\begin{bmatrix}
\vec(\Delta'_{H})\\
\vec(\frac1n \Delta'_{\Gamma})
\end{bmatrix}^T\left(\sum_{i\neq j}\left(\frac{e^{P_{ij}^t}}{(1+e^{P_{ij}^t})^2}-\frac{e^{P_{ij}^*}}{(1+e^{P_{ij}^*})^2}\right)
   \begin{bmatrix}
  \vec(z_iz_j^T)\\
  \vec(e_ie_j^T)
  \end{bmatrix}^{\otimes 2}
  \right)
  \begin{bmatrix}
\vec(\Delta_{H})\\
\vec(\frac1n \Delta_{\Gamma})
\end{bmatrix} \\
=&
\begin{bmatrix}
\vec(\Delta'_{H})\\
\vec(\frac1n \Delta'_{\Gamma})
\end{bmatrix}^T\left(\sum_{i\neq j}\left(\frac{e^{P_{ij}^t}}{(1+e^{P_{ij}^t})^2}-\frac{e^{P_{ij}^*}}{(1+e^{P_{ij}^*})^2}\right)
   \cP_c\begin{bmatrix}
  \vec(z_iz_j^T)\\
  \vec(e_ie_j^T)
  \end{bmatrix}^{\otimes 2}\cP_c
  \right)
  \begin{bmatrix}
\vec(\Delta_{H})\\
\vec(\frac1n \Delta_{\Gamma})
\end{bmatrix}.
\end{align*}
Note that
\begin{align*}
&\left\|\sum_{i\neq j}\left(\frac{e^{P_{ij}^t}}{(1+e^{P_{ij}^t})^2}-\frac{e^{P_{ij}^*}}{(1+e^{P_{ij}^*})^2}\right)
   \cP_c\begin{bmatrix}
  \vec(z_iz_j^T)\\
  \vec(e_ie_j^T)
  \end{bmatrix}^{\otimes 2}\cP_c\right\|\\
\leq&
\left\|\sum_{i\neq j}\left|\frac{e^{P_{ij}^t}}{(1+e^{P_{ij}^t})^2}-\frac{e^{P_{ij}^*}}{(1+e^{P_{ij}^*})^2}\right|
   \cP_c\begin{bmatrix}
  \vec(z_iz_j^T)\\
  \vec(e_ie_j^T)
  \end{bmatrix}^{\otimes 2}\cP_c\right\|\\
\leq& \frac14
\left\|\sum_{i\neq j}\left|P_{ij}^t-P_{ij}^*\right|
   \cP_c\begin{bmatrix}
  \vec(z_iz_j^T)\\
  \vec(e_ie_j^T)
  \end{bmatrix}^{\otimes 2}\cP_c\right\|\tag{mean-value theorem}\\
\leq& \max_{i\neq j}\left|P_{ij}^t-P_{ij}^*\right|\cdot \left\|\cP_c\sum_{i, j}
\begin{bmatrix}
  \vec(z_iz_j^T)\\
  \vec(e_ie_j^T)
  \end{bmatrix}
  \begin{bmatrix}
  \vec(z_iz_j^T)\\
  \vec(e_ie_j^T)
\end{bmatrix}^T\cP_c\right\|\\
\leq&\su \max_{i\neq j}\left|P_{ij}^t-P_{ij}^*\right|.
\end{align*}
By the definition of $P^t_{ij}$, we have
\begin{align*}
\max_{i\neq j}\left|P_{ij}^t-P_{ij}^*\right|
&\leq (\max_{i\neq j}\|z_i\|\|z_j\|  )\cdot\|H^t-H^*\|_{F}+\frac1n \|\Gamma^t-\Gamma^*\|_{\infty}\\
&\leq   \frac{\sz}{n}\|H^t-H^*\|_{F}+\frac1n \|\Gamma^t-\Gamma^*\|_{\infty}\\
&\leq   \frac{\sz}{n}\|\hat{H}^d-H^*\|_{F}+\frac1n \|\hat{\Gamma}^d-\Gamma^*\|_{\infty},
\end{align*}
where the last inequality follows from the definition of $H^t,\Gamma^t$. Moreover, notice that
\begin{align*}
\|\hat{\Gamma}^d-\Gamma^*\|_{\infty}    
&=\|(\hat{X}^d-X^*)(\hat{Y}^d)^T+X^{*}(\hat{Y}^d-Y^*)^T\|_{\infty}\\
&\leq \|\hat{X}^d-X^*\|_{2,\infty}\|\hat{Y}^d\|_{2,\infty}+\|X^{*}\|_{2,\infty}\|\hat{Y}^d-Y^*\|_{2,\infty},
\end{align*}
where $\|\hat{Y}^d\|_{2,\infty}\leq \|\hat{Y}^d-\bar{Y}\|_{2,\infty}+\|\bar{Y}-{Y}^*\|_{2,\infty}+\|{Y}^*\|_{2,\infty}\leq 3\|{Y}^*\|_{2,\infty}$. Thus, we have
\begin{align*}
\|\hat{\Gamma}^d-\Gamma^*\|_{\infty} \leq 4\|F^{*}\|_{2,\infty}\|\hat{F}^d-F^*\|_{2,\infty}
\end{align*}
and 
\begin{align*}
\max_{i\neq j}\left|P_{ij}^t-P_{ij}^*\right|
&\lesssim  \frac{\sz}{n}\|\hat{H}^d-H^*\|_{F}+\frac1n\|F^{*}\|_{2,\infty}\|\hat{F}^d-F^*\|_{2,\infty}\\
&\lesssim \frac{\sz}{n}(\|\hat{H}^d-\hat{H}\|_{F}+\|\hat{H}-H^*\|_{F}) +\frac1n\|F^{*}\|_{2,\infty}(\|\hat{F}^d-\bar{F}\|_{2,\infty}+\|\bar{F}-F^*\|_{2,\infty})\\
&\lesssim\frac{\sz (c_{11}+c_a)+\sqrt{\mu r^2\sigma_{\max}/n^2}(c_{\infty}+c_b)}{\sqrt{n}}\\
&\lesssim\frac{\sqrt{\mu r^2\sigma_{\max}/n^2}(c_{\infty}+c_b)}{\sqrt{n}},
\end{align*}
where the third inequality follows from Lemma \ref{lem:debias_bar}, Proposition \ref{prop:debias_dist}, Proposition \ref{prop:bar_dist} and Lemma \ref{lem:ncv1}.
Consequently, we have
\begin{align*}
\left\|\sum_{i\neq j}\left(\frac{e^{P_{ij}^t}}{(1+e^{P_{ij}^t})^2}-\frac{e^{P_{ij}^*}}{(1+e^{P_{ij}^*})^2}\right)
   \cP_c\begin{bmatrix}
  \vec(z_iz_j^T)\\
  \vec(e_ie_j^T)
  \end{bmatrix}
  \begin{bmatrix}
  \vec(z_iz_j^T)\\
  \vec(e_ie_j^T)
  \end{bmatrix}^\top\cP_c\right\| \lesssim \frac{\su \sqrt{\mu r^2\sigma_{\max}/n^2}(c_{\infty}+c_b)}{\sqrt{n}}
\end{align*}
and for all $t\in [0,1]$
\begin{align*}
  &\left|\left(c(\bar{V})-c(\hat{V}^d)\right)^\top\left(\nabla^2L_c\left(c({V}^*)+t(c(\hat{V}^d)-c({V}^*))\right)-\nabla^2 L_c(c(V^*))\right)(c(\hat{V}^d)-c({V}^*))\right|\\
  &\lesssim \frac{\su\sqrt{\mu r^2\sigma_{\max}/n^2}(c_{\infty}+c_b)}{\sqrt{n}}\sqrt{\|\Delta_{H}\|^2_F+\frac{1}{n^2}\|\Delta_{\Gamma}\|^2_F}  \cdot\sqrt{\|\Delta'_{H}\|^2_F+\frac{1}{n^2}\|\Delta'_{\Gamma}\|^2_F}  .
\end{align*}
As a result, we obtain
\begin{align*}
|(1)|\lesssim\frac{C_1}{\sqrt{n}}\sqrt{\|\Delta_{H}\|^2_F+\frac{1}{n^2}\|\Delta_{\Gamma}\|^2_F}  \cdot\sqrt{\|\Delta'_{H}\|^2_F+\frac{1}{n^2}\|\Delta'_{\Gamma}\|^2_F} ,
\end{align*}
where $C_1:= \su\sqrt{\mu r^2\sigma_{\max}/n^2}(c_{\infty}+c_b)$.

\item We then bound (2).
Note that
\begin{align}\label{eq4:pf_debias}
&\nabla L_c(c(\hat{V}^d))^T\left(c(\bar{V})-c(\hat{V}^d)\right)\notag\\
=&\nabla L_c(c(\hat{V}))^\top \left(c(\bar{V})-c(\hat{V}^d)\right)\notag\\
&+\left(c(\bar{V})-c(\hat{V}^d)\right)^\top\int^1_0\nabla^2 L_c\left(c(\hat{V})+t(c(\hat{V}^d)-c(\hat{V}))\right)dt \left(c(\hat{V}^d)-c(\hat{V})\right).
\end{align}
For the first term, recall the definition of $\Delta''$, we have
\begin{align*}
\nabla L_c(c(\hat{V}))^T \left(c(\bar{V})-c(\hat{V})\right)=\sum_{i\neq j}\left(\frac{e^{\hat{P}_{ij}}}{1+e^{\hat{P}_{ij}}}-A_{ij}\right)\left(\langle\Delta''_{H},z_iz_j^T\rangle+\frac1n\langle \bar{X}\bar{Y}^T-\hat{X}\hat{Y}^T,e_i e_j^T\rangle\right).
\end{align*}
Note that
\begin{align*}
\langle \bar{X}\bar{Y}^T-\hat{X}\hat{Y}^T,e_i e_j^T\rangle=\langle \Delta''_X\hat{Y}^T+\hat{X}\Delta^{''T}_Y,e_i e_j^T\rangle+\langle \Delta''_X\Delta^{''T}_Y,e_i e_j^T\rangle.
\end{align*}
Thus we have
\begin{align}\label{eq6:pf_debias}
 &L_c(c(\hat{V}))^T \left(c(\bar{V})-c(\hat{V})\right)   \notag\\
 =&\sum_{i\neq j}\left(\frac{e^{\hat{P}_{ij}}}{1+e^{\hat{P}_{ij}}}-A_{ij}\right)\left(\langle\Delta''_{H},z_iz_j^T\rangle+\frac1n\langle \Delta''_X,e_i e_j^T\hat{Y}\rangle+\frac1n\langle \Delta^{''}_Y,e_j e_i^T\hat{X}\rangle\right) \notag\\
 &+ \frac1n\sum_{i\neq j}\left(\frac{e^{\hat{P}_{ij}}}{1+e^{\hat{P}_{ij}}}-A_{ij}\right)\langle \Delta''_X\Delta^{''T}_Y,e_i e_j^T\rangle \notag\\
 =&\nabla L(\hat{V})^T(\bar{V}-\hat{V}) + \frac1n\sum_{i\neq j}\left(\frac{e^{\hat{P}_{ij}}}{1+e^{\hat{P}_{ij}}}-A_{ij}\right)\langle \Delta''_X\Delta^{''T}_Y,e_i e_j^T\rangle.
\end{align}
By \eqref{property:debias}, we have 
\begin{align*}
\cP \nabla L(\hat{V})=-\cP \left(\sum_{i\neq j}\frac{e^{\hat{P}_{ij}}}{(1+e^{\hat{P}_{ij}})^2}
\begin{bmatrix}
z_iz_j^T\\
\frac1ne_ie_j^T\hat{Y}\\
\frac1ne_je_i^T\hat{X}
\end{bmatrix}^{\otimes 2}\right)(\hat{V}^d-\hat{V}).
\end{align*}
Moreover, note that $\cP(\bar{V})=\bar{V}, \cP(\hat{V})=\hat{V}, \cP(\hat{V}^d)=\hat{V}^d$.
Thus, it holds that
\begin{align*}
&\nabla L(\hat{V})^T(\bar{V}-\hat{V})\\
=& \left(\cP\nabla L(\hat{V}) \right)^T(\bar{V}-\hat{V})\\
=&-(\bar{V}-\hat{V})^T\left(\sum_{i\neq j}\frac{e^{\hat{P}_{ij}}}{(1+e^{\hat{P}_{ij}})^2}
\begin{bmatrix}
z_iz_j^T\\
\frac1ne_ie_j^T\hat{Y}\\
\frac1ne_je_i^T\hat{X}
\end{bmatrix}^{\otimes 2}\right)(\hat{V}^d-\hat{V})\\
=&-\sum_{i\neq j}\frac{e^{\hat{P}_{ij}}}{(1+e^{\hat{P}_{ij}})^2}\cdot\left(\langle\Delta''_{H},z_iz_j^T\rangle+\frac1n\langle \Delta''_X\hat{Y}^T+\hat{X}\Delta^{''T}_Y,e_i e_j^T\rangle\right)\cdot \left(\langle\hat{\Delta}_{H},z_iz_j^T\rangle+\frac1n\langle \hat{\Delta}_X\hat{Y}^T+\hat{X}\hat{\Delta}^{T}_Y,e_i e_j^T\rangle\right)\\
=&-\sum_{i\neq j}\frac{e^{\hat{P}_{ij}}}{(1+e^{\hat{P}_{ij}})^2}\cdot\left(\langle\Delta''_{H},z_iz_j^T\rangle+\frac1n\langle\bar{X}\bar{Y}^T-\hat{X}\hat{Y}^T,e_i e_j^T\rangle-\frac1n \langle \Delta''_X\Delta^{''T}_Y,e_i e_j^T\rangle\right)\\
&\qquad\qquad\cdot \left(\langle\hat{\Delta}_{H},z_iz_j^T\rangle+\frac1n\langle\hat{X}^d\hat{Y}^{dT}-\hat{X}\hat{Y}^T,e_i e_j^T\rangle-\frac1n \langle \hat{\Delta}_X\hat{\Delta}^{T}_Y,e_i e_j^T\rangle\right)\\
=&-\sum_{i\neq j}\frac{e^{\hat{P}_{ij}}}{(1+e^{\hat{P}_{ij}})^2}\cdot\left(\langle\Delta''_{H},z_iz_j^T\rangle+\frac1n\langle\bar{X}\bar{Y}^T-\hat{X}\hat{Y}^T,e_i e_j^T\rangle\right)\cdot \left(\langle\hat{\Delta}_{H},z_iz_j^T\rangle+\frac1n\langle\hat{X}^d\hat{Y}^{dT}-\hat{X}\hat{Y}^T,e_i e_j^T\rangle\right)\\
&+\frac1n\sum_{i\neq j}\frac{e^{\hat{P}_{ij}}}{(1+e^{\hat{P}_{ij}})^2}\cdot\left(\langle\Delta''_{H},z_iz_j^T\rangle+\frac1n\langle\bar{X}\bar{Y}^T-\hat{X}\hat{Y}^T,e_i e_j^T\rangle\right)\langle \hat{\Delta}_X\hat{\Delta}^{T}_Y,e_i e_j^T\rangle\\
&+\frac1n\sum_{i\neq j}\frac{e^{\hat{P}_{ij}}}{(1+e^{\hat{P}_{ij}})^2}\cdot\left(\langle\hat{\Delta}_{H},z_iz_j^T\rangle+\frac1n\langle\hat{X}^d\hat{Y}^{dT}-\hat{X}\hat{Y}^T,e_i e_j^T\rangle\right)\langle \Delta''_X\Delta^{''T}_Y,e_i e_j^T\rangle\\
&-\frac{1}{n^2}\sum_{i\neq j}\frac{e^{\hat{P}_{ij}}}{(1+e^{\hat{P}_{ij}})^2}\langle \hat{\Delta}_X\hat{\Delta}^{T}_Y,e_i e_j^T\rangle \langle \Delta''_X\Delta^{''T}_Y,e_i e_j^T\rangle\\
=&-\left(c(\bar{V})-c(\hat{V})\right)^T\nabla^2L_c(c(\hat{V})) \left(c(\hat{V}^d)-c(\hat{V})\right)\\
&+\frac1n\sum_{i\neq j}\frac{e^{\hat{P}_{ij}}}{(1+e^{\hat{P}_{ij}})^2}\cdot\left(\langle\Delta''_{H},z_iz_j^T\rangle+\frac1n\langle\bar{X}\bar{Y}^T-\hat{X}\hat{Y}^T,e_i e_j^T\rangle\right)\langle \hat{\Delta}_X\hat{\Delta}^{T}_Y,e_i e_j^T\rangle\\
&+\frac1n\sum_{i\neq j}\frac{e^{\hat{P}_{ij}}}{(1+e^{\hat{P}_{ij}})^2}\cdot\left(\langle\hat{\Delta}_{H},z_iz_j^T\rangle+\frac1n\langle\hat{X}^d\hat{Y}^{dT}-\hat{X}\hat{Y}^T,e_i e_j^T\rangle\right)\langle \Delta''_X\Delta^{''T}_Y,e_i e_j^T\rangle\\
& -\frac{1}{n^2}\sum_{i\neq j}\frac{e^{\hat{P}_{ij}}}{(1+e^{\hat{P}_{ij}})^2}\langle \hat{\Delta}_X\hat{\Delta}^{T}_Y,e_i e_j^T\rangle \langle \Delta''_X\Delta^{''T}_Y,e_i e_j^T\rangle
\end{align*}
Combine the above result with \eqref{eq6:pf_debias}, we have
\begin{align}\label{eq7:pf_debias}
 &L_c(c(\hat{V}))^T \left(c(\bar{V})-c(\hat{V})\right)   \notag\\ 
 =&-\left(c(\bar{V})-c(\hat{V})\right)^T\nabla^2L_c(c(\hat{V})) \left(c(\hat{V}^d)-c(\hat{V})\right)\notag\\
&+\frac1n\sum_{i\neq j}\frac{e^{\hat{P}_{ij}}}{(1+e^{\hat{P}_{ij}})^2}\cdot\left(\langle\Delta''_{H},z_iz_j^T\rangle+\frac1n\langle\bar{X}\bar{Y}^T-\hat{X}\hat{Y}^T,e_i e_j^T\rangle\right)\langle \hat{\Delta}_X\hat{\Delta}^{T}_Y,e_i e_j^T\rangle\notag\\
&+\frac1n\sum_{i\neq j}\frac{e^{\hat{P}_{ij}}}{(1+e^{\hat{P}_{ij}})^2}\cdot\left(\langle\hat{\Delta}_{H},z_iz_j^T\rangle+\frac1n\langle\hat{X}^d\hat{Y}^{dT}-\hat{X}\hat{Y}^T,e_i e_j^T\rangle\right)\langle \Delta''_X\Delta^{''T}_Y,e_i e_j^T\rangle\notag\\
&-\frac{1}{n^2}\sum_{i\neq j}\frac{e^{\hat{P}_{ij}}}{(1+e^{\hat{P}_{ij}})^2}\langle \hat{\Delta}_X\hat{\Delta}^{T}_Y,e_i e_j^T\rangle \langle \Delta''_X\Delta^{''T}_Y,e_i e_j^T\rangle\notag\\
&+\frac1n\sum_{i\neq j}\left(\frac{e^{\hat{P}_{ij}}}{1+e^{\hat{P}_{ij}}}-A_{ij}\right)\langle \Delta''_X\Delta^{''T}_Y,e_i e_j^T\rangle.
\end{align}
Similarly, we have
\begin{align}\label{eq8:pf_debias}
 &L_c(c(\hat{V}))^T \left(c(\hat{V}^d)-c(\hat{V})\right)   \notag\\ 
 =&-\left(c(\hat{V}^d)-c(\hat{V})\right)^T\nabla^2L_c(c(\hat{V})) \left(c(\hat{V}^d)-c(\hat{V})\right)\notag\\
&+\frac2n\sum_{i\neq j}\frac{e^{\hat{P}_{ij}}}{(1+e^{\hat{P}_{ij}})^2}\cdot\left(\langle\hat{\Delta}_{H},z_iz_j^T\rangle+\frac1n\langle\hat{X}^d\hat{Y}^{dT}-\hat{X}\hat{Y}^T,e_i e_j^T\rangle\right)\langle \hat{\Delta}_X\hat{\Delta}^{T}_Y,e_i e_j^T\rangle\notag\\
&-\frac{1}{n^2}\sum_{i\neq j}\frac{e^{\hat{P}_{ij}}}{(1+e^{\hat{P}_{ij}})^2}|\langle \hat{\Delta}_X\hat{\Delta}^{T}_Y,e_i e_j^T\rangle|^2 \notag\\
&+\frac1n\sum_{i\neq j}\left(\frac{e^{\hat{P}_{ij}}}{1+e^{\hat{P}_{ij}}}-A_{ij}\right)\langle \hat{\Delta}_X\hat{\Delta}^{T}_Y,e_i e_j^T\rangle.
\end{align}
Combine \eqref{eq7:pf_debias} and \eqref{eq8:pf_debias}, we obtain
\begin{align}\label{eq9:pf_debias}
   L_c(c(\hat{V}))^T \left(c(\bar{V})-c(\hat{V}^d)\right)  =-\left(c(\bar{V})-c(\hat{V}^d)\right)^T\nabla^2L_c(c(\hat{V})) \left(c(\hat{V}^d)-c(\hat{V})\right)+(r),
\end{align}
where 
\begin{align*}
(r)
=&\frac1n\sum_{i\neq j}\frac{e^{\hat{P}_{ij}}}{(1+e^{\hat{P}_{ij}})^2}\cdot\left(\langle\Delta''_{H},z_iz_j^T\rangle+\frac1n\langle\bar{X}\bar{Y}^T-\hat{X}\hat{Y}^T,e_i e_j^T\rangle\right)\langle \hat{\Delta}_X\hat{\Delta}^{T}_Y,e_i e_j^T\rangle\notag\\
&+\frac1n\sum_{i\neq j}\frac{e^{\hat{P}_{ij}}}{(1+e^{\hat{P}_{ij}})^2}\cdot\left(\langle\hat{\Delta}_{H},z_iz_j^T\rangle+\frac1n\langle\hat{X}^d\hat{Y}^{dT}-\hat{X}\hat{Y}^T,e_i e_j^T\rangle\right)\langle \Delta''_X\Delta^{''T}_Y,e_i e_j^T\rangle\notag\\
&-\frac2n\sum_{i\neq j}\frac{e^{\hat{P}_{ij}}}{(1+e^{\hat{P}_{ij}})^2}\cdot\left(\langle\hat{\Delta}_{H},z_iz_j^T\rangle+\frac1n\langle\hat{X}^d\hat{Y}^{dT}-\hat{X}\hat{Y}^T,e_i e_j^T\rangle\right)\langle \hat{\Delta}_X\hat{\Delta}^{T}_Y,e_i e_j^T\rangle\notag\\
&-\frac{1}{n^2}\sum_{i\neq j}\frac{e^{\hat{P}_{ij}}}{(1+e^{\hat{P}_{ij}})^2}\langle \hat{\Delta}_X\hat{\Delta}^{T}_Y,e_i e_j^T\rangle \langle \Delta''_X\Delta^{''T}_Y,e_i e_j^T\rangle+\frac{1}{n^2}\sum_{i\neq j}\frac{e^{\hat{P}_{ij}}}{(1+e^{\hat{P}_{ij}})^2}|\langle \hat{\Delta}_X\hat{\Delta}^{T}_Y,e_i e_j^T\rangle|^2 \notag\\
&+\frac1n\sum_{i\neq j}\left(\frac{e^{\hat{P}_{ij}}}{1+e^{\hat{P}_{ij}}}-A_{ij}\right)\langle \Delta''_X\Delta^{''T}_Y,e_i e_j^T\rangle-\frac1n\sum_{i\neq j}\left(\frac{e^{\hat{P}_{ij}}}{1+e^{\hat{P}_{ij}}}-A_{ij}\right)\langle \hat{\Delta}_X\hat{\Delta}^{T}_Y,e_i e_j^T\rangle.
\end{align*}
By \eqref{eq4:pf_debias} and \eqref{eq9:pf_debias}, we have
\begin{align}\label{eq10:pf_debias}
&\nabla L_c(c(\hat{V}^d))^T\left(c(\bar{V})-c(\hat{V}^d)\right)\notag\\
&=\left(c(\bar{V})-c(\hat{V}^d)\right)^T\left(\int^1_0\nabla^2 L_c\left(c(\hat{V})+t(c(\hat{V}^d)-c(\hat{V}))\right)-\nabla^2L_c(c(\hat{V})) dt\right)\left(c(\hat{V}^d)-c(\hat{V})\right)+(r).
\end{align}
As bounding (1), we have
\begin{align}\label{eq11:pf_debias}
&\left|\left(c(\bar{V})-c(\hat{V}^d)\right)^T\left(\int^1_0\nabla^2 L_c\left(c(\hat{V})+t(c(\hat{V}^d)-c(\hat{V}))\right)-\nabla^2L_c(c(\hat{V})) dt\right)\left(c(\hat{V}^d)-c(\hat{V})\right)\right|\notag\\
\lesssim&\frac{C_2}{\sqrt{n}}\sqrt{\|\Delta_{H}\|^2_F+\frac{1}{n^2}\|\Delta_{\Gamma}\|^2_F}  \cdot\sqrt{\|\hat{\Delta}_{H}\|^2_F+\frac{1}{n^2}\|\hat{\Delta}_{\Gamma}\|^2_F}
\end{align}
for $C_2= \su\sqrt{\mu r \sigma_{\max}/n^2}(\sqrt{r}(c_{\infty}+c_b)+c_{43})$.
It remains to bound $(r)$.
Note that
\begin{align*}
&\left|\frac1n\sum_{i\neq j}\frac{e^{\hat{P}_{ij}}}{(1+e^{\hat{P}_{ij}})^2}\cdot\left(\langle\Delta''_{H},z_iz_j^T\rangle+\frac1n\langle\bar{X}\bar{Y}^T-\hat{X}\hat{Y}^T,e_i e_j^T\rangle\right)\langle \hat{\Delta}_X\hat{\Delta}^{T}_Y,e_i e_j^T\rangle\right|\notag\\
\lesssim& \frac1n \|\hat{\Delta}_X\hat{\Delta}^{T}_Y\|_F\sqrt{\sum_{i\neq j}\left(\langle\Delta''_{H},z_iz_j^T\rangle+\frac1n\langle\bar{X}\bar{Y}^T-\hat{X}\hat{Y}^T,e_i e_j^T\rangle\right)^2}\tag{Cauchy-Schwarz}\\
\lesssim&  \frac{\sqrt{\su}}{n} \left(\|\hat{\Delta}_X\|_F^2+\|\hat{\Delta}_Y\|_F^2\right)\sqrt{\|\Delta''_{H}\|^2_F+\frac{1}{n^2}\|\Delta''_{\Gamma}\|^2_F} \tag{by Assumption \ref{assumption:eigenvalues}}.
\end{align*}
Similarly, we have
\begin{align*}
&\left|\frac1n\sum_{i\neq j}\frac{e^{\hat{P}_{ij}}}{(1+e^{\hat{P}_{ij}})^2}\cdot\left(\langle\hat{\Delta}_{H},z_iz_j^T\rangle+\frac1n\langle\hat{X}^d\hat{Y}^{dT}-\hat{X}\hat{Y}^T,e_i e_j^T\rangle\right)\langle \Delta''_X\Delta^{''T}_Y,e_i e_j^T\rangle\right|\\
\lesssim&  \frac{\sqrt{\su}}{n} \left(\|\Delta''_X\|_F^2+\|\Delta''_Y\|_F^2\right)\sqrt{\|\hat{\Delta}_{H}\|^2_F+\frac{1}{n^2}\|\hat{\Delta}_{\Gamma}\|^2_F}
\end{align*}
and 
\begin{align*}
&\left|\frac2n\sum_{i\neq j}\frac{e^{\hat{P}_{ij}}}{(1+e^{\hat{P}_{ij}})^2}\cdot\left(\langle\hat{\Delta}_{H},z_iz_j^T\rangle+\frac1n\langle\hat{X}^d\hat{Y}^{dT}-\hat{X}\hat{Y}^T,e_i e_j^T\rangle\right)\langle \hat{\Delta}_X\hat{\Delta}^{T}_Y,e_i e_j^T\rangle\right|\\
\lesssim&  \frac{\sqrt{\su}}{n}\left(\|\hat{\Delta}_X\|_F^2+\|\hat{\Delta}_Y\|_F^2\right)\sqrt{\|\hat{\Delta}_{H}\|^2_F+\frac{1}{n^2}\|\hat{\Delta}_{\Gamma}\|^2_F}.
\end{align*}
Also, we have
\begin{align*}
&\left|\frac{1}{n^2}\sum_{i\neq j}\frac{e^{\hat{P}_{ij}}}{(1+e^{\hat{P}_{ij}})^2}\langle \hat{\Delta}_X\hat{\Delta}^{T}_Y,e_i e_j^T\rangle \langle \Delta''_X\Delta^{''T}_Y,e_i e_j^T\rangle\right|\\
\lesssim& \frac{1}{n^2}\|\Delta''_X\Delta^{''T}_Y\|_F\|\hat{\Delta}_X\hat{\Delta}^{T}_Y\|_F\tag{Cauchy-Schwarz}\\
\leq& \frac{1}{n^2}\left(\|\Delta''_X\|_F^2+\|\Delta''_Y\|_F^2\right)\left(\|\hat{\Delta}_X\|_F^2+\|\hat{\Delta}_Y\|_F^2\right)
\end{align*}
and 
\begin{align*}
&\left|\frac{1}{n^2}\sum_{i\neq j}\frac{e^{\hat{P}_{ij}}}{(1+e^{\hat{P}_{ij}})^2}|\langle \hat{\Delta}_X\hat{\Delta}^{T}_Y,e_i e_j^T\rangle|^2\right|\leq\frac{1}{n^2}\|\hat{\Delta}_X\hat{\Delta}^{T}_Y\|_F^2\leq \frac{1}{n^2}\left(\|\hat{\Delta}_X\|_F^2+\|\hat{\Delta}_Y\|_F^2\right)^2.
\end{align*}
Follow the proofs in Appendix \ref{sec:pf_geometry}, we have
\begin{align*}
&\left|\frac1n\sum_{i\neq j}\left(\frac{e^{\hat{P}_{ij}}}{1+e^{\hat{P}_{ij}}}-A_{ij}\right)\langle \Delta''_X\Delta^{''T}_Y,e_i e_j^T\rangle\right|\\
\lesssim&\frac{\sqrt{\su}}{n}\left(\|\Delta''_X\|_F^2+\|\Delta''_Y\|_F^2\right)\left(\|\hat{H}-H^{*}\|_{F}+\frac{1}{n}\|X^{*}\|\|\hat{F}-F^{*}\|_{F}+{\sqrt{\log n}}\right)
\end{align*}
and 
\begin{align*}
&\left|\frac1n\sum_{i\neq j}\left(\frac{e^{\hat{P}_{ij}}}{1+e^{\hat{P}_{ij}}}-A_{ij}\right)\langle \hat{\Delta}_X\hat{\Delta}^{T}_Y,e_i e_j^T\rangle\right|\\
\lesssim&\frac{\sqrt{\su}}{n}\left(\|\hat{\Delta}_X\|_F^2+\|\hat{\Delta}_Y\|_F^2\right)\left(\|\hat{H}-H^{*}\|_{F}+\frac{1}{n}\|X^{*}\|\|\hat{F}-F^{*}\|_{F}+{\sqrt{\log n}}\right).
\end{align*}
As a result, we obtain
\begin{align*}
|(r)|\lesssim&  \frac{\sqrt{\su}}{n}\left(\|\hat{\Delta}_X\|_F^2+\|\hat{\Delta}_Y\|_F^2+\|\Delta''_X\|_F^2+\|\Delta''_Y\|_F^2\right)\\
&\cdot\sqrt{\|\hat{\Delta}_{H}\|^2_F+\frac{1}{n^2}\|\hat{\Delta}_{\Gamma}\|^2_F+\|\Delta''_{H}\|^2_F+\frac{1}{n^2}\|\Delta''_{\Gamma}\|^2_F}\\
&+\frac{1}{n^2}\left(\|\hat{\Delta}_X\|_F^2+\|\hat{\Delta}_Y\|_F^2+\|\Delta''_X\|_F^2+\|\Delta''_Y\|_F^2\right)\left(\|\hat{\Delta}_X\|_F^2+\|\hat{\Delta}_Y\|_F^2\right)\\
&+\frac{\sqrt{\su}}{n}\left(\|\hat{\Delta}_X\|_F^2+\|\hat{\Delta}_Y\|_F^2+\|\Delta''_X\|_F^2+\|\Delta''_Y\|_F^2\right) \\
&\cdot\left(\|\hat{H}-H^{*}\|_{F}+\frac{1}{n}\|X^{*}\|\|\hat{F}-F^{*}\|_{F}+{\sqrt{\log n}}\right)\\
\lesssim& \sqrt{\frac{\su \mu r\sigma_{\max}}{n^2}}(c_a+c_{11})^3 \sqrt{n}.
\end{align*}
Combine this with \eqref{eq10:pf_debias} and \eqref{eq11:pf_debias}, we show that
\begin{align*}
|(2)|&=\left|L_c(c(\hat{V}^d))^T\left(c(\bar{V})-c(\hat{V}^d)\right)\right|\\
&\lesssim \frac{C_2}{\sqrt{n}}\sqrt{\|\Delta_{H}\|^2_F+\frac{1}{n^2}\|\Delta_{\Gamma}\|^2_F}  \cdot\sqrt{\|\hat{\Delta}_{H}\|^2_F+\frac{1}{n^2}\|\hat{\Delta}_{\Gamma}\|^2_F}+\sqrt{\frac{\su \mu r\sigma_{\max}}{n^2}}(c_a+c_{11})^3 \sqrt{n}. 
\end{align*}

\item We finally bound (3). By \eqref{eq2:pf_debias} with $V=\bar{V}$, we have
\begin{align*}
\nabla\bar{L}_c\left(c(\bar{V})\right)=\nabla {L}_c\left(c(V^*)\right)+\nabla^2 {L}_c\left(c(V^*)\right)\left(c(\bar{V})-c(V^*)\right),
\end{align*}
which then implies 
\begin{align}\label{eq24:pf_debias}
&\nabla\bar{L}_c\left(c(\bar{V})\right)^T\left(c(\bar{V})-c(\hat{V}^d)\right)\notag\\
=&\nabla {L}_c\left(c(V^*)\right)^T\left(c(\bar{V})-c(\hat{V}^d)\right)+\left(c(\bar{V})-c(\hat{V}^d)\right)^T\nabla^2 {L}_c\left(c(V^*)\right)\left(c(\bar{V})-c(V^*)\right).
\end{align}
We will deal with the first term in the following. Recall the definition of $\Delta'''$, follow the same argument as in \eqref{eq6:pf_debias}, we have
\begin{align}\label{eq20:pf_debias}
 L_c(c(V^*))^T \left(c(\bar{V})-c(V^*)\right) =  \nabla L(V^*)^T(\bar{V}-V^*) + \frac1n\sum_{i\neq j}\left(\frac{e^{{P}^*_{ij}}}{1+e^{{P}^*_{ij}}}-A_{ij}\right)\langle \Delta^{'''}_X\Delta^{'''T}_Y,e_i e_j^T\rangle.
\end{align}
By \eqref{property:bar}, we have $\cP\nabla L(V^*)=-\cP D^*\left(\bar{V}-V^*\right)$. Moreover, note that $\cP(\bar{V})=\bar{V}, \cP(V^*)=V^*$. Thus, we have
\begin{align*}
&\nabla L(V^*)^T(\bar{V}-V^*) =\left(\cP\nabla L(V^*)\right)^T(\bar{V}-V^*) \\
=&-\left(\bar{V}-V^*\right)^T\left(\sum_{i\neq j}\frac{e^{P^*_{ij}}}{(1+e^{P^*_{ij}})^2}
\begin{bmatrix}
z_iz_j^T\\
\frac1ne_ie_j^T{Y^*}\\
\frac1ne_je_i^T{X^*}
\end{bmatrix}^{\otimes 2}\right)(\bar{V}-V^*)\\
=&-\left(c(\bar{V})-c({V}^*)\right)^T\nabla^2L_c(c({V^*})) \left(c(\bar{V})-c({V}^*)\right)\\
&+\frac2n\sum_{i\neq j}\frac{e^{{P}^*_{ij}}}{(1+e^{{P}^*_{ij}})^2}\cdot\left(\langle\Delta'''_{H},z_iz_j^T\rangle+\frac1n\langle\bar{X}\bar{Y}^T-X^*Y^{*T},e_i e_j^T\rangle\right)\langle \Delta'''_X\Delta^{'''T}_Y,e_i e_j^T\rangle\\
&-\frac{1}{n^2}\sum_{i\neq j}\frac{e^{{P}^*_{ij}}}{(1+e^{{P}^*_{ij}})^2}\left|\langle \Delta'''_X\Delta^{'''T}_Y,e_i e_j^T\rangle\right|^2,
\end{align*}
where the last equation follows the same argument as that in bounding (2). Plug this into \eqref{eq20:pf_debias}, we have
\begin{align}\label{eq21:pf_debias}
  &L_c(c(V^*))^T \left(c(\bar{V})-c(V^*)\right)   \notag\\
  =&-\left(c(\bar{V})-c({V}^*)\right)^T\nabla^2L_c(c({V^*})) \left(c(\bar{V})-c({V}^*)\right)\notag\\
&+\frac2n\sum_{i\neq j}\frac{e^{{P}^*_{ij}}}{(1+e^{{P}^*_{ij}})^2}\cdot\left(\langle\Delta'''_{H},z_iz_j^T\rangle+\frac1n\langle\bar{X}\bar{Y}^T-X^*Y^{*T},e_i e_j^T\rangle\right)\langle \Delta'''_X\Delta^{'''T}_Y,e_i e_j^T\rangle\notag\\
& -\frac{1}{n^2}\sum_{i\neq j}\frac{e^{{P}^*_{ij}}}{(1+e^{{P}^*_{ij}})^2}\left|\langle \Delta'''_X\Delta^{'''T}_Y,e_i e_j^T\rangle\right|^2+ \frac1n\sum_{i\neq j}\left(\frac{e^{{P}^*_{ij}}}{1+e^{{P}^*_{ij}}}-A_{ij}\right)\langle \Delta'''_X\Delta^{'''T}_Y,e_i e_j^T\rangle.
\end{align}
Similarly, recall the definition of $\Delta'$, we have
\begin{align}\label{eq22:pf_debias}
  &L_c(c(V^*))^T \left(c(\hat{V}^d)-c(V^*)\right)   \notag\\
  =&-\left(c(\hat{V}^d)-c({V}^*)\right)^T\nabla^2L_c(c({V^*})) \left(c(\bar{V})-c({V}^*)\right)\notag\\
&+\frac1n\sum_{i\neq j}\frac{e^{{P}^*_{ij}}}{(1+e^{{P}^*_{ij}})^2}\cdot\left(\langle\Delta'_{H},z_iz_j^T\rangle+\frac1n\langle\hat{X}^d\hat{Y}^{dT}-X^*Y^{*T},e_i e_j^T\rangle\right)\langle \Delta'''_X\Delta^{'''T}_Y,e_i e_j^T\rangle\notag\\
&+\frac1n\sum_{i\neq j}\frac{e^{{P}^*_{ij}}}{(1+e^{{P}^*_{ij}})^2}\cdot\left(\langle\Delta'''_{H},z_iz_j^T\rangle+\frac1n\langle\bar{X}\bar{Y}^T-X^*Y^{*T},e_i e_j^T\rangle\right)\langle \Delta'_X\Delta^{'T}_Y,e_i e_j^T\rangle\notag\\
&-\frac{1}{n^2}\sum_{i\neq j}\frac{e^{{P}^*_{ij}}}{(1+e^{{P}^*_{ij}})^2}\langle \Delta'_X\Delta^{'T}_Y,e_i e_j^T\rangle\langle \Delta'''_X\Delta^{'''T}_Y,e_i e_j^T\rangle+ \frac1n\sum_{i\neq j}\left(\frac{e^{{P}^*_{ij}}}{1+e^{{P}^*_{ij}}}-A_{ij}\right)\langle \Delta'_X\Delta^{'T}_Y,e_i e_j^T\rangle.
\end{align}
Combine \eqref{eq21:pf_debias} and \eqref{eq22:pf_debias}, we then have
\begin{align}\label{eq23:pf_debias}
\nabla {L}_c\left(c(V^*)\right)^T\left(c(\bar{V})-c(\hat{V}^d)\right) =-\left(c(\bar{V})-c(\hat{V}^d)\right)^T\nabla^2 {L}_c\left(c(V^*)\right)\left(c(\bar{V})-c(V^*)\right)+(r),
\end{align}
where 
\begin{align*}
(r)=&\frac2n\sum_{i\neq j}\frac{e^{{P}^*_{ij}}}{(1+e^{{P}^*_{ij}})^2}\cdot\left(\langle\Delta'''_{H},z_iz_j^T\rangle+\frac1n\langle\bar{X}\bar{Y}^T-X^*Y^{*T},e_i e_j^T\rangle\right)\langle \Delta'''_X\Delta^{'''T}_Y,e_i e_j^T\rangle\notag\\
&-\frac{1}{n^2}\sum_{i\neq j}\frac{e^{{P}^*_{ij}}}{(1+e^{{P}^*_{ij}})^2}\left|\langle \Delta'''_X\Delta^{'''T}_Y,e_i e_j^T\rangle\right|^2+ \frac1n\sum_{i\neq j}\left(\frac{e^{{P}^*_{ij}}}{1+e^{{P}^*_{ij}}}-A_{ij}\right)\langle \Delta'''_X\Delta^{'''T}_Y,e_i e_j^T\rangle\\
&-\frac1n\sum_{i\neq j}\frac{e^{{P}^*_{ij}}}{(1+e^{{P}^*_{ij}})^2}\cdot\left(\langle\Delta'_{H},z_iz_j^T\rangle+\frac1n\langle\hat{X}^d\hat{Y}^{dT}-X^*Y^{*T},e_i e_j^T\rangle\right)\langle \Delta'''_X\Delta^{'''T}_Y,e_i e_j^T\rangle\notag\\
&-\frac1n\sum_{i\neq j}\frac{e^{{P}^*_{ij}}}{(1+e^{{P}^*_{ij}})^2}\cdot\left(\langle\Delta'''_{H},z_iz_j^T\rangle+\frac1n\langle\bar{X}\bar{Y}^T-X^*Y^{*T},e_i e_j^T\rangle\right)\langle \Delta'_X\Delta^{'T}_Y,e_i e_j^T\rangle\notag\\
&+\frac{1}{n^2}\sum_{i\neq j}\frac{e^{{P}^*_{ij}}}{(1+e^{{P}^*_{ij}})^2}\langle \Delta'_X\Delta^{'T}_Y,e_i e_j^T\rangle\langle \Delta'''_X\Delta^{'''T}_Y,e_i e_j^T\rangle- \frac1n\sum_{i\neq j}\left(\frac{e^{{P}^*_{ij}}}{1+e^{{P}^*_{ij}}}-A_{ij}\right)\langle \Delta'_X\Delta^{'T}_Y,e_i e_j^T\rangle.
\end{align*}
By \eqref{eq24:pf_debias} and \eqref{eq23:pf_debias}, we have
\begin{align*}
    \nabla\bar{L}_c\left(c(\bar{V})\right)^T\left(c(\bar{V})-c(\hat{V}^d)\right)=(r).
\end{align*}
Following the same argument as in bounding (2), we have
\begin{align*}
|(3)|=&|(r)|\\
\lesssim&  \frac{\sqrt{\su}}{n}\left(\|\Delta'_X\|_F^2+\|\Delta'_Y\|_F^2+\|\Delta'''_X\|_F^2+\|\Delta'''_Y\|_F^2\right)\\
&\cdot\sqrt{\|\Delta'_{H}\|^2_F+\frac{1}{n^2}\|\Delta'_{\Gamma}\|^2_F+\|\Delta'''_{H}\|^2_F+\frac{1}{n^2}\|\Delta'''_{\Gamma}\|^2_F}\\
&+\frac{1}{n^2}\left(\|\Delta'_X\|_F^2+\|\Delta'_Y\|_F^2+\|\Delta'''_X\|_F^2+\|\Delta'''_Y\|_F^2\right)\left(\|\Delta'''_X\|_F^2+\|\Delta'''_Y\|_F^2\right)\\
& +\frac{\sqrt{\log n}}n \left(\|\Delta'_X\|_F^2+\|\Delta'_Y\|_F^2+\|\Delta'''_X\|_F^2+\|\Delta'''_Y\|_F^2\right)\\
\lesssim& \sqrt{\frac{\su \mu r\sigma_{\max}}{n^2}}(c_a+c_{11})^3 \sqrt{n}.
\end{align*}

\end{enumerate}

Combine the results for (1)-(3), by \eqref{eq3:pf_debias}, we have
\begin{align*}
&\left|\left(c(\bar{V})-c(\hat{V}^d)\right)^T \nabla^2{L}_c\left(c(V^*)\right)\left(c(\bar{V})-c(\hat{V}^d)\right)\right|\notag\\
\lesssim& \frac{C_1}{\sqrt{n}}\sqrt{\|\Delta_{H}\|^2_F+\frac{1}{n^2}\|\Delta_{\Gamma}\|^2_F}  \cdot\sqrt{\|\Delta'_{H}\|^2_F+\frac{1}{n^2}\|\Delta'_{\Gamma}\|^2_F}\\
&+\frac{C_2}{\sqrt{n}}\sqrt{\|\Delta_{H}\|^2_F+\frac{1}{n^2}\|\Delta_{\Gamma}\|^2_F}  \cdot\sqrt{\|\hat{\Delta}_{H}\|^2_F+\frac{1}{n^2}\|\hat{\Delta}_{\Gamma}\|^2_F}+\sqrt{\frac{\su \mu r\sigma_{\max}}{n^2}}(c_a+c_{11})^3 \sqrt{n}\\
\lesssim& C_2(c_a+c_{11})\sqrt{\frac{\mu r\sigma_{\max}}{n^2}}\sqrt{\|\Delta_{H}\|^2_F+\frac{1}{n^2}\|\Delta_{\Gamma}\|^2_F} +\sqrt{\frac{\su \mu r\sigma_{\max}}{n^2}}(c_a+c_{11})^3 \sqrt{n}.
\end{align*}
Further combine with \eqref{eq26:pf_debias}, we have
\begin{align*}
&\frac{\sl e^{c_P}}{2(1+e^{c_P})^2} \left(\|\Delta_{H}\|^2_F+\frac{1}{n^2}\|\Delta_{\Gamma}\|^2_F\right)-C_0 \left(\sqrt{\frac{\mu r \sigma_{max}}{n^2}}+c_a'+c_a\right)^2\\
 \lesssim&   C_2(c_a+c_{11})\sqrt{\frac{\mu r\sigma_{\max}}{n^2}}\sqrt{\|\Delta_{H}\|^2_F+\frac{1}{n^2}\|\Delta_{\Gamma}\|^2_F} +\sqrt{\frac{\su \mu r\sigma_{\max}}{n^2}}(c_a+c_{11})^3 \sqrt{n},
\end{align*}
which implies
\begin{align*}
\sqrt{\|\Delta_{H}\|^2_F+\frac{1}{n^2}\|\Delta_{\Gamma}\|^2_F}\lesssim \sqrt{\frac{2(1+e^{c_P})^2}{ \sl e^{c_P}}}\left(\frac{\su \mu r\sigma_{\max}}{n^2}\right)^{1/4}(c_a+c_{11})^{3/2}\cdot n^{1/4}
\end{align*}
as long as $n$ is large enough to make $C_0 \left(\sqrt{\frac{\mu r \sigma_{max}}{n^2}}+c_a'+c_a\right)^2$ and $C_2(c_a+c_{11})\sqrt{\frac{\mu r\sigma_{\max}}{n^2}}$ negligible terms compared to $\sqrt{n}$. Thus, we finish the proofs.
\end{proof}

%% file: pf_main_results.tex
\section{Proofs of Section \ref{sec:results}}\label{sec:pf_results}

\subsection{Proofs of Proposition \ref{prop:assumption_example}}\label{pf:prop_assumption_example}
In this example, we have
\begin{align*}
\Gamma^*=
\begin{bmatrix}
p\mathbf{1}\mathbf{1}^\top & q\mathbf{1}\mathbf{1}^\top\\
q\mathbf{1}\mathbf{1}^\top & p\mathbf{1}\mathbf{1}^\top\\
\end{bmatrix}
=\Pi
\begin{bmatrix}
p & q\\
q & p\\
\end{bmatrix}
\Pi^{T}
=\Pi
\begin{bmatrix}
1 & 1\\
1 & -1\\
\end{bmatrix}
\begin{bmatrix}
\frac{p+q}{2} & 0\\
0 & \frac{p-q}{2}\\
\end{bmatrix}
\begin{bmatrix}
1 & 1\\
1 & -1\\
\end{bmatrix}
\Pi^{T},
\end{align*}
where $\Pi\in\bbR^{n\times 2}$ with the first $n/2$ rows being $[1,0]$ and the last $n/2$ rows being $[0,1]$. It then holds that
\begin{align}\label{eq3:pf_example}
X^{*}=Y^{*}
=\Pi
\begin{bmatrix}
1 & 1\\
1 & -1\\
\end{bmatrix}
\begin{bmatrix}
\sqrt{\frac{p+q}{2}} & 0\\
0 & \sqrt{\frac{p-q}{2}}\\
\end{bmatrix}
=\begin{bmatrix}
\sqrt{\frac{p+q}{2}}\mathbf{1} & \sqrt{\frac{p-q}{2}} \mathbf{1}\\
\sqrt{\frac{p+q}{2}}\mathbf{1} & -\sqrt{\frac{p-q}{2}} \mathbf{1}
\end{bmatrix},
\end{align}
where $\mathbf{1}\in\bbR^{\frac{n}{2}\times 1}$.
We denote
\begin{align*}
&\cA_1:=\{(i,j)\mid 1\leq i,j\leq n/2, i\neq j\},\\
&\cA_2:=\{(i,j)\mid 1\leq i\leq n/2,n/2< j\leq n, i\neq j\},\\
&\cA_3:=\{(i,j)\mid n/2< i\leq n,1\leq j\leq n/2,i\neq j\},\\
&\cA_4:=\{(i,j)\mid n/2< i,j\leq n, i\neq j\}.
\end{align*}
Since $H^*=0$, we then have
\begin{align*}
D^{*}
&=
\sum_{(i,j)\in\cA_1}\frac{e^{\frac{p}{n}}}{(1+e^{\frac{p}{n}})^2}
\begin{bmatrix}
\frac1n e_i \left[\sqrt{\frac{p+q}{2}}, \sqrt{\frac{p-q}{2}}\right]\\[10pt]
\frac1n e_j \left[\sqrt{\frac{p+q}{2}}, \sqrt{\frac{p-q}{2}}\right]
\end{bmatrix}^{\otimes 2}
+\sum_{(i,j)\in\cA_2}\frac{e^{\frac{q}{n}}}{(1+e^{\frac{q}{n}})^2}
\begin{bmatrix}
\frac1n e_i \left[\sqrt{\frac{p+q}{2}}, -\sqrt{\frac{p-q}{2}}\right]\\[10pt]
\frac1n e_j \left[\sqrt{\frac{p+q}{2}}, \sqrt{\frac{p-q}{2}}\right]
\end{bmatrix}^{\otimes 2}\\
&\quad +\sum_{(i,j)\in\cA_3}\frac{e^{\frac{q}{n}}}{(1+e^{\frac{q}{n}})^2}
\begin{bmatrix}
\frac1n e_i \left[\sqrt{\frac{p+q}{2}}, \sqrt{\frac{p-q}{2}}\right]\\[10pt]
\frac1n e_j \left[\sqrt{\frac{p+q}{2}}, -\sqrt{\frac{p-q}{2}}\right]
\end{bmatrix}^{\otimes 2}
+\sum_{(i,j)\in\cA_4}\frac{e^{\frac{p}{n}}}{(1+e^{\frac{p}{n}})^2}
\begin{bmatrix}
\frac1n e_i \left[\sqrt{\frac{p+q}{2}}, -\sqrt{\frac{p-q}{2}}\right]\\[10pt]
\frac1n e_j \left[\sqrt{\frac{p+q}{2}}, -\sqrt{\frac{p-q}{2}}\right]
\end{bmatrix}^{\otimes 2}.
\end{align*}
In what follows, we view $D^*$ as a $4n\times 4n$ matrix and denote
\begin{align*}
D^*=
\begin{bmatrix}
D_{11}^* & D_{12}^* & D_{13}^* & D_{14}^*\\
D_{21}^* & D_{22}^* & D_{23}^* & D_{24}^*\\
D_{31}^* & D_{32}^* & D_{33}^* & D_{34}^*\\
D_{41}^* & D_{42}^* & D_{43}^* & D_{44}^*
\end{bmatrix},
\end{align*}
where $D_{ij}^*\in\bbR^{n\times n}$.
Recall that we have
\begin{align*}
D^*
\begin{bmatrix}
\Delta_X\\
\Delta_Y
\end{bmatrix}
=
\begin{bmatrix}
X^*\\
Y^*
\end{bmatrix},
\end{align*}
where we view $\Delta_X,\Delta_Y,X^*,Y^*$ as $2n\times 1$ vectors. Thus, we have
\begin{align}\label{eq1:pf_example}
\begin{bmatrix}
D_{11}^* & D_{12}^* \\
D_{21}^* & D_{22}^*    
\end{bmatrix}\Delta_X
+\begin{bmatrix}
D_{13}^* & D_{14}^*\\
D_{23}^* & D_{24}^*\\  
\end{bmatrix}\Delta_Y
=X^*.
\end{align}
Based on the specific form of $X^*$ and $Y^*$, it can be seen that the solutions must have the following form:\textcolor{red}{Why???}
\begin{align}\label{eq2:pf_example}
\Delta_X=\Delta_Y
=\vec 
\begin{bmatrix}
a & b\\
\vdots & \vdots\\
a & b\\
c & d\\
\vdots & \vdots\\
c & d\\
\end{bmatrix}.
\end{align}
The matrix on the RHS has its first $n$ rows being $[a,b]$ and last $n$ rows being $[c,d]$ for some $a,b,c,d$.
By \eqref{eq1:pf_example}, we have
\begin{align}\label{eq4:pf_example}
\begin{bmatrix}
D_{11}^* + D_{13}^* & D_{12}^* + D_{14}^* \\
D_{21}^* + D_{23}^* & D_{22}^* + D_{24}^*  
\end{bmatrix}\Delta_X
=X^*.
\end{align}
With a little abuse of notion, we denote $e_i,e_j\in\bbR^{\frac{n}{2}\times 1}$ the one-hot vector in the following. And we denote $s=\sqrt{\frac{p+q}{2}}, t=\sqrt{\frac{p-q}{2}}$. It then holds that
\begin{align}\label{eq5:pf_example}
&D_{11}^* + D_{13}^*
=\frac{1}{n^2}\sum_{\substack{
        i \neq j \\
        1\leq i, j \leq \frac{n}{2}
    }}\frac{e^{\frac{p}{n}}}{(1+e^{\frac{p}{n}})^2}
\begin{bmatrix}
e_i [s, t]
\end{bmatrix}^{\otimes 2}
+\frac{1}{n^2}\sum_{\substack{
        i \neq j \\
        1\leq i, j \leq \frac{n}{2}
    }}\frac{e^{\frac{p}{n}}}{(1+e^{\frac{p}{n}})^2}
\begin{bmatrix}
e_i [s, t]
\end{bmatrix}^{\otimes}
\begin{bmatrix}
e_j [s, t]
\end{bmatrix}\notag\\
&\qquad\qquad\qquad +
\frac{1}{n^2}\sum_{
        1\leq i, j \leq \frac{n}{2}}\frac{e^{\frac{q}{n}}}{(1+e^{\frac{q}{n}})^2}
\begin{bmatrix}
e_i [s, -t]
\end{bmatrix}^{\otimes 2}\notag\\
&D_{12}^* + D_{14}^*
=
\frac{1}{n^2}\sum_{
        1\leq i, j \leq \frac{n}{2}}\frac{e^{\frac{q}{n}}}{(1+e^{\frac{q}{n}})^2}
\begin{bmatrix}
e_i [s, -t]
\end{bmatrix}^{\otimes}
\begin{bmatrix}
e_j [s, t]
\end{bmatrix}\notag\\
&D_{21}^* + D_{23}^*
=
\frac{1}{n^2}\sum_{
        1\leq i, j \leq \frac{n}{2}}\frac{e^{\frac{q}{n}}}{(1+e^{\frac{q}{n}})^2}
\begin{bmatrix}
e_i [s, t]
\end{bmatrix}^{\otimes}
\begin{bmatrix}
e_j [s, -t]
\end{bmatrix}\notag\\
&D_{22}^* + D_{24}^*
=\frac{1}{n^2}\sum_{\substack{
        i \neq j \\
        1\leq i, j \leq \frac{n}{2}
    }}\frac{e^{\frac{p}{n}}}{(1+e^{\frac{p}{n}})^2}
\begin{bmatrix}
e_i [s, -t]
\end{bmatrix}^{\otimes 2}
+\frac{1}{n^2}\sum_{\substack{
        i \neq j \\
        1\leq i, j \leq \frac{n}{2}
    }}\frac{e^{\frac{p}{n}}}{(1+e^{\frac{p}{n}})^2}
\begin{bmatrix}
e_i [s, -t]
\end{bmatrix}^{\otimes}
\begin{bmatrix}
e_j [s, -t]
\end{bmatrix}\notag\\
&\qquad\qquad\qquad +
\frac{1}{n^2}\sum_{
        1\leq i, j \leq \frac{n}{2}}\frac{e^{\frac{q}{n}}}{(1+e^{\frac{q}{n}})^2}
\begin{bmatrix}
e_i [s, t]
\end{bmatrix}^{\otimes 2}.
\end{align}
Combine \eqref{eq3:pf_example}, \eqref{eq2:pf_example}, \eqref{eq4:pf_example}, \eqref{eq5:pf_example}, we have
\begin{align*}
&\mathbf{1}_{\frac{n}{2}}\cdot[s,t]
=(n-2)\frac{as+bt}{n^2}\frac{e^{\frac{p}{n}}}{(1+e^{\frac{p}{n}})^2}\cdot (\mathbf{1}_{\frac{n}{2}}\cdot[s,t])
+\frac{n}{2}\frac{as-bt}{n^2}\frac{e^{\frac{q}{n}}}{(1+e^{\frac{q}{n}})^2}\cdot (\mathbf{1}_{\frac{n}{2}}\cdot[s,-t])\\
&\qquad\qquad\qquad +\frac{n}{2}\frac{cs+dt}{n^2}\frac{e^{\frac{q}{n}}}{(1+e^{\frac{q}{n}})^2}\cdot (\mathbf{1}_{\frac{n}{2}}\cdot[s,-t])\\
&\mathbf{1}_{\frac{n}{2}}\cdot[s,-t]
=(n-2)\frac{cs-dt}{n^2}\frac{e^{\frac{p}{n}}}{(1+e^{\frac{p}{n}})^2}\cdot (\mathbf{1}_{\frac{n}{2}}\cdot[s,-t])
+\frac{n}{2}\frac{cs+dt}{n^2}\frac{e^{\frac{q}{n}}}{(1+e^{\frac{q}{n}})^2}\cdot (\mathbf{1}_{\frac{n}{2}}\cdot[s,t])\\
&\qquad\qquad\qquad +\frac{n}{2}\frac{as-bt}{n^2}\frac{e^{\frac{q}{n}}}{(1+e^{\frac{q}{n}})^2}\cdot (\mathbf{1}_{\frac{n}{2}}\cdot[s,t]),
\end{align*}
which implies
\begin{align*}
\begin{cases}
as-bt=-(cs+dt)\\
(n-2)\frac{as+bt}{n^2}\frac{e^{\frac{p}{n}}}{(1+e^{\frac{p}{n}})^2}=1\\
(n-2)\frac{cs-dt}{n^2}\frac{e^{\frac{p}{n}}}{(1+e^{\frac{p}{n}})^2}=1
\end{cases}.
\end{align*}
We denote
\begin{align*}
N:=\frac{n^2}{n-2}\frac{(1+e^{\frac{p}{n}})^2}{e^{\frac{p}{n}}}.
\end{align*}
It then holds that
\begin{align*}
a=c=\frac{N}{2s}, \quad b=-d=\frac{N}{2t}
\end{align*}
and 
\begin{align}\label{eq6:pf_example}
\Delta_X
=\begin{bmatrix}
\frac{N}{2s}\mathbf{1} & \frac{N}{2t}\mathbf{1}\\[10pt]
\frac{N}{2s}\mathbf{1} & -\frac{N}{2t}\mathbf{1}
\end{bmatrix}.
\end{align}
Combine with \eqref{eq3:pf_example}, we have
\begin{align*}
\Delta_X Y^{*\top}+X^{*}\Delta^{\top}_Y=\Delta_X X^{*\top}+X^{*}\Delta^{\top}_X
=\begin{bmatrix}
2N\mathbf{1}\mathbf{1}^{\top} & 0\\
0 & 2N\mathbf{1}\mathbf{1}^{\top}
\end{bmatrix}\in\bbR^{n\times n}.
\end{align*}
By the definition of $M^*$, we have
\begin{align*}
M^*=
\begin{bmatrix}
\frac{e^{\frac{p}{n}}}{(1+e^{\frac{p}{n}})^2}\mathbf{1}\mathbf{1}^{\top} & \frac{e^{\frac{q}{n}}}{(1+e^{\frac{q}{n}})^2}\mathbf{1}\mathbf{1}^{\top}\\
\frac{e^{\frac{q}{n}}}{(1+e^{\frac{q}{n}})^2}\mathbf{1}\mathbf{1}^{\top} & \frac{e^{\frac{p}{n}}}{(1+e^{\frac{p}{n}})^2}\mathbf{1}\mathbf{1}^{\top}
\end{bmatrix}
-
\begin{bmatrix}
\frac{e^{\frac{p}{n}}}{(1+e^{\frac{p}{n}})^2} & & \\
 & \ddots & \\
 & & \frac{e^{\frac{p}{n}}}{(1+e^{\frac{p}{n}})^2}
\end{bmatrix}.
\end{align*}
Thus, we have
\begin{align*}
\frac1n M^*\odot \left(\frac1n(\Delta_X Y^{*\top}+X^{*}\Delta^{\top}_Y) \right)
&=\begin{bmatrix}
\frac{2}{n-2}\mathbf{1}\mathbf{1}^{\top} & 0\\
0 & \frac{2}{n-2}\mathbf{1}\mathbf{1}^{\top}
\end{bmatrix}
-
\begin{bmatrix}
\frac{2}{n-2} & & \\
 & \ddots & \\
 & & \frac{2}{n-2}
\end{bmatrix}\\
&=\frac{2}{n-2}
\begin{bmatrix}
\mathbf{1}\mathbf{1}^{\top}-I_{\frac{n}{2}} & 0\\
0 & \mathbf{1}\mathbf{1}^{\top}-I_{\frac{n}{2}} 
\end{bmatrix}.
\end{align*}
Note that
\begin{align*}
  \begin{bmatrix}
\mathbf{1}\mathbf{1}^{\top}-I_{\frac{n}{2}} & 0\\
0 & \mathbf{1}\mathbf{1}^{\top}-I_{\frac{n}{2}} 
\end{bmatrix}  
\end{align*}
has $2$ singular values equal to $\frac{n}{2}-1$ and $n-2$ singular values equal to $1$. And in this example, we have $r=2$. Thus we verify
\begin{align*}
\sigma_{3}\left(\frac1n M^*\odot \left(\frac1n(\Delta_X Y^{*\top}+X^{*}\Delta^{\top}_Y) \right)\right)=\frac{2}{n-2}<\frac12.
\end{align*}
That is to say Assumption \ref{assumption:r+1} holds with $\eps = 1/2$. Thus, we finish the proofs.

\subsection{Bridge convex optimizer and approximate nonconvex optimizer}\label{pf:thm_cv_ncv}

In this section, we aim to prove the following theorem.
\begin{theorem}\label{thm:cv_ncv}
Suppose Assumption \ref{assumption:r+1} holds. We then have
\begin{align*}
\left\|
\begin{bmatrix}
\hat{H}_c-\hat{H}\\
\frac{1}{n}(\hat{\Gamma}_c-\hat{X}\hat{Y}^\top)
\end{bmatrix}
\right\|_F
\lesssim \frac{(1+e^{c c_P})^2}{\min\{\sl/3, 1/6\} e^{c c_P}}\left(1+\frac{ 72 \kappa n}{\sqrt{\sigma_{\min}}}\right)\|\cP(\nabla f(\hat{H},\hat{X},\hat{Y}))\|_F
\end{align*}
for some constant $c$.
\end{theorem}

\subsubsection{Useful claims and lemmas}
In this section, we establish several useful claims and lemmas that will support the proof of Theorem \ref{thm:cv_ncv}.
For notation simplicity, in this section, we denote the solution given by gradient descent as $(H,X,Y)$ instead of $(\hat{H},\hat{X},\hat{Y})$.

\begin{claim}\label{claim1:pf_cv_ncv}
Let $U\Sigma V^\top$ be the SVD of $XY^\top$. There exists an invertible matrix $Q\in\bbR^{r\times r}$ such that $X=U\Sigma^{1/2}Q$, $Y=V\Sigma^{1/2}Q^{-T}$ and
\begin{align}\label{ineq1:pf_cv_ncv}
\left\|\Sigma_Q-\Sigma_Q^{-1}\right\|_F\leq\frac{8 \sqrt{\kappa} }{\lambda \sqrt{\sigma_{\min}}} \|\cP(\nabla f(H,X,Y))\|_F.
\end{align}
Here $U_Q\Sigma_Q V_Q^\top$ is the SVD of Q. 
\end{claim}
\begin{proof}[Proof of Claim \ref{claim1:pf_cv_ncv}]
Let
\begin{align*}
B_1:=\cP_Z^{\perp}\nabla_{\Gamma}L_c(H,XY^{\top})Y + \lambda X \quad \text{and} \quad B_2:=\cP_Z^{\perp} \nabla_{\Gamma}L_c(H,XY^{\top})^{\top}X + \lambda Y.
\end{align*}
By the definition of $\nabla_X f(H,X,Y)$ and $\nabla_Y f(H,X,Y)$, we have
\begin{align*}
    \max\{\| {B}_1\|_F, \|{B}_2\|_F\} 
   & =\max\{\|\cP^{\perp}_Z \nabla_X f(H,X,Y)\|_F, \|\cP^{\perp}_Z \nabla_Y f(H,X,Y)\|_F\}\\
&\leq \|\cP(\nabla f(H,X,Y))\|_F.
\end{align*}
In addition, the definition of $B_1$ and $B_2$ allow us to obtain
\begin{align*}
    \|X^\top X - Y^\top Y\|_F &= \frac{1}{\lambda} \|X^\top ({B}_1 - \cP_Z^{\perp}\nabla_{\Gamma}L_c(H,XY^{\top}) Y) - ({B}_2 - \cP_Z^{\perp}\nabla_{\Gamma}L_c(H,XY^{\top})^{\top}X)^{\top}Y\|_F \\
    &= \frac{1}{\lambda} \|X^\top {B}_1 - {B}_2^\top Y\|_F \\
    &\leq \frac{1}{\lambda} \|X\| \|{B}_1\|_F + \frac{1}{\lambda} \|{B}_2\|_F \|Y\| \\
    &\leq \frac{4}{\lambda} \sqrt{ \sigma_{\max}} \|\cP(\nabla f(H,X,Y))\|_F.
\end{align*}
Here, the last inequality follows from the fact that $\|X\|, \|Y\| \leq 2\sqrt{\sigma_{\max}}$. In view of Lemma \ref{ar:lem6}, one can find an invertible ${Q}$ such that $X = U \Sigma^{1/2} {Q}$, $Y = V \Sigma^{1/2} {Q}^{-T}$ and
\begin{align*}
    \|\Sigma_{Q} - \Sigma_{Q}^{-1}\|_F &\leq \frac{1}{\sigma_{\min} (\Sigma)} \|X^\top X - Y^\top Y\|_F \\
    &\leq \frac{2}{\sigma_{\min}} \cdot \frac{4}{\lambda} \sqrt{ \sigma_{\max}} \|\cP(\nabla f(H,X,Y))\|_F\\
    &= \frac{8 \sqrt{\kappa} }{\lambda \sqrt{\sigma_{\min}}} \|\cP(\nabla f(H,X,Y))\|_F,
\end{align*}
where $\Sigma_Q$ is a diagonal matrix consisting of all singular values of ${Q}$. Here the second inequality follows from 
\begin{align*}
\left|\sigma_{\min}(XY^{\top})-\sigma_{\min}(X^* Y^{*\top})\right|\leq \left\|XY^{\top}-X^* Y^{*\top}\right\|\lesssim \sqrt{\sigma_{\max}}\|X-X^*\|\lesssim c_{11}\sqrt{\sigma_{\max} n},
\end{align*}
which implies $\sigma_{\min}(XY^{\top})\geq \sigma_{\min}/2$ as long as $c_{11}\sqrt{\sigma_{\max} n}\ll \sigma_{\min}$.
This completes the proof.   
\end{proof}

With Claim \ref{claim1:pf_cv_ncv} in hand, we are ready to establish the following claim.
\begin{claim}\label{claim2:pf_cv_ncv}
Suppose that Assumption \ref{assumption:r+1} holds. Under the same notations in Claim \ref{claim1:pf_cv_ncv}, let $\cP_Z^{\perp}\nabla_{\Gamma}L_c(H, XY^{\top})\cP_Z^{\perp}=-\lambda UV^{\top}+R$.
Then $R$ satisfies:
\begin{align*}
\|\cP_{\cT}(R)\|_F\leq \frac{ 72 \kappa}{\sqrt{\sigma_{\min}}} \|\cP(\nabla f(H,X,Y))\|_F\quad\text{and}\quad \|\cP_{\cT^{\perp}}(R)\|<\left(1-\frac{\eps}{4}\right)\lambda
\end{align*}
for $\eps>0$ specified in Assumption \ref{assumption:r+1}. Here $\cT$ is the tangent space for $U\Sigma V^{\top}$ defined as 
\begin{align*}
\cT:=\{UA^{\top}+BV^{\top}\mid A,B\in\bbR^{n\times r}\},
\end{align*}
$\cT^{\perp}$ is the orthogonal complement of $\cT$, and $\cP_{\cT}$, $\cP_{\cT^{\perp}}$ are the orthogonal projection onto the subspace $\cT$, $\cT^{\perp}$, respectively.
\end{claim}
\begin{proof}[Proof of Claim \ref{claim2:pf_cv_ncv}]
Recall the notations in the proof of Claim \ref{claim1:pf_cv_ncv}, we have
\begin{align}\label{eq:B}
B_1=\cP_Z^{\perp}\nabla_{\Gamma}L_c(H,XY^{\top})Y + \lambda X \quad \text{and} \quad B_2=\cP_Z^{\perp}\nabla_{\Gamma}L_c(H,XY^{\top})^{\top}X + \lambda Y.
\end{align}
By the definition of $R$, we have
\begin{align}\label{eq:R}
\cP_Z^{\perp}\nabla_{\Gamma}L_c(H, XY^{\top})\cP_Z^{\perp}=R-\lambda UV^{\top}.
\end{align}
From the definition of $\mathcal{P}_{\cT}$, we have
\begin{align}\label{eq4:pf_cv_ncv}
    \|\mathcal{P}_{\cT}({R})\|_F &= \|UU^\top {R} (I - VV^\top) + RVV^\top\|_F \notag\\
    &\leq \|U^\top {R} (I - VV^\top)\|_F + \|RVV^\top\|_F \notag\\
    &\leq \|U^\top {R}\|_F + \|RV\|_F.
\end{align}
In addition, by Claim \ref{claim1:pf_cv_ncv}, we have
\begin{align}\label{eq2:pf_cv_ncv}
    X &= U \Sigma^{1/2} {Q} \quad \text{and} \quad Y = V \Sigma^{1/2} {Q}^{-T} 
\end{align}
for some invertible matrix ${Q} \in \mathbb{R}^{r \times r}$, whose SVD $U_Q \Sigma_Q V_Q^\top$ obeys \eqref{ineq1:pf_cv_ncv}. Combine \eqref{eq:B} and \eqref{eq:R}, we have
\begin{align*}
    -\lambda UV^\top Y + RY &= -\lambda X + {B}_1,
\end{align*}
which together with \eqref{eq2:pf_cv_ncv} yields
\begin{align*}
    RV &= \lambda U \Sigma^{1/2} (I_r - {Q} {Q}^\top) \Sigma^{-1/2} + {B}_1 {Q}^\top \Sigma^{-1/2}.
\end{align*}
Apply the triangle inequality to get
\begin{align}\label{eq3:pf_cv_ncv}
    \|RV\|_F &\leq \|\lambda U \Sigma^{1/2} (I_r - {Q} {Q}^\top) \Sigma^{-1/2}\|_F + \|{B}_1 {Q}^\top \Sigma^{-1/2}\|_F \notag\\
    &\leq \lambda \|\Sigma^{1/2}\| \|\Sigma^{-1/2}\| \| {Q} {Q}^\top - I_r \|_F + \|{Q}\| \|\Sigma^{-1/2}\| \|{B}_1\|_F.
\end{align}
In order to further upper bound \eqref{eq3:pf_cv_ncv}, we first recognize the fact that as long as $c_{11}\sqrt{\sigma_{\max} n}\ll \sigma_{\min}$,
\begin{align*}
    \|\Sigma^{1/2}\| &\leq \sqrt{2 \sigma_{\max}}, \quad \text{and} \quad \|\Sigma^{-1/2}\| = 1 / \sqrt{\sigma_{\min} (\Sigma)} \leq \sqrt{2 / \sigma_{\min}},
\end{align*}
which holds since
\begin{align*}
\left|\sigma_{i}(XY^{\top})-\sigma_{i}(X^* Y^{*\top})\right|\leq \left\|XY^{\top}-X^* Y^{*\top}\right\|\lesssim \sqrt{\sigma_{\max}}\|X-X^*\|\lesssim c_{11}\sqrt{\sigma_{\max} n}.
\end{align*}
Second, Claim \ref{claim1:pf_cv_ncv} and \eqref{ncv:grad} yields
\begin{align*}
    \|\Sigma_{Q} - \Sigma_{Q}^{-1}\|_F &\leq \frac{8 \sqrt{\kappa} }{\lambda \sqrt{\sigma_{\min}}} \|\cP(\nabla f(H,X,Y))\|_F \ll 1 .
\end{align*}
This in turn implies that $\|{Q}\| = \|\Sigma_Q\| \leq 2$. Putting the above bounds together yields
\begin{align*}
    \|RV\|_F &\leq \lambda \sqrt{2 \sigma_{\max}} \sqrt{\frac{2}{\sigma_{\min}}} \|\Sigma_Q^2 - I_r\|_F + 2 \sqrt{\frac{2}{\sigma_{\min}}} \|\cP(\nabla f(H,X,Y))\|_F \\
    &\leq \lambda \sqrt{2 \sigma_{\max}} \sqrt{\frac{2}{\sigma_{\min}}} \|\Sigma_Q \| \|\Sigma_Q - \Sigma_Q^{-1}\|_F + 2 \sqrt{\frac{2}{\sigma_{\min}}}\|\cP(\nabla f(H,X,Y))\|_F \\
    &\leq 2 \lambda \sqrt{2 \sigma_{\max}} \sqrt{\frac{2}{\sigma_{\min}}}\frac{8 \sqrt{\kappa} }{\lambda \sqrt{\sigma_{\min}}} \|\cP(\nabla f(H,X,Y))\|_F + 2 \sqrt{\frac{2}{\sigma_{\min}}} \|\cP(\nabla f(H,X,Y))\|_F  \\
    &\leq \frac{ 36 \kappa}{\sqrt{\sigma_{\min}}} \|\cP(\nabla f(H,X,Y))\|_F .
\end{align*}
Similarly, we can show that
\begin{align*}
    \|U^\top {R}\|_F &\leq \frac{ 36 \kappa}{\sqrt{\sigma_{\min}}} \|\cP(\nabla f(H,X,Y))\|_F.
\end{align*}
These bounds together with \eqref{eq4:pf_cv_ncv} result in
\begin{align}\label{eq6:pf_cv_ncv}
    \|\mathcal{P}_{\cT}({R})\|_F &\leq \frac{ 72 \kappa}{\sqrt{\sigma_{\min}}} \|\cP(\nabla f(H,X,Y))\|_F.
\end{align}

In the following, we establish the bound for $\|\cP_{\cT^{\perp}}(R)\|$. Note that
\begin{align*}
\cP_Z^{\perp}\nabla_{\Gamma}L_c(H,XY^{\top})\cP_Z^{\perp}=-\lambda UV^{\top}+R=-\lambda UV^{\top}+\mathcal{P}_{\cT}({R})+\mathcal{P}_{\cT^{\perp}}({R}).
\end{align*}
Suppose for the moment that
\begin{align}\label{eq5:pf_cv_ncv}
  \|\mathcal{P}_{\cT}({R})\|_F \leq\frac{\eps}{4}\lambda \quad\text{and} \quad \sigma_{r+1}\left(\cP_Z^{\perp}\nabla_{\Gamma}L_c(H,XY^{\top})\cP_Z^{\perp}\right)<\left(1-\frac{\eps}{2}\right)\lambda.
\end{align}
Then, by Weyl's inequality, we have
\begin{align*}
\sigma_{r+1}\left(-\lambda UV^{\top}+\mathcal{P}_{\cT^{\perp}}({R})\right)
&\leq\sigma_{r+1}\left( -\lambda UV^{\top}+\mathcal{P}_{\cT}({R})+\mathcal{P}_{\cT^{\perp}}({R})\right)+\|\mathcal{P}_{\cT}({R})\|_F\\
&=\sigma_{r+1}\left(\cP_Z^{\perp}\nabla_{\Gamma}L_c(H,XY^{\top})\cP_Z^{\perp}\right)+\|\mathcal{P}_{\cT}({R})\|_F\\
&< \left(1-\frac{\eps}{2}\right)\lambda+\frac{\eps}{4}\lambda\\
&=\left(1-\frac{\eps}{4}\right)\lambda.
\end{align*}
Since $-\lambda UV^\top\in\cT$ with $r$ singular values of $\lambda$, we conclude that 
\begin{align*}
\|\cP_{\cT^{\perp}}(R)\|<\left(1-\frac{\eps}{4}\right)\lambda.
\end{align*}
Thus, it remains to verify the two conditions listed in \eqref{eq5:pf_cv_ncv}. For the first condition, by \eqref{eq6:pf_cv_ncv}, we have
\begin{align*}
\|\mathcal{P}_{\cT}({R})\|_F \leq\frac{\eps}{4}\lambda\quad\text{as long as} \quad 
\|\cP(\nabla f(H,X,Y))\|_F\leq \frac{\eps\lambda}{288\kappa}\sqrt{\sigma_{\min}},
\end{align*}
which is guaranteed by \eqref{ncv:grad}. 
We then verify the second condition in the following. For notation simplicity, we denote
\begin{align*}
\ell_{ij}(x):=\log (1+e^x)-A_{ij}x.
\end{align*}
Then we have
\begin{align*}
&\left(\nabla_{\Gamma} L_c(H,XY^{\top})-\nabla_{\Gamma} L_c(H^{*},{\Gamma}^{{*}})\right)_{ij}\\
&=
\begin{cases}
\int^1_{0}\ell''_{ij}\left(P^{*}_{ij}+\tau(P_{ij}-P^{*}_{ij})\right)d\tau\cdot \frac{P_{ij}-P^{*}_{ij}}{n} & i\neq j\\
0 & i=j
\end{cases}.
\end{align*}
We define a matrix $H\in\bbR^{n\times n}$ such that
\begin{align*}
H_{ij}=
\begin{cases}
\int^1_{0}\ell''_{ij}\left(P^{*}_{ij}+\tau(P_{ij}-P^{*}_{ij})\right)d\tau & i\neq j\\
0 & i=j
\end{cases}.
\end{align*}
It then holds that
\begin{align}\label{eq14:pr_cv_ncv}
&\nabla_{\Gamma} L_c(H,XY^{\top})\notag\\
&=\nabla_{\Gamma} L_c(H^{*},{\Gamma}^{{*}})+\frac1n H\odot(P-P^*)\notag\\
&=\nabla_{\Gamma} L_c(H^{*},{\Gamma}^{{*}})+\frac1n (H-M^*)\odot(P-P^*)+\frac1n M^*\odot(P-P^*).
\end{align}
Recall the definition of $M^*$ where 
\begin{align*}
M^{*}_{ij}=
\begin{cases}
    \frac{e^{P_{ij}^*}}{(1+e^{P_{ij}^*})^2} & i\neq j\\
    0 & i=j
\end{cases}
=\begin{cases}
    \ell''_{ij}(P^{*}_{ij}) & i\neq j\\
    0 & i=j
\end{cases}.
\end{align*}
We then have for $i\neq j$:
\begin{align*}
|M^{*}_{ij}-H_{ij}|\leq \int^1_{0}\left|\ell''_{ij}\left(P^{*}_{ij}+\tau(P_{ij}-P^{*}_{ij})\right)- \ell''_{ij}(P^{*}_{ij})\right|d\tau\leq \frac14 |P_{ij}-P^{*}_{ij}|,
\end{align*}
where the last inequality follows from the mean-value theorem and the fact that $|\ell'''_{ij}(x)|\leq \ell''_{ij}(x)\leq 1/4$ for any $x$.
This further implies that
\begin{align*}
\left\|\frac1n (H-M^*)\odot(P-P^*)\right\|&\leq\left\|\frac1n (H-M^*)\odot(P-P^*)\right\|_F\notag\\
&\leq \frac1n \cdot n\|P-P^{*}\|_{\infty}^2\notag\\
&\leq \left(\frac{\sz}{n}\|H-H^*\|_F+\frac1n \|XY^{\top}-\Gamma^{*}\|_{\infty}\right)^2\\
&\lesssim \left(\frac{\sz}{n}\|H-H^*\|_F+\frac1n\|F^*\|_{2,\infty} \|F-F^*\|_{2,\infty}\right)^2\\
&\lesssim \frac1n \left(\sz^2 c_{11}^2+\frac{\mu r \sigma_{\max}}{n^2} c_{41}^2\right),
\end{align*}
where the last inequality follows from \eqref{ncv:est} and Assumption \ref{assumption:incoherent}.
Thus, as long as $n$ large enough, we have
\begin{align}\label{eq12:pf_cv_ncv}
\left\|\frac1n (H-M^*)\odot(P-P^*)\right\|
\lesssim \frac1n \left(\sz^2 c_{11}^2+\frac{\mu r \sigma_{\max}}{n^2} c_{41}^2\right)<\frac{\eps\lambda}{10}.
\end{align}
Moreover, as shown in the proofs of Lemma \ref{lem:gradient_norm}, we have
\begin{align}\label{eq13:pf_cv_ncv}
\left\|\nabla_{\Gamma}L_c(H^*,\Gamma^*)\right\|
\lesssim\sqrt{\frac{\log n}{n}}<\frac{\eps\lambda}{10}
\end{align}
as long as $\eps \lambda >\sqrt{\frac{\log n}{n}}$.
Thus, it remains to deal with $\frac1n M^*\odot(P-P^*)$.
Note that
\begin{align}\label{eq18:pf_cv_ncv}
\frac1n M^*\odot(P-P^*)=
\underbrace{\frac1n M^*\odot(\hat{P}^d-\bar{P})}_{(1)}
+\underbrace{\frac1n M^*\odot(\bar{P}-P^*)}_{(2)}
+\underbrace{\frac1n M^*\odot(P-\hat{P}^d)}_{(3)}.
\end{align}
In what follows, we will deal with (1),(2) and (3), respectively.
\begin{enumerate}
\item For (1), note that
\begin{align*}
\|(1)\|
&\leq \left\|\frac1n M^*\odot(\hat{P}^d-\bar{P})\right\|_F\\
&\leq \frac{1}{4n}\sqrt{\sum_{i\neq j}(\hat{P}^d_{ij}-\bar{P}_{ij})^2}\\
&\lesssim \frac{1}{n}\left(\sqrt{\su}\|\hat{H}^d-\bar{H}\|_{F}+\frac{1}{n}\|\hat{X}^d\hat{Y}^{d\top}-\bar{X}\bar{Y}^{\top}\|_{F}\right),
\end{align*}
where the last inequality follows the same trick we used before and we omit the details here. By Theorem \ref{thm:ncv_debias}, we obtain
\begin{align}\label{eq14:pf_cv_ncv}
\|(1)\|\lesssim \frac{\sqrt{\su}c_d}{n^{3/4}}<\frac{\eps\lambda}{10}
\end{align}
as long as $n$ is large enough. 

\item 
For (2), note that
\begin{align*}
\|(2)\|
&\leq \left\|\frac1n M^*\odot(\bar{P}-P^*)\right\|_F\\
&\leq \frac{1}{4n}\sqrt{\sum_{i\neq j}(\bar{P}_{ij}-P^*_{ij})^2}\\
&\lesssim \frac{1}{n}\left(\sqrt{\su}\|\bar{H}-H^*\|_{F}+\frac{1}{n}\|\bar{X}\bar{Y}^{\top}-X^{*}Y^{*\top}\|_{F}\right).
\end{align*}
By Proposition \ref{prop:bar_dist}, we have
\begin{align}\label{eq15:pf_cv_ncv}
\|(2)\|\lesssim c'_a\left(\sqrt{\frac{\su}{n}}+\frac{\sqrt{\sigma_{\max}}}{n^{3/2}}\right)<\frac{\eps\lambda}{10}
\end{align}
as long as $\eps \lambda> c'_a\left(\sqrt{\frac{\su}{n}}+\frac{\sqrt{\sigma_{\max}}}{n^{3/2}}\right)$.

\item 
We denote $\hat\Delta$ such that
\begin{align*}
\begin{bmatrix}
\hat{\Delta}_{H}\\
\hat{\Delta}_{X}\\
\hat{\Delta}_{Y}
\end{bmatrix}
=
\begin{bmatrix}
\hat{H}^d-H\\
\hat{X}^d-X\\
\hat{Y}^d-Y
\end{bmatrix}.
\end{align*}
With a little abuse of notation, we denote $(\Delta_{H}, \Delta_{X}, \Delta_{Y})$ such that
\begin{align*}
\begin{bmatrix}
\Delta_{H}\\
\Delta_{X}\\
\Delta_{Y}
\end{bmatrix} =(\cP D^*\cP)^{\dagger}\begin{bmatrix}
0\\
\lambda X^{*}\\
\lambda Y^{*}
\end{bmatrix},
\end{align*}
Then by Assumption \ref{assumption:r+1}, we have
\begin{align*}
    \sigma_{r+1}\left(\cP_Z^{\perp}\left(\frac{1}{n}{M}^* \odot \left({z}_i^\top\Delta_{H}{z}_j+\frac{(\Delta_{X}{Y}^{*T}+{X}^*{\Delta^\top_{Y}})_{ij}}{n}\right)_{ij}\right)\cP_Z^{\perp}\right) <(1-\epsilon)\lambda
\end{align*}
for some $\eps>0$.
We further have the following claim.
\begin{claim}\label{claim1:delta_diff}
It holds that
\begin{align*}
 \left\|\begin{bmatrix}
\hat{\Delta}_{H}-{\Delta}_{H}\\
\hat{\Delta}_{X}-{\Delta}_{X}\\
\hat{\Delta}_{Y}-{\Delta}_{Y}
\end{bmatrix}
\right\|_F   \leq c_{\text{diff}},
\end{align*}
where 
\begin{align}
c_{\text{diff}}\asymp \lambda\left(\frac{\hat{c}}{\slD^2}\sqrt{\frac{\mu r\sigma_{\max}}{n}}+\frac{c_{11}\sqrt{n}}{\slD}\right).
\end{align}
\end{claim}
\begin{proof}[Proof of Claim \ref{claim1:delta_diff}]
By the definition of $(\Delta_{H}, \Delta_{X}, \Delta_{Y})$, we have
\begin{align*}
\begin{bmatrix}
\Delta_{H}\\
\Delta_{X}\\
\Delta_{Y}
\end{bmatrix} =-(\cP D^*\cP)^{\dagger}\begin{bmatrix}
0\\
-\lambda X^{*}\\
-\lambda Y^{*}
\end{bmatrix}.
\end{align*}
By the definition of the debiased estimator, we have
\begin{align*}
\begin{bmatrix}
\hat{\Delta}_{H}\\
\hat{\Delta}_{X}\\
\hat{\Delta}_{Y}
\end{bmatrix} 
=-(\cP\hat{D}\cP)^{\dagger}\left(\cP\nabla L(\hat{H},\hat{X},\hat{Y})\right),
\end{align*}
For notation simplicity, we denote
\begin{align*}
& A_1:=(\cP D^*\cP)^{\dagger},\ 
b_1:=\begin{bmatrix}
0\\
-\lambda X^{*}\\
-\lambda Y^{*}
\end{bmatrix}\\
\text{and }
& A_2:=(\cP\hat{D}\cP)^{\dagger},\ 
b_2:=\cP\nabla L(\hat{H},\hat{X},\hat{Y}).
\end{align*}
It then holds that
\begin{align}\label{eq8:delta_diff}
\left\|\begin{bmatrix}
\hat{\Delta}_{H}-{\Delta}_{H}\\
\hat{\Delta}_{X}-{\Delta}_{X}\\
\hat{\Delta}_{Y}-{\Delta}_{Y}
\end{bmatrix}
\right\|_F
&=\|A_1b_1-A_2b_2\|_F\notag\\
&\leq \|A_1-A_2\|\|b_2\|_F+\|A_1\|\|b_1-b_2\|_F.
\end{align}
In the following, we control $ \|A_1-A_2\|$, $\|b_1-b_2\|_F$, $\|A_1\|$, and $\|b_2\|_F$, respectively.
\begin{enumerate}
    
\item For $ \|A_1-A_2\|$, by Theorem 3.3 in \cite{stewart1977perturbation}, we have
\begin{align}\label{eq4:delta_diff}
\|A_1-A_2\|
&\leq \frac{1+\sqrt{5}}{2}\max\{\|(\cP\hat{D}\cP)^{\dagger}\|^2, \|(\cP D^*\cP)^{\dagger}\|^2\}\cdot \|\cP\hat{D}\cP-\cP D^*\cP\|.
\end{align}
By Lemma \ref{lem:D}, we have
\begin{align*}
\max\{\|(\cP\hat{D}\cP)^{\dagger}\|^2, \|(\cP D^*\cP)^{\dagger}\|^2\}\lesssim \frac{1}{\slD^2}, \quad
\|\cP\hat{D}\cP-\cP D^*\cP\|\lesssim  \frac{\hat{c}}{\sqrt{n}}.
\end{align*}
Finally, by \eqref{eq4:delta_diff}, we obtain
\begin{align}\label{eqA:delta_diff}
\|A_1-A_2\|\lesssim \frac{\hat{c}}{\slD^2\sqrt{n}}.
\end{align}

\item For $\|b_1-b_2\|_F$, we have
\begin{align}\label{eq5:delta_diff}
\|b_1-b_2\|_F
&=\left\|\begin{bmatrix}
0\\
-\lambda X^{*}\\
-\lambda Y^{*}
\end{bmatrix}-\cP\nabla L(\hat{H},\hat{X},\hat{Y})\right\|_F\notag\\
&\leq \left\|\begin{bmatrix}
0\\
\lambda X^{*}\\
\lambda Y^{*}
\end{bmatrix}-
\begin{bmatrix}
0\\
\lambda \hat{X}\\
\lambda \hat{Y}
\end{bmatrix}\right\|_F
+\left\|\cP\nabla L(\hat{H},\hat{X},\hat{Y})-\begin{bmatrix}
0\\
-\lambda \hat{X}\\
-\lambda \hat{Y}
\end{bmatrix}\right\|_F
\end{align}
For the first term, we have
\begin{align}\label{eq6:delta_diff}
 \left\|\begin{bmatrix}
0\\
\lambda X^{*}\\
\lambda Y^{*}
\end{bmatrix}-
\begin{bmatrix}
0\\
\lambda \hat{X}\\
\lambda \hat{Y}
\end{bmatrix}\right\|_F
\leq \lambda (\|X^*-\hat{X}\|_F+(\|Y^*-\hat{Y}\|_F)\leq 2 \lambda c_{11}\sqrt{n},
\end{align}
where the last inequality follows from Lemma \ref{lem:ncv1}.
For the second term, we have
\begin{align}\label{eq7:delta_diff}
\left\|\cP\nabla L(\hat{H},\hat{X},\hat{Y})-\begin{bmatrix}
0\\
-\lambda \hat{X}\\
-\lambda \hat{Y}
\end{bmatrix}\right\|_F  
=\|\cP\nabla f(\hat{H},\hat{X},\hat{Y})\|_F
\lesssim n^{-5}.
\end{align}
Combine \eqref{eq5:delta_diff}, \eqref{eq6:delta_diff} and \eqref{eq7:delta_diff}, we have
\begin{align*}
\|b_1-b_2\|_F\lesssim \lambda c_{11}\sqrt{n}.
\end{align*}

\item For $\|A_1\|$, by Lemma \ref{lem:D}, we have $\|A_1\|\leq 1/\slD$.

\item For $\|b_2\|_F$, we have
\begin{align*}
\|b_2\|_F
&\leq \|b_1-b_2\|_F+\|b_1\|_F\\
&\leq \|b_1-b_2\|_F+\lambda(\|X^*\|_F+\|Y^*\|_F)\\
&\lesssim \lambda \sqrt{\mu r\sigma_{\max}}
\end{align*}
as long as $ \sqrt{\mu r\sigma_{\max}}\gg c_{11}\sqrt{n}$.

\end{enumerate}

Combine the above results with \eqref{eq8:delta_diff}, we have
\begin{align*}
\left\|\begin{bmatrix}
\hat{\Delta}_{H}-{\Delta}_{H}\\
\hat{\Delta}_{X}-{\Delta}_{X}\\
\hat{\Delta}_{Y}-{\Delta}_{Y}
\end{bmatrix}
\right\|_F
&\lesssim \frac{\hat{c}}{\slD^2\sqrt{n}}\cdot \lambda \sqrt{\mu r\sigma_{\max}}+\frac{1}{\slD}\cdot\lambda c_{11}\sqrt{n}\\
&=\lambda\left(\frac{\hat{c}}{\slD^2}\sqrt{\frac{\mu r\sigma_{\max}}{n}}+\frac{c_{11}\sqrt{n}}{\slD}\right).
\end{align*}
We then finish the proofs.

\end{proof}

By Weyl's inequality, we have
\begin{align*}
&\sigma_{r+1}\left(\cP_Z^{\perp}\left(\frac{1}{n}{M}^* \odot \left((\hat{P}^d-P)-\frac1n(\hat{X}^d-X)(\hat{Y}^d-Y)^{\top}\right)\right)\cP_Z^{\perp}\right)\\
&=\sigma_{r+1}\left(\cP_Z^{\perp}\left(\frac{1}{n}{M}^* \odot \left({z}_i^\top\hat{\Delta}_{H}{z}_j+\frac{(\hat{\Delta}_{X}{Y}^{T}+{X}{\hat{\Delta}^\top_{Y}})_{ij}}{n}\right)_{ij}\right)\cP_Z^{\perp}\right)\\
&\leq \sigma_{r+1}\left(\cP_Z^{\perp}\left(\frac{1}{n}{M}^* \odot \left({z}_i^\top\Delta_{H}{z}_j+\frac{(\Delta_{X}{Y}^{*T}+{X}^*{\Delta^\top_{Y}})_{ij}}{n}\right)_{ij}\right)\cP_Z^{\perp}\right)\\
&\quad +\frac1n \left\|{M}^* \odot\left({z}_i^\top(\hat{\Delta}_{H}-{\Delta}_{H}){z}_j\right)_{ij} \right\|_F\\
&\quad + \frac{1}{n^2} \left\|{M}^* \odot \left((\hat{\Delta}_{X}{Y}^{T}+{X}{\hat{\Delta}^\top_{Y}})-(\Delta_{X}{Y}^{*T}+{X}^*{\Delta^\top_{Y}})\right) \right\|_F\\
&\leq \sigma_{r+1}\left(\cP_Z^{\perp}\left(\frac{1}{n}{M}^* \odot \left({z}_i^\top\Delta_{H}{z}_j+\frac{(\Delta_{X}{Y}^{*T}+{X}^*{\Delta^\top_{Y}})_{ij}}{n}\right)_{ij}\right)\cP_Z^{\perp}\right)\\
&\quad +\frac{1}{4n}\left\|\left({z}_i^\top(\hat{\Delta}_{H}-{\Delta}_{H}){z}_j\right)_{ij} \right\|_F\\
&\quad + \frac{1}{4n^2} \left\|(\hat{\Delta}_{X}{Y}^{T}+{X}{\hat{\Delta}^\top_{Y}})-(\Delta_{X}{Y}^{*T}+{X}^*{\Delta^\top_{Y}})\right\|_F.
\end{align*}
By Claim \ref{claim1:delta_diff}, we can bound the Frobenius norms on the RHS respectively.
\begin{enumerate}

      \item For the first Frobenius norm, we have
      \begin{align*}
      \left\|\left({z}_i^\top(\hat{\Delta}_{H}-{\Delta}_{H}){z}_j\right)_{ij} \right\|_F
      &=\sqrt{\sum_{i,j}\left({z}_i^\top(\hat{\Delta}_{H}-{\Delta}_{H}){z}_j\right)^2}    \\
      &\leq \|\hat{\Delta}_{H}-{\Delta}_{H}\|\sqrt{\sum_{i,j}\|z_i\|_2^2\|z_j\|_2^2} \\
      &\leq \sz \|\hat{\Delta}_{H}-{\Delta}_{H}\| \\
      &\leq \sz c_{\text{diff}}.
      \end{align*}

      \item For the second Frobenius norm, we have
      \begin{align*}
        &\left\|(\hat{\Delta}_{X}{Y}^{T}+{X}{\hat{\Delta}^\top_{Y}})-(\Delta_{X}{Y}^{*T}+{X}^*{\Delta^\top_{Y}})\right\|_F\\
        &\leq \|\hat{\Delta}_{X}{Y}^{T}-\Delta_{X}{Y}^{*T}\|_F
        + \|{X}{\hat{\Delta}^\top_{Y}}-{X}^*{\Delta^\top_{Y}}\|_F\\
        &\leq \|\hat{\Delta}_{X}\|_F \|Y-Y^*\|_F+ \|Y^*\|_F \|\hat{\Delta}_{X}-{\Delta}_{X}\|_F+\|\hat{\Delta}_{Y}\|_F \|X-X^*\|_F+ \|X^*\|_F \|\hat{\Delta}_{Y}-{\Delta}_{Y}\|_F\\
        &\lesssim c_a c_{11} n+\sqrt{\mu r \sigma_{\max}}c_{\text{diff}},
      \end{align*}
where the last inequality follows from Lemma \ref{lem:ncv1}, Proposition \ref{prop:debias_dist}, Claim \ref{claim1:delta_diff}, and Assumption \ref{assumption:incoherent}.

\end{enumerate}

Combine the above bounds, we have
\begin{align*}
&\sigma_{r+1}\left(\cP_Z^{\perp}\left(\frac{1}{n}{M}^* \odot \left((\hat{P}^d-P)-\frac1n(\hat{X}^d-X)(\hat{Y}^d-Y)^{\top}\right)\right)\cP_Z^{\perp}\right)\\
&\leq \sigma_{r+1}\left(\cP_Z^{\perp}\left(\frac{1}{n}{M}^* \odot \left({z}_i^\top\Delta_{H}{z}_j+\frac{(\Delta_{X}{Y}^{*T}+{X}^*{\Delta^\top_{Y}})_{ij}}{n}\right)_{ij}\right)\cP_Z^{\perp}\right)\\
&\quad+\frac{c}{n}\left( c_a c_{11}+\frac{\sqrt{\mu r \sigma_{\max}}}{n}c_{\text{diff}}+\sz c_{\text{diff}}\right)\\
&< (1-\eps)\lambda +\frac{\eps\lambda}{20}\\
&=\left(1-\frac{19\eps}{20}\right)\lambda
\end{align*}
as long as $ c_a c_{11}+\frac{\sqrt{\mu r \sigma_{\max}}}{n}c_{\text{diff}}+\sz c_{\text{diff}}\ll \eps\lambda n$ ($n$ is large enough).
By Weyl's inequality, we have
\begin{align}
&\sigma_{r+1}\left(\cP_Z^{\perp}\left(\frac{1}{n}{M}^* \odot (\hat{P}^d-P)\right)\cP_Z^{\perp}\right)\notag\\
&\leq \sigma_{r+1}\left(\cP_Z^{\perp}\left(\frac{1}{n}{M}^* \odot \left((\hat{P}^d-P)-\frac1n(\hat{X}^d-X)(\hat{Y}^d-Y)^{\top}\right)\right)\cP_Z^{\perp}\right)+\frac{1}{n^2}\left\|{M}^* \odot(\hat{X}^d-X)(\hat{Y}^d-Y)^{\top}\right\|_F\notag\\
&\leq \sigma_{r+1}\left(\cP_Z^{\perp}\left(\frac{1}{n}{M}^* \odot \left((\hat{P}^d-P)-\frac1n(\hat{X}^d-X)(\hat{Y}^d-Y)^{\top}\right)\right)\cP_Z^{\perp}\right)+\frac{1}{4n^2}\|\hat{X}^d-X\|_F\|\hat{Y}^d-Y\|_F\notag\\
&\leq \sigma_{r+1}\left(\cP_Z^{\perp}\left(\frac{1}{n}{M}^* \odot \left((\hat{P}^d-P)-\frac1n(\hat{X}^d-X)(\hat{Y}^d-Y)^{\top}\right)\right)\cP_Z^{\perp}\right)+\frac{c_a^2}{4n}\tag{by Proposition \ref{prop:debias_dist}}\\
&<\left(1-\frac{19\eps}{20}\right)\lambda+\frac{\eps\lambda}{20}\notag\\\label{eq16:pf_cv_ncv}
&=\left(1-\frac{9\eps}{10}\right)\lambda
\end{align}
as long as $c_a^2\ll \eps\lambda n$ ($n$ is large enough).
\end{enumerate}

Combine \eqref{eq18:pf_cv_ncv}, \eqref{eq14:pf_cv_ncv}, \eqref{eq15:pf_cv_ncv}, \eqref{eq16:pf_cv_ncv} and apply Weyl's inequality, we have
\begin{align}\label{eq19:pf_cv_ncv}
\sigma_{r+1}\left(\cP_Z^{\perp}\left(\frac{1}{n}{M}^* \odot (P-P^*)\right)\cP_Z^{\perp}\right)
&\leq \sigma_{r+1}\left(\cP_Z^{\perp}\left(\frac{1}{n}{M}^* \odot (P-\hat{P}^d)\right)\cP_Z^{\perp}\right)+\|(1)\|+\|(2)\|\notag\\    
&<\left(1-\frac{9\eps}{10}\right)\lambda+\frac{\eps\lambda}{10}+\frac{\eps\lambda}{10}\notag\\
&=\left(1-\frac{7\eps}{10}\right)\lambda.
\end{align}

Finally, combine \eqref{eq14:pr_cv_ncv}, \eqref{eq12:pf_cv_ncv}, \eqref{eq13:pf_cv_ncv}, \eqref{eq19:pf_cv_ncv} and apply Weyl's inequality, we have
\begin{align*}
&\sigma_{r+1}\left(\cP_Z^{\perp}\left(\nabla_{\Gamma} L_c(H,XY^{\top})\right)\cP_Z^{\perp}\right)\notag\\
&\leq \sigma_{r+1}\left(\cP_Z^{\perp}\left(\frac{1}{n}{M}^* \odot (P-P^*)\right)\cP_Z^{\perp}\right)+\left\|\frac1n (H-M^*)\odot(P-P^*)\right\|+\left\|\nabla_{\Gamma}L_c(H^*,\Gamma^*)\right\|\notag\\
&<\left(1-\frac{7\eps}{10}\right)\lambda+\frac{\eps\lambda}{10}+\frac{\eps\lambda}{10}\notag\\
&=\left(1-\frac{\eps}{2}\right)\lambda.
\end{align*}
Thus, we finish the proof of Claim \ref{claim2:pf_cv_ncv}.

\end{proof}

\begin{lemma}\label{lem:cvx_lb}
Given any $\Delta_H\in \mathbb{R}^{p\times p}$ and $\Delta_\Gamma \in \mathbb{R}^{n\times n}$ which satisfies
\begin{align*}
    \cP_Z \Delta_\Gamma = \Delta_\Gamma\cP_Z = 0,
\end{align*}
then we have
    \begin{align*}
        \begin{bmatrix}
\Delta_{H}\\
\Delta_{\Gamma}
\end{bmatrix}^{\top}
\left(\sum_{i\neq j}
\begin{bmatrix}
{z}_i{z}_j^\top\\
\frac{{e}_i{e}_j^\top}{n}
\end{bmatrix}^{\otimes 2}\right)
\begin{bmatrix}
\Delta_{H}\\
\Delta_{\Gamma}
\end{bmatrix}\geq \frac{\sl}{2} \left\|\Delta_H\right\|_F^2 + \left\|\frac{\Delta_\Gamma}{n}\right\|_F^2 -2 \sum_{i=1}^n \left(\frac{(\Delta_\Gamma)_{ii}}{n}\right)^2.
    \end{align*}
\end{lemma}
\begin{proof}
    We have
    \begin{align*}
        &\begin{bmatrix}
\Delta_{H}\\
\Delta_{\Gamma}
\end{bmatrix}^{\top}
\left(\sum_{i\neq j}
\begin{bmatrix}
{z}_i{z}_j^\top\\
\frac{{e}_i{e}_j^\top}{n}
\end{bmatrix}^{\otimes 2}\right)
\begin{bmatrix}
\Delta_{H}\\
\Delta_{\Gamma}
\end{bmatrix} \\
&= \begin{bmatrix}
\Delta_{H}\\
\Delta_{\Gamma}
\end{bmatrix}^{\top}
\left(\sum_{1\leq i,j\leq n}
\begin{bmatrix}
{z}_i{z}_j^\top\\
\frac{{e}_i{e}_j^\top}{n}
\end{bmatrix}^{\otimes 2}\right)
\begin{bmatrix}
\Delta_{H}\\
\Delta_{\Gamma}
\end{bmatrix} - \begin{bmatrix}
\Delta_{H}\\
\Delta_{\Gamma}
\end{bmatrix}^{\top}
\left(\sum_{i=1}^n
\begin{bmatrix}
{z}_i{z}_i^\top\\
\frac{{e}_i{e}_i^\top}{n}
\end{bmatrix}^{\otimes 2}\right)
\begin{bmatrix}
\Delta_{H}\\
\Delta_{\Gamma}
\end{bmatrix} \\
&=\left\|Z \Delta_H Z^\top + \frac{\Delta_\Gamma}{n}\right\|_F^2  - \sum_{i=1}^n\left(z_i^\top \Delta_H z_j +\frac{(\Delta_\Gamma)_{ii}}{n}\right)^2 \\
&\geq_{(i)}  \left\|Z \Delta_H Z^\top\right\|_F^2 + \left\|\frac{\Delta_\Gamma}{n}\right\|_F^2 - 2\left( \sum_{i=1}^n \left(z_i^\top \Delta_H z_i \right)^2+\left(\frac{(\Delta_\Gamma)_{ii}}{n}\right)^2\right) \\
&\geq \sl \left\|\Delta_H\right\|_F^2 -2\sum_{i=1}^n \left\|z_i\right\|_2^4 \left\|\Delta_H\right\|^2 +\left\|\frac{\Delta_\Gamma}{n}\right\|_F^2-2 \sum_{i=1}^n \left(\frac{(\Delta_\Gamma)_{ii}}{n}\right)^2 \\
&\geq  \frac{\sl}{2} \left\|\Delta_H\right\|_F^2 + \left\|\frac{\Delta_\Gamma}{n}\right\|_F^2 -2 \sum_{i=1}^n \left(\frac{(\Delta_\Gamma)_{ii}}{n}\right)^2,
    \end{align*} 
as long as $n\geq 2c_z^2/\sl$. Here (i) follows from the fact that $\cP_Z (\Delta_{\Gamma})=0$, $\cP_Z  (Z\Delta_H Z^\top)=Z \Delta_H Z^\top$, and $(a+b)^2\leq 2(a^2+b^2)$.
\end{proof}

\subsubsection{Proof of Theorem \ref{thm:cv_ncv}}

With Claim \ref{claim2:pf_cv_ncv} and Lemma \ref{lem:cvx_lb} in hand, we are ready to prove Theorem \ref{thm:cv_ncv} in the following.

\begin{proof}[Proof of Theorem \ref{thm:cv_ncv}]
In the following, we fix a constant $c=c_P$. We define a constraint convex optimization problem as
\begin{align}\label{obj:con}
(\hat{H}^{con}, \hat{\Gamma}^{con}) 
:=\argmin_{\substack{\cP_Z\Gamma=0,\quad \Gamma\cP_Z = 0,\\
\|{H}-H^*\|_{F}\leq c n, \ \|{\Gamma}-\Gamma^*\|_{\infty}\leq c n}} f_c(H,\Gamma),
\end{align}
where $f_c$ is the convex objective defined in \eqref{obj:cv}.
By \eqref{ncv:est}, $(\hat{H}, \hat{X}\hat{Y}^\top)$ is feasible for the constraint of \eqref{obj:con}. 
By the optimality of $(\hat{H}^{con}, \hat{\Gamma}^{con})$, we have
\begin{align}\label{eq7:pf_MLE_bound}
L_c(\hat{H}^{con}, \hat{\Gamma}^{con})+\lambda\|\hat{\Gamma}^{con}\|_{*}\leq L_c(\hat{H}, \hat{X}\hat{Y}^{\top})+\lambda\|\hat{X}\hat{Y}^{\top}\|_{*}.
\end{align}
We denote
\begin{align*} 
\Delta_{H}^{con}:=\hat{H}^{con}-\hat{H},\ 
\Delta_{\Gamma}^{con}:=\hat{\Gamma}^{con}-\hat{X}\hat{Y}^{\top}.
\end{align*}
By mean value theorem, there exists a set of parameter $(\tilde{H}, \tilde{\Gamma})$ which is a convex combination of $(\hat{H}^{con}, \hat{\Gamma}^{con})$ and $(\hat{H}, \hat{X}\hat{Y}^\top)$ such that
\begin{align}\label{eq30:pf_MLE_bound}
&L_c(\hat{H}^{con}, \hat{\Gamma}^{con})-L_c(\hat{H}, \hat{X}\hat{Y}^{\top}) \notag \\
=&\nabla L_c(\hat{H}, \hat{X}\hat{Y}^{\top})^{\top}
\begin{bmatrix}
\Delta_{H}^{con}\\
\Delta_{\Gamma}^{con}
\end{bmatrix}
+\frac12 
\begin{bmatrix}
\Delta_{H}^{con}\\
\Delta_{\Gamma}^{con}
\end{bmatrix}^{\top}
\nabla^2 L_c(\tilde{H}, \tilde{\Gamma})
\begin{bmatrix}
\Delta_{H}^{con}\\
\Delta_{\Gamma}^{con}
\end{bmatrix}\notag\\
=&\nabla L_c(\hat{H}, \hat{X}\hat{Y}^{\top})^{\top}
\begin{bmatrix}
\Delta_{H}^{con}\\
\Delta_{\Gamma}^{con}
\end{bmatrix}
+\frac12 
\begin{bmatrix}
\Delta_{H}^{con}\\
\Delta_{\Gamma}^{con}
\end{bmatrix}^{\top}
\left(\sum_{i\neq j}\frac{e^{ \tilde{P}_{ij}}}{(1+e^{ \tilde{P}_{ij}})^2} 
\begin{bmatrix}
{z}_i{z}_j^\top\\
\frac{{e}_i{e}_j^\top}{n}
\end{bmatrix}^{\otimes 2}\right)
\begin{bmatrix}
\Delta_{H}^{con}\\
\Delta_{\Gamma}^{con}
\end{bmatrix}.
\end{align}
Therefore, we have
\begin{align*}
    L_c(\hat{H}^{con}, \hat{\Gamma}^{con})-L_c(\hat{H}, \hat{X}\hat{Y}^{\top})\geq \nabla L_c(\hat{H}, \hat{X}\hat{Y}^{\top})^{\top}
\begin{bmatrix}
\Delta_{H}^{con}\\
\Delta_{\Gamma}^{con}
\end{bmatrix}.
\end{align*}
Combine this with \eqref{eq7:pf_MLE_bound} we get 
\begin{align}
 0
 \leq -\nabla L_c(\hat{H}, \hat{X}\hat{Y}^\top)^{\top}
\begin{bmatrix}
\Delta_{H}^{con}\\
\Delta_{\Gamma}^{con}
\end{bmatrix}+\lambda\|\hat{X}\hat{Y}^\top\|_{*}-\lambda\|\hat{\Gamma}^{con}\|_{*}.\label{eq9:pf_MLE_bound}
\end{align}
In addition, by the convexity of $\|\cdot\|_{*}$, we have
\begin{align*}
\|\hat{\Gamma}^{con}\|_{*}-\|\hat{X}\hat{Y}^\top\|_{*}=\|\hat{X}\hat{Y}^\top+\Delta_{\Gamma}^{con}\|_{*}-\|\hat{X}\hat{Y}^\top\|_{*}\geq \langle UV^{\top}+W,\Delta_{\Gamma}^{con}\rangle
\end{align*}
for any $W\in\cT^{\perp}$ obeying $\|W\|\leq 1$. In the following, we pick $W$ such that $\langle W, \Delta_{\Gamma}^{con}\rangle=\|\cP_{\cT^{\perp}}(\Delta_{\Gamma}^{con})\|_{*}$. We then obtain $\|\hat{\Gamma}^{con}\|_{*}-\|\hat{X}\hat{Y}^\top\|_{*}\geq  \langle UV^{\top},\Delta_{\Gamma}^{con}\rangle+\|\cP_{\cT^{\perp}}(\Delta_{\Gamma}^{con})\|_{*}$, and consequently, by \eqref{eq9:pf_MLE_bound}, we have
\begin{align*}
0
 &\leq -\nabla L_c(\hat{H}, \hat{X}\hat{Y}^\top)^{\top}
\begin{bmatrix}
\Delta_{H}^{con}\\
\Delta_{\Gamma}^{con}
\end{bmatrix}-\lambda\langle UV^{\top},\Delta_{\Gamma}^{con}\rangle-\lambda\|\cP_{\cT^{\perp}}(\Delta_{\Gamma}^{con})\|_{*} \\
&=-\nabla L_c(\hat{H}, \hat{X}\hat{Y}^\top)^{\top}
\begin{bmatrix}
\Delta_{H}^{con}\\
\cP^{\perp}_Z\Delta_{\Gamma}^{con}\cP^{\perp}_Z
\end{bmatrix}-\lambda\langle UV^{\top},\Delta_{\Gamma}^{con}\rangle-\lambda\|\cP_{\cT^{\perp}}(\Delta_{\Gamma}^{con})\|_{*} \\
&=-\nabla_H L_c(\hat{H}, \hat{X}\hat{Y}^\top)^{\top}\Delta_{H}^{con} - (\cP^{\perp}_Z\nabla_\Gamma L_c(\hat{H}, \hat{X}\hat{Y}^\top)\cP^{\perp}_Z)^\top\Delta_{\Gamma}^{con}
-\lambda\langle UV^{\top},\Delta_{\Gamma}^{con}\rangle-\lambda\|\cP_{\cT^{\perp}}(\Delta_{\Gamma}^{con})\|_{*}.
\end{align*}
Recall the definition of $R$ in \eqref{eq:R}, we further have
\begin{align*}
0\leq&  -\nabla_{H} L_c(\hat{H}, \hat{X}\hat{Y}^\top)^\top \Delta_{H}^{con}-\langle R,\Delta_{\Gamma}^{con}\rangle-\lambda\|\cP_{\cT^{\perp}}(\Delta_{\Gamma}^{con})\|_{*}\notag\\
=&  -\nabla_{H} L_c(\hat{H}, \hat{X}\hat{Y}^\top)^{\top} \Delta_{H}^{con}-\langle \cP_{\cT}(R),\Delta_{\Gamma}^{con}\rangle-\langle \cP_{\cT^{\perp}}(R),\Delta_{\Gamma}^{con}\rangle-\lambda\|\cP_{\cT^{\perp}}(\Delta_{\Gamma}^{con})\|_{*}\notag\\
\leq& -\nabla_{H} L_c(\hat{H}, \hat{X}\hat{Y}^\top)^{\top} \Delta_{H}^{con} +\|\cP_{\cT}(R)\|_F\|\cP_{\cT}(\Delta_{\Gamma}^{con})\|_F-\langle \cP_{\cT^{\perp}}(R),\Delta_{\Gamma}^{con}\rangle-\lambda\|\cP_{\cT^{\perp}}(\Delta_{\Gamma}^{con})\|_{*}.
\end{align*}
By Claim \ref{claim2:pf_cv_ncv}, we have
\begin{align*}
 -\langle \cP_{\cT^{\perp}}(R),\Delta_{\Gamma}^{con}\rangle
 \leq  \| \cP_{\cT^{\perp}}(R)\|\|\cP_{\cT^{\perp}}(\Delta_{\Gamma}^{con})\|_{*}
 \leq\left(1-\frac{\eps}{4}\right)\lambda\cdot\|\cP_{\cT^{\perp}}(\Delta_{\Gamma}^{con})\|_{*}
 \end{align*}
As a result, one can see that
\begin{align}
    \frac{\epsilon}{4}\lambda\|\cP_{\cT^{\perp}}(\Delta_{\Gamma}^{con})\|_{*}&\leq -\nabla_{H} L_c(\hat{H}, \hat{X}\hat{Y}^\top)^{\top} \Delta_{H}^{con} +\|\cP_{\cT}(R)\|_F\|\cP_{\cT}(\Delta_{\Gamma}^{con})\|_F  \notag\\
    &\leq \|\nabla_{H} L_c(\hat{H}, \hat{X}\hat{Y}^\top)\|_F\|\Delta_{H}^{con}\|_F+\|\cP_{\cT}(R)\|_F\|\cP_{\cT}(\Delta_{\Gamma}^{con})\|_F \notag \\
    &\leq \frac{c'}{n^5}\left(\|\Delta_{H}^{con}\|_F+\frac{ 72 \kappa}{\sqrt{\sigma_{\min}}}\|\cP_{\cT}(\Delta_{\Gamma}^{con})\|_F\right)\label{eq1:pf_MLE_bound}
\end{align}
with some constant $c'>0$.

On the other hand, note that for some constant $c''$
\begin{align}\label{ineq30:pf_MLE_bound}
&\begin{bmatrix}
\Delta_{H}^{con}\\
\Delta_{\Gamma}^{con}
\end{bmatrix}^{\top}
\left(\sum_{i\neq j}\frac{e^{ \tilde{P}_{ij}}}{(1+e^{ \tilde{P}_{ij}})^2} 
\begin{bmatrix}
{z}_i{z}_j^\top\\
\frac{{e}_i{e}_j^\top}{n}
\end{bmatrix}^{\otimes 2}\right)
\begin{bmatrix}
\Delta_{H}^{con}\\
\Delta_{\Gamma}^{con}
\end{bmatrix} \notag\\
&\geq_{(1)}
 \frac{e^{c'' c_P}}{(1+e^{c'' c_P})^2}
\begin{bmatrix}
\Delta_{H}^{con}\\
\Delta_{\Gamma}^{con}
\end{bmatrix}^{\top}
\left(\sum_{i\neq j}
\begin{bmatrix}
{z}_i{z}_j^\top\\
\frac{{e}_i{e}_j^\top}{n}
\end{bmatrix}^{\otimes 2}\right)
\begin{bmatrix}
\Delta_{H}^{con}\\
\Delta_{\Gamma}^{con}
\end{bmatrix}\notag\\
&\geq  \frac{  e^{c'' c_P}}{(1+e^{c'' c_P})^2}\left( \frac{\sl}{2} \left\|\Delta_H^{con}\right\|_F^2 + \left\|\frac{\Delta_\Gamma^{con}}{n}\right\|_F^2 -2 \sum_{i=1}^n \left(\frac{(\Delta_\Gamma^{con})_{ii}}{n}\right)^2\right),
\end{align}
where the last inequality follows Lemma \ref{lem:cvx_lb}.
Here (1) follows from the following argument: note that
\begin{align*}
|\hat{P}_{ij}|
:=\left|{z}_i^\top\hat{H}{z}_j + \frac{(\hat{X}\hat{Y}^{\top})_{ij}}{n}\right|\leq |P_{ij}^{*}|+\frac{\sz}{n}\|\hat{H}-H^*\|_F+\frac{1}{n}\|\hat{X}\hat{Y}^{\top}-\Gamma^*\|_{\infty}.
\end{align*}
Further, by \eqref{ncv:est}, we have
\begin{align*}
&\|\hat{X}\hat{Y}^\top-\Gamma^*\|_{\infty}\\
=&\|(\hat{X}-X^*)\hat{Y}^\top+X^{*}(\hat{Y}-Y^*)^\top\|_{\infty}\\
\leq& \|\hat{X}-X^*\|_{2,\infty}\|\hat{Y}\|_{2,\infty}+\|X^{*}\|_{2,\infty}\|\hat{Y}-Y^*\|_{2,\infty} \tag{Cauchy}\\
\lesssim& c_{41}\sqrt{\frac{\mu r \sigma_{\max}}{n}}.
\end{align*}
Thus, by \eqref{ncv:est} and Assumption \ref{assumption:scales}, as long as $n$ is large enough, we have $|\hat{P}_{ij}|\leq 2c_P$.
Similarly, we have
\begin{align*}
|(\hat{P}^{con})_{ij}|
:=\left|{z}_i^\top{\hat{H}^{con}}{z}_j + \frac{(\hat{\Gamma}^{con})_{ij}}{n}\right|&\leq |P_{ij}^{*}|+\frac{\sz}{n}\|\hat{H}^{con}-H^*\|_F+\frac{1}{n}\|\hat{\Gamma}^{con}-\Gamma^*\|_{\infty}.
\end{align*}
Further, by the constraints, we have $|(\hat{P}^{con})_{ij}|\lesssim c_P$. Since $\tilde{P}_{ij}$ lies between $P_{ij}$ and $(\hat{P}^{con})_{ij}$, we conclude that $|\tilde{P}_{ij}|\lesssim c_P$.
Consequently, we have $\frac{e^{ \tilde{P}_{ij}}}{(1+e^{ \tilde{P}_{ij}})^2} \geq  \frac{e^{c'' c_P}}{(1+e^{c'' c_P})^2}$ for some constant $c''$, which implies (1). 

Combine \eqref{eq30:pf_MLE_bound} and \eqref{ineq30:pf_MLE_bound}, we have
\begin{align}
&L_c(\hat{H}^{con}, \hat{\Gamma}^{con})-L_c(\hat{H}, \hat{X}\hat{Y}^\top) \notag \\
\geq &\nabla L_c(\hat{H}, \hat{X}\hat{Y}^\top)^{\top}
\begin{bmatrix}
\Delta_{H}^{con}\\
\Delta_{\Gamma}^{con}
\end{bmatrix}
+C'\left(\frac{\sl}{2} \left\|\Delta_H^{con}\right\|_F^2 + \left\|\frac{\Delta_\Gamma^{con}}{n}\right\|_F^2 -2 \sum_{i=1}^n \left(\frac{(\Delta_\Gamma^{con})_{ii}}{n}\right)^2\right)\label{eq8:pf_MLE_bound}
\end{align}
for $C':= \frac{e^{c'' c_P}}{2(1+e^{c'' c_P})^2}$. By \eqref{eq7:pf_MLE_bound} and \eqref{eq8:pf_MLE_bound}, we obtain
\begin{align}
 &C'\left(\frac{\sl}{2} \left\|\Delta_H^{con}\right\|_F^2 + \left\|\frac{\Delta_\Gamma^{con}}{n}\right\|_F^2 -2 \sum_{i=1}^n \left(\frac{(\Delta_\Gamma^{con})_{ii}}{n}\right)^2\right) \notag\\ 
 \leq& -\nabla L_c(\hat{H}, \hat{X}\hat{Y}^\top)^{\top}
\begin{bmatrix}
\Delta_{H}^{con}\\
\Delta_{\Gamma}^{con}
\end{bmatrix}+\lambda\|\hat{X}\hat{Y}^\top\|_{*}-\lambda\|\hat{\Gamma}^{con}\|_{*}.\label{eq9:pf_cv_ncv}
\end{align}
Recall that $\|\hat{\Gamma}^{con}\|_{*}-\|\hat{X}\hat{Y}^\top\|_{*}\geq  \langle UV^{\top},\Delta_{\Gamma}^{con}\rangle+\|\cP_{\cT^{\perp}}(\Delta_{\Gamma}^{con})\|_{*}$. Consequently, we have
\begin{align*}
&C'\left(\frac{\sl}{2} \left\|\Delta_H^{con}\right\|_F^2 + \left\|\frac{\Delta_\Gamma^{con}}{n}\right\|_F^2 -2 \sum_{i=1}^n \left(\frac{(\Delta_\Gamma^{con})_{ii}}{n}\right)^2\right) \\
 \leq& -\nabla L_c(\hat{H}, \hat{X}\hat{Y}^\top)^{\top}
\begin{bmatrix}
\Delta_{H}^{con}\\
\cP^{\perp}_Z\Delta_{\Gamma}^{con}\cP^{\perp}_Z
\end{bmatrix}-\lambda\langle UV^{\top},\Delta_{\Gamma}^{con}\rangle-\lambda\|\cP_{\cT^{\perp}}(\Delta_{\Gamma}^{con})\|_{*} \\
\leq &-\nabla_H L_c(\hat{H}, \hat{X}\hat{Y}^\top)^{\top}\Delta_{H}^{con} - (\cP^{\perp}_Z\nabla_\Gamma L_c(\hat{H}, \hat{X}\hat{Y}^\top)\cP^{\perp}_Z)^\top\Delta_{\Gamma}^{con}
-\lambda\langle UV^{\top},\Delta_{\Gamma}^{con}\rangle-\lambda\|\cP_{\cT^{\perp}}(\Delta_{\Gamma}^{con})\|_{*}.
\end{align*}
Recall the definition of $R$ in \eqref{eq:R}, we further have
\begin{align}\label{eq10:pf_MLE_bound}
&C'\left(\frac{\sl}{2}\|\Delta_{H}^{con}\|_F^2+\frac{1}{n^2}\|\Delta_{\Gamma}^{con}\|_F^2 - \frac{2}{n^2}\sum_{i=1}^n(\Delta_\Gamma^{con})_{ii}^2\right) \notag \\
\leq&  -\nabla_{H} L_c(\hat{H}, \hat{X}\hat{Y}^\top)^\top \Delta_{H}^{con}-\langle R,\Delta_{\Gamma}^{con}\rangle-\lambda\|\cP_{\cT^{\perp}}(\Delta_{\Gamma}^{con})\|_{*} \notag\\
=&  -\nabla_{H} L_c(\hat{H}, \hat{X}\hat{Y}^\top)^{\top} \Delta_{H}^{con}-\langle \cP_{\cT}(R),\Delta_{\Gamma}^{con}\rangle-\langle \cP_{\cT^{\perp}}(R),\Delta_{\Gamma}^{con}\rangle-\lambda\|\cP_{\cT^{\perp}}(\Delta_{\Gamma}^{con})\|_{*}\notag\\
\leq& -\nabla_{H} L_c(\hat{H}, \hat{X}\hat{Y}^\top)^{\top} \Delta_{H}^{con} +\|\cP_{\cT}(R)\|_F\|\cP_{\cT}(\Delta_{\Gamma}^{con})\|_F-\langle \cP_{\cT^{\perp}}(R),\Delta_{\Gamma}^{con}\rangle-\lambda\|\cP_{\cT^{\perp}}(\Delta_{\Gamma}^{con})\|_{*}.
\end{align}
By Claim \ref{claim2:pf_cv_ncv}, we have
\begin{align}\label{eq11:pf_MLE_bound}
 -\langle \cP_{T^{\perp}}(R),\Delta_{\Gamma}^{con}\rangle
 \leq  \| \cP_{T^{\perp}}(R)\|\|\cP_{T^{\perp}}(\Delta_{\Gamma}^{con})\|_{*}
 \leq\left(1-\frac{\eps}{4}\right)\lambda\cdot\|\cP_{T^{\perp}}(\Delta_{\Gamma}^{con})\|_{*}
 \end{align}
Combine \eqref{eq10:pf_MLE_bound} and \eqref{eq11:pf_MLE_bound} we get
\begin{align}
    &C'\left(\frac{\sl}{2}\|\Delta_{H}^{con}\|_F^2+\frac{1}{n^2}\|\Delta_{\Gamma}^{con}\|_F^2 - \frac{2}{n^2}\sum_{i=1}^n(\Delta_\Gamma^{con})_{ii}^2\right) \notag \\ 
    \leq& -\nabla_{H} L_c(\hat{H}, \hat{X}\hat{Y}^\top)^{\top} \Delta_{H}^{con} +\|\cP_{\cT}(R)\|_F\|\cP_{\cT}(\Delta_{\Gamma}^{con})\|_F.\label{eq5:pf_MLE_bound}
\end{align}
In the sequel, we deal with $\sum_{i=1}^n(\Delta_\Gamma^{con})_{ii}^2$. One can see that
\begin{align}
    \sum_{i=1}^n(\Delta_\Gamma^{con})_{ii}^2 &= \sum_{i=1}^n(P_{\cT}(\Delta_\Gamma^{con}) + P_{\cT^\perp}(\Delta_\Gamma^{con}))_{ii}^2 \leq 2\sum_{i=1}^n(P_{\cT}(\Delta_\Gamma^{con}))_{ii}^2 + 2\sum_{i=1}^n(P_{\cT^\perp}(\Delta_\Gamma^{con}))_{ii}^2. \label{eq2:pf_MLE_bound}
\end{align}
By Lemma \ref{lem:POmega} we know that
\begin{align}
    \sum_{i=1}^n(P_{\cT}(\Delta_\Gamma^{con}))_{ii}^2\leq \frac{1}{5} \left\|P_{\cT}(\Delta_\Gamma^{con})\right\|_F^2\leq \frac{1}{5} \left\|\Delta_\Gamma^{con}\right\|_F^2.\label{eq3:pf_MLE_bound}
\end{align}
On the other hand, by \eqref{eq1:pf_MLE_bound} we have
\begin{align}
    \sum_{i=1}^n(P_{\cT^\perp}(\Delta_\Gamma^{con}))_{ii}^2\leq \left\|P_{\cT^\perp}(\Delta_\Gamma^{con})\right\|_F^2\leq\left\|P_{\cT^\perp}(\Delta_\Gamma^{con})\right\|_*^2 &\lesssim \left(\frac{4}{\epsilon \lambda n^5}\right)^2\left(\|\Delta_{H}^{con}\|_F+\frac{ 72 \kappa}{\sqrt{\sigma_{\min}}}\|\cP_{\cT}(\Delta_{\Gamma}^{con})\|_F\right)^2 \notag\\
    &\ll \frac{1}{n}\left(\|\Delta_{H}^{con}\|_F^2+\frac{1}{n^2}\|\Delta_{\Gamma}^{con}\|_F^2\right).\label{eq4:pf_MLE_bound}
\end{align}
Combine \eqref{eq2:pf_MLE_bound}, \eqref{eq3:pf_MLE_bound} and \eqref{eq4:pf_MLE_bound} we know that
\begin{align*}
    \frac{\sl}{2}\|\Delta_{H}^{con}\|_F^2+\frac{1}{n^2}\|\Delta_{\Gamma}^{con}\|_F^2 - \frac{2}{n^2}\sum_{i=1}^n(\Delta_\Gamma^{con})_{ii}^2 \geq \frac{\sl}{3}\|\Delta_{H}^{con}\|_F^2+\frac{1}{6 n^2}\|\Delta_{\Gamma}^{con}\|_F^2.
\end{align*}
Combine this with \eqref{eq5:pf_MLE_bound}, we get
\begin{align*}
    C''\left(\|\Delta_{H}^{con}\|_F^2+\frac{1}{n^2}\|\Delta_{\Gamma}^{con}\|_F^2 \right) &\leq -\nabla_{H} L_c(\hat{H}, \hat{X}\hat{Y}^\top)^{\top} \Delta_{H}^{con} +\|\cP_{\cT}(R)\|_F\|\cP_{\cT}(\Delta_{\Gamma}^{con})\|_F \notag\\
    &\leq \|\nabla_{H} L_c(\hat{H}, \hat{X}\hat{Y}^\top)\|_F\left\|\Delta_{H}^{con}\right\|_F+\|\cP_{\cT}(R)\|_F\|\cP_{\cT}(\Delta_{\Gamma}^{con})\|_F \notag \\
    &\leq \left(1+ \frac{ 72 \kappa n }{\sqrt{\sigma_{\min}}}\right)\|\cP(\nabla f(\hat{H},\hat{X},\hat{Y}))\|_F\left(\|\Delta_{H}^{con}\|_F+\frac{1}{n}\|P_{\cT}(\Delta_{\Gamma}^{con})\|_F\right) \\
    &\leq \left(1+ \frac{ 72 \kappa n }{\sqrt{\sigma_{\min}}}\right)\|\cP(\nabla f(\hat{H},\hat{X},\hat{Y}))\|_F\left(\|\Delta_{H}^{con}\|_F+\frac{1}{n}\|\Delta_{\Gamma}^{con}\|_F\right),
\end{align*}
where $C'':=C'\min\{\sl/3, 1/6\}$. Since $\|\Delta_{H}^{con}\|_F^2+\frac{1}{n^2}\|\Delta_{\Gamma}^{con}\|_F^2\geq (\|\Delta_{H}^{con}\|_F+\frac{1}{n}\|\Delta_{\Gamma}^{con}\|_F)^2/2$, we know that
\begin{align*}
    \frac{C''}{2}\left(\|\Delta_{H}^{con}\|_F+\frac{1}{n}\|\Delta_{\Gamma}^{con}\|_F\right)^2 \leq \left(1+ \frac{ 72 \kappa n }{\sqrt{\sigma_{\min}}}\right)\|\cP(\nabla f(\hat{H},\hat{X},\hat{Y}))\|_F\left(\|\Delta_{H}^{con}\|_F+\frac{1}{n}\|\Delta_{\Gamma}^{con}\|_F\right). 
\end{align*}
As a result, we get
\begin{align}\label{ineq100:pf_cn_ncv}
\|\Delta_{H}^{con}\|_F+\frac{1}{n}\|\Delta_{\Gamma}^{con}\|_F \leq \frac{2}{C''}\left(1+\frac{ 72 \kappa n}{\sqrt{\sigma_{\min}}}\right)\|\cP(\nabla f(\hat{H},\hat{X},\hat{Y}))\|_F.
\end{align}
Further by \eqref{ncv:grad}, we obtain
\begin{align*}
\|\Delta_{H}^{con}\|_F \lesssim n^{-5},\quad
\|\Delta_{\Gamma}^{con}\|_F\lesssim n^{-4}.
\end{align*}
Consequently, we show that
\begin{align*}
&\|\hat{H}^{con}-H^*\|_F\leq \|\Delta_{H}^{con}\|_F+\|\hat{H}-H^*\|_F\lesssim c_{11}\sqrt{n}<cn,\\
&\|\hat{\Gamma}^{con}-\Gamma^*\|_{\infty}\leq \|\Delta_{\Gamma}^{con}\|_F+\|\hat{X}\hat{Y}^\top-\Gamma^*\|_{\infty}\lesssim c_{41}\sqrt{\frac{\mu r\sigma_{\max}}{n}}<cn
\end{align*}
as long as $n$ is large enough.
In other words, the minimizer of \eqref{obj:con} is in the interior of the constraint. By the convexity of \eqref{obj:con}, we have $(\hat{H}^{con}, \hat{\Gamma}^{con}) =(\hat{H}_c, \hat{\Gamma}_c)$. 
Consequently, by \eqref{ineq100:pf_cn_ncv}, we have
\begin{align*}
\left\|
\begin{bmatrix}
\hat{H}_c-\hat{H}\\
\frac{1}{n}(\hat{\Gamma}_c-\hat{X}\hat{Y}^\top)
\end{bmatrix}
\right\|_F
& = \left\|
\begin{bmatrix}
\hat{H}^{con}-\hat{H}\\
\frac{1}{n}(\hat{\Gamma}^{con}-\hat{X}\hat{Y}^\top)
\end{bmatrix}
\right\|_F\\
&\leq \|\Delta_{H}^{con}\|_F+\frac{1}{n}\|\Delta_{\Gamma}^{con}\|_F\\
&\leq \frac{2}{C''}\left(1+\frac{ 72 \kappa n}{\sqrt{\sigma_{\min}}}\right)\|\cP(\nabla f(\hat{H},\hat{X},\hat{Y}))\|_F.
\end{align*}
Thus, we prove Theorem \ref{thm:cv_ncv}.
\end{proof}

\subsection{Proofs of Theorem \ref{thm:cv_est}}\label{pf:cor_cv_ncv}
Note that $\|\cP(\nabla f(\hat{H},\hat{X},\hat{Y}))\|_F\lesssim n^{-5}$.
By Theorem \ref{thm:cv_ncv} and Lemma \ref{lem:ncv4}, we then have
\begin{align*}
\|\hat{H}_c-H^*\|_F\lesssim \|\hat{H}-H^*\|_F\lesssim c_{11}\sqrt{n}.
\end{align*}
Note that by Lemma \ref{lem:ncv1}, we have
\begin{align*}
\|\hat{X}\hat{Y}^{\top}-\Gamma^{*}\|_F\leq 
\|\hat{X}-{X}^{*}\|_F\|\hat{Y}\|+\|\hat{Y}-{Y}^{*}\|_F\|{X}^{*}\|
\lesssim \sqrt{\sigma_{\max}n}c_{11}.
\end{align*}
By Lemma \ref{lem:ncv4}, we have
\begin{align*}
\|\hat{X}\hat{Y}^{\top}-\Gamma^{*}\|_{\infty}\leq 
\|\hat{X}-{X}^{*}\|_{2,\infty}\|\hat{Y}\|_{2,\infty}+\|\hat{Y}-{Y}^{*}\|_{2,\infty}\|{X}^{*}\|_{2,\infty}
\lesssim \sqrt{\frac{\mu r\sigma_{\max}}{n}}c_{41}.
\end{align*}
Further by Theorem \ref{thm:cv_ncv}, we have
\begin{align*}
&\|\hat{\Gamma}_c-\Gamma^{*}\|_{F}\lesssim \|\hat{X}\hat{Y}^{\top}-\Gamma^{*}\|_F\lesssim \sqrt{\sigma_{\max}n}c_{11}\\
&\|\hat{\Gamma}_c-\Gamma^{*}\|_{\infty}\lesssim \|\hat{X}\hat{Y}^{\top}-\Gamma^{*}\|_{\infty}\lesssim \sqrt{\frac{\mu r\sigma_{\max}}{n}}c_{41}.
\end{align*}

\section{Proofs of Proposition \ref{identidiablecondition} and Theorem \ref{thm:mem_est}}

\subsection{Proofs of Proposition \ref{identidiablecondition}}

Compared to \cite[Proposition A.1]{jin2023mixed}, the only difference between our identifiable condition and theirs is the sign of diagonal entries of $W$. In fact, as to the proof of this condition, we only need to make a slight modification on the basis of \cite[Proposition A.1]{jin2023mixed}.

    Assume that we have two sets of $(\Theta, \Pi, W)$ and $(\tilde{\Theta}, \tilde{\Pi}, \tilde{W})$ which satisfy Proposition \ref{identidiablecondition} and $\Theta\Pi W \Pi^\top \Theta = \tilde{\Theta}\tilde{\Pi}\tilde{W}\tilde{\Pi}^\top\tilde{\Theta}$. According to \cite[Proof of Proposition A.1]{jin2023mixed}, if the row $i$ of $\Pi$ (or $\tilde{\Pi}$) represents a pure node, then the row $i$ of $\tilde{\Pi}$ (or $\Pi$) also represents a pure node, and these two sets of pure nodes are identical up to a permutation of the columns. Therefore, without loss of generality, we assume that $\Pi_{1:K, :}$ and $\tilde{\Pi}_{1:K, :}$ are all equal to the identity matrix. Comparing the submatrices $(\Theta\Pi W \Pi^\top \Theta)_{1:K, 1:K}$ and $(\tilde{\Theta}\tilde{\Pi}\tilde{W}\tilde{\Pi}^\top\tilde{\Theta})_{1:K, 1:K}$, which should be identical, we get
    \begin{align}
        \Theta_{1:K, 1:K} W\Theta_{1:K, 1:K} = \tilde{\Theta}_{1:K, 1:K} \tilde{W}\tilde{\Theta}_{1:K, 1:K}. \label{identifiableconditioneq1}
    \end{align}
    Particularly, we know that $\theta_i^2W_{ii} = \tilde{\theta}_i^2\tilde{W}_{ii}$ for all $i\in [r]$. Since $\theta_i^2, \tilde{\theta}_i^2 >0$, we know that $W_{ii}$ and $\tilde{W}_{ii}$ must have the same sign. By Proposition \ref{identidiablecondition}, $|W_{ii}| = |\tilde{W}_{ii}| = 1$. Therefore, we know that $W_{ii} = \tilde{W}_{ii}$. This also implies $\theta_i = \tilde{\theta}_i$, and thus $\Theta_{1:K, 1:K} = \tilde{\Theta}_{1:K, 1:K}$. Plugging this back in \eqref{identifiableconditioneq1}, we get $W = \tilde{W}$. The rest of the proof is the same as \cite[Proof of Proposition A.1]{jin2023mixed}, and we finally reach $(\Theta, \Pi, W) = (\tilde{\Theta}, \tilde{\Pi}, \tilde{W})$. That is to say, the DCMM model $\Gamma = \Theta\Pi W \Pi^\top \Theta$ is identifiable under the conditions in Proposition \ref{identidiablecondition}.

\subsection{Proofs of Theorem \ref{thm:mem_est}}\label{pf:thm_mem_est}
In this section, we will frequently using \cite[Lemma C.2, C.3, C.4]{jin2023mixed}. Since our condition on $W^*$ is slightly adapted from \cite[Proposition A.1]{jin2023mixed}, the \cite[Lemma C.2]{jin2023mixed} has to be modified here. We state the result we are going to use as follows.

\begin{lemma}\label{modifiedLemmaC.2}
    Under Proposition \ref{identidiablecondition} and Assumption \ref{assumption:mem_est}, we have
    \begin{itemize}
        \item $r^{-1}\left\|\theta^*\right\|_2^2\lesssim |\lambda_1|\lesssim \left\|\theta^*\right\|_2^2$.
        \item $|\lambda_1| - |\lambda_2|\asymp |\lambda_1|$.
        \item $|\lambda_k|\asymp\beta_n r^{-1}\left\|\theta^*\right\|_2^2$ for all $1\leq k\leq r$. 
    \end{itemize}
\end{lemma}

Let $\hat{U}_{full}\hat{\Sigma}_{full}\hat{V}_{full}^\top$ be the SVD of $\hat{\Gamma}_c$ and assume the diagonal entries of $\hat{\Sigma}_{full}$ are sorted in a descending order. We denote by
\begin{align*}
    \hat{U}_c = (\hat{U}_{full})_{:, 1:r}, \quad \hat{\Sigma}_c = (\hat{\Sigma}_{full})_{1:r, 1:r}, \quad \hat{V}_c = (\hat{V}_{full})_{:, 1:r}. 
\end{align*}
We choose the signs such that the left singular vectors $\hat{U}_c$ are coincident with the eigenvectors associated with the largest $r$ (in magnitude) eigenvalues. Also, let $\hat{U}\hat{\Sigma}\hat{V}^\top$ be the SVD of $\hat{X}\hat{Y}^\top$, where $\hat{X},\hat{Y}$ are the nonconvex estimators given in Section \ref{sec:nonconvex_debias}. Define 
\begin{align*}
    R_U = \argmin_{R\in \mathbb{R}^{r \times r}} \left\|\hat{U}_c R - \hat{U}\right\|_F \text{ and } R_V = \argmin_{R\in \mathbb{R}^{r \times r}} \left\|\hat{V}_c R - \hat{V}\right\|_F.
\end{align*}
We begin with the following lemma.

\begin{lemma}\label{lem1:U_2inflem1}
Define
\begin{align*}
    R:=\argmin_{L\in \cO^{r \times r}} \left\|\hat{U}_c L - U^*\right\|_F.
\end{align*}
Then it holds that
\begin{align*}
    \|\hat{U}_c R-U^*\|_{2,\infty}\lesssim\frac{c_{41}+c_{11} \kappa^{1.5}\sqrt{\mu r}}{\sqrt{\sigma_{\min}}}.
\end{align*}
\end{lemma}

\begin{proof}
    We define
    \begin{align*}
    &R_1:=\argmin_{L\in \cO^{r \times r}} \left\|\hat{U} L - \hat{U}_cR\right\|_F, \\
    &R_2:=\argmin_{L\in \cO^{r \times r}} \left\|\hat{U}R_1 L - U^*\right\|_F.
    \end{align*}
By Weyl’s inequality and the proof of Theorem \ref{thm:cv_est}, we have
\begin{align*}
|\sigma_{\min}(\hat{\Gamma})-\sigma_{\min}|\leq \|\hat{\Gamma}-{\Gamma}^*\|\lesssim \sqrt{\sigma_{\max}n}c_{11},
\end{align*}
which implies $\sigma_{\min}(\hat{\Gamma})\geq \sigma_{\min}/2$. By Davis-Kahan Theorem, we have
\begin{align*}
&\|\hat{U}_c R - \hat{U} R_1 \|_F = \min_{L\in \cO^{r \times r}} \left\|\hat{U} L - \hat{U}_cR\right\|_F= \min_{L\in \cO^{r \times r}} \left\|\hat{U} L - \hat{U}_c\right\|_F \\
\lesssim&  \frac{\sqrt{r}}{\sigma_{\min}(\hat{\Gamma})}\|\hat{\Gamma}_c -\hat{\Gamma}\|\leq \frac{\sqrt{r}}{\sigma_{\min}}\|\hat{\Gamma}_c -\hat{\Gamma}\|_F\lesssim_{(1)}\frac{1}{\sigma_{\min}}\frac{(1+e^{c c_P})^2}{\sl e^{c c_P}}\left(1+\frac{ 72 \kappa n}{\sqrt{\sigma_{\min}}}\right)n^{-4}\lesssim n^{-5},
\end{align*}
where (1) follows from Theorem \ref{thm:cv_ncv} and \eqref{ncv:grad}. Also, by Davis-Kahan Theorem we know that
\begin{align*}
    \left\|\hat{U}_cR-U^*\right\|\lesssim \frac{\left\|\hat{\Gamma}_c - \Gamma^*\right\|}{\sigma_{r}(\Gamma^*)}\lesssim \frac{\sqrt{\sigma_{\max}n}c_{11}}{\sigma_{\min}}.
\end{align*}
Since
\begin{align*}
    &\left\|\hat{U}_cR-U^*\right\|\left\|U^*\right\| \lesssim \frac{\sqrt{\sigma_{\max}n}c_{11}}{\sigma_{\min}}\ll \frac{1}{2}\sigma^2_r(U^*),\\
    &\left\|\hat{U}_c R-\hat{U}R_1\right\|\left\|U^*\right\| \lesssim n^{-5}\sigma^2_r(U^*)
\end{align*}
as long as $n$ is large enough, by Lemma \ref{ar:lem3} we have
\begin{align}
    \left\|\hat{U}_cR - \hat{U}R_1R_2\right\|_F\leq \frac{5\sigma^2_1(U^*)}{\sigma^2_r(U^*)}\left\|\hat{U}_c R - \hat{U} R_1\right\|_F\lesssim n^{-5}.\label{pf_memlem1eq2}
\end{align}
We define $\hat{R} = R_1R_2$. It's easy to see that
\begin{align}
    \hat{R}=\argmin_{L\in \cO^{r \times r}} \left\|\hat{U} L - U^*\right\|_F.\label{pf_memlem1eq1}
\end{align}
We then turn to control $\|\hat{U}\hat{R}-U^*\|_{2, \infty}$. By Claim \ref{claim1:pf_cv_ncv}, there exists an invertible matrix $Q$ such that $\hat{X}=\hat{U}\hat{\Sigma}^{\frac{1}{2}}Q$ and \eqref{ineq1:pf_cv_ncv} holds. By the definition of $X^{*}$, we have $X^*=U^*\Sigma^{* \frac{1}{2}}$. Thus, we have
\begin{align}\label{eq1:pf_mem}
&\hat{U}=  \hat{X}(\hat{\Sigma}^{\frac{1}{2}}Q)^{-1},\quad \hat{U}^{\top}\hat{X}=\hat{\Sigma}^{\frac{1}{2}}Q\notag\\
& U^*=X^*(\Sigma^{* \frac{1}{2}})^{-1},\quad U^{*\top}X^*=\Sigma^{* \frac{1}{2}}.
\end{align}
It then holds that
\begin{align}\label{ineq2:pf_mem}
\|\hat{U}\hat{R}-U^*\|_{2,\infty}&=\|\hat{X}(\hat{\Sigma}^{\frac{1}{2}}Q)^{-1}\hat{R}-X^*(\Sigma^{* \frac{1}{2}})^{-1}\|_{2,\infty}\notag\\
&\leq \|(\hat{X}-X^*)(\Sigma^{* \frac{1}{2}})^{-1}\|_{2,\infty}+\|\hat{X}(Q^{-1}\hat{\Sigma}^{-\frac{1}{2}}\hat{R}-\Sigma^{* -\frac{1}{2}})\|_{2,\infty}\notag\\
&\leq \|\Sigma^{* -\frac{1}{2}}\|\|\hat{X}-X^*\|_{2,\infty}+\|Q^{-1}\hat{\Sigma}^{-\frac{1}{2}}\hat{R}-\Sigma^{* -\frac{1}{2}}\|\|\hat{X}\|_{2,\infty}\notag\\
&\leq \frac{1}{\sqrt{\sigma_{\min}}}c_{41}+2\|Q^{-1}\hat{\Sigma}^{-\frac{1}{2}}\hat{R}-\Sigma^{* -\frac{1}{2}}\|\sqrt{\frac{\mu r \sigma_{\max}}{n}},
\end{align}
where the last inequality follows from \eqref{ncv:est} and the fact that
\begin{align*}
\|\hat{X}\|_{2,\infty}\leq \|\hat{X}-X^*\|_{2,\infty}+\|{X}^*\|_{2,\infty} \leq 2   \|{X}^*\|_{2,\infty}
\leq 2\sqrt{\frac{\mu r \sigma_{\max}}{n}}.
\end{align*}
Note that
\begin{align*}
\|Q^{-1}\hat{\Sigma}^{-\frac{1}{2}}\hat{R}-\Sigma^{* -\frac{1}{2}}\|&=\|\Sigma^{* -\frac{1}{2}}(\hat{R}^{\top} \hat{\Sigma}^{\frac{1}{2}}Q-\Sigma^{* \frac{1}{2}})Q^{-1}\hat{\Sigma}^{-\frac{1}{2}}\hat{R}\|\\
&\leq \|\Sigma^{* -\frac{1}{2}}\| \|\hat{R}^{\top} \hat{\Sigma}^{\frac{1}{2}}Q-\Sigma^{* \frac{1}{2}}\| \|Q^{-1}\hat{\Sigma}^{-\frac{1}{2}}\hat{R}\|\\
&\leq \|\Sigma^{* -\frac{1}{2}}\| \|\hat{\Sigma}^{ -\frac{1}{2}}\|\|Q^{-1}\|\|\hat{R}^{\top} \hat{\Sigma}^{\frac{1}{2}}Q-\Sigma^{* \frac{1}{2}}\|.
\end{align*}
Since $Q$ satisfies \eqref{ineq1:pf_cv_ncv}, we have $\|Q^{-1}\|\leq 2$. Moreover, we have show that $\|\hat{\Sigma}^{-\frac{1}{2}}\|=\sqrt{1/\sigma_{\min}(\hat{\Gamma})}\leq \sqrt{2/\sigma_{\min}}$.
Thus, we obtain
\begin{align}\label{ineq3:pf_mem}
\|Q^{-1}\hat{\Sigma}^{-\frac{1}{2}}\hat{R}-\Sigma^{* -\frac{1}{2}}\|\leq \frac{4}{\sigma_{\min}}\|\hat{R}^{\top} \hat{\Sigma}^{\frac{1}{2}}Q-\Sigma^{* \frac{1}{2}}\|.
\end{align}
By \eqref{eq1:pf_mem}, we have
\begin{align*}
\|\hat{R}^{\top} \hat{\Sigma}^{\frac{1}{2}}Q-\Sigma^{* \frac{1}{2}}\|&=\|\hat{R}^{\top}\hat{U}^{\top}\hat{X}-{U}^{* \top}{X}^{*}\|\\
&\leq \|(\hat{R}^{\top}\hat{U}^{\top}-{U}^{* \top})\hat{X}\|+\|{U}^{* \top}(\hat{X}-{X}^{*})\|\\
&\leq \|\hat{X}\|\|\hat{U}\hat{R}-U^{*}\|+\|\hat{X}-{X}^{*}\|\\
&\leq 2\sqrt{\sigma_{\max}}\|\hat{U}\hat{R}-U^{*}\|_F+c_{11}\sqrt{n},
\end{align*}
where the last inequality follows from Lemma \ref{lem:ncv1} and the fact that
\begin{align*}
\|\hat{X}\|\leq \|\hat{X}-X^*\|+\|{X}^*\|\leq 2\|{X}^*\|=2\sqrt{\sigma_{\max}}.
\end{align*}
By Davis-Kahan Theorem and \eqref{pf_memlem1eq1}, we have
\begin{align*}
 \|\hat{U}\hat{R}-U^{*}\|_F\lesssim  \frac{\|\hat{X}\hat{Y}^\top - \Gamma^*\|}{\sigma_{r}(\Gamma^*)}\lesssim \frac{\sqrt{\sigma_{\max}n}c_{11}}{\sigma_{\min}}.
\end{align*}
Thus, we have
\begin{align}\label{ineq4:pf_mem}
 \|\hat{R}^{\top} \hat{\Sigma}^{\frac{1}{2}}Q-\Sigma^{* \frac{1}{2}}\|\lesssim 2\sqrt{\sigma_{\max}}\frac{\sqrt{\sigma_{\max}n}c_{11}}{\sigma_{\min}} +c_{11}\sqrt{n}\lesssim   \kappa c_{11}\sqrt{n} .
\end{align}
Combine \eqref{ineq3:pf_mem} and \eqref{ineq4:pf_mem}, we get
\begin{align}\label{ineq5:pf_mem}
 \|Q^{-1}\hat{\Sigma}^{-\frac{1}{2}}\hat{R}-\Sigma^{* -\frac{1}{2}}\|\lesssim \frac{\kappa c_{11}\sqrt{n}}{\sigma_{\min}}  .  
\end{align}
Further by \eqref{ineq2:pf_mem}, we have
\begin{align*}
\|\hat{U}\hat{R}-U^*\|_{2,\infty}\lesssim \frac{c_{41}+c_{11} \kappa^{1.5}\sqrt{\mu r}}{\sqrt{\sigma_{\min}}}.
\end{align*}
Combine above inequality with \eqref{pf_memlem1eq2}, we obtain
\begin{align*}
\|\hat{U}_c R-U^*\|_{2,\infty}&\leq \|\hat{U}_c R-\hat{U}\hat{R}\|_{2,\infty}+ \|\hat{U}\hat{R} R-U^*\|_{F}\\
&\lesssim n^{-5}+    \frac{c_{41}+c_{11} \kappa^{1.5}\sqrt{\mu r}}{\sqrt{\sigma_{\min}}} \lesssim \frac{c_{41}+c_{11} \kappa^{1.5}\sqrt{\mu r}}{\sqrt{\sigma_{\min}}}.
\end{align*}
We then finish the proof of Lemma \ref{lem1:U_2inflem1}.
\end{proof}

Let $\check{U}_c := (\hat{U}_c)_{:, 2:r}\in \mathbb{R}^{n\times (r-1)}$ be the $2$-th to $r$-th column of $\hat{U}_c$ and $\check{U}^*:= (U^*)_{:, 2:r}\in \mathbb{R}^{n\times (r-1)} $ be the $2$-th to $r$-th column of $U^*$. Define $\check{R}\in \mathbb{R}^{(r-1)\times (r-1)}$ as the rotation matrix aligns $\check{U}_c$ and $\check{U}^*$, i.e.,
\begin{align*}
    \check{R}:=\argmin_{L\in \cO^{(r-1)\times (r-1)}} \left\|\check{U}_c L - \check{U}^*\right\|_F.
\end{align*}
Moreover, without loss of generality, we choose the direction of $(U^*)_{:, 1}$ such that $(\hat{U}_c)_{:, 1}^\top (U^*)_{:, 1}\geq 0$.
Then we have the following results.

\begin{lemma}\label{lem:U_2inflem2}
It holds that
\begin{align*}
    &\|\check{U}_c \check{R}-\check{U}^*\|_{2,\infty}\lesssim\frac{c_{41}+c_{11} \kappa^{1.5}\sqrt{\mu r}+c_{11}\sqrt{\mu\kappa}r^{5/4}}{\sqrt{\sigma_{\min}}},\\
    &\|(\hat{U}_c)_{:, 1} - (U^*)_{:, 1}\|_{\infty}\lesssim \frac{c_{41}+c_{11} \kappa^{1.5}\sqrt{\mu r}+c_{11}\sqrt{\mu\kappa}r^{5/4}}{\sqrt{\sigma_{\min}}}.
\end{align*}
\end{lemma}
\begin{proof}
    We define 
    \begin{align*}
        &H:= \hat{U}_c^\top U^*,\quad \check{H}:= \check{U}_c^\top \check{U}^*.
    \end{align*}
    According to this definition, one can see that
    \begin{align*}
        H_{2:r, 2:r} = (\hat{U}_c)_{:, 2:r}^\top (U^*)_{:, 2:r} = \check{U}_c^\top \check{U}^* = \check{H}.
    \end{align*}
    Therefore, we can control the different between $R_{2:r, 2:r}$ and $\check{R}$ as 
    \begin{align*}
        \left\|R_{2:r, 2:r} - \check{R}\right\|\leq \left\|R_{2:r, 2:r} - H_{2:r, 2:r}\right\| +\left\|H_{2:r, 2:r}-\check{H}\right\| + \left\|\check{H}-\check{R}\right\|\leq \left\|R - H\right\| + \left\|\check{H}-\check{R}\right\|.
    \end{align*}
    By Davis-Karhan Theorem and Lemma 2 in \cite{yan2024inference}, we have
    \begin{align*}
        \left\|R - H\right\|\lesssim \left(\frac{\|\hat{\Gamma}_c-\Gamma^*\|}{\sigma_{\min}}\right)^2.
    \end{align*}
    Similarly, according to Assumption \ref{assumption:mem_est}, we have
    \begin{align*}
        \left\|\check{H}-\check{R}\right\|\lesssim \left(\frac{\|\hat{\Gamma}_c-\Gamma^*\|}{\sigma_{\min}}\right)^2.
    \end{align*}
    Combine these two results we get
    \begin{align}
        \left\|R_{2:r, 2:r} - \check{R}\right\|\lesssim \left(\frac{\|\hat{\Gamma}_c-\Gamma^*\|}{\sigma_{\min}}\right)^2\lesssim \frac{c_{11}^2\sigma_{\max}n}{\sigma_{\min}^2}. \label{mem_est_proof_eq1}
    \end{align}
    On the other hand, $\check{U}_c \check{R}-\check{U}^*$ can be written as
    \begin{align}\label{mem_est_proof_eq10}
        \check{U}_c \check{R}-\check{U}^* = \check{U}_c R_{2:r,2:r}-\check{U}^* + \check{U}_c(\check{R} - R_{2:r,2:r}).
    \end{align}
    It remains to control $\check{U}_c R_{2:r,2:r}-\check{U}^*$. Notice that
    \begin{align*}
        (\hat{U}_c R-U^*)_{:, 2:r} = \hat{U}_c R_{:, 2:r} - U^*_{:, 2:r} = (\hat{U}_c)_{:, 1} R_{1, 2:r} + \check{U}_c R_{2:r, 2:r} - \check{U}^*.
    \end{align*}
    Therefore, $\check{U}_c R_{2:r, 2:r} - \check{U}^*$ can be controlled as
    \begin{align}
        \left\|\check{U}_c R_{2:r, 2:r} - \check{U}^*\right\|_{2,\infty}&\leq \left\|(\hat{U}_c R-U^*)_{:, 2:r}\right\|_{2,\infty} + \left\|(\hat{U}_c)_{:, 1} R_{1, 2:r}\right\|_{2,\infty} \notag\\
        &\leq \left\|\hat{U}_c R-U^*\right\|_{2,\infty} + \left\|(\hat{U}_c)_{:, 1} \right\|_{\infty} \left\|R_{1, 2:r}\right\|_2 \notag\\
        &\lesssim \frac{c_{41}+c_{11} \kappa^{1.5}\sqrt{\mu r}}{\sqrt{\sigma_{\min}}} + \sqrt{\frac{\mu r}{n}} \left\|R_{1, 2:r}\right\|_2.\label{mem_est_proof_eq2}
    \end{align}
    The term $\left\|R_{1, 2:r}\right\|_2$ can be further controlled as
    \begin{align}
        \left\|R_{1, 2:r}\right\|_2 = \sqrt{\left\|R_{:, 2:r}\right\|_F^2 - \left\|R_{2:r, 2:r}\right\|_F^2} &\leq \sqrt{r-1-\left(\left\|\check{R}\right\|_F - \left\|\check{R} - R_{2:r, 2:r}\right\|_F\right)^2} \notag\\
        &\lesssim \sqrt{r}\sqrt{\left\|\check{R} - R_{2:r, 2:r}\right\|_F}\lesssim \frac{c_{11}r^{3/4}\sqrt{\sigma_{\max}n}}{\sigma_{\min}}. \label{mem_est_proof_eq3}
    \end{align}
    Plugging \eqref{mem_est_proof_eq3} in \eqref{mem_est_proof_eq2} we get
    \begin{align*}
        \left\|\check{U}_c R_{2:r, 2:r} - \check{U}^*\right\|_{2,\infty}\lesssim \frac{c_{41}+c_{11} \kappa^{1.5}\sqrt{\mu r}+c_{11}\sqrt{\mu\kappa}r^{5/4}}{\sqrt{\sigma_{\min}}}.
    \end{align*}
    Combing this with \eqref{mem_est_proof_eq1} and \eqref{mem_est_proof_eq10} we have
    \begin{align*}
        \left\|\check{U}_c \check{R} - \check{U}^*\right\|_{2,\infty} &\leq \left\|\check{U}_c R_{2:r, 2:r} - \check{U}^*\right\|_{2,\infty} + \left\|\check{U}_c(\check{R} - R_{2:r, 2:r})\right\|_{2,\infty} \\
        &\leq \left\|\check{U}_c R_{2:r, 2:r} - \check{U}^*\right\|_{2,\infty} + \left\|\check{U}_c\right\|_{2,\infty}\left\|\check{R} - R_{2:r, 2:r}\right\|  \\
        &\lesssim \left\|\check{U}_c R_{2:r, 2:r} - \check{U}^*\right\|_{2,\infty}+\left(\sqrt{\frac{\mu r}{n}}+\left\|\hat{U}_c R-U^*\right\|_{2,\infty}\right)\frac{c_{11}^2\sigma_{\max}n}{\sigma_{\min}^2} \\
        &\lesssim \frac{c_{41}+c_{11} \kappa^{1.5}\sqrt{\mu r}+c_{11}\sqrt{\mu\kappa}r^{5/4}}{\sqrt{\sigma_{\min}}}.
    \end{align*}

    Now we turn to control $\|(\hat{U}_c)_{:, 1} - (U^*)_{:, 1}\|_\infty$. Similarly, one can see that
    \begin{align*}
        \left|R_{1, 1} - 1\right|&\leq \left|R_{1, 1} - H_{1, 1}\right| +\left|H_{1, 1}-(\hat{U}_c)_{:, 1}^\top(U^*)_{:, 1}\right| + \left|(\hat{U}_c)_{:, 1}^\top(U^*)_{:, 1} - 1\right|   \\
        &\leq \left\|R - H\right\| + \left|(\hat{U}_c)_{:, 1}^\top(U^*)_{:, 1} - 1\right|  \lesssim \frac{c_{11}^2\sigma_{\max}n}{\sigma_{\min}^2}.
    \end{align*}
    On the other hand, $(\hat{U}_c)_{:, 1} - (U^*)_{:, 1}$ can be decomposed into
    \begin{align*}
        (\hat{U}_c)_{:, 1} - (U^*)_{:, 1} &= (\hat{U}_c)_{:, 1}R_{1, 1} - (U^*)_{:, 1} +(\hat{U}_c)_{:, 1} (1 - R_{1, 1}) \\
        &= (\hat{U}_c R - U^*)_{:, 1} - (\hat{U}_c)_{:, 2:r}R_{2:r, 1} + (\hat{U}_c)_{:, 1} (1 - R_{1, 1}).
    \end{align*}
    Since $\|R_{2:r, 1}\|_2^2 = r-1-\|R_{2:r, 2:r}\|_F^2 = \|R_{1, 2:r}\|_2^2$, the inequality \eqref{mem_est_proof_eq3} also applies to $\|R_{2:r, 1}\|_2$. Therefore, $\|(\hat{U}_c)_{:, 1} - (U^*)_{:, 1}\|_\infty$ can be controlled as 
    \begin{align*}
        \left\|(\hat{U}_c)_{:, 1} - (U^*)_{:, 1}\right\|_\infty&\leq \left\|(\hat{U}_c R - U^*)_{:, 1}\right\|_\infty + \left\|(\hat{U}_c)_{:, 2:r}R_{2:r, 1}\right\|_\infty + \left\|(\hat{U}_c)_{:, 1} (1 - R_{1, 1})\right\|_\infty  \\
        &\leq \left\|\hat{U}_c R - U^*\right\|_{2, \infty}+  \left\|(\hat{U}_c)_{:, 2:r}\right\|_{2,\infty}\left\|R_{2:r, 1}\right\|_2 + \left\|(\hat{U}_c)_{:, 1} \right\|_\infty |1 - R_{1, 1}| \\
        &\lesssim \frac{c_{41}+c_{11} \kappa^{1.5}\sqrt{\mu r}+c_{11}\sqrt{\mu\kappa}r^{5/4}}{\sqrt{\sigma_{\min}}}.
    \end{align*}
\end{proof}
As a direct corollary of Lemma \ref{lem:U_2inflem2}, we have the following result.
\begin{corollary}
$\forall i\in [n]$, it holds that
    \begin{align*}
        (\hat{U}_c)_{1, i}\asymp \frac{1}{\sqrt{n}}.
    \end{align*}
\end{corollary}
\begin{proof}
    \cite[Lemma C.3]{jin2023mixed} shows that $(U^*)_{1, i}\asymp 1/\sqrt{n}$. Combine this with Lemma \ref{lem:U_2inflem2}, as long as 
    \begin{align*}
        \frac{c_{41}+c_{11} \kappa^{1.5}\sqrt{\mu r}+c_{11}\sqrt{\mu\kappa}r^{5/4}}{\sqrt{\sigma_{\min}}} \ll \frac{1}{\sqrt{n}},
    \end{align*}
    we have $(\hat{U}_c)_{1, i}\asymp 1/\sqrt{n}$.
\end{proof}
Then we are ready to control the estimation error of the eigen ratio $\hat{r}_i$.
\begin{lemma}\label{lem:restimation}
    \begin{align*}
        \max_{1\leq i\leq n}\|\check{R}^\top \hat{r}_i - r_i^*\|_2\lesssim \sqrt{\frac{\mu rn}{\sigma_{\min}}}c_{41}+\frac{\mu r\sqrt{\kappa \sigma_{\max}n}}{\sigma_{\min}}c_{11} + \frac{c_{11}\mu r^{7/4}\sqrt{\sigma_{\max} n}}{\sigma_{\min}}.
    \end{align*}
\end{lemma}

\begin{proof}
    By definition we can write 
    \begin{align*}
        \check{R}^\top \hat{r}_i - r_i^* = \frac{(\check{U}_c\check{R})_{i, :}^\top}{(\hat{U}_c)_{i, 1}} - \frac{(\check{U}^*)_{i, :}^\top}{(U^*)_{i, 1}} = \frac{(\check{U}_c\check{R} - \check{U}^*)_{i, :}^\top}{(\hat{U}_c)_{i, 1}} + \frac{(U^*)_{i, 1} - (\hat{U}_c)_{i, 1}}{(\hat{U}_c)_{i, 1}(U^*)_{i, 1}}(\check{U}^*)_{i, :}^\top.
    \end{align*}
    Therefore, we have
    \begin{align*}
        \left\|\check{R}^\top \hat{r}_i - r_i^*\right\|_2 &\lesssim \frac{\|\check{U}_c \check{R}-\check{U}^*\|_{2,\infty}}{|(\hat{U}_c)_{i, 1}|} + \frac{\|(U^*)_{:, 1} - (\hat{U}_c)_{:, 1}\|_\infty}{|(\hat{U}_c)_{i, 1}(U^*)_{i, 1}|}\|\check{U}^*\|_{2,\infty}  \\
        &\lesssim \sqrt{n}\|\check{U}_c \check{R}-\check{U}^*\|_{2,\infty} +n \|(U^*)_{:, 1} - (\hat{U}_c)_{:, 1}\|_\infty\sqrt{\frac{\mu r}{n}}  \\
        &\lesssim \left(c_{41}+c_{11} \kappa^{1.5}\sqrt{\mu r}+c_{11}\sqrt{\mu\kappa}r^{5/4}\right)\sqrt{\frac{\mu r n}{\sigma_{\min}}}.
    \end{align*}
\end{proof}

By Lemma \ref{lem:restimation}, the eigen ratio $r_i^*$ can be estimated uniformly well in the sense that 
\begin{align}\label{prop1:pf_mem}
  \max_{1\leq i\leq n}\|\check{R}^\top \hat{r}_i - r_i^*\|_2\lesssim \left(c_{41}+c_{11} \kappa^{1.5}\sqrt{\mu r}+c_{11}\sqrt{\mu\kappa}r^{5/4}\right)\sqrt{\frac{\mu r n}{\sigma_{\min}}}.
\end{align}
Recall the definition of efficient vertex-hunting algorithms, we have
\begin{align}\label{prop2:pf_mem}
 \max_{1\leq \ell \leq r}\|\check{R}^\top\hat{v}_{\ell}-v_{\ell}^*\|_2\lesssim  \max_{1\leq i\leq n}\|\check{R} \hat{r}_i - r_i^*\|_2\lesssim \left(c_{41}+c_{11} \kappa^{1.5}\sqrt{\mu r}+c_{11}\sqrt{\mu\kappa}r^{5/4}\right)\sqrt{\frac{\mu r n}{\sigma_{\min}}}.
\end{align}
We then prove Theorem \ref{thm:mem_est} in the following.
\begin{proof}[Proof of Theorem \ref{thm:mem_est}]

Note that
\begin{align*}
\underbrace{
\begin{bmatrix}
v_1^* &  \dots & v_r^*\\
1 & \dots & 1
\end{bmatrix}}_{=:Q}w^{*}_i=\begin{bmatrix}
r^*_i\\
1 
\end{bmatrix},\quad
\underbrace{\begin{bmatrix}
\check{R}\hat{v}_1 &  \dots & \check{R}\hat{v}_r\\
1 & \dots & 1
\end{bmatrix}}_{=:\hat{Q}}\hat{w}_i=\begin{bmatrix}
\check{R}\hat{r}_i\\
1 
\end{bmatrix}.
\end{align*}
The following claim is from \cite[(C.26)]{jin2023mixed}.
\begin{claim}\label{claim1:pf_mem}
Under Assumption \ref{assumption:mem_est}, it holds that \begin{align*}
\|Q\|\lesssim \sqrt{r}\text{ and }
\|Q^{-1}\|\lesssim \sqrt{1/r}.
\end{align*}
\end{claim}

By Claim \ref{claim1:pf_mem}, we have $\sigma_{r}(Q)\gtrsim \sqrt{r}$.
By Weyl's inequality, we have
\begin{align*}
|\sigma_{r}(\hat{Q})-\sigma_{r}(Q)|
&\leq \|\hat{Q}-Q\|\leq \|\hat{Q}-Q\|_F\leq \sqrt{r}\max_{1\leq \ell \leq r}\|\check{R}\hat{v}_{\ell}-v_{\ell}^*\|\\
&\lesssim \sqrt{r}\left(c_{41}+c_{11} \kappa^{1.5}\sqrt{\mu r}+c_{11}\sqrt{\mu\kappa}r^{5/4}\right)\sqrt{\frac{\mu r n}{\sigma_{\min}}}\\
&\ll \sqrt{r}.
\end{align*}
Thus, it holds that
\begin{align}\label{ineq6:pf_mem}
\|\hat{Q}^{-1}\|=\frac{1}{\sigma_{r}(\hat{Q})}\leq \frac{2}{\sigma_{r}({Q})}\lesssim \sqrt{1/r}.
\end{align}
Note that
\begin{align*}
\hat{w}_i-w^{*}_i
&=\hat{Q}^{-1}\begin{bmatrix}
\check{R}\hat{r}_i\\
1 
\end{bmatrix}-Q^{-1}\begin{bmatrix}
r^*_i\\
1 
\end{bmatrix}\\
&=\hat{Q}^{-1}\begin{bmatrix}
\check{R}\hat{r}_i - r^*_i \\
0
\end{bmatrix}-\hat{Q}^{-1}(\hat{Q}-Q)Q^{-1}\begin{bmatrix}
r^*_i\\
1 
\end{bmatrix}\\
&=\hat{Q}^{-1}\begin{bmatrix}
\check{R}\hat{r}_i - r^*_i \\
0
\end{bmatrix}-\hat{Q}^{-1}(\hat{Q}-Q)w^{*}_i.
\end{align*}
Consequently, we have
\begin{align*}
\|\hat{w}_i-w^{*}_i\|_2
&\leq \|\hat{Q}^{-1}\|\left(\|\check{R}\hat{u}_i-u^{*}_i\|_2+\|(\hat{Q}-Q)w^{*}_i\|_2\right)\\
&\lesssim \sqrt{1/r}\left(\|\check{R}\hat{r}_i-r^{*}_i\|_2+\|(\hat{Q}-Q)w^{*}_i\|_2\right),
\end{align*}
where the last inequality follows from \eqref{ineq6:pf_mem}.
Note that
\begin{align*}
\|(\hat{Q}-Q)w^{*}_i\|_2
&=\left\|\sum^r_{\ell=1}w^{*}_i(\ell)\begin{bmatrix}
\check{R}\hat{v}_\ell - v_\ell^* \\
0
\end{bmatrix}\right\|_2\\
&\leq \sum^r_{\ell=1}w^{*}_i(\ell)\|\check{R}\hat{v}_{\ell}-v_{\ell}^*\|_2\\
&\leq \max_{1\leq\ell \leq r}\|\check{R}\hat{v}_{\ell}-v_{\ell}^*\|_2,
\end{align*}
where the last inequality follows from the fact that $\sum^r_{\ell=1}w^{*}_i(\ell)=1$. Thus, we obtain
\begin{align*}
\|\hat{w}_i-w^{*}_i\|_2\lesssim    \sqrt{1/r}\left(\|\check{R}\hat{r}_i-r^{*}_i\|_2+\max_{1\leq\ell \leq r}\|\check{R}\hat{v}_{\ell}-v_{\ell}^*\|_2\right).
\end{align*}
Further by \eqref{prop1:pf_mem} and \eqref{prop2:pf_mem}, we obtain
\begin{align}
 \max_{1\leq i\leq n}\|\hat{w}_i-w^{*}_i\|_2
 &\lesssim   \left(c_{41}+c_{11} \kappa^{1.5}\sqrt{\mu r}+c_{11}\sqrt{\mu\kappa}r^{5/4}\right)\sqrt{\frac{\mu n}{\sigma_{\min}}}. \label{piestimationeq7}
\end{align}

Next, let's control $|(\hat{b}_1(l))^{-1} - (b^*_1(l))^{-1}|$. By definition we have
\begin{align}
    \left|\frac{1}{\hat{b}_1(l)} - \frac{1}{b^*_1(l)}\right| &= \left|\sqrt{\left|\hat{\lambda}_1+\hat{v}_{\ell}^\top\diag(\hat{\lambda}_2,\ldots, \hat{\lambda}_r)\hat{v}_{\ell}\right|} - \sqrt{\left|\lambda_1+v_{\ell}^{*\top}\diag(\lambda_2,\ldots, \lambda_r)v^*_{\ell}\right|}\right| \notag \\
    &=\left|\frac{\left|\hat{\lambda}_1+\hat{v}_{\ell}^\top\diag(\hat{\lambda}_2,\ldots, \hat{\lambda}_r)\hat{v}_{\ell}\right| -\left| \lambda_1+v_{\ell}^{*\top}\diag(\lambda_2,\ldots, \lambda_r)v^*_{\ell}\right|}{\sqrt{\left|\hat{\lambda}_1+\hat{v}_{\ell}^\top\diag(\hat{\lambda}_2,\ldots, \hat{\lambda}_r)\hat{v}_{\ell}\right|} +\sqrt{\left|\lambda_1+v_{\ell}^{*\top}\diag(\lambda_2,\ldots, \lambda_r)v^*_{\ell}\right|}}\right| \notag \\
    &\leq \frac{\left|\hat{\lambda}_1+\hat{v}_{\ell}^\top\diag(\hat{\lambda}_2,\ldots, \hat{\lambda}_r)\hat{v}_{\ell} -\lambda_1 - v_{\ell}^{*\top}\diag(\lambda_2,\ldots, \lambda_r)v^*_{\ell}\right|}{\sqrt{\lambda_1+v_{\ell}^{*\top}\diag(\lambda_2,\ldots, \lambda_r)v^*_{\ell}}}  \notag\\
    &\leq b_1^*(\ell)\left(\left|\hat{\lambda}_1-\lambda_1\right|+\left|\hat{v}_{\ell}^\top\diag(\hat{\lambda}_2,\ldots, \hat{\lambda}_r)\hat{v}_{\ell} -v_{\ell}^{*\top}\diag(\lambda_2,\ldots, \lambda_r)v^*_{\ell}\right|\right).\label{piestimationeq6}
\end{align}
By \cite[Eq. (C.22)]{jin2023mixed} we know that $b_1^*(\ell)\asymp (\sqrt{n}\bar{\theta}_2^*)^{-1}$. On the other hand, by Weyl's inequality we know that
\begin{align}
    \left|\hat{\lambda}_1-\lambda_1\right|\leq \left\|\hat{\Gamma}_c - \Gamma^*\right\|\lesssim c_{11}\sqrt{\sigma_{\max}n}.\label{piestimationeq5}
\end{align}
It remains to control $\left|\hat{v}_{\ell}^\top\diag(\hat{\lambda}_2,\ldots, \hat{\lambda}_r)\hat{v}_{\ell} -v_{\ell}^{*\top}\diag(\lambda_2,\ldots, \lambda_r)v^*_{\ell}\right|$. Define 
\begin{align*}
    \overline{\Lambda}' = \diag(\hat{\lambda}_2,\ldots, \hat{\lambda}_r), \overline{\Lambda} = \diag({\lambda}_2,\ldots, {\lambda}_r),
\end{align*}
we can write
\begin{align}
    &\left|\hat{v}_{\ell}^\top\diag(\hat{\lambda}_2,\ldots, \hat{\lambda}_r)\hat{v}_{\ell} -v_{\ell}^{*\top}\diag(\lambda_2,\ldots, \lambda_r)v^*_{\ell}\right| \notag \\
    =& \left|\hat{v}_{\ell}^\top\overline{\Lambda}'\hat{v}_{\ell} -v_{\ell}^{*\top}\overline{\Lambda}v^*_{\ell}\right| = \left|(\check{R}^\top\hat{v}_{\ell})^\top\check{R}^\top\overline{\Lambda}'\check{R}\check{R}^\top\hat{v}_{\ell} -v_{\ell}^{*\top}\overline{\Lambda}v^*_{\ell}\right| \notag \\
    \leq&\left|(\check{R}^\top\hat{v}_{\ell})^\top\check{R}^\top\overline{\Lambda}'\check{R}\check{R}^\top\hat{v}_{\ell} -v_{\ell}^{*\top}\check{R}^\top\overline{\Lambda}'\check{R}\check{R}^\top\hat{v}_{\ell} \right| + \left|v_{\ell}^{*\top}\check{R}^\top\overline{\Lambda}'\check{R}\check{R}^\top\hat{v}_{\ell} - v_{\ell}^{*\top}\overline{\Lambda}\check{R}^\top\hat{v}_{\ell} \right| \notag \\
    &+\left| v_{\ell}^{*\top}\overline{\Lambda}\check{R}^\top\hat{v}_{\ell} -v_{\ell}^{*\top}\overline{\Lambda}v_{\ell}^*\right| \notag \\
    \leq & \left\|\check{R}^\top\hat{v}_{\ell} - v_{\ell}^*\right\|_2 \left\|\overline{\Lambda}'\right\|\left\|\hat{v}_{\ell}\right\|_2  + \left\|\hat{v}_{\ell}\right\|_2 \left\|v_{\ell}^*\right\|_2\left\|\check{R}^\top\overline{\Lambda}'\check{R} - \overline{\Lambda}\right\| + \left\|v_{\ell}^*\right\|_2\left\|\overline{\Lambda}\right\|\left\|\check{R}^\top\hat{v}_{\ell} - v_{\ell}^*\right\|_2.\label{piestimationeq4}
\end{align}
Furthermore, we write
\begin{align}
    \check{R}^\top\overline{\Lambda}'\check{R} - \overline{\Lambda} = \check{R}^\top\overline{\Lambda}'\check{R} - \check{H}^\top\overline{\Lambda}'\check{H} + \check{H}^\top\overline{\Lambda}'\check{H} -\overline{\Lambda}.\label{piestimationeq2}
\end{align}
The first term on the RHS can be controlled as
\begin{align}
    \left\|\check{R}^\top\overline{\Lambda}'\check{R} - \check{H}^\top\overline{\Lambda}'\check{H}\right\| &= \left\|\check{R}^\top\overline{\Lambda}'\check{R} - \check{R}^\top\overline{\Lambda}'\check{H}+\check{R}^\top\overline{\Lambda}'\check{H} - \check{H}^\top\overline{\Lambda}'\check{H}\right\| \notag\\
    &\leq \left\|\check{R}^\top\overline{\Lambda}'\check{R} - \check{R}^\top\overline{\Lambda}'\check{H}\right\|  + \left\|\check{R}^\top\overline{\Lambda}'\check{H} - \check{H}^\top\overline{\Lambda}'\check{H}\right\| \leq 2\left\|\check{R} - \check{H}\right\| \left\|\overline{\Lambda}'\right\| \notag\\
    &\lesssim \left(\frac{\|\hat{\Gamma}_c-\Gamma^*\|}{\sigma_{\min}}\right)^2 \sigma_2(\hat{\Gamma}_c).\label{piestimationeq3}
\end{align}
The second term can be controlled as
\begin{align}
    \left\|\check{H}^\top\overline{\Lambda}'\check{H} -\overline{\Lambda}\right\| &= \left\|\check{U}^{*\top}\check{U}_c\overline{\Lambda}'\check{U}_c^\top \check{U}^* -  \overline{\Lambda}\right\| = \left\|\check{U}^{*\top}(\check{U}_c\overline{\Lambda}'\check{U}_c^\top -  \check{U}^*\overline{\Lambda}\check{U}^{*\top})\check{U}^*\right\| \notag \\
    &\leq \left\|\check{U}_c\overline{\Lambda}'\check{U}_c^\top -  \check{U}^*\overline{\Lambda}\check{U}^{*\top}\right\| \notag \\
    &= \left\|\hat{\Gamma}_c - \hat{\lambda}_1(\hat{U}_c)_{:, 1}(\hat{U}_c)_{:, 1}^\top -\sum_{k=r+1}^n\hat{\lambda}_k(\hat{U}_c)_{:, k}(\hat{U}_c)_{:, k}^\top -  \left(\Gamma^* - \lambda_1(U^*)_{:, 1}(U^*)_{:, 1}^\top\right)\right\| \notag \\
    &\leq \left\|\hat{\Gamma}_c - \Gamma^*\right\| + \left\|\sum_{k=r+1}^n\hat{\lambda}_k(\hat{U}_c)_{:, k}(\hat{U}_c)_{:, k}^\top\right\| + \left\|\hat{\lambda}_1(\hat{U}_c)_{:, 1}(\hat{U}_c)_{:, 1}^\top - \lambda_1(U^*)_{:, 1}(U^*)_{:, 1}^\top\right\| \notag \\
    &=\left\|\hat{\Gamma}_c - \Gamma^*\right\| + \sigma_{r+1}(\hat{\Gamma}_c) + \left\|\hat{\lambda}_1(\hat{U}_c)_{:, 1}(\hat{U}_c)_{:, 1}^\top - \lambda_1(U^*)_{:, 1}(U^*)_{:, 1}^\top\right\| \notag \\
    &\leq 2\left\|\hat{\Gamma}_c - \Gamma^*\right\|+ \left\|\hat{\lambda}_1(\hat{U}_c)_{:, 1}(\hat{U}_c)_{:, 1}^\top - \lambda_1(U^*)_{:, 1}(U^*)_{:, 1}^\top\right\|. \label{piestimationeq1}
\end{align}
The last term can be further controlled by
\begin{align*}
    \left\|\hat{\lambda}_1(\hat{U}_c)_{:, 1}(\hat{U}_c)_{:, 1}^\top - \lambda_1(U^*)_{:, 1}(U^*)_{:, 1}^\top\right\|&\leq \left\|(\hat{\lambda}_1-\lambda_1)(\hat{U}_c)_{:, 1}(\hat{U}_c)_{:, 1}^\top \right\| + \left\|\lambda_1\left((\hat{U}_c)_{:, 1}(\hat{U}_c)_{:, 1}^\top -(U^*)_{:, 1}(U^*)_{:, 1}^\top\right)\right\| \\
    &\leq |\hat{\lambda}_1-\lambda_1| + \lambda_1\left\|(\hat{U}_c)_{:, 1}(\hat{U}_c)_{:, 1}^\top -(U^*)_{:, 1}(U^*)_{:, 1}^\top\right\| \\
    &= |\hat{\lambda}_1-\lambda_1| + \lambda_1\sqrt{1-\left((\hat{U}_c)_{:, 1}^\top(U^*)_{:, 1}\right)^2} \\
    &\leq \left\|\hat{\Gamma}_c - \Gamma^*\right\| + \lambda_1\sqrt{2 - 2(\hat{U}_c)_{:, 1}^\top(U^*)_{:, 1}} \\
    &= \left\|\hat{\Gamma}_c - \Gamma^*\right\| + \lambda_1\sqrt{\|(\hat{U}_c)_{:, 1}\|_2+\|(U^*)_{:, 1}\|_2 - 2(\hat{U}_c)_{:, 1}^\top(U^*)_{:, 1}} \\
    &= \left\|\hat{\Gamma}_c - \Gamma^*\right\| + \lambda_1\left\|(\hat{U}_c)_{:, 1} - (U^*)_{:, 1}\right\|_2 \\
    &\lesssim \left\|\hat{\Gamma}_c - \Gamma^*\right\| + \lambda_1 \frac{\left\|\hat{\Gamma}_c - \Gamma^*\right\|}{\lambda_1-\sigma_2(\Gamma^*)} \lesssim \left\|\hat{\Gamma}_c - \Gamma^*\right\|.
\end{align*}
Combine this with \eqref{piestimationeq1} we know that 
\begin{align*}
    \left\|\check{H}^\top\overline{\Lambda}'\check{H} -\overline{\Lambda}\right\|\lesssim \left\|\hat{\Gamma}_c - \Gamma^*\right\|.
\end{align*}
Plugging this and \eqref{piestimationeq3} in \eqref{piestimationeq2} we get
\begin{align*}
    \left\|\check{R}^\top\overline{\Lambda}'\check{R} - \overline{\Lambda}\right\|\lesssim \left(\frac{\|\hat{\Gamma}_c-\Gamma^*\|}{\sigma_{\min}}\right)^2 \sigma_2(\hat{\Gamma}_c) + \left\|\hat{\Gamma}_c - \Gamma^*\right\|\lesssim \left\|\hat{\Gamma}_c - \Gamma^*\right\|
\end{align*}
as long as $\sigma_{\min}\gtrsim \kappa \left\|\hat{\Gamma}_c - \Gamma^*\right\|$.

Before we go back to \eqref{piestimationeq4}, we need to control $\left\|v_{\ell}^*\right\|_2$ and $\left\|\hat{v}_{\ell}\right\|_2$. By Assumption \ref{assumption:incoherent} and \cite[Lemma C.3]{jin2023mixed}, it can be controlled as 
\begin{align*}
    \left\|v_{\ell}^*\right\|_2 \leq \frac{\sqrt{\mu r/n}}{1/\sqrt{n}} = \sqrt{\mu r}.
\end{align*}
And, as long as 
\begin{align*}
    \left(c_{41}+c_{11} \kappa^{1.5}\sqrt{\mu r}+c_{11}\sqrt{\mu\kappa}r^{5/4}\right)\sqrt{\frac{n}{\sigma_{\min}}} \ll 1,
\end{align*}
we also have $\left\|\hat{v}_{\ell}\right\|_2\lesssim \sqrt{\mu r}$. As a result, from \eqref{piestimationeq4} we know that
\begin{align*}
    &\left|\hat{v}_{\ell}^\top\diag(\hat{\lambda}_2,\ldots, \hat{\lambda}_r)\hat{v}_{\ell} -v_{\ell}^{*\top}\diag(\lambda_2,\ldots, \lambda_r)v^*_{\ell}\right| \\
    \lesssim & \sigma_{\max}\sqrt{\mu r}\left(c_{41}+c_{11} \kappa^{1.5}\sqrt{\mu r}+c_{11}\sqrt{\mu\kappa}r^{5/4}\right)\sqrt{\frac{\mu r n}{\sigma_{\min}}} +  \mu r\left\|\hat{\Gamma}_c - \Gamma^*\right\| \\
    \lesssim & \mu r\left(c_{41}+c_{11} \kappa^{1.5}\sqrt{\mu r}+c_{11}\sqrt{\mu\kappa}r^{5/4}\right)\sqrt{\kappa n \sigma_{\max}}.
\end{align*}
Combine this with \eqref{piestimationeq5} and plug them back in \eqref{piestimationeq6} we get
\begin{align*}
     \left|\frac{1}{\hat{b}_1(l)} - \frac{1}{b^*_1(l)}\right|\lesssim b_1^*(\ell) \mu r\left(c_{41}+c_{11} \kappa^{1.5}\sqrt{\mu r}+c_{11}\sqrt{\mu\kappa}r^{5/4}\right)\sqrt{\kappa n \sigma_{\max}}.
\end{align*}
By Lemma \ref{modifiedLemmaC.2} we know that
\begin{align*}
    \sigma_{\max}\lesssim (\sqrt{n}\bar{\theta}_2^*)^2 \asymp \frac{1}{b_1^{*2}(\ell)}.
\end{align*}
Therefore, we have
\begin{align*}
    \left|\frac{1}{\hat{b}_1(l)} - \frac{1}{b^*_1(l)}\right|\lesssim  \mu r\left(c_{41}+c_{11} \kappa^{1.5}\sqrt{\mu r}+c_{11}\sqrt{\mu\kappa}r^{5/4}\right)\sqrt{\kappa n}.
\end{align*}

Combine this with \eqref{piestimationeq7}, we are able to control $\tilde{\pi}_i(\ell)-\tilde{\pi}^*_i(\ell)$, where $\tilde{\pi}^*_i(\ell)$ is defined as
\begin{align*}
    \tilde{\pi}^*_i(\ell):=\frac{w^*_i(\ell)}{b^*_1(\ell)},\quad \forall i \in [n], \ell \in [r].
\end{align*}
Since $\tilde{\pi}^*_i(\ell)\geq 0$, one can see that
\begin{align}
    \left\|\tilde{\pi}_i-\tilde{\pi}^*_i\right\|_1 =& \sum_{\ell=1}^r\left|\max\left\{\frac{\hat{w}_i(\ell)}{\hat{b}_1(\ell)}, 0\right\} - \frac{w^*_i(\ell)}{b^*_1(\ell)}\right| \leq \sum_{\ell=1}^r\left|\frac{\hat{w}_i(\ell)}{\hat{b}_1(\ell)} - \frac{w^*_i(\ell)}{b^*_1(\ell)}\right|  \notag \\
    \leq& \frac{1}{\min_{\ell\in [r]}\hat{b}_1(\ell)}\left\|\hat{w}_i - w_i^*\right\|_1 + \left\|w^*_i\right\|_1\max_{\ell\in [r]}\left|\frac{1}{\hat{b}_1(l)} - \frac{1}{b^*_1(l)}\right| \notag \\
    \lesssim& \sqrt{n}\bar{\theta}_2^*\left(c_{41}+c_{11} \kappa^{1.5}\sqrt{\mu r}+c_{11}\sqrt{\mu\kappa}r^{5/4}\right)\sqrt{\frac{\mu r n}{\sigma_{\min}}}  \notag \\
    &+\mu r\left(c_{41}+c_{11} \kappa^{1.5}\sqrt{\mu r}+c_{11}\sqrt{\mu\kappa}r^{5/4}\right)\sqrt{\kappa n}.\label{piestimationeq8}
\end{align}
Note that
\begin{align}
    \left\|\hat{\pi}_i - \pi_i^*\right\|_1 = \left\|\frac{\tilde{\pi}_i}{\left\|\tilde{\pi}_i\right\|_1} - \frac{\tilde{\pi}^*_i}{\left\|\tilde{\pi}^*_i\right\|_1}\right\|_1 &\leq \left\|\tilde{\pi}_i\right\|_1\left|\frac{1}{\left\|\tilde{\pi}_i\right\|_1} - \frac{1}{\left\|\tilde{\pi}^*_i\right\|_1}\right| + \frac{\left\|\tilde{\pi}_i-\tilde{\pi}^*_i\right\|_1}{\left\|\tilde{\pi}^*_i\right\|_1} \notag\\
    &\leq \frac{\left|\left\|\tilde{\pi}_i\right\|_1 - \left\|\tilde{\pi}^*_i\right\|_1\right|}{\left\|\tilde{\pi}^*_i\right\|_1} + \frac{\left\|\tilde{\pi}_i-\tilde{\pi}^*_i\right\|_1}{\left\|\tilde{\pi}^*_i\right\|_1} \leq 2\frac{\left\|\tilde{\pi}_i-\tilde{\pi}^*_i\right\|_1}{\left\|\tilde{\pi}^*_i\right\|_1}.\label{piestimationeq9}
\end{align}
Since 
\begin{align*}
    \left\|\tilde{\pi}^*_i\right\|_1 = \sum_{\ell=1}^r \frac{w^*_i(\ell)}{b^*_1(\ell)}\asymp \sqrt{n}\bar{\theta}_2^*\sum_{\ell=1}^r w^*_i(\ell) = \sqrt{n}\bar{\theta}_2^*,
\end{align*}
by \eqref{piestimationeq8} and \eqref{piestimationeq9}, we know that
\begin{align*}
    \left\|\hat{\pi}_i - \pi_i^*\right\|_1\lesssim& \left(c_{41}+c_{11} \kappa^{1.5}\sqrt{\mu r}+c_{11}\sqrt{\mu\kappa}r^{5/4}\right)\left(\sqrt{\frac{\mu r n}{\sigma_{\min}}}+\frac{\sqrt{\kappa } \mu r}{\bar{\theta}_2^*}\right).
\end{align*}
Again by Lemma \ref{modifiedLemmaC.2} we know that
\begin{align*}
    \sigma_{\min} \asymp \beta_n r^{-1}(\sqrt{n}\bar{\theta}_2^*)^2.
\end{align*}
Therefore, we know that 
\begin{align*}
    \left\|\hat{\pi}_i - \pi_i^*\right\|_1\lesssim& \left(c_{41}+c_{11} \kappa^{1.5}\sqrt{\mu r}+c_{11}\sqrt{\mu\kappa}r^{5/4}\right)\left(\sqrt{\frac{\mu }{\beta_n}}+\sqrt{\kappa } \mu \right)\frac{r}{\bar{\theta}_2^*}.
\end{align*}
\end{proof}

%% file: technical_lemmas.tex
\section{Technical lemmas}\label{sec:echnical_lemmas}

\begin{lemma}\label{ar:lem1}
For matrix $A\in \mathbb{R}^{n_1\times m}, B\in \mathbb{R}^{n_2\times m}$, we have
\begin{align*}
\max\{\|A\|,\|B\|\}\leq \left\|
\begin{bmatrix}
A\\
B
\end{bmatrix}
\right\|
\leq \|A\|+\|B\|.
\end{align*}
\end{lemma}

\begin{proof}[Proof of Lemma \ref{ar:lem1}]
Given $A\in \mathbb{R}^{n_1\times m}, B\in \mathbb{R}^{n_2\times m}$. We have
\begin{align*}
\left\|
\begin{bmatrix}
A\\
B
\end{bmatrix}
\right\| &= \sup_{v\in \mathbb{R}^{n_1+n_2},u\in \mathbb{R}^{m}, \|u\|_2,\|v\|_2\leq 1} v^\top \begin{bmatrix}
A\\
B
\end{bmatrix} u  \\
&= \sup_{v_1\in \mathbb{R}^{n_1},v_2\in \mathbb{R}^{n_2}, u\in \mathbb{R}^{m}, \|u\|_2 \leq 1, \sqrt{\|v_1\|_2^2+\|v_2\|_2^2} \leq 1} v_1^\top A u +v_2^\top B u \\
&\leq \sup_{v_1\in \mathbb{R}^{n_1},v_2\in \mathbb{R}^{n_2}, u\in \mathbb{R}^{m}, \|u\|_2 ,\|v_1\|_2, \|v_2\|_2\leq 1} v_1^\top A u +v_2^\top B u \\
&\leq \sup_{v_1\in \mathbb{R}^{n_1}, u\in \mathbb{R}^{m}, \|u\|_2 \|v_1\|_2\leq 1} v_1^\top A u+\sup_{v_2\in \mathbb{R}^{n_2}, u\in \mathbb{R}^{m}, \|u\|_2 ,\|v_2\|_2\leq 1} v_2^\top B u \\
&= \|A\|_2 + \|B\|_2.
\end{align*}

Given $A\in \mathbb{R}^{n_1\times m}, B\in \mathbb{R}^{n_2\times m}$. We have
\begin{align*}
\left\|
\begin{bmatrix}
A\\
B
\end{bmatrix}
\right\| &= \sup_{v\in \mathbb{R}^{n_1+n_2},u\in \mathbb{R}^{m}, \|u\|_2,\|v\|_2\leq 1} v^\top \begin{bmatrix}
A\\
B
\end{bmatrix} u  \\
&= \sup_{v_1\in \mathbb{R}^{n_1},v_2\in \mathbb{R}^{n_2}, u\in \mathbb{R}^{m}, \|u\|_2 \leq 1, \sqrt{\|v_1\|_2^2+\|v_2\|_2^2} \leq 1} v_1^\top A u +v_2^\top B u \\
&\geq \sup_{v_1\in \mathbb{R}^{n_1},v_2\in \mathbb{R}^{n_2}, u\in \mathbb{R}^{m}, \|u\|_2 \leq 1, \sqrt{\|v_1\|_2^2+\|v_2\|_2^2} \leq 1, v_2 = 0} v_1^\top A u +v_2^\top B u \\
&= \sup_{v_1\in \mathbb{R}^{n_1}, u\in \mathbb{R}^{m}, \|u\|_2 \leq 1, \|v_1\|_2 <= 1} v_1^\top A u = \|A\|_2.
\end{align*}

\end{proof}

\begin{lemma}\label{ar:lem3}
Suppose \( F_1, F_2, F_0 \in \mathbb{R}^{2n \times r} \) are three matrices such that
\[
\|F_1 - F_0\| \|F_0\| \leq \sigma_r^2(F_0)/2 \quad \text{and} \quad \|F_1 - F_2\| \|F_0\| \leq \sigma_r^2(F_0)/4,
\]
where \(\sigma_i(A)\) stands for the \(i\)-th largest singular value of \(A\). Denote
\[
\mathbf{R}_1 \triangleq \arg \min_{\mathbf{R} \in \cO^{r \times r}} \|F_1 \mathbf{R} - F_0\|_F \quad \text{and} \quad \mathbf{R}_2 \triangleq \arg \min_{\mathbf{R} \in \cO^{r \times r}} \|F_2 \mathbf{R} - F_0\|_F.
\]
Then the following two inequalities hold:
\[
\|F_1 \mathbf{R}_1 - F_2 \mathbf{R}_2\| \leq \frac{5 \sigma_1^2(F_0)}{\sigma_r^2(F_0)} \|F_1 - F_2\| \quad \text{and} \quad \|F_1 \mathbf{R}_1 - F_2 \mathbf{R}_2\|_F \leq \frac{5 \sigma_1^2(F_0)}{\sigma_r^2(F_0)} \|F_1 - F_2\|_F.
\]
\end{lemma}
\begin{proof}[Proof of Lemma \ref{ar:lem3}]
See Lemma 37 in \cite{ma2018implicit}.
\end{proof}

\begin{lemma}\label{ar:lem2}
Suppose Lemma \ref{lem:ncv1}-Lemma \ref{lem:ncv5} hold for the $t$-th iteration. We then have
\begin{align*}
\left\|F^{t,(m)}R^{t,(m)}-F^{t}R^t\right\|_F\leq 5\kappa  \left\|F^{t,(m)}O^{t,(m)}-F^{t}R^t\right\|_F.    
\end{align*}
\end{lemma}
\begin{proof}[Proof of Lemma \ref{ar:lem2}]
By Lemma \ref{lem:ncv1}, we have
\begin{align*}
\|F^tR^t-F^*\|\|F^*\|\leq \sigma_r^2(F^*)/2.
\end{align*}
Note that
\begin{align*}
  \|F^{t,(m)}R^{t,(m)}-F^{t}R^t\|_F
  &\leq \|F^{t,(m)}R^{t,(m)}-F^*\|_F+\|F^*-F^{t}R^t\|_F\\
  &\leq  \|F^{t,(m)}O^{t,(m)}-F^*\|_F+\|F^*-F^{t}R^t\|_F \tag{by the definition of $R^{t,(m)}$}\\
  &\leq   \|F^{t,(m)}O^{t,(m)}-F^{t}R^t\|_F+2\|F^*-F^{t}R^t\|_F\\
  &\leq c_{21}+2c_{11}\sqrt{n}.
\end{align*}
Thus, it holds that
\begin{align*}
\|F^{t,(m)}R^{t,(m)}-F^{t}R^t\|\|F^*\|
\leq \|F^{t,(m)}R^{t,(m)}-F^{t}R^t\|_F\|F^*\|
\leq \sigma_r^2(F^*)/4.
\end{align*}
Then by Lemma \ref{ar:lem3} with $F_0=F^*, F_1=F^tR^t, F_2=F^{t,(m)}O^{t,(m)}$, we have
\begin{align*}
\left\|F^{t,(m)}R^{t,(m)}-F^{t}R^t\right\|_F
&\leq  
\frac{5 \sigma_1^2(F^*)}{\sigma_r^2(F^*)} 
\left\|F^{t,(m)}O^{t,(m)}-F^{t}R^t\right\|_F\\
&=\frac{5 \sigma_{\max}}{\sigma_{\min}} 
\left\|F^{t,(m)}O^{t,(m)}-F^{t}R^t\right\|_F.
\end{align*}
\end{proof}

\begin{lemma}\label{ar:lem4}
Suppose Lemma \ref{lem:ncv1}-Lemma \ref{lem:ncv5} hold for the $t$-th iteration. We then have
\begin{align*}
\sigma_{\min}/2\leq \sigma_{\min}\left(\left(Y^{t,(m)}R^{t,(m)}\right)^TY^{t,(m)}R^{t,(m)}\right)\leq \sigma_{\max}\left(\left(Y^{t,(m)}R^{t,(m)}\right)^TY^{t,(m)}R^{t,(m)}\right)\leq 2\sigma_{\max}.
\end{align*}
\end{lemma}
\begin{proof}[Proof of Lemma \ref{ar:lem4}]
By Weyl's inequality, we have
\begin{align*}
&\left|\sigma_{\min}\left(\left(Y^{t,(m)}R^{t,(m)}\right)^TY^{t,(m)}R^{t,(m)}\right)-\sigma_{\min}\right|\\
    &=\left|\sigma_{\min}\left(\left(Y^{t,(m)}R^{t,(m)}\right)^TY^{t,(m)}R^{t,(m)}\right)-\sigma_{\min}\left(Y^{*T}Y^*\right)\right|\\
    &\leq \left\|\left(Y^{t,(m)}R^{t,(m)}\right)^TY^{t,(m)}R^{t,(m)}-Y^{*T}Y^*\right\|\\
    &\leq \left\|Y^{t,(m)}R^{t,(m)}-Y^*\right\|\left(\left\|Y^{t,(m)}R^{t,(m)}\right\|+\left\|Y^*\right\|\right)\\
    &\leq  \left\|Y^{t,(m)}R^{t,(m)}-Y^*\right\|\left(\left\|Y^{t,(m)}R^{t,(m)}-Y^*\right\|+2\left\|Y^*\right\|\right).
\end{align*}
Note that
\begin{align*}
\|Y^{t,(m)}R^{t,(m)}-Y^*\|
&\leq \|F^{t,(m)}R^{t,(m)}-F^*\|\\
&\leq \|F^{t,(m)}R^{t,(m)}-F^tR^t\|+\|F^tR^t-F^*\|\\
&\leq 5\kappa\|F^{t,(m)}O^{t,(m)}-F^tR^t\|_F+\|F^tR^t-F^*\|\tag{by Lemma \ref{ar:lem2}}\\
&\lesssim c_{11}\sqrt{n}
\end{align*}
and $\|Y^*\|=\sqrt{\sigma_{\max}}$.
We then have
\begin{align*}
 \left|\sigma_{\min}\left(\left(Y^{t,(m)}R^{t,(m)}\right)^TY^{t,(m)}R^{t,(m)}\right)-\sigma_{\min}\right|\lesssim c_{11}\sqrt{n\sigma_{\max}}\leq    \sigma_{\min}/2,
\end{align*}
which implies 
\begin{align*}
\sigma_{\min}\left(\left(Y^{t,(m)}R^{t,(m)}\right)^TY^{t,(m)}R^{t,(m)}\right)\geq  \sigma_{\min}/2.   
\end{align*}
Similarly, we can show that
\begin{align*}
\sigma_{\max}\left(\left(Y^{t,(m)}R^{t,(m)}\right)^TY^{t,(m)}R^{t,(m)}\right)\leq 2\sigma_{\max}.
\end{align*}
\end{proof}

\begin{lemma}\label{ar:lem5}
Let \( S \in \mathbb{R}^{r \times r} \) be a nonsingular matrix. Then for any matrix \( K \in \mathbb{R}^{r \times r} \) with \(\|K\| \leq \sigma_{\min}(S)\), one has
\[
\| \text{sgn}(S + K) - \text{sgn}(S) \| \leq \frac{2}{\sigma_{r-1}(S) + \sigma_r(S)} \|K\|,
\]
where \(\text{sgn}(\cdot)\) denotes the matrix sign function, i.e. \(\text{sgn}(A) = UV^\top\) for a matrix \( A \) with SVD \( U \Sigma V^\top \).
\end{lemma}
\begin{proof}
See Lemma 36 in \cite{ma2018implicit}. 
\end{proof}

\begin{lemma}\label{ar:lem6}
Let $U \Sigma V^\top$ be the SVD of a rank-$r$ matrix $XY^\top$ with $X, Y \in \mathbb{R}^{n \times r}$. Then there exists an invertible matrix ${Q} \in \mathbb{R}^{r \times r}$ such that $X = U \Sigma^{1/2} {Q}$ and $Y = V \Sigma^{1/2} {Q}^{-T}$. In addition, one has
\begin{align*}
    \| \Sigma_{Q} - \Sigma_{{Q}}^{-1} \|_F &\leq \frac{1}{\sigma_{\min} (\Sigma)} \| X^\top X - Y^\top Y \|_F, \tag{140}
\end{align*}
where $U_Q \Sigma_Q V_Q^\top$ is the SVD of ${Q}$. In particular, if $X$ and $Y$ have balanced scale, i.e., $X^\top X - Y^\top Y = 0$, then ${Q}$ must be a rotation matrix.
\end{lemma}
\begin{proof}
See Lemma 20 in \cite{ma2018implicit}. 
\end{proof}